\documentclass[11pt]{preprint}
\usepackage[english]{babel}
\usepackage{amssymb}
\usepackage{amsmath}
\usepackage{hyperref}
\usepackage{breakurl}
\usepackage{pifont}
\usepackage{mhenvs}
\usepackage{mhequ}
\usepackage{mhsymb}
\usepackage{booktabs}
\usepackage{float}
\usepackage{tikz}
\usetikzlibrary{external}
\usepackage{graphicx}
\graphicspath{{figures/}}
\usepackage{mathrsfs}
\usepackage{array}
\usepackage{times}
\usepackage{microtype}
\usepackage{comment}
\usepackage{slashed}
\usepackage{mathtools}
\usepackage{centernot}
\usepackage{leftidx}
\usepackage{accents}
\usepackage{arydshln}
\usepackage{footnote}
\makesavenoteenv{tabular}
\usepackage{enumitem}
\usepackage{stackrel}
\usepackage{halloweenmath}
\usepackage{longtable}
\usepackage{array}
\usepackage{cprotect}
\usepackage{xstring}
\usepackage{ulem}
\usepackage{stmaryrd}
\usepackage{empheq}
\usepackage{color}
\usepackage{framed}
\definecolor{shadecolor}{gray}{0.875}

\DeclareSymbolFont{timesoperators}{T1}{ptm}{m}{n}
\SetSymbolFont{timesoperators}{bold}{T1}{ptm}{b}{n}

\makeatletter
\renewcommand{\operator@font}{\mathgroup\symtimesoperators}
\makeatother

\colorlet{symbolsgrey}{blue!30!black!50}
\colorlet{testcolor}{green!60!black}
\definecolor{purple}{rgb}{0.55,0.05,0.8}
\definecolor{symbols}{rgb}{0.55,0.05,0.8}

\usetikzlibrary{shapes.misc}
\usetikzlibrary{shapes.symbols}
\usetikzlibrary{shapes.geometric}
\usetikzlibrary{snakes}
\usetikzlibrary{decorations}
\usetikzlibrary{decorations.markings}

\usetikzlibrary{calc}
\usetikzlibrary{external}

\let\oldskull\skull
\def\skull{\mathord{\oldskull}}

\DeclareMathAlphabet{\mathbbm}{U}{bbm}{m}{n}

\DeclareFontFamily{U}{BOONDOX-calo}{\skewchar\font=45 }
\DeclareFontShape{U}{BOONDOX-calo}{m}{n}{
  <-> s*[1.05] BOONDOX-r-calo}{}
\DeclareFontShape{U}{BOONDOX-calo}{b}{n}{
  <-> s*[1.05] BOONDOX-b-calo}{}
\DeclareMathAlphabet{\mcb}{U}{BOONDOX-calo}{m}{n}
\SetMathAlphabet{\mcb}{bold}{U}{BOONDOX-calo}{b}{n}

\setlist{noitemsep,topsep=4pt}

\makeatletter 
\newcommand*{\bigcdot}{}
\DeclareRobustCommand*{\bigcdot}{%
  \mathbin{\mathpalette\bigcdot@{}}%
}
\newcommand*{\bigcdot@scalefactor}{.5}
\newcommand*{\bigcdot@widthfactor}{1.15}
\newcommand*{\bigcdot@}[2]{%
  \sbox0{$#1\vcenter{}$}
  \sbox2{$#1\cdot\m@th$}%
  \hbox to \bigcdot@widthfactor\wd2{%
    \hfil
    \raise\ht0\hbox{%
      \scalebox{\bigcdot@scalefactor}{%
        \lower\ht0\hbox{$#1\bullet\m@th$}%
      }%
    }%
    \hfil
  }%
}
\makeatother

\def\symbol#1{\textcolor{black}{#1}}
\def\1{\mathbf{\symbol{1}}}

\def\bone{\mathbf{1}}

\def\BG{\textnormal{\scriptsize \textsc{BG}}}
\def\LD{\textnormal{\scriptsize \textsc{LD}}}
\def\sto{\textnormal{\scriptsize \tiny{sto}}}
\def\det{\textnormal{\scriptsize \tiny{det}}}

\def\dash{\leavevmode\unskip\kern0.18em--\penalty\exhyphenpenalty\kern0.18em}
\def\slash{\leavevmode\unskip\kern0.15em/\penalty\exhyphenpenalty\kern0.15em}

\newcommand{\mrd}{\mathrm{d}}

\newcommand{\cmark}{\text{\ding{51}}}
\newcommand{\xmark}{\text{\ding{55}}}

\colorlet{darkblue}{blue!90!black}
\colorlet{darkgreen}{green!50!black}
\colorlet{darkyellow}{yellow!92!black}
\colorlet{darkred}{red!50!black}
\let\comm\comment

\newcommand{\mfc}{\mathfrak{c}}

\newcommand{\mfC}{\mathfrak{C}}
\newcommand{\mfR}{\mathfrak{R}}

\newcommand{\mfM}{\mathfrak{M}}

\newcommand{\mfp}{\mathfrak{p}}

\newcommand{\mfP}{\mathfrak{P}}

\newcommand{\mcC}{\mathcal{C}}

\newcommand{\mcS}{\mathcal{S}}

\newcommand{\mcL}{\mathcal{L}}

\newcommand{\mcF}{\mathcal{F}}

\newcommand{\mcN}{\mathcal{N}}

\newcommand{\mcD}{\mathcal{D}}

\newcommand{\rmT}{\mathrm{T}}
\newcommand{\tI}{\tt{I}}

\newcommand{\tL}{\Tilde{L}}
\newcommand{\tV}{\Tilde{V}}

\usetikzlibrary{calc}
\usetikzlibrary{shapes.misc}
\usetikzlibrary{shapes.symbols}
\usetikzlibrary{shapes.geometric}
\usetikzlibrary{snakes}
\usetikzlibrary{decorations}
\usetikzlibrary{decorations.markings}

\tikzset{
	root/.style={circle,fill=testcolor,inner sep=0pt, minimum size=2mm},
	dot/.style={circle,fill=symbols,draw=symbols,inner sep=0pt,minimum size=0.5mm},
	bdot/.style={circle,fill=symbols,draw=symbols,inner sep=0pt,minimum size=1mm},
	bdotsml/.style={circle,fill=symbols,draw=symbols,inner sep=0pt,minimum size=0.75mm},
	square/.style={regular polygon,regular polygon sides=4,fill=black,draw=black,inner sep=0pt,minimum size=1.2mm},
	wsquare/.style={regular polygon,regular polygon sides=4,fill=white,draw=black, inner sep=0pt,minimum size=1.2mm},
	squaresml/.style={regular polygon,regular polygon sides=4,fill=black,draw=black,inner sep=0pt,minimum size=0.9mm},
	wsquaresml/.style={regular polygon,regular polygon sides=4,fill=white,draw=black, inner sep=0pt,minimum size=0.9mm},
	eps/.style={circle,fill=white,draw=symbols,inner sep=0pt,minimum size=1mm},
	int/.style={circle,fill=black,draw=black,inner sep=0pt,minimum size=0.7mm},
	var/.style={circle,fill=black!10,draw=black,inner sep=0pt, minimum size=2mm},
	dotred/.style={circle,fill=black!50,inner sep=0pt, minimum size=2mm},
	generic/.style={semithick,shorten >=1pt,shorten <=1pt},
	dist/.style={ultra thick,draw=testcolor,shorten >=1pt,shorten <=1pt},
	testfcn/.style={ultra thick,testcolor,shorten >=1pt,shorten <=1pt,<-},
	testfcnx/.style={ultra thick,testcolor,shorten >=1pt,shorten <=1pt,<-,
		postaction={decorate,decoration={markings,mark=at position 0.6 with {\drawx}}}},
	keps/.style={semithick,shorten >=1pt,shorten <=1pt,densely dashed,->},
	kprimex/.style={semithick,shorten >=1pt,shorten <=1pt,densely dashed,->,
		postaction={decorate,decoration={markings,mark=at position 0.4 with {\drawx}}}},
	kernel/.style={semithick,shorten >=1pt,shorten <=1pt,->},
	multx/.style={shorten >=1pt,shorten <=1pt,
		postaction={decorate,decoration={markings,mark=at position 0.5 with {\drawx}}}},
	kernelx/.style={semithick,shorten >=1pt,shorten <=1pt,->,
		postaction={decorate,decoration={markings,mark=at position 0.4 with {\drawx}}}},
	kernel1/.style={->,semithick,shorten >=1pt,shorten <=1pt,postaction={decorate,decoration={markings,mark=at position 0.45 with {\draw[-] (0,-0.1) -- (0,0.1);}}}},
	kernel2/.style={->,semithick,shorten >=1pt,shorten <=1pt,postaction={decorate,decoration={markings,mark=at position 0.45 with {\draw[-] (0.05,-0.1) -- (0.05,0.1);\draw[-] (-0.05,-0.1) -- (-0.05,0.1);}}}},
	kernelBig/.style={semithick,shorten >=1pt,shorten <=1pt,decorate, decoration={zigzag,amplitude=1.5pt,segment length = 3pt,pre length=2pt,post length=2pt}},
	rho/.style={dotted,semithick,shorten >=1pt,shorten <=1pt},
	renorm/.style={shape=circle,fill=white,inner sep=1pt},
	labl/.style={shape=rectangle,fill=white,inner sep=1pt},
	xi/.style={circle,fill=symbols!10,draw=symbols,inner sep=0pt,minimum size=1.2mm},
	xix/.style={crosscircle,fill=symbols!10,draw=symbols,inner sep=0pt,minimum size=1.2mm},
	xib/.style={circle,fill=symbols!10,draw=symbols,inner sep=0pt,minimum size=1.6mm},
	xibx/.style={crosscircle,fill=symbols!10,draw=symbols,inner sep=0pt,minimum size=1.6mm},
	not/.style={circle,fill=symbols,draw=symbols,inner sep=0pt,minimum size=0.5mm},
cumu2n/.style={inner sep=3pt},
cumu2/.style={draw=red!80,fill=red!40},
cumu2b/.style={draw=blue!80,fill=blue!40},
cumu2nv/.style={inner sep=3pt},
cumu2v/.style={draw=red!80,fill=white,very thick},
cumu3/.style={regular polygon, regular polygon sides=3,draw=red!80,rounded corners=3pt,fill=red!40,minimum size=5mm},
cumu4/.style={regular polygon, regular polygon sides=4,draw=red!80,rounded corners=3pt,fill=red!40,minimum size=7mm},
cumu5/.style={regular polygon, regular polygon sides=5,draw=red!80,rounded corners=3pt,fill=red!40,minimum size=7mm},
	>=stealth,
	}

\newcommand{\ang}[1]{\langle #1\rangle }
\newcommand{\na}{\ang{\nabla}}
\newcommand{\angp}[2]{\langle #1\rangle^{#2}}
\newcommand{\bbX}{\mathbb{X}}
\newcommand{\G}{\mathbb{G}}
\newcommand{\bH}{\mathbb{H}}
\newcommand{\cC}{\mathcal{C}}

\newcommand{\cG}{\mathcal{G}}
\newcommand{\cL}{\mathcal{L}}

\newcommand{\cB}{\mathcal{B}}
\newcommand{\cA}{\mathcal{A}}

\newcommand{\cI}{\mathfrak{I}}
\newcommand{\cK}{\mathfrak{K}}
\newcommand{\cS}{\mathcal{S}}
\newcommand{\cN}{\mathcal{N}}
\newcommand{\cZ}{\mathcal{Z}}
\newcommand{\cF}{\mathcal{F}}
\newcommand{\bv}{\big|}
\newcommand{\bV}{\Big\Vert}
\newcommand{\hc}{\hat{c}}
\newcommand{\rmd}{\mathrm{d}}
\newcommand{\tA}{\Tilde{A}}

\makeatletter

\DeclareRobustCommand{\TitleEquation}[2]{\texorpdfstring{\StrLeft{\f@series}{1}[\@firstchar]$\if%
b\@firstchar\boldsymbol{#1}\else#1\fi$}{#2}}

\makeatother

\usepackage[margin=1.15in]{geometry}

\begin{document}

\title{A Stochastic Analysis Approach to Tensor Field Theories}

\author{Ajay Chandra$^{1}$, L\'{e}onard Ferdinand$^{2}$}

\institute{Imperial College London, UK 
\and Universit\'{e} Paris-Saclay, France\\[1em]
\email{a.chandra@ic.ac.uk, lferdinand@ijclab.in2p3.fr}}

\maketitle

\newcommand{\noise}{\raisebox{0ex}{\includegraphics[scale=1]{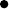}}}
\newcommand{\bnoise}{\raisebox{0ex}{\includegraphics[scale=1]{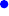}}}
\newcommand{\rnoise}{\raisebox{0ex}{\includegraphics[scale=1]{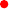}}}
\newcommand{\gnoise}{\raisebox{0ex}{\includegraphics[scale=1]{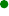}}}
\newcommand{\X}{\raisebox{0ex}{\includegraphics[scale=1]{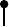}}}
\newcommand{\bX}{\raisebox{0ex}{\includegraphics[scale=1]{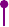}}}
\newcommand{\rX}{\raisebox{0ex}{\includegraphics[scale=1]{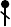}}}
\newcommand{\Xtau}{\raisebox{0ex}{\includegraphics[scale=1]{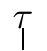}}}
\newcommand{\rXtau}{\raisebox{0ex}{\includegraphics[scale=1]{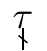}}}
\newcommand{\cherry}{\raisebox{0ex}{\includegraphics[scale=1]{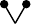}}}
\newcommand{\Xtrident}{\raisebox{-.5 ex}{\includegraphics[scale=1]{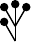}}}
\newcommand{\Xtwofork}{\raisebox{-.5 ex}{\includegraphics[scale=1]{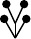}}}
\newcommand{\Xthreeloc}{\raisebox{-.3ex}{\includegraphics[scale=1]{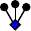}}}
\newcommand{\pXthreeloc}{\raisebox{-.3ex}{\includegraphics[scale=1]{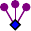}}}
\newcommand{\XtwoPictwo}{\raisebox{-1 ex}{\includegraphics[scale=1]{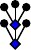}}}
\newcommand{\Pictwo}{\raisebox{-1 ex}{\includegraphics[scale=1]{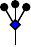}}}
\newcommand{\rPictwo}{\raisebox{-1 ex}{\includegraphics[scale=1]{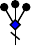}}}
\newcommand{\Picthree}{\raisebox{-1.5 ex}{\includegraphics[scale=1]{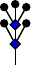}}}
\newcommand{\rPicthree}{\raisebox{-1.5 ex}{\includegraphics[scale=1]{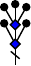}}}
\newcommand{\Xtwom}{\raisebox{0ex}{\includegraphics[scale=1]{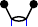}}}
\newcommand{\pXtwom}{\raisebox{0ex}{\includegraphics[scale=1]{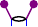}}}
\newcommand{\Xtwonm}{\raisebox{-1 ex}{\includegraphics[scale=1]{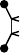}}}
\newcommand{\pXtwonm}{\raisebox{-1 ex}{\includegraphics[scale=1]{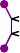}}}
\newcommand{\XdotX}{\raisebox{-1 ex}{\includegraphics[scale=1]{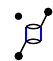}}}
\newcommand{\pXdotX}{\raisebox{-1 ex}{\includegraphics[scale=1]{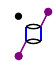}}}
\newcommand{\Sym}{\raisebox{-1 ex}{\includegraphics[scale=1]{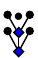}}}
\newcommand{\pSym}{\raisebox{-1 ex}{\includegraphics[scale=1]{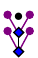}}}
\newcommand{\Xthree}{\raisebox{-1 ex}{\includegraphics[scale=1]{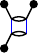}}}
\newcommand{\pXthree}{\raisebox{-1 ex}{\includegraphics[scale=1]{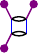}}}
\newcommand{\vXX}{\raisebox{-1 ex}{\includegraphics[scale=1]{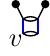}}}
\newcommand{\pvXX}{\raisebox{-1 ex}{\includegraphics[scale=1]{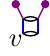}}}
\newcommand{\XvX}{\raisebox{-1 ex}{\includegraphics[scale=1]{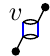}}}
\newcommand{\pXvX}{\raisebox{-1 ex}{\includegraphics[scale=1]{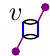}}}
\newcommand{\XXv}{\raisebox{-1 ex}{\includegraphics[scale=1]{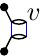}}}
\newcommand{\pXXv}{\raisebox{-1 ex}{\includegraphics[scale=1]{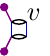}}}
\newcommand{\vvX}{\raisebox{-1 ex}{\includegraphics[scale=1]{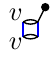}}}
\newcommand{\pvvX}{\raisebox{-1 ex}{\includegraphics[scale=1]{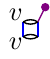}}}
\newcommand{\Xvv}{\raisebox{-1 ex}{\includegraphics[scale=1]{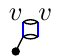}}}
\newcommand{\pXvv}{\raisebox{-1 ex}{\includegraphics[scale=1]{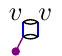}}}
\newcommand{\vXv}{\raisebox{-1 ex}{\includegraphics[scale=1]{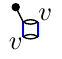}}}
\newcommand{\pvXv}{\raisebox{-1 ex}{\includegraphics[scale=1]{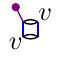}}}
\newcommand{\vvv}{\raisebox{-1 ex}{\includegraphics[scale=1]{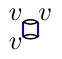}}}
\newcommand{\XXXX}{\raisebox{-1.2 ex}{\includegraphics[scale=1]{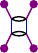}}}
\newcommand{\vXXv}{\raisebox{-1.3 ex}{\includegraphics[scale=1]{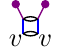}}}
\newcommand{\XXXv}{\raisebox{-1.3 ex}{\includegraphics[scale=1]{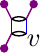}}}
\newcommand{\XvXv}{\raisebox{-1.3 ex}{\includegraphics[scale=1]{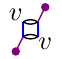}}}
\newcommand{\XXvv}{\raisebox{-1.3 ex}{\includegraphics[scale=1]{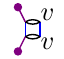}}}
\newcommand{\Xvvv}{\raisebox{-1.3 ex}{\includegraphics[scale=1]{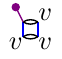}}}
\newcommand{\vvvv}{\raisebox{-1.3 ex}{\includegraphics[scale=1]{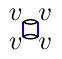}}}
\newcommand{\cXdotXloc}{\raisebox{-.3ex}{\includegraphics[scale=1]{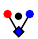}}}
\newcommand{\XdotXloc}{\raisebox{-.3ex}{\includegraphics[scale=1]{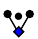}}}
\newcommand{\XXdotloc}{\raisebox{-.3ex}{\includegraphics[scale=1]{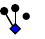}}}
\newcommand{\dotXXloc}{\raisebox{-.3ex}{\includegraphics[scale=1]{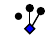}}}
\newcommand{\cXthreeloc}{\raisebox{-.3ex}{\includegraphics[scale=1]{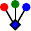}}}
\newcommand{\cdotXXloc}{\raisebox{-.3ex}{\includegraphics[scale=1]{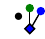}}}
\newcommand{\cXXdotloc}{\raisebox{-.3ex}{\includegraphics[scale=1]{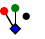}}}

\begin{abstract}
We present two different arguments using stochastic analysis to construct super-renormalizable tensor field theories, namely the $\rmT^4_3$ and $\rmT^4_4$ models. 
The first approach is the construction of a Langevin dynamic \cite{ParisiWu,GH21} combined with a PDE energy estimate while the second is an application of the variational approach of Barashkov and Gubinelli \cite{BG20}. 
By leveraging the melonic structure of divergences, regularising properties of non-local products, and controlling certain random operators, we demonstrate that for tensor field theories these arguments can be significantly simplified in comparison to what is required for $\Phi^4_d$ models. 
 \end{abstract}

\setcounter{tocdepth}{2}
\tableofcontents

\section{Introduction}

A class of heavily studied and paradigmatic models in constructive quantum field theory are the $\Phi^4_d$ measures. 
In finite volume, these are non-Gaussian probability measures $\mu$ supported on distributions (over space) which are formally written
\begin{equ}\label{eq:local_phi_4}
    \mathrm{d}\mu  (\phi) \propto  \exp \Big( - \frac\lambda4 \Vert \phi\Vert_{L^4(\T^d)}^4+\frac{a}{2} \Vert \phi\Vert_{L^2(\T^d)}^2 \Big) \mathrm{d}\mcb{g}(\phi)\,.
\end{equ}
where $\lambda > 0$ and $\mcb{g}$ denotes the Gaussian measure with covariance $(1-\Delta)^{-1}$, that is a massive Gaussian free field on $\T^{d}$. 
When $d=1$ it is straightforward to make \eqref{eq:local_phi_4} rigorous but much less so when $d > 1$ \dash in this case the measure $\mcb{g}$ is supported on distributions so rough that the nonlinear expressions  $\Vert \phi\Vert_{L^4(\T^d)}^4$ and $\Vert \phi\Vert_{L^2(\T^d)}^2$ above are ill-defined. 

When $d=2$ and $d=3$  the measure $\mu$ can be obtained as the weak limit of regularized and appropriately renormalized measures \cite{Nel73,MR0408581,Feldman,MR0416337} \dash for instance the $N \uparrow \infty$ limit of
\begin{equ}\label{eq:local_phi_4_renorm}
    \mathrm{d}\mu_{N}  (\phi) \propto \exp \Big( - \frac\lambda4 \Vert \phi \Vert_{L^4(\T^d)}^4+ \frac{a+a_N}{2} \Vert \phi \Vert_{L^2(\T^d)}^2  \Big) \mathrm{d}\mcb{g}_{N}(\phi)\,.
\end{equ}
where $\mcb{g}_{N}$ is the pushforward of $\mcb{g}$ under $\varrho_{N}$, which is a Fourier multiplier in space with  $0 \leqslant \hat{\varrho} \leqslant 1$, $\hat{\varrho}_{N}(k) = 1$ for $|k| \leqslant N$ and $\varrho_{N} (k) = 0$ for $|k| \geqslant 2N$.

The constant $a_{N}$ in \eqref{eq:local_phi_4_renorm} is a renormalization constant diverging to infinity as $N \uparrow \infty$. 
One can identify the suitable choices of $a_{N}$ using perturbation theory, that is formally calculating moments of $\mu_{N}$ by expanding the exponential on the right hand side of \eqref{eq:local_phi_4_renorm} as a formal series and using Wick's rule for Gaussian moments to integrate out $\phi$ term by term \dash one chooses $a_{N}$ to make these expansions finite order by order in $\lambda$ as $N \uparrow \infty$. 
However this series itself is far from convergent and proving $\mu_{N}$ weakly converges as $N \uparrow \infty$ requires other tools. 

In the case of $d=2$, one can show that $\mu$ is absolutely continuous with respect to $\mcb{g}$, the renormalization procedure is nothing more than passing from ill-defined polynomials of a rough Gaussian field to the corresponding  well-defined Hermite polynomials \slash Wick powers. 
The situation when $d=3$ is much more complicated due to the presence of so-called ``non-local divergences'' \dash it has been shown in \cite{BG21} that $\mu$ is singular with respect to $\mcb{g}$. 
Even though $a_{N}$ is chosen to cancel the divergences of infinitely many terms in the perturbation expansion, when $d=2,3$ one can write $a_{N}$ explicitly as the expectation of a polynomial in $\phi$, $\lambda$ and $a$ under $\mcb{g}$.
For this reason  $\Phi^4_2$ and $\Phi^4_3$ are called \textit{super-renormalizable}.

In the case where $d = 4$ the model is \textit{just-renormalizable}, in particular perturbative renormalization requires including a diverging renormalization counterterm $\lambda_{N} \Vert  \phi \Vert_{L^4(\T^4)}^4$ in the exponential in \eqref{eq:local_phi_4_renorm} along with a multiplicative wave-function renormalisation $Z_{N}$. 
Moreover, perturbation theory gives infinite series $\lambda_{N}$, $a_{N}$, and $Z_{N}$ in $\lambda$ which are not  all summable. 
When $d > 4$ the model is \textit{non-renormalizable}, and in addition to the above perturbative renormalization would require inserting infinitely many other counterterms (such as terms $\Vert \phi\Vert_{L^{2j}(\T^d)}^{2j}$ for instance) inside of the exponential factor in \eqref{eq:local_phi_4_renorm}. 

Switching from questions of perturbative renormalization to the question of the convergence of measures, in \cite{MR678000,MR643591} it was shown that one cannot obtain non-Gaussian limits when $d \geqslant 5$ while the case of $d=4$ is more subtle and was only settled more recently in \cite{ADC21}. 
The analysis of renormalization problems in just-renormalizable models is quite subtle and there are other examples \cite{GKNonLinSigma} where one can obtain non-Gaussian limits. 
An important step of one of the Millenium Problems is investigating whether one can obtain a non-trivial limit for non-abelian Yang-Mills in $4$-dimensions.  
A major frustration has been the lack of a simple scalar field theory to investigate the construction of a just renormalizable model.
However, a newer class of models, so called ``non-local models'' of which tensor field theories are an example, are promising candidates for obtaining non-Gaussian limits in the just renormalizable case. 

\subsection{Tensor field theories}

The regularized, renormalized measures $\mu_{N}$ in \eqref{eq:local_phi_4_renorm} are non-Gaussian because of the quartic $L^{4}$ term. 
The models we study in this article are $\mathrm{T}^4_d$ models, they are examples of \textit{tensor field models} \cite{Bengeloun} and they are obtained by replacing this quartic $L^{4}$ term with a non-local quartic interaction which we introduce now. 

For $d\geqslant 2$ and any fixed $c\in [d] \eqdef \{1,\dots,d\}$, we define the function  
\[
 \T^d\times\T^d \ni (x,y) \mapsto \chi^{c}(x,y) = \big( \chi^{c}(x,y)_{i} \big)_{i=1}^{d} \in \T^{d}
 \]  
by setting, for each $1 \leqslant i \leqslant d$ and $x = (x_j)_{j=1}^{d},\; y= (y_j)_{j=1}^{d} \in \T^{d}$,  
\begin{equ}\label{eq:nonlocalpairing}
    \chi^c(x,y)_i=\begin{cases}x_i \text{ if } i\neq c\;, 
    \\y_c \text{ if } i=c\;.\end{cases}\
\end{equ}

For each $c \in[d]$ we introduce a norm $\Vert \bigcdot \Vert_{M^4_c}$ on scalar functions $\phi$ on $\T^{d}$ by setting 
\begin{equ}\label{eq:1colornonlin}
    \Vert\phi\Vert_{M^4_c(\T^d)}^4\eqdef \int_{\T^d}\int_{\T^d}\phi(x)\phi(y)\phi\big(\chi^c(x,y)\big)\phi \big(\Bar{\chi}^c(x,y) \big)\; \mathrm{d}x \mathrm{d}y \,,
\end{equ} 
where $\Bar{\chi}^c(x,y) \eqdef \chi^c(y,x)$. 
We also write 
\[
\Vert \bigcdot \Vert_{M^4(\T^d)} \eqdef  \Big( \displaystyle\sum_{c=1}^{d} \Vert \bigcdot \Vert_{M_{c}^4(\T^d)}^{4} \Big)^{1/4}\;.
\]
In Appendix~\ref{sec:nonlin_facts} we confirm that $\Vert \bigcdot \Vert_{M^4_c}$ is indeed a norm and why it may be called a tensor norm.  
 
Formally, the $\mathrm{T}_{d}^4$ model is given by
\begin{equ}\label{eq:tensor_phi_4}
    \mathrm{d}\nu  (\phi) \propto  \exp \Big( - \frac\lambda4 \Vert \phi\Vert_{M^4(\T^d)}^4+ \frac{a}{2} \Vert \phi\Vert_{L^2(\T^d)}^2 \Big) \mathrm{d}\mcb{g}(\phi)\;. 
\end{equ}

Like the local $\Phi^4_d$ models, analysis of higher dimensional $\mathrm{T}_{d}^4$ models goes via regularization, introducing renormalization, and then obtaining bounds that are uniform in the regularization.   
Heuristically, the level of renormalization needed for a $\mathrm{T}_{d}^4$ model is comparable to that of the corresponding $\Phi^4_{d-1}$ model. 

In particular, $\mathrm{T}^4_{d}$ is super-renormalizable for $d < 5$ and $\mathrm{T}_{2}^4$ can be defined without renormalization.
For $d  \in \{3,4\}$ the measures $\nu$ can be obtained as weak $N \uparrow \infty $ limits of the measures
\begin{equ}\label{eq:tensor_phi_4_renorm}
    \mathrm{d}\nu_{N}  (\phi) \propto  \exp \Big( - \frac\lambda4 \Vert \phi\Vert_{M^4(\T^d)}^4+ \frac{a+a_N}{2} \Vert \phi\Vert_{L^2(\T^d)}^2 \Big) \mathrm{d} \mcb{g}_{N}(\phi)\,,
\end{equ}
where again $a_{N}$ diverges as $N \uparrow \infty$. In our context of a subcritical measure on a compact space, the renormalized coupling constant $\lambda$ and the renormalized mass $1-a$ can be set to one, which we enforce in the sequel. 

Moreover, $\mathrm{T}_{3}^4$ is absolutely continuous with respect to $\mcb{g}$ and only requires Wick type renormalization while the renormalization of $\mathrm{T}_{4}^4$ is more difficult and its limit is expected to be singular with respect to  $\mcb{g}$. 
The models $\mathrm{T}_{3}^4$  and $\mathrm{T}_{4}^4$ \cite{T44} are also both super-renormalizable and were constructed and shown to be Borel summable in \cite{T43} and \cite{T44}. 
However, the technique used in these constructions seems to break-down in $d=4$, thereby not allowing for an exploration of the full subcritical regime. The model $\mathrm{T}_{5}^4$ is just-renormalizable and requires an additional counterterm $\lambda_{N} \Vert \phi\Vert_{M^4(\T^4)}^4$ along with a multiplicative wave-function renormalization $Z_{N}$. 
However, in contrast with $\Phi^4_4$, there is evidence \cite{RV21} obtaining a non-Gaussian limit for $\mathrm{T}^4_5$  might be possible.

Another feature of the $\mathrm{T}^4_{d}$ models that makes them seem more tractable is a topological constraint that must be satisfied for a Feynman graph to be divergent, such graphs must be ``melonic''. 
For the $\mathrm{T}^{3}_{4}$ and $\mathrm{T}^{4}_{4}$ models the melonic constraint plays a role in greatly simplifying the stochastic analysis approach to these models. 
For $\mathrm{T}^{4}_{5}$ the melonic constraint imposes that the infinite collection of divergent Feynman graphs in this model, when organized by the number of their edges, proliferate like trees instead of connected graphs \dash in particular one could hope that perturbative formula for renormalization constants would be summable for small coupling $\lambda$.

Finally, we mention that tensor field theories are not the only examples of toy non-local field theories that have been investigated. 
For instance, Hartree type nonlinearities can be defined by substituting the $L^4$ norm with $\Vert\phi^2\Vert^2_{H^{-\frac{\beta}{2}}}$ where for the Sobolev space $H^{-\beta/2}$ one takes $\beta\in(0,d)$.
 However, the Hartree nonlinearity is closer to being a regularized version of the local product while the tensor field theory interaction behaves very differently from the local product in its renormalization. 
Another non-local model which is closer to our setting and which has been shown to give a non-Gaussian limit in the just-renormalizable case is the Moyal model \cite{GW14}. 
\subsection{Main results}
Much of the earlier work on both local and non-local field theories constructs obtains controls over measures by showing convergence of their moments as cut-offs are removed using expansions and computations of expectations under $\mcb{g}$. 
Broadly speaking, the stochastic analysis approach to studying these problems takes a different point - the field theory is again analysed as a perturbation of a gaussian measure, but the aim is to show convergence of the non-Gaussian field as a random field which is coupled to the underlying Gaussian field. 

In this paper we construct the $\mathrm{T}^{4}_{3}$ and $\mathrm{T}^{4}_{4}$ measures, employing two approaches \dash the first approach is a Langevin-type dynamical stochastic quantization \cite{ParisiWu,GH21} while the other is a variational approach introduced in \cite{BG20} reminiscent of the Gibbs variational principle. 
The two approaches give two independent ways of constructing the $\mathrm{T}^{4}_{3}$ and $\mathrm{T}^{4}_{4}$ measures. 

\subsubsection{Langevin dynamic}
A natural Langevin dynamic for the $\Phi^4_d$ model \eqref{eq:local_phi_4} is formally given by
\begin{equ}\label{eq:LD}
\partial_{t} \phi = (\Delta-1) \phi -  \lambda \phi^3 -  a  \phi + \sqrt{2} \xi\;,
\end{equ}
where $\xi$ is a $d$-dimensional space-time white noise on a probability space $(\Omega,\cB,\P)$ and now $\phi$ is also a distribution over space and a new ``fictitious''  time $t$. 
When $d \geqslant 2$,  \eqref{eq:LD} is an example of singular stochastic partial differential equation \dash  the roughness of the noise $\xi$ prevents us from using classical arguments for local well-posedness since $\phi$ belongs to spaces where $\phi^3$ is not well-defined. 
Local well-posedness for \eqref{eq:LD} was first obtained for $d=2$ in \cite{DPD2} and was one of the first examples in combining probablistic estimates with path-wise analysis to solve a singular SPDE. The case $d=3$ was much more difficult and remained open until the development of more sophisticated path-wise methods such as regularity structures and paracontrolled calculus \cite{Hai14,CC13}. 

At a formal level, the $\Phi^4_d$ measure should be invariant for the dynamic \eqref{eq:LD} and this can be made rigorous for the regularized dynamic and measure.
The corresponding Langevin dynamic for the $\mathrm{T}^{4}_{d}$ model is given (with $\lambda$ and $1-a$ set to one) by 
\begin{equation}\label{eq:tensorLD}
 \partial_{t} \phi= (\Delta-1) \phi - \cN(\phi,\phi,\phi)+\sqrt{2}\xi\,.
\end{equation}
where, $\xi$ is as in \eqref{eq:LD} and, for  $f,g,h: \T^{d} \times \R \rightarrow \R$, we have 

\begin{equs}
\cN(f,g,h)(x,t) &\eqdef \displaystyle\sum_{c=1}^d\cN^c(f,g,h)(x,t)\;,  \enskip \text{where}\\
\cN^c(f,g,h)(x,t) &\eqdef \int_{\T^d}f(\chi^c(x,y),t)g(y,t)h(\Bar{\chi}^c(x,y),t)\; \mathrm{d}y\;.
\end{equs}

The regularized dynamic we study is given by
\begin{equ}\label{eq:regTensorLD}
\partial_{t} \phi= (\Delta-1) \phi -  \Pi_N \cN(\phi,\phi,\phi)+ a_{N} \phi+\sqrt{2} \Pi_N\xi\;,
\end{equ}
where $\Pi_{N}=\1\{  |\nabla|_{\infty}\leqslant N \}$ projects onto Fourier modes $k$ with $|k|_{\infty} \leqslant N$ and $a_{N}$ again refers to a renormalization constant. 
Note if one takes $\varrho_{N} = \Pi_N$ in \eqref{eq:tensor_phi_4_renorm} then the dynamic \eqref{eq:regTensorLD} keeps the $\Pi_N$ marginal of \eqref{eq:tensor_phi_4_renorm} invariant. 
A main step in showing solutions to \eqref{eq:regTensorLD} converge is writing $\phi = \bX + v$ where $\bX$ is an explicit, rough space-time stochastic process and $v$ is the local in time solution to a better behaved remainder equation. 
Our main result on the Langevin dynamics is given by the following theorem.
\begin{theorem}\label{dynamic_theorem1}
Let $d \in \{3,4\}$, then there exist a choice of the constants $a_{N}$ such that one has the $N \uparrow \infty$ convergence of local (in time) solutions to \eqref{eq:regTensorLD} when the dynamic is started from data that is appropriately close to equilibrium, see Theorem~\ref{thm:3d} and Proposition~\ref{prop:4dfixedpoint} for a precise statement.

Moreover, one has an $L^2$-coming down from infinity bound for $v = \phi - \bX$ , see Proposition~\ref{prop:comingdown} for an precise statement. 
\end{theorem} 

In contrast with known results for the $\Phi^4_d$ theory for $d=2,3$ \cite{DPD2,tsatsoulis2016spectral,MW17Phi43}, we are not able to obtain global in time solutions. 
However, the local well-posedness result and $L^2$ bounds are sufficient to use the dynamic to obtain tightness for the measures $\nu_{N}$.
\begin{theorem} 
 For $d=3,4$ the family $(\nu_N)_{N\in\N}$ is tight on $H^{-\frac{d-2}{2}-\epsilon}(\T^d)$. Moreover, for all $p\in[1,\infty)$, any subsequential limit $\nu$ satisfies the bounds: 
   \begin{equs}
       \E_\nu[\Vert\phi\Vert_{H^{-\frac{d-2}{2}-\epsilon}(\T^d)}^p]<\infty \quad \text{and} \quad
       \E_\nu[\Vert \phi-\bX\Vert_{L^2(\T^d)}^p]<\infty\;.
   \end{equs} 
\end{theorem}
While the power counting of $\mathrm{T}^{4}_{4}$ resembles that of $\Phi^4_{3}$, it turns out we can bypass the use of more advanced analytic theories for singular SPDE (such as the theory of regularity structures or paracontrolled calculus) when proving local well-posedness of dynamical  $\mathrm{T}^{4}_{4}$. 
In fact, if one sets up the model with the dimension as a real parameter (say by changing the regularity of the noise $\xi$ or using fractional Laplacians) we conjecture that Da Prato - Debussche style arguments (with progressively more terms in $\bX$)  can be used to obtain local well-posedness for the Langevin dynamic in the entire subcritical \slash super-renormalizable regime. 
While the result proven in \cite{T44} for $\rmT^4_4$ is stronger than ours, the machinery is heavier and it seems unlikely that the arguments of \cite{T44} generalize as easily to the full subcritical \slash super-renormalizable regime.

The following technical remark, which assumes some familiarity with Feynman diagrams in field theory and perturbative tree expansions for singular SPDE, gives some intuition for why the Da Prato - Debussche argument suffices for $\mathrm{T}^4_{4}$ . 
 \begin{remark}\label{remark:DPD}
We first give more detail about the power counting for Feynman graphs in $\Phi^4_d$ and $\mathrm{T}^4_{d}$.
The superficial degree of divergence of a Feynman graph $G$ in the the local $\Phi^4_d$ theory is given by 
\begin{equs}
    \omega_{\Phi^4_d}(G)=d-(4-d)|V(G)|-\frac{d-2}{2}|L^{ext}(G)|\,,
\end{equs}
where $|V(G)|$ and $|L^{ext}(G)|$ denote the number of vertices and external edges of $G$. 

In the critical/just-renormalizable dimension $d=4$, the power counting of a graph does not depend on its number of vertices, while in the subcritical/super-renormalizable regime only a finite number of graphs are superficially divergent. 
More precisely, apart from the vacuum graphs, in two dimensions, only the so-called tadpole graph (the one-vertex two-point graph) is superficially divergent, while in three dimensions, the sunset graph and the snowball graph (the two two-vertex two-point graphs) are superficially divergent. 

Moreover, in the case of $\Phi^4_3$, the snowball graph does not require the introduction of a counterterm, since it already contains a nested tadpole graph whose renormalized amplitude is equal to zero. Renormalizing the tadpole graph thus renormalizes the snowball graph too. Therefore, in the local case, one only has to add two counterterms, one for the tadpole and the other for the sunset.

The superficial degree of divergence of the $\rmT^4_d$ theories is given by 
\begin{equs}\label{eq:perturbative_degree_tensor}
    \omega_{\rmT^4_d}(G)&=d-(5-d)|V(G)|-\frac{d-3}{2}|L^{ext}(G)|-\delta(G)-C(\partial G)\\&= \omega_{\Phi^4_{d-1}}(G)-\delta(G)-\big(C(\partial G)-1\big)
    \,,
\end{equs}
where $C(\partial G)$, the number of connected components of the boundary graph, and $\delta(G)$, the degree of the graph which is a positive integer, are defined in Section~\ref{sec:diagram}. This power counting strongly suggests that $\rmT^4_d$ is to be compared with $\Phi^4_{d-1}$. However a difference with $\Phi^4_{d-1}$ is that (apart the vacuum graphs) only graphs $G$ with $\delta(G)  = 0$ can be superficially divergent. Such graphs $G$ are called melonic graphs.

The non-local vertex for $\T^4_{d}$ is less symmetric than the local vertex, so in the non-local theory there are many different Feynman graphs corresponding to the local tadpole, sunset, and snowball graphs.
Some of these are melonic, some of them are not - in particular there are tadpole and snowball type non-local graphs that can be melonic but none of the sunset type non-local graphs are. Another difference in the non-local case is that the renormalized amplitude of the tadpole is no longer equal to zero (the tadpole is no longer a completely local divergence since the vertex itself is non-local) which is why one has to add two counterterms, one for the (melonic) tadpole and the other one for the (melonic) snowball. 
We summarise this in Figure~\ref{fig:powercounting} below.

\begin{figure}[H]
    \centering
    \tikzsetnextfilename{fig1}
    \begin{tikzpicture}
\draw[black] (0,0)..controls(-1,1)and(1,1)..(0,0);
\draw[black] (-.2,-.2)--(0,0);
\draw[black] (.2,-.2)--(0,0);
\draw[black] (2,0.5)..controls(1,1.5)and(3,1.5)..(2,0.5);
\draw[black] (2,0.5)..controls(2.35,0.25)and(2.35,.05)..(2,-.2);
\draw[black] (2,0.5)..controls(1.65,0.25)and(1.65,.05)..(2,-.2);
\draw[black] (1.8,-.4)--(2,-.2);
\draw[black] (2.2,-.4)--(2,-.2);
\draw (3.5,0.25)--(4.5,0.25);
\draw (4.3,.25)arc(0:180:.3);
\draw (3.7,.25)arc(180:360:.3);
\node at (0,-.8){tadpole};
\node at (2,-.8){snowball};
\node at (4,-.8){sunset};
\node at (-1.3,-1.5){$\Phi^4_3$};
\node at (0,-1.5){$\cmark$};
\node at (2,-1.5){$\xmark$};
\node at (4,-1.5){$\cmark$};
\node at (-1.3,-2.5){$\rmT^4_4$};
\node at (0,-2.5){$\cmark$};
\node at (2,-2.5){$\cmark$};
\node at (4,-2.5){$\xmark$};
\node at (-1.3,-3.5){Tree};
\node at (0,-3.5){$\E[\X^2]$};
\node at (2,-3.5){$\E[\Xtrident]$};
\node at (4,-3.5){$\E[\Xtwofork]$};
\end{tikzpicture}
    \caption{Comparison between the primary divergences of the $\Phi^4_3$ and $\mathrm{T}^4_4$ measures. The top row lists divergent Feynman graphs and the  bottom row lists the expectations of the stochastic trees that generate these graphs in the dynamic picture. }
    \label{fig:powercounting}
\end{figure}
The main takeaway is that the divergence $\Xtwofork$, which is the significant non-local divergence of $\Phi^4_3$, does not pose a problem for $\mathrm{T}^4_4$ and instead we must be careful with $\Xtrident$ (which doesn't directly pose an issue in $\Phi^4_3$ if the tadpole has already been renormalized) - but $\Xtrident$  is ``missing'' a branch at the root rather than the top internal vertex which allows us to use a Da Prato - Debussche argument.
\end{remark}

\subsubsection{The variational approach}
Our second approach using the variational method of \cite{BG20} proceeds by representing the Laplace transform of the regularized and renormalized $\mathrm{T}^4_{3}$ and $\mathrm{T}^4_{4}$ measures as a stochastic control problem.

In this approach instead of introducing a time that corresponds to evolution under a Langevin dynamic we instead introduce a time that represents scale. 
More precisely one constructs a martingale field $(\X_{t})_{t \in [0,\infty]}$ where for each fixed $t$, $\X_{t}$ is a Gaussian random field (over space) whose law is given by $\varrho_{t} \phi$ where $\phi \sim \mcb{g}$ and $\varrho_{t}$ is a specially chosen (but now $t$ is a continuous parameter and one enforces additional constraints on how $\varrho_{t}$ depends on $t$ given in Section~\ref{sec:BG}). Consider the measure $\nu_t$ defined in \eqref{eq:tensor_phi_4_renorm}, with $\varrho_t$ as given in Section~\ref{sec:BG}. The work \cite{BG20} uses the Bou\'e-Depuis formula, for each $t > 0$ and $f :\cC^{-\frac{d-2}{2}-\epsilon}(\T^d)\rightarrow\R$ Lipschitz, to obtain the identity
\begin{equs}
{}&  - \log
  \E_{\nu_t}[ e^{-f(\phi)}] =\log\cZ_t +\inf_{u \in \mathbb{H}_{a}} \E_{\P}
  \Bigg[ 
\begin{array}{c} 
\| \X_{t} +I_t( u) \|_{M^4(\T^d)}^4 
    -a_{t}  \| \X_{t} +I_t (u) \|_{L^2(\T^d)}^2-b_t \\
+ f(\X_t+I_t(u)) + 
\mathrm{Ent}(u)
\end{array} 
\Bigg] \nonumber
\end{equs}
where $\log\cZ_t$ is a constant uniformly bounded in $t$, $\mathbb{H}_{a}$ is a certain space of adapted random drifts $u: \R_{\geqslant 0}\times\T^{d}  \rightarrow \R$ which should be thought of as a shift of an underlying white noise process, $I_{\bullet}(u)$ is the corresponding shift of the free field process $\X_{t}$, $\P$ is the law of the entire process $(\X_{t})_{t \in [0,T]}$, $\mathrm{Ent}(u)$ is a relative entropy term which is a quadratic in $u$, and $b_{t}$ is a constant independent of $f$. 
This variational representation can be used to obtain bounds on the Laplace transform of $\nu_t$ uniform in $t$, which gives the following main result (proven at the end of Section~\ref{sec:BG}).  
\begin{theorem}\label{BG_theorem1} 
Let $d\in\{2,3,4\}$ and $f:\cC^{-\frac{d-2}{2}-\epsilon}(\T^d)\rightarrow\R$ have at most linear growth. 
Then, uniform in $t \geqslant 0$ one has  
   \begin{equs}\label{eq:laplace_bound}
       \E_{\nu_t}[e^{-f(\phi)}]\lesssim 1\,,
   \end{equs} 
Since the embedding $\cC^{-\frac{d-2}{2}- 2\epsilon}\hookrightarrow \cC^{-\frac{d-2}{2}-\epsilon}$ is compact, the family $\big(\nu_t\big)_{t\geqslant0}$ is tight on $\cC^{-\frac{d-2}{2}-2\epsilon}(\T^d)$. 
\end{theorem} 
\subsection*{Acknowledgements}
{\small
The authors thank Chunqiu Song and Hendrik Weber for pointing out a mistake in an earlier version of the paper. 
AC gratefully acknowledges partial support by the EPSRC through EP/S023925/1 through the ``Mathematics of Random Systems'' CDT EP/S023925/1.
AC also thanks MPI MiS in Leipzig for their hospitality, where part of this work was completed during an extended research visit. AC and LF thank Fabien Vignes-Tourneret, who was involved at the early stage of this project. AC and LF also thank Vincent Rivasseau for having suggested to them this investigation. }
\subsection{Notation and conventions}
Given $n \in \N$ we write $[n] = \{1,\dots,n\}$. 
For Banach spaces $(\cA,\Vert\bigcdot\Vert_{\cA})$ and $(\cB,\Vert\bigcdot\Vert_{\cB})$, we write $\cL(\cA,\cB) $ for the Banach space of bounded linear operators $\cA\rightarrow\cB$ equipped with the operator norm:
\begin{equs}
    \Vert L\Vert_{\cL(\cA,\cB)}=\sup_{ u :\Vert u \Vert_{\cA}=1}\Vert L( u )\Vert_{\cB}\,.
\end{equs}

For any $\theta > 0$, Banach space $(\cB,\Vert\bigcdot\Vert_{\cB})$, and $T \geqslant 0$, we write $C^\theta_T\cB\eqdef C^\theta([0,T],\cB)$ for the Banach space of $\theta$-H\"older continuous functions from $[0,T]$ into $\cB$ (with respect to the parabolic scaling) and equip this space with the norm $\Vert\bigcdot\Vert_{C_T^\theta\cB}\eqdef \Vert  \bigcdot \Vert_{C_T\cB}+ \Vert  \bigcdot \Vert_{{\dot{C}}^\theta_T\cB}$ where
 \begin{equs}
 \Vert  u  \Vert_{C_T\cB}\eqdef \sup_{ t\in [0,T]}\Vert  u (t)\Vert_{\cB}\;\;\text{ and }\;\;   \Vert  u  \Vert_{\dot{C}^\theta_T\cB}\eqdef \sup_{\substack{0 \leqslant s, t \leqslant T\\|t-s|\leqslant1}}\frac{\Vert  u (t)- u (s)\Vert_{\cB}}{|t-s|^{\frac\theta2}}\,.
 \end{equs}

We write $C_T\cB$ for the Banach space of bounded continuous $\cB$-valued functions on  $[0,T]$ with norm $   \Vert  \bigcdot \Vert_{C_T\cB}$. 

For any $d\geqslant2$ and $x,y \in \R^d$ we use the scalar product $x\cdot y=\sum_{c=1}^d x_cy_c$,   $\ell^2$ norm $ |x|\eqdef \sqrt{x\cdot x}$, $\ell^\infty$ norm $ |x|_\infty\eqdef \max_{i\in[d]}|x_i|$, and the ``bracket'' norm $\ang{x}\eqdef \sqrt{1+x\cdot x}$.
  
We often work on the torus $\T^d=\big(\R/\Z\big)^d$ and denote by $|\cdot|_\infty$ the distance on $\T^{d}$ induced by the $\ell^\infty$ norm on $\R^{d}$. We also write, for $N \in \N$, $\Z^{(d)}_N\eqdef\{k\in\Z^{(d)}:|k|_\infty\leqslant N\}=\{-N,\dots,N\}^{(d)}$.

For any spatial function $ u :\T^d\rightarrow\R$, we use the (spatial) Fourier transform to define a function $\hat{u}$ on $\Z^{d}$ by setting
\begin{equs}
    \hat{ u }_m=\cF_x( u )(m)\eqdef\int_{\T^d}e^{\imath x\cdot m} u (x) \rmd x\,,
\end{equs}
whose inverse is given by 
\begin{equs}
     u (x)=\cF_x^{-1}(\hat{ u })(x)=\sum_{m\in\Z^d}e^{-\imath x\cdot m}\hat{ u }_m\,.
\end{equs}

Note that we only perform Fourier transforms in space, and extend the above notation to space-time functions $u(x,t)$ by writing $\CF_x(u(\cdot,t))(m) = \hat{u}_m(t)$. 

The Littlewood-
Paley blocks $(\Delta^j)_{j\geqslant-1}$ that we define in Appendix~\ref{app:besov} act in Fourier space by multiplication with $\hat\Delta^j_m=\1_{[2^{-1},1)}(2^{-j}|m|_\infty)$ for $i\geqslant0$, which is why for any sequence $v:\Z^d\rightarrow\C$ we introduce the notation
\begin{equs}
    \sum_{m\sim 2^i} v_m\eqdef\sum_{m\in\Z^d}\hat\Delta^i_m v_m =\sum_{m\in\Z^d}\1_{[2^{-1},1)}(2^{-j}|m|_\infty) v_m \,,
\end{equs}
with the usual understanding when $i=-1$.

For $p,q \in [1,\infty]$ and $\alpha \in \R$ we denote by $B^\alpha_{p,q}(\T^d)$  the corresponding Besov spaces on $\T^d$ 
\dash see Appendix~\ref{app:besov}.
In the special cases $(p,q) = (\infty,\infty)$ and $(p,q) = (2,2)$ we denote by $\cC^\alpha\eqdef B^\alpha_{\infty,\infty}$ the corresponding  H\"{o}lder-Besov space and $H^\alpha\eqdef B^\alpha_{2,2}$ the Sobolev spaces. 
We denote the $L^2$ pairing of two functions $ u , v \in L^2(\T^d)$ as
 \begin{equs}
     ( u , v )_{L^2(\T^d)}\eqdef \int_{\T^d} u (x) v (x)\rmd x\,,
 \end{equs}
often dropping the subscript $L^2(\T^d)$. 

For $c\in[d]$, we allow $\cN^c:(f,g,h)\mapsto\cN^c(f,g,h)$ to act on a function $f$ (resp. $h$) depending only on $y_c$ (resp. $y_{\hc}$). When this is the case, $\cN^c(f,g,h)(x,t)$ is independent of $x_{\hc}$ (resp. $x_c$), and we write that $f:\T_c\rightarrow\R$ (resp. $h:\T^{d-1}_{\hc}\rightarrow\R$).

We make frequent use of Kolmogorov estimates that reference various parameters $p,T,\epsilon$ where $\epsilon > 0$ is an exponent drop to turn a statement in expectation into a pathwise one, $p \in 2\N$ is a degree of stochastic integrability, and $T \in (0,1]$ is a finite time cut-off on the relevant process. 
It is always implied that $p,T > 0$ can be taken arbitrarily large and $\epsilon > 0$ arbitrarily small, with this only changing implicit or explicit constants.   

We introduce the shorthand $\cI(\bigcdot)\eqdef\Vert\bigcdot\Vert^4_{M^4(\T^d)}$ for the tensor interaction and, for $c\in[d]$, write $\cI^c(\bigcdot)\eqdef\Vert\bigcdot\Vert^4_{M^4_c(\T^d)}$.
We fix $\epsilon > 0$, which later will be taken sufficiently small, and write $\cK(\bigcdot)\eqdef\Vert\bigcdot\Vert^2_{H^{1-\epsilon}(\T^d)}$ for kinetic term (the Gaussian action). 
 
For $t\geqslant0$, we denote by $P_t\eqdef e^{-t(1-\Delta)}$ the massive heat operator on $\T^{d}$, we view $P_{t}$ as an operator on functions\slash distributions over space.
For nice enough space-time distributions $F$ we write
\[
\cL^{-1}F(x,t) \eqdef\int_0^t P_{t-s}(F(\cdot,s)\big)(x) \mathrm{d}s\;,\;
\underline{\cL}^{-1}F(x,t)  \eqdef\int_{-\infty}^t P_{t-s}\big(F(\cdot,s)\big)(x)  \mathrm{d}s\;.
\]

\subsubsection{Some graphical notation}\label{subsec:graphical_notation}
We use a variety of graphical representations in the article, some of which might be unfamiliar even if one is familiar with diagrams in local field theories. 
A more careful explanation of our graphical notation can be found in Section~\ref{subsubsec:strandedgraphs}.

Suppose $\phi \in C^{\infty}(\T^d)$, and let us write $\hat{\phi}_{m} = \hat{\phi}(m)$ viewing the Fourier mode as an index. 
Then we can write the local interaction as 
 \begin{equ}\label{eq:fourier_local_vertex}
 \Vert \phi\Vert^4_{L^4}=\sum_{m,n,p,q\in\Z^d}\hat{\phi}_{m}\hat{\phi}_{n}\hat{\phi}_{p}\hat{\phi}_{q} \delta_{m+n+p+q,0}\;.
 \end{equ}
In Figure~\ref{fig:local_vertex} we draw the local $ \Vert \phi\Vert^4_{L^4}$ interaction\slash vertex as a diagram on the left (in Fourier space) and the corresponding dynamical\slash stochastic vertex $\phi^{3}(x)$ on the right (in direct space). 
\begin{figure}
\centering
    \tikzsetnextfilename{fig2}
    \begin{tikzpicture}
\draw[densely dotted, thick](0,0)--(1,0);
\node[black] at (-.3,.0){$m$};
\node[black] at (.5,.7){$n$};
\node[black] at (1.25,.0){$p$};
\node[black] at (.5,-.7){$-(m+n+p)$};
\draw[densely dotted, thick](0.5,-.5)--(.5,.5);
\draw[densely dotted, thick](3,-.25)--(3,.5);
\draw[densely dotted, thick](3,-.25)--(2.5,.235);
\draw[densely dotted, thick](3,-.25)--(3.5,.235);
\node[black] at (3,-.4){$x$};
\node[black] at (3,.6){$\phi$};
\node[black] at (3.5,.4){$\phi$};
\node[black] at (2.5,.4){$\phi$};
\end{tikzpicture}
\caption{Local vertices}\label{fig:local_vertex}
\end{figure}


In Figure~\ref{fig:local_vertex} the dashed lines on the left correspond to fields $\phi$, and the vertex in the middle represents the delta function imposed on the sum of incoming momentum from the dashed lines. 

For $c \in [d]$ and $m\in\Z^d$, we write $m_{\hc}=\chi^c(m,0)$. We have the following formula for the monochrome interaction of the non-local theory
\begin{equ}\label{eq:fourier_nonlocal_vertex}
\cI^c(\phi)=
\sum_{m,n,p,q \in \Z^{d}}\delta_{m_{c},-n_{c}}
\delta_{p_{c},-q_{c}}
\delta_{m_{\hc},-q_{\hc}}
\delta_{n_{\hc},-p_{\hc}}
\hat{\phi}_{m}\hat{\phi}_{n}\hat{\phi}_{p}\hat{\phi}_{q}\,.
\end{equ}
The corresponding vertex for the non-local interaction $\cI^c(\phi)$ is drawn on the left below (in Fourier space) and we draw the dynamical  \slash stochastic vertex $\cN^{c}(f,g,h)(x)$ on the right (in direct space). 
\begin{figure}
\centering 
 \tikzsetnextfilename{fig3}
 \begin{tikzpicture}
 \draw[gray] (0,-1)--(0,0);
\draw[gray] (1,-1)--(1,0);
\draw[gray] (0,-1)--(1,-1);
\draw[gray] (0,0)--(1,0);
\draw[gray] (0,-1)..controls(.3,-.8)and(.7,-.8)..(1,-1);
\draw[gray] (0,0)..controls(.3,-.2)and(.7,-.2)..(1,0);
\draw[gray] (0,0)..controls(.3,.2)and(.7,.2)..(1,0);
\draw[gray] (0,-1)..controls(.3,-1.2)and(.7,-1.2)..(1,-1);
\draw[thick,densely dotted](0,0)--(-.3,.3);
\draw[thick,densely dotted](1,0)--(1.3,.3);
\draw[thick,densely dotted](0,-1)--(-.3,-1.3);
\draw[thick,densely dotted](1,-1)--(1.3,-1.3);
\node[gray] at (1.12,-.5){$c$};
\node[gray] at (-.12,-.5){$c$};
\node[black] at (1.3,-1.5){$m$};
\node[black] at (1.3,.5){$(-n_{\hc},-m_c)$};
\node[black] at (-.3,.5){$n$};
\node[black] at (-.3,-1.5){$(-m_{\hc},-n_c)$};
\draw[gray] (4,-1)--(4,0);
\draw[gray] (5,-1)--(5,0);
\draw[gray] (4,-1)--(5,-1);
\draw[gray] (4,0)--(5,0);
\draw[gray] (4,-1)..controls(4.3,-.8)and(4.7,-.8)..(5,-1);
\draw[gray] (4,0)..controls(4.3,-.2)and(4.7,-.2)..(5,0);
\draw[gray] (4,0)..controls(4.3,.2)and(4.7,.2)..(5,0);
\draw[gray] (4,-1)..controls(4.3,-1.2)and(4.7,-1.2)..(5,-1);
\draw[thick,densely dotted](4,0)--(3.7,.3);
\draw[thick,densely dotted](5,0)--(5.3,.3);
\draw[thick,densely dotted](4,-1)--(3.7,-1.3);
\node[black] at (3.6,.-1.5){$f$};
\node[black] at (3.6,.4){$g$};
\node[black] at (5.4,.45){$h$};
\node[gray] at (5.12,-.5){$c$};
\node[gray] at (3.88,-.5){$c$};
\node[black] at (5.2,-1.2){$x$};
\end{tikzpicture} 
\caption{Non-local vertices}\label{fig:non_local_vertex}
\end{figure}
For the picture on the right, we have labeled the dashed lines corresponding to fields $f,g,h$ to clarify the asymmetric nature of the vertex.

One way to introduce a diagrammatic approach for the tensor field theory is to start with the graphs of the local theory and replace the local vertices of Figure~\ref{fig:local_vertex}
 with the colored non-local vertices of Figure~\ref{fig:non_local_vertex} \dash we will call the result graphs tensor graphs.

Note that the non-local vertex has less symmetries than the local vertex, in particular a single Feynman graph in the local theory gives rise to many different graphs in the non-local theory.  
\begin{definition}
We define the \textit{melonic pairing} of $\phi$ and $\psi$ as $\cN^c(\bigcdot,\phi,\psi)$  \dash for $\phi, \psi \in C^{\infty}(\T^{d})$, it can be viewed as an operator $C^{\infty}(\T^d) \ni h \mapsto  \cN^c(h,\phi,\psi) \in C^{\infty}(\T^d)$ (which also acts naturally on $h\in C^\infty(\T_c)$. 
We draw a melonic pairing as
\begin{figure}[H]
    \centering
   \tikzsetnextfilename{fig4}
   \begin{tikzpicture}
\draw[gray] (0,-.5)--(0,0);
\draw[gray] (1,-.5)--(1,0);
\draw[gray] (0,0)--(1,0);
\draw[gray] (0,0)..controls(.3,-.2)and(.7,-.2)..(1,0);
\draw[gray] (0,0)..controls(.3,.2)and(.7,.2)..(1,0);
\node[gray] at (1.12,-.35){$c$};
\node[gray] at (-.12,-.35){$c$};
\node[black] at (0,.25){$\phi$};
\node[black] at (1,.25){$\psi$};
\end{tikzpicture}
    \caption{The melonic pairing of $\phi$ and $\psi$}  
\end{figure}
\noindent
Similarly, $\cN^c(\phi,\psi,\bigcdot)$ is the \textit{non-melonic pairings} of $\phi$ and $\psi$ that acts on $h\in C^\infty(\T^d)$ (and also $h\in C^\infty(\T_{\hc}^{d-1})$) and  $\cN^c(\phi,\bigcdot,\psi)$ is the \textit{exterior pairing} that acts on $h\in C^\infty(\T^d)$. Note that contrarily to the melonic and non-melonic pairings, the exterior pairing can be extended to an operator $C^{\infty}(\T^d) \ni h \mapsto  \cN^c(\phi,h,\psi) \in C^{\infty}(\T^d)$ for $\phi,\psi\in\mcD'(\T^d)$.
\end{definition}
We also talk about melonic pairings in the context of the particular pairings in Wick's formula that lead to divergences, this again refer to pairings that maximize the number of Wick pairings between terms sitting in the second and third arguments of $\cN^{c}(\bigcdot,\bigcdot,\bigcdot)$ \dash see the discussion above \eqref{eq:Wick_constant} and at the beginning of Section~\ref{subsect:secondrenorm}. 
\begin{definition}
Following the discussion in Figure~\ref{fig:powercounting}, we define the primary divergent subgraphs for the $\rmT^4_4$ measure. 
For $c \in [d]$,  $\mfM^{1,c}$ is the melonic tadpole of color $c$ and for $c,c'\in[d]$, $c\neq c'$,  $\mfM^{2,c,c'}$is the melonic snowball of colors $c,c'$. 
The color index can be dropped when the color of the graph is of no importance. They are pictured as
\begin{figure}[H]
    \centering
   \tikzsetnextfilename{fig5}
   \begin{tikzpicture}
\draw[gray] (0,-1)--(0,0);
\draw[gray] (1,-1)--(1,0);
\draw[gray] (0,-1)--(1,-1);
\draw[gray] (0,0)--(1,0);
\draw[gray] (0,-1)..controls(.3,-.8)and(.7,-.8)..(1,-1);
\draw[gray] (0,0)..controls(.3,-.2)and(.7,-.2)..(1,0);
\draw[gray] (0,0)..controls(.3,.2)and(.7,.2)..(1,0);
\draw[gray] (0,-1)..controls(.3,-1.2)and(.7,-1.2)..(1,-1);
\draw[black,thick,densely dotted] (1,0)arc(-45:225:.71);
\draw[black,thick,densely dotted] (0,-1)--(-.3,-1.3);
\draw[black,thick,densely dotted] (1,-1)--(1.3,-1.3);
\node[gray] at (-.12,-.5){$c$};
\node[gray] at (1.12,-.5){$c$};

\draw[gray] (3,0)--(3,1);
\draw[gray] (4,0)--(4,1);
\draw[gray] (3,0)--(4,0);
\draw[gray] (3,1)--(4,1);
\draw[gray] (3,0)..controls(3.3,.2)and(3.7,.2)..(4,0);
\draw[gray] (3,1)..controls(3.3,.8)and(3.7,.8)..(4,1);
\draw[gray] (3,1)..controls(3.3,1.2)and(3.7,1.2)..(4,1);
\draw[gray] (3,0)..controls(3.3,-.2)and(3.7,-.2)..(4,0);
\draw[gray] (3,-2)--(3,-1);
\draw[gray] (4,-2)--(4,-1);
\draw[gray] (3,-2)--(4,-2);
\draw[gray] (3,-1)--(4,-1);
\draw[gray] (3,-2)..controls(3.3,-1.8)and(3.7,-1.8)..(4,-2);
\draw[gray] (3,-1)..controls(3.3,-1.2)and(3.7,-1.2)..(4,-1);
\draw[gray] (3,-1)..controls(3.3,-.8)and(3.7,-.8)..(4,-1);
\draw[gray] (3,-2)..controls(3.3,-2.2)and(3.7,-2.2)..(4,-2);
\draw[black,thick,densely dotted] (4,1)arc(-45:225:.71);
\draw[black,thick,densely dotted] (3,-2)--(2.7,-2.3);
\draw[black,thick,densely dotted] (4,-2)--(4.3,-2.3);
\draw[black,thick,densely dotted] (3,-1)..controls(2.7,-.7)and(2.7,-.3)..(3,0) ;
\draw[black,thick,densely dotted] (4,-1)..controls(4.3,-.7)and(4.3,-.3)..(4,0) ;
\node[gray] at (2.88,-1.5){$c$};
\node[gray] at (4.12,-1.5){$c$};
\node[gray] at (2.85,.5){$c'$};
\node[gray] at (4.15,.5){$c'$};
\node[black] at(.5,-2){$\mfM^{1,c}$};
\node[black] at(3.5,-3){$\mfM^{2,c,c'}$};
\end{tikzpicture}
    \caption{The two primary divergent graphs $\mfM^{1,c}$ and $\mathfrak \mfM^{2,c,c'}$}
\end{figure}
\end{definition}
\section{ Local well-posedness for dynamical \TitleEquation{\rmT^4_3}{T43}  }\label{sec:3d}
In this section we prove the following local (in both time and space) well-posedness result for the parabolic stochastic quantization of the $\rmT^4_{d}$ equation for $d\in\{2,3\}$. 
This will also be a warm-up for describing some of the power-counting and stochastic estimates needed for tensor field theories before we treat the $\rmT^4_4$ model.

Recall from the introduction that we define a regularized noise $\xi_N\eqdef \Pi_{N}\xi$ where $\Pi_{N}$ is the Fourier multiplier that projects onto the Fourier modes of size lower or equal to $N$. We introduce the stationary (in both space and time) process $\X_{N}$ over $\T^{d} \times \R$ given by 
\[  
\X_{N}\eqdef\sqrt{2}\,\underline{\cL}^{-1}\xi_{N}\,.
\]
\begin{theorem}\label{thm:3d}
For $d = 2,3$,  there exist constants $a_{N}$ such that one has uniform in $N$ control of the local in time solutions to \eqref{eq:regTensorLD} with initial condition of the form $\X_N(0)+ \Pi_{N}v(0)$ with $v(0) \in \cC^{\frac{3}{2}-\epsilon}$. In particular, using the ansatz $\phi=\X_N+v$ for the equation with cut-off $N$, one has that $v$ converges in 
$C\big([0,\Bar{T}),\cC^{\frac{3}{2}-\epsilon}(\T^3)\big)$ where $\Bar{T} \in (0,\infty]$ is a random blow-up time. 
\end{theorem}

As mentioned earlier the local theory for the dynamical $\rmT^4_3$ is quite similar to that of dynamical $\Phi^4_2$, which will allow us to use an analog of the Da Prato - Debussche \cite{DPD2} argument. The idea of the argument is that $\lim_{N\uparrow \infty} \X_{N} = \X$ is rough, but explicit and Gaussian, while $v$ is an inexplicit remainder that should have better regularity. One then derives a remainder PDE for $v$ involving Wick powers/Hermite polynomials of $\X$ (constructed via a probabilistic argument) and show this PDE for $v$ is locally well-posed (with stability as the regularization is removed).

We start with the following standard estimate. 
\begin{lemma}
For any fixed $T > 0$ and $\epsilon > 0$,
$\X_{N}$ converges to a limit $\X = \X_{\infty}$ in $C_T\cC^{-\frac{d-2}{2}-\epsilon}$. 
Moreover, for any $N$ and $t_1,t_2 \in \R_{\geqslant 0}$, 

\begin{equs}  \label{eq:covariance} 
 \E[\X_{N}(t_1)\X_{N}(t_2)] = \Pi_{N} \frac{P_{|t_1-t_2|}}{1-\Delta}\,.
\end{equs}
In particular, for fixed $t$ and as a random field on $\T^{d}$, one has $\X_{\infty}(t)\overset{law}{=}\mcb{g}$. 
\end{lemma}
In what follows we write  $\cN_N^{(c)}\eqdef\Pi_N\cN^{(c)}$. 
A natural first question is how to renormalize $\cN_{N}(\X_{N},\X_{N},\X_{N})$ so it has a meaningful limit as $N \uparrow \infty$. 
In $\Phi^4_d$ the corresponding (Wick) renormalization is $\X_{N}^3(z) \mapsto \X_{N}^3(z) - 3 \E[ \X_{N}^2(z)] \X_{N}(z)$ which cancels the divergence $E[ \X_{N}^2(z)]$ produced when any two of the three copies of $\X_{N}$ Wick contract.

The product $\cN$ is less symmetric, each choice of which two copies of $\X_{N}$ Wick contract, that is $\cN_N(\bigcdot,\X_{N}, \X_{N})$, $\cN_N(\X_{N},\bigcdot, \X_{N})$ or $\cN_N(\X_{N},\X_{N}, \bigcdot)$, should be treated differently. 
The second two terms are well-defined without renormalization as $N\uparrow\infty$ in any dimension $d$, while the first term is melonic \dash it is well-defined for $d=2$ but requires renormalization for all $d \geqslant 3$.

Our ``Wick'' renormalization is given by setting, for any $d \geqslant 2$ and $c \in [d]$, 
\begin{equ}\label{eq:Wick_constant}
   \mfC^{1,c}_{N}(d)\eqdef  \E[\cN_N^c(1,\X_{N}(t),\X_{N}(t))(x)]\;.
\end{equ} 
Note that by stationarity $ \mfC^{1,c}_{N}(d)$ does not depend on $(t,x)$. 

\begin{notation}
Recall that we can view $\cN^{c}$ and $\cN$ as either acting on functions/distributions over space and returning a function over space, or acting on and producing functions/distributions over space-time. 
We jump between the two viewpoints often.
\end{notation}

We can then write the promised renormalized product by setting, for any  $\psi: \T^{d} \rightarrow \R$,  
\begin{equ}\label{eq:renorm_melonic_square}
\cN_N^{c}(\psi,\X_{N}, \X_{N})(x)
-
\mfC^{1,c}_{N}(d) \psi(x)\;.
\end{equ}
The analysis required to prove the above object is well-defined as $N \uparrow \infty$ is more complicated than the corresponding Wick renormalization in the local theory. 
Our cancellation of the divergence of the expectation of $\cN_N^{c}(\psi,\X_{N}, \X_{N})(x)$ is less exact than the local theory since we have replaced $\psi$ with $1$ in \eqref{eq:Wick_constant}.
In \eqref{eq:renorm_melonic_square} we compensate the non-local divergence in the first term with a local counterterm. 
Recall (see Remark~\ref{remark:DPD}) the non-locality of Wick divergences for our model has  We renormalization of this divergence in the following lemma. 
\begin{lemma}\label{lem:CT1}
For $d\geqslant2$ and $c\in[d]$, one has
\[   
\mfC^{1,c}_{N}(d)=\sum_{m_{\hc}\in\Z_N^{d-1}}\frac{1}{\langle m_{\hc}\rangle^2}\,.
\]
In particular $\mfC_{N}^{1,c}(d)$ is independent of $c \in [d]$. 
Writing 
\begin{equ}\label{eq:CT1_eq2}
\mfC^1_{N}(d)\eqdef\sum_{c=1}^d \mfC^{1,c}_{N}(d)=d\mfC^{1,c}_{N}(d)\;,
\end{equ}
we have, for $d\in\{2,3,4\}$ and any $\psi,f \in C^{\infty}(\T^d)$,
\begin{equs}\label{eq:first_renorm_bound}
\sup_{N\in\N} \E[ \big( \cN_N(\psi,\X_{N},\X_{N})-\mfC^1_{N}\psi \big)(f)^2]  &< \infty\\
 \sup_{N\in\N} \E[ \big( \cN_N(\X_{N}, \psi, \X_{N})\big)(f)^2]  \vee \E[ \big( \cN_N(\X_{N},\X_{N}, \psi)\big)(f)^2] &< \infty\;.
\end{equs}
\end{lemma}
We sometimes suppress $d$ from the notation, just writing $\mfC^1_{N}$.
\begin{proof}
We start with \eqref{eq:first_renorm_bound}, but we only sketch this bound in order to demonstrate the need for renormalization. We state more detailed estimates in Lemma~\ref{lem:randop1}, steps skipped over here can be reviewed in the proof of that lemma. 
In what follows we sometimes drop $N$ from our notation and write Fourier variables as subscripts. 

Thanks to hypercontractivity and orthogonality of the homogeneous Gaussian chaoses it suffices to bound the second moments of projections onto each homogeneous chaos
We first deal with $\cN_N(\X,\psi,\X)=\cN_N(\X,\psi,\X)^{(2)}+\cN_N(\X,\psi,\X)^{(0)}$ \dash superscripts indicate the Gaussian chaos. 
As in the local case, super-renormalizability, one doesn't need renormalization to obtain uniform in $N$ bounds on the highest chaos.
Regarding the $\cN_N(\X,\psi,\X)^{(0)}$ \dash the expectation of $\cN_N(\X,\psi,\X)$ \dash we have 
\begin{equs}
    \cN_N^c(\X(t),\psi,\X(t))^{(0)}(x)&=\sum_{m\in\Z_N^d}e^{-\imath x\cdot m}\sum_{n\in\Z^d}
    \hat{\psi}_n\E[
    \hat{\X}_{-m_{\hc},-n_c}(t)\hat{\X}_{-n_{\hc},-m_c}(t)]\\&=
   \sum_{m\in\Z_N^d}e^{-\imath x\cdot m}\sum_{n\in\Z^d}
    \hat{\psi}_n\delta_{m,-n} \frac{1}{\langle m\rangle^2}\\&= \Pi_N(1-\Delta)^{-1}\psi(x)\,,
\end{equs}
which is is clearly well behaved as $N\uparrow \infty$. 
Turning to $\cN_N(\X(t),\X(t),\psi)$, we have
\begin{equs}
   \cN_N^c(\X(t),\X(t),\psi)^{(0)}(x)&=
   \sum_{m\in\Z_N^d}e^{-\imath x\cdot m}\sum_{n\in\Z^d}
    \hat{\psi}_{-n_{\hc},-m_c}\E[\hat{\X}_{-m_{\hc},-n_c}(t)
    \hat{\X}_{n}(t)]\\&=
    \sum_{m\in\Z_N^d}e^{-\imath x\cdot m}\sum_{n\in\Z^d}
    \hat{\psi}_{-n_{\hc},-m_c}\delta_{m_{\hc},n_{\hc}} \frac{\Pi_N(m_{\hc},n_c)}{\langle (m_{\hc},n_c)\rangle^2}\\
    &= \sum_{m\in\Z_N^d}e^{-\imath x\cdot m}\hat{\psi}_{-m}\sum_{n_c\in\Z}
   \frac{\Pi_N (m_{\hc},n_c)}{\langle (m_{\hc},n_c)\rangle^2}\,,
\end{equs}
which poses no problem as $N\uparrow \infty$ since the sum over $n_c$ above is convergent. 
Now we turn to $\cN_N(\psi,\X(t),\X(t))$ where we see the need for renormalization. 
\begin{equs}\label{eq:first_renorm_comp}
     \cN_N^c(\psi,\X(t),\X(t))^{(0)}(x)&=
  \sum_{m\in\Z_N^d}e^{-\imath x\cdot m}\sum_{n\in\Z^d}
    \hat{\psi}_{-m_{\hc},-n_c}\E[ \hat{\X}_{n}(t)\hat{\X}_{-n_{\hc},-m_c}(t)
   ]
   \\&=
   \sum_{m\in\Z_N^d}e^{-\imath x\cdot m}\sum_{n\in\Z^d}
    \hat{\psi}_{-m_{\hc},-n_c}\delta_{m_{c},n_{c}} \frac{\Pi_N (n_{\hc},m_c)}{\langle (n_{\hc},m_c)\rangle^2}\\
    &=\sum_{m\in\Z_N^d}e^{-\imath x\cdot m}\hat{\psi}_{-m}\sum_{n_{\hc}\in\Z^{d-1}}
   \frac{\Pi_N (n_{\hc},m_c)}{\langle (n_{\hc},m_c)\rangle^2}\,.
\end{equs}
Observe that in $d\geqslant3$ the sum over $n_{\hc}$ becomes divergent when we take  $N\uparrow\infty$. 
We renormalize to compensate this divergence, and since we want to accomplish this with a local counter-term we only subtract the divergent term evaluated at $m_c=0$ rather cancelling the whole quantity above. 
Indeed, subtracting $\mfC^{1,c}_{N}\psi$ yields
\begin{equs}\label{eq:renorm1}
     \cN_N^c(\psi,\X_{N},\X_{N})^{(0)}(x)-\mfC^{1,c}_{N}\psi(x)&=
    \sum_{m\in\Z_N^d}e^{-\imath x\cdot m}\hat{\psi}_{-m}\sum_{n_{\hc}\in\Z^{d-1}}\mathfrak R^{1,c}_{N}(\chi^c(n,m))\,,
     \end{equs}
where 
\begin{equ}\label{eq:R1}
    \mfR^{1,c}_{N}(m)\eqdef 
   \frac{\Pi_N( m)}{\langle m\rangle^2}-\frac{\Pi_N( m_{\hc})}{\langle m_{\hc}\rangle^2}\;.
\end{equ}
We see that \eqref{eq:renorm1} is uniformly bounded as we take $N\uparrow \infty$ when $d\leqslant4$. 

Finally, the formula for $\mfC^1_{N}$ in \eqref{eq:CT1_eq2} directly comes from inserting  $\psi=1$ in \eqref{eq:first_renorm_comp} since the Fourier transform of $1$ is $\delta_{m,0}$.
\end{proof}
\begin{remark}
$\mfC^1_\infty$ is finite when $d=2$ and we see that the dynamic $\mathrm T^4_2$ model, while being singular\footnote{$\cN(\phi,\phi,\phi)$  is not defined for $ \phi \in \cC^{0-}$ so the equation is indeed singular, but this term can be made sense of when  $\phi = \X$ with a stochastic estimate without renormalization.}, does not require renormalization.   
$\mfC^1_{N}$ diverges like $ \log N$ when $d=3$ and like $N^{d-3}$ when $d \geqslant 4$. 
Viewing  $\mfC^1_{N}$ as the analog of the ``Wick constant'' in the local theory,  this is consistent with the idea that divergences in the $\rmT^4_d$ model resemble those in the $\Phi^4_{d-1}$ model for $d \in \{3,4,5\}$.  
\end{remark}

We introduce some symbolic notation, writing, for $N \in\N$, 
\begin{equ}\label{def:cube}
\Xthreeloc_{N} = \Xthree_{N} 
\eqdef
\cN(\X_{N},\X_{N},\X_{N}) - \mfC^1_{N} \X_{N}\;.
\end{equ}
In Lemma~\ref{prop:stochastic_1} we prove that, for $d \in \{3,4\}$, $\Xthreeloc_{N}$ converges to a limiting distribution over space-time as $N \uparrow \infty$. Note that $\Xthreeloc_{N}$ is not in the homogeneous third Gaussian chaos associated to $\xi$, since only one of the three Wick self-contractions is canceled, and it is not even fulled canceled.


With this notation at hand, we are going to study the following regularized equation:

\begin{equs}
\label{eq:T^4_3}
    \cL\phi_{N}=-\cN_N(\phi_{N},\phi_{N},\phi_{N})+\mfC^1_{N}\phi_{N}+\sqrt{2}\xi_{N}\,,\,\,\phi_{N}(0)=\X_{N}(0)+\Pi_Nv_0\,.
\end{equs} 
Given that we can expect that the noise is the most singular term of the equation, in order to solve the equation, we first have to perform an expansion around the solution to the linear equation, which is known as Da Prato-Debussche method. Let $v_{N}\eqdef \phi_{N}-\X_{N}$ that solves
\begin{equs}\label{eq:DPD1}\nonumber
    \cL v_{N}= &- \cN_{N}(v_{N},v_{N},v_{N})-\cN_{N}(\X_{N},v_{N},v_{N}) -\cN_{N}(v_{N},\X_{N},v_{N}) \\
    &-\cN_{N}(v_{N},v_{N},\X_{N})-\cN_{N}(\X_{N},\X_{N},v_{N})-\cN(\X_{N},v_{N},\X_{N})\\&-\big(\cN_{N}(v_{N},\X_{N},\X_{N})-\mfC_{N}^1 v_{N}\big)-\big(\cN_{N}(\X_{N},\X_{N},\X_{N})-\mfC_{N}^1 \X_{N}\big)\,,
\end{equs}
with initial condition $\Pi_Nv_0$. 
In the rest of the section, the cut-off $N$ is often suppressed from notation \dash for instance we write $\cN(v,\X,v)$ to denote $\cN_{N}(v_{N},\X_{N},v_{N})$. 
However, all estimates in this section will be uniform in $N$.  
We now state regularity estimates for the stochastic objects appearing in \eqref{eq:DPD1}, all of our proofs of stochastic estimates are deferred to Section~\ref{sec:stochasticobjects}. 

\begin{lemma}[Random fields 1]\label{prop:stochastic_1}
 Let
\begin{equs}
 \Xtwom^c_{N}(x_c,y_c)& \eqdef \int_{\T^{d-1}}\X_{N}(y_{\hc},y_c)\X_{N}(y_{\hc},x_c)\rmd y_{\hc}-\mfC^{1,c}_{N}(d)\delta(x_c-y_c)\,,\\
\Xtwonm^c_{N}(x_{\hc},y_{\hc})&\eqdef \int_{\T}\X_{N}(x_{\hc},y_c)\X_{N}(y_{\hc},y_c)\rmd y_{c}\,, 
\quad  \enskip
\Xthree_{N}\eqdef \cN_N(\X_{N},\X_{N},\X_{N})-\mfC^1_{N}(d)\X_{N}\,.
\end{equs}
Then, in any dimension $d\geqslant2$, 
\begin{equ}
 \sup_{N\in\N}\E[\Vert \Xtwonm^c_{N}\Vert_{C_T\cC^{-\frac{2d-5}{2}-\epsilon}(\T^{2d-2})}^p]<\infty\,.
\end{equ}
Furthermore, for $d\in\{2,3,4\}$,
\begin{equs}
 \label{reg:nonlocal_wicksquare}  \sup_{N\in\N} &\E[\Vert \Xtwom^c_{N}\Vert_{C_T\cC^{-\frac{2d-5}{2}-\epsilon}(\T^2)}^p] <\infty\,,\\  
\label{reg:nonlocal_wickcube} \sup_{N\in\N} &\E[\Vert \Xthree_{N}\Vert_{C_T\cC^{-\frac{3d-8}{2}-\epsilon}(\T^{d})}^p]<\infty\,.
\end{equs}
\end{lemma}
\begin{remark}
Comparing \eqref{reg:nonlocal_wicksquare} and \eqref{reg:nonlocal_wickcube} with local \slash pointwise products, we see that the regularity in \eqref{reg:nonlocal_wicksquare} shows an improvement of $1/2$, while that of \eqref{reg:nonlocal_wickcube} is improved by $1$. 
 \end{remark}
\begin{notation}
We often use a pictorial representation of the non-linearity $\cN(v,v,v)$ by writing $\vvv$, and will use the same notation for the mixed terms of $v$ and the stochastic objects. 
For instance, We let $\Xvv$ stand for $\cN(\X,v,v)$, and $\XXv$ for $\cN(\X,\X,v)$. 
We stress the fact that with this notation we let $\vXX=\cN(v,\X,\X)-\mfC^1v=\sum_{c=1}^d\int_{\T}\Xtwom^c(x_c,y_c)v(x_{\hc},y_c)\rmd y_c$.
\end{notation}
We can then rewrite \eqref{eq:DPD1} as
\begin{equs}\label{eq:DPD2}
    \cL v=-\vvv-\Xvv-\vXv-\vvX-\XXv-\XvX-\vXX-\Xthree\,,
\end{equs}
with initial condition $v(0)=\Pi_Nv_0$.

\subsection{Random operators for \TitleEquation{\rmT^4_3}{T43}}

In $d=3$, the fact that $\X$ is of regularity  $-\frac{1}{2}-$ implies, via deterministic estimates, that the same can be inferred for $\Xtwom$ and $\Xtwonm$ and that  
\begin{equ}\label{eq:T43_mixedterms}
\cN(\X, v, v)\;,\enskip
\cN(v,\X, v)\;,\enskip
\cN(v, v,\X)\;,\enskip
\cN(v,\X,\X) -\mfC^1v\,
\text{  and  }
\cN(\X,\X,v)
\end{equ} 
are all well-defined in the limit as $N \uparrow \infty$ as long as $v$ can be taken to be of regularity better than $\frac{1}{2}$, and are of regularity at worst $-\frac12-$.
However, the term $\cN(\X,v,\X) = \XvX$ remains problematic. 
We get for free that the exterior product of two copies of the free field, denoted $\X\otimes\X$, is well-defined without any renormalization but we cannot assume regularity better than $-1-$.
Feeding only this into a deterministic argument would give $\XvX$ of regularity $-1-$ which means the equation for $v$ could only be solved in $C_T\cC^{1-\epsilon}$, but pairing $v$ with $\X\otimes\X$ to obtain $\cN(\X,v,\X)$ requires $v$ of regularity better than $1$. 

The same issue is encountered in dynamical $\Phi^4_3$ where the product $\cherry  v$ is problematic \dash it requires that $v$ is of regularity greater than $1$, which is hindered by the regularity $-1-$ of $\cherry $. 
However we can overcome this issue for $\rmT^4_3$ with much less work - since the roughness of the two free fields appearing is (in law) isotropic, we could hope the integration in the non-local product $\XvX$ compensates their irregularity. 
We make this intuition rigorous by promoting $\X\otimes \X$ to a random operator, and viewing $\XvX$ as this random operator acting on $v$. 
To simplify our presentation, we adopt the same language for all the mixed terms in \eqref{eq:T43_mixedterms}.

\begin{lemma}[Random operators 1]\label{lem:randop1}
For $k\in [d-1]$, define linear random operators $\X^{(k)}$:
\begin{align*}  
 \X^{(k)}_{N}(f)&\eqdef \int_{\T^{d-k}} \X_{N}(\cdot,y) f(y)\rmd y\,,&&\text{ for }f:\R_{\geqslant 0}\times\T^{d-k}\rightarrow\R\,.
\end{align*}

Then, in any dimension $d\geqslant2$ and for all $\alpha>\frac{d-k-2}{2}$,
\begin{equ}    \label{eq:Xd3}
\sup_{N\in\N}  \E[\Vert\X^{(k)}_{N}\Vert^p_{\cL(C_TH^{\alpha}(\T^{d-k}),C_T\cC^{\beta-\epsilon}(\T^k))}]<\infty\,
\text{ with $\beta=\min\big(-\frac{k-2}{2},\alpha-\frac{d-2}{2}\big)$}\;.
\end{equ}
We also define for $c\in[d]$ the three quadratic random operators 
\begin{align*}
    \Xtwom^c_{N}(f)&\eqdef \cN^c_N(f,\X_{N},\X_{N})-\mfC^{1,c}_N(d)f
    \,,&&\text{ for }f:\R_{\geqslant 0}\times\T_c\rightarrow\R\,,
    \\ \Xtwonm^c_{N}(f)&\eqdef \cN^c_N(\X_{N},\X_{N},f)\,,&&\text{ for }f:\R_{\geqslant 0}\times\T_{\hc}^{d-1}\rightarrow\R\,,\\ \XdotX_{N}(f)&\eqdef \cN_N(\X_{N},f,\X_{N})\,,&&\text{ for }f:\R_{\geqslant 0}\times\T^{d}\rightarrow\R\,.
\end{align*}
Then, in any dimension $d\geqslant2$, for all $c\in[d]$ and $\alpha>\frac{d-3}{2}$,
\begin{equs}
\sup_{N\in\N} & \E[\Vert\Xtwonm^c_{N}\Vert^p_{\cL(C_TH^{\alpha}(\T^{d-1}),C_T\cC^{\beta-\epsilon}(\T^{d-1}))}]<\infty  \text{ with $\beta=\min\big(-\frac{d-3}{2},\alpha-\frac{2d-5}{2}\big)$}\,,\\
\sup_{N\in\N} &  \E[\Vert\XdotX_{N}\Vert^p_{\cL(C_TH^{\alpha}(\T^d),C_T\cC^{\beta-\epsilon}(\T^d))}]
    <\infty \text{ with $\beta=\min\big(-\frac{d-3}{2},\alpha-(d-2)\big)$}\,.\label{eq:XdotX}
\end{equs}  
Finally, for $d\in\{2,3,4\}$, for all $c\in[d]$ and $\alpha>-\frac12$,
\begin{equ}  \label{eq:XvXd3}   \sup_{N\in\N}\E[\Vert\Xtwom^c_{N}\Vert^p_{\cL (C_TH^{\alpha}(\T),C_T\cC^{\beta-\epsilon}(\T))}]<\infty \text{ with $\beta=\min\big(\frac12,\alpha-\frac{2d-5}{2}\big)$}\,.
\end{equ}
\end{lemma}
\begin{notation}\label{not:cnotc} 
From now on, the notations $\Xtwom^c$ and $\Xtwonm^c$ always stand for the random operators (and not the random fields defined in Lemma~\ref{prop:stochastic_1}), and we sometimes omit the superscript $c$ when it is understood that the color $c$ is summed on, writing for instance $\Xtwom=\sum_{c\in[d]}\Xtwom^c$. Overloading the notation, we will write
\begin{equs}    \|\Xtwom\|_{\cL (C_TH^{\alpha}(\T),C_T\cC^{\beta-\epsilon}(\T))}=\max_{c\in[d]}\|\Xtwom^c\|_{\cL (C_TH^{\alpha}(\T),C_T\cC^{\beta-\epsilon}(\T))}\,,    
\end{equs}
and
\begin{equs}\|\Xtwonm\|_{\cL(C_TH^{\alpha}(\T^{d-1}),C_T\cC^{\beta-\epsilon}(\T^{d-1}))}=\max_{c\in[d]}\|\Xtwonm^c\|_{\cL(C_TH^{\alpha}(\T^{d-1}),C_T\cC^{\beta-\epsilon}(\T^{d-1}))}\,.
\end{equs}
\end{notation}
\begin{remark}
   The previous estimates show that while $\Xtwom$ and $\Xtwonm$ benefit the $+1/2$ regularizing effect of the non-local product, this is not true of $\XdotX$. 
We obtain the stochastic regularizing effect of the nonlinearity $\cN(f,g,h)$ only when two consecutive arguments of $\cN$ \dash that is $(f,g)$ or/and $(g,h)$ \dash are explicit stochastic objects.
   In $d=4$ this makes $\XdotX$ especially problematic (more so than $\Xtwom$ or $\Xtwonm$).
\end{remark}
\subsection{Closing the fixed point problem}
We now formulate \eqref{eq:DPD2} as a fixed point problem in $C_T\cC^{\frac{3}{2}-}$. 
We will use the following lemma regarding mixed terms involving the solution and stochastic objects. 
\begin{lemma}\label{lem:mixednonlin}
Let $d\in\{3,4\}$. The mixed nonlinearities obey the following bounds:
\begin{equs}
     \Vert \Xvv\Vert_{C_T\cC^{-\frac{1+\epsilon}{2}}}&\lesssim
      \Vert \X^{(d-1)}\Vert_{\cL(C_TH^{\frac32-3\epsilon}(\T),C_T\cC^{-\frac{1}{2}-\frac{\epsilon}{4}}(\T^{d-1}))} \Vert v\Vert_{C_TL^\infty}\Vert v\Vert_{C_T\cC^{\frac{3}{2}-2\epsilon}}\,,\label{eq:upXvv}\\
       \Vert \vXv\Vert_{C_T\cC^{-\frac{1+\epsilon}{2}}}&\lesssim
      \Vert \X^{(1)}\Vert_{\cL(C_TH^{\frac32-3\epsilon}(\T^{d-1}),C_T\cC^{-\frac{1}{2}-\frac{\epsilon}{4}}(\T))} \Vert v\Vert_{C_TL^\infty}\Vert v\Vert_{C_T\cC^{\frac{3}{2}-2\epsilon}}\,,\label{eq:upvXv}
      \\
 \Vert \vvX\Vert_{C_T\cC^{-\frac{1+\epsilon}{2}}}&\lesssim
      \Vert \X^{(1)}\Vert_{\cL(C_TH^{\frac32-3\epsilon}(\T^{d-1}),C_T\cC^{-\frac{1}{2}-\frac{\epsilon}{4}}(\T))} \Vert v\Vert_{C_TL^\infty}\Vert v\Vert_{C_T\cC^{\frac{3}{2}-2\epsilon}}\,,\label{eq:upvvX}
      \\
 \Vert \XXv \Vert_{C_T\cC^{-\frac{1+\epsilon}{2}}}&\lesssim \Vert\Xtwonm\Vert_{\cL(C_TH^{\frac32-3\epsilon},C_T\cC^{-\frac{1}{2}-\frac{\epsilon}{4}})}\Vert 
 v\Vert_{C_T\cC^{\frac{3}{2}-2\epsilon}}
 \,,\label{eq:upXXv}    
 \\
 \Vert \vXX \Vert_{C_T\cC^{-\frac{1+\epsilon}{2}}}&\lesssim \Vert\Xtwom\Vert_{\cL(C_TH^{\frac32-3\epsilon},C_T\cC^{-\frac{1}{2}-\frac{\epsilon}{4}})}\Vert 
 v\Vert_{C_T\cC^{\frac{3}{2}-2\epsilon}}
 \,.\label{eq:upvXX} 
\end{equs}
The $2\epsilon$'s in the regularity of $v$ in the RHS are here for convenience, for a reason that will become clear in when we study $\rmT^4_4$.\end{lemma}
\begin{proof}
We only prove the inequalities for one term quadratic in $v$, $\Xvv$,  and one term linear in $v$, $\XXv$. The proofs for the three other terms are very similar (taking into account the renormalization in $\vXX$) and left to the reader. 
The bounds here rely on the random operators estimates from Lemma~\ref{lem:randop1} as well as the bilocal Besov embeddings given in Appendix~\ref{app:bilocbesov}, equations~\eqref{eq:bilocbesov+} and \ref{eq:bilocbesov-}. 

We first prove \eqref{eq:upXvv}. Using the embedding $L^\infty_{x_{c}}\cC_{x_{\hc}}^{-\frac{1}{2}-\frac{\epsilon}{4}}\hookrightarrow\cC_{x_c,x_{\hc}}^{-\frac{1+\epsilon}{2}}$ given in \eqref{eq:bilocbesov-}, we have 
\begin{equs}
    \Vert \Xvv\Vert_{C_T\cC^{-\frac{1+\epsilon}{2}}}&=\sum_{c=1}^d\Vert \cN^c(\X,v,v)(x_c,x_{\hc}) \Vert_{C_T\cC_{x_c,x_{\hc}}^{-\frac{1+\epsilon}{2}}}\\&\lesssim
\sum_{c=1}^d\Vert \cN^c(\X,v,v)(x_c,x_{\hc}) \Vert_{C_TL^\infty_{x_{c}}\cC_{x_{\hc}}^{-\frac{1}{2}-\frac{\epsilon}{4}}}    \\
&\lesssim \Vert v\Vert_{C_TL^\infty}
\sum_{c=1}^d\Vert\X(v(\cdot,y_{\hc}))(x_{\hc}) \Vert_{C_TL^\infty_{y_{\hc}}\cC_{x_{\hc}}^{-\frac{1}{2}-\frac{\epsilon}{4}}} \,.   
\end{equs}
Here, we can use the estimate \eqref{eq:Xd3} on the random tensor $\X^{(d-1)}$:
\begin{equs}
     \Vert \Xvv\Vert_{C_T\cC^{-\frac{1+\epsilon}{2}}}&\lesssim \Vert \X^{(d-1)}\Vert_{\cL(C_TH^{\frac32-3\epsilon},C_T\cC^{-\frac{1}{2}-\frac{\epsilon}{4}})} \Vert v\Vert_{C_TL^\infty}\Vert v(y_c,y_{\hc})\Vert_{C_TL^\infty_{y_{\hc}}\cC_{y_{c}}^{\frac{3}{2}-3\epsilon}}\,.
\end{equs}
Finally, we can conclude using the embedding $\cC^{\frac{3}{2}-2\epsilon}_{y_c,y_{\hc}}\hookrightarrow L^\infty_{y_{\hc}}\cC_{y_{c}}^{\frac{3}{2}-3\epsilon}$ \eqref{eq:bilocbesov+}.

Let us now deal with \eqref{eq:upXXv}. Once again, using the embedding $L^\infty_{x_{c}}\cC_{x_{\hc}}^{-\frac{1}{2}-\frac{\epsilon}{4}}\hookrightarrow\cC_{x_c,x_{\hc}}^{-\frac{1+\epsilon}{2}}$ yields:
\begin{equs}
    \Vert &\XXv \Vert_{C_T\cC^{-\frac{1+\epsilon}{2}}}=\sum_{c=1}^d\Vert \cN^c(\X,\X,v)(x_c,x_{\hc}) \Vert_{C_T\cC_{x_c,x_{\hc}}^{-\frac{1+\epsilon}{2}}}  \\
&\lesssim
\sum_{c=1}^d\Vert \cN^c(\X,\X,v)(x_c,x_{\hc}) \Vert_{C_TL^\infty_{x_{c}}\cC_{x_{\hc}}^{-\frac{1}{2}-\frac{\epsilon}{4}}}  \lesssim
\sum_{c=1}^d\Vert\Xtwonm^c(v(x_{c},\cdot))(x_{\hc}) \Vert_{C_TL^\infty_{x_{c}}\cC_{x_{\hc}}^{-\frac{1}{2}-\frac{\epsilon}{4}}} \,.
\end{equs}
Using \eqref{eq:boundpXtwonm}, we have
\begin{equs}
     \Vert \XXv \Vert_{C_T\cC^{-\frac{1+\epsilon}{2}}}&\lesssim \Vert\Xtwonm\Vert_{\cL(C_TH^{\frac32-3\epsilon},C_T\cC^{-\frac{1}{2}-\frac{\epsilon}{4}})}\Vert v(x_c,y_{\hc})\Vert_{C_TL_{x_c}^\infty\cC_{y_{\hc}}^{\frac32-3\epsilon}}\,,
\end{equs}
and we can conclude using \eqref{eq:bilocbesov+}.
\end{proof}

Define $\bbX^{\LD}_{f,3}$ the tuple of all the random fields defined in Lemma~\ref{prop:stochastic_1} and $\bbX^{\LD}_{o,3}$ the tuple of all the random operators defined in Lemma~\ref{lem:randop1}. 
We define our enhanced noise in $d=3$, $\bbX_3^{\LD}=\bbX^{\LD}_{f,3}\sqcup\bbX^{\LD}_{o,3}$ to be the concatenation of these tuples. It is an element of the product of the Banach spaces in which the random fields and random operators live. In particular, we introduce the following norm:
\begin{equs}\label{eq:normbbX}
     \Vert \bbX_3^{\LD} \Vert_{T}=\max\Big(  \max_{\tau\in\bbX^{\LD}_{f,3}}\Vert \tau\Vert_{C_T\cC^{\beta_\tau-\epsilon}},\max_{\tau\in\bbX^{\LD}_{o,3}}\Vert \tau\Vert_{\cL(C_TH^{\alpha_\tau},C_T\cC^{\beta_\tau-\epsilon})} \Big)\,,
 \end{equs}
where for any stochastic object $\tau$, $\alpha_\tau$ and $\beta_\tau$ are the inner and outer regularities of $\tau$ as stated in Lemmas~\ref{prop:stochastic_1} and \ref{lem:randop1}. Note that if $T<\infty$, then $\Vert\bbX_3^{\LD}\Vert_T$ is in $L^p(\Omega)$ for all $1\leqslant p<\infty$. We write $A\lesssim_{\bbX^{\LD}_3}B$ whenever there exists two positive constants $C$ and $c$ such that $A \leqslant C \Vert \bbX^{\LD}_3\Vert_T^c B$. We this notation, we have the following proposition: 
\begin{proposition}\label{prop:3dfixedpoint}
For any $v_0\in\cC^{\frac32-\epsilon}(\T^3)$, there exists a random blow-up time $\Bar{T}$ such that equation \eqref{eq:DPD2} with initial condition $v(0)=\Pi_Nv_0$ admits uniformly in $N$ a unique solution $v$ in $C\big([0,\Bar{T}),\cC^{\frac32-\epsilon}(\T^3)\big)$. This solution is maximal, and we have $\lim_{t\uparrow \Bar{T}} \Vert v(t)\Vert_{\cC^{\frac32-\epsilon}} =+\infty$. Moreover, if $T<\Bar{T}$, the solution on $[0,T]$ depends continuously on $x$ and on the enhanced noise set $\bbX^{\LD}_3$ (w.r.t the topology of $\Vert  \cdot \Vert_{T}$). 
\end{proposition}
\begin{proof}
    Let us consider the fixed point map of \eqref{eq:DPD2} in its mild formulation:
   \begin{equs}
       \Phi(v)&=P_tx-\int_0^tP_{t-s}\big(\vvv+\Xvv+\vXv+\vvX+\XXv+\XvX+\vXX-\Xthree\big)(s)\rmd s\,.
   \end{equs}
   We aim to show that it maps a ball of $C_T\cC^{\frac{
    3
    }{2}-\epsilon}$ to itself, and that it is a contraction. While the proof of the former is left to the reader, the latter is shown by evaluating the $C_T\cC^{\frac{3}{2}-\epsilon}$ norm of $\Phi(v_1)-\Phi(v_2)$. For the sake of notation, in the proof, we always let $v$ stand for $v_1-v_2$. Using the Schauder estimate \eqref{eq:schauder2}, one has 
\begin{equs}
    \Vert \Phi(v_1)-&\Phi(v_2)\Vert_{C_T\cC^{\frac{3}{2}-\epsilon}}   \lesssim T^{\frac{\epsilon}{4}}\Vert \vvv+\Xvv+\vXv+\vvX+\XXv+\XvX+\vXX\Vert_{C_T\cC^{-\frac{1+\epsilon}{2}}}\,.
\end{equs}
Let us first deal with the non-local third power of $v$:
\begin{equs}\label{eq:boundvvv}
 \nonumber  \Vert \vvv\Vert_{C_T\cC^{-\frac{1+\epsilon}{2}}} &\lesssim  \Vert \vvv\Vert_{C_TL^\infty}\lesssim \Vert v\Vert_{C_TL^\infty}^3\\&\lesssim\big(1+\Vert v_1\Vert_{C_T\cC^{\frac{3}{2}-\epsilon}}^2+\Vert v_2\Vert_{C_T\cC^{\frac{3}{2}-\epsilon}}^2\big)\Vert v_1-v_2\Vert_{C_T\cC^{\frac{3}{2}-\epsilon}}\,.
\end{equs}
We can now turn to the mixed terms. Using \eqref{eq:upXvv} to \eqref{eq:upvXX}, we obtain
\begin{equs}
    \Vert 
    \Xvv+\vXv+\vvX+\XXv+\vXX
    \Vert_{C_T\cC^{-\frac{1+\epsilon}{2}}}&\lesssim_{\bbX^{\LD}_3}\big(1+\Vert v \Vert_{C_TL^\infty}\big)\Vert v\Vert_{C_T\cC^{\frac{3}{2}-\epsilon}}\\&\lesssim_{\bbX^{\LD}_3}
   \big(1+\Vert v_1\Vert_{C_T\cC^{\frac{3}{2}-\epsilon}}+\Vert v_2\Vert_{C_T\cC^{\frac{3}{2}-\epsilon}}\big)\Vert v_1-v_2\Vert_{C_T\cC^{\frac{3}{2}-\epsilon}}\,,
\end{equs}
Finally, by the random operator estimate \eqref{eq:XvXd3}, 
\begin{equs}
       \Vert \XvX\Vert_{C_T\cC^{-\frac{1+\epsilon}{2}}}&\lesssim \Vert\XdotX\Vert_{C_T\cL(H^{\frac32-\epsilon},\cC^{-\frac{1+\epsilon}{2}})}  \Vert v_1-v_2\Vert_{C_T\cC^{\frac32-\epsilon}}\,.
\end{equs}
We have thus obtained that 
\begin{equs}
       \Vert \Phi(v_1)-\Phi(v_2)\Vert_{C_T\cC^{\frac{3}{2}-\epsilon}}
       \lesssim_{\bbX^{\LD}_3}
     T^{\frac{\epsilon}{4}}     
     \big(1+\Vert v_1\Vert_{C_T\cC^{\frac32-\epsilon}}^2+\Vert v_2\Vert_{C_T\cC^{\frac32-\epsilon}}^2\big)\Vert v_1-v_2\Vert_{C_T\cC^{\frac32-\epsilon}}\,.
\end{equs}
Hence, by choosing $T$ small enough, depending only on $\Vert\bbX^{\LD}_3\Vert_T$ and $\Vert v_0\Vert_{\cC^{\frac32-\epsilon}}$, $\Phi$ is indeed a contraction. We can iterate the argument until the solution blows up. Continuity in the data set can be proven in the same way, by studying the difference between two solutions with different data sets, and the proof that $\Phi$ maps a ball of $C_T\cC^{\frac{3}{2}-\epsilon}$ into itself is totally similar. 
\end{proof}

\section{Local well-posedness and for dynamical \TitleEquation{\rmT^4_4}{T44} and tightness of the measure}\label{sec:T44}
While shifting $\phi$ by $\X$ was enough in the previous section, many of the estimates break down when $d=4$.  
The worst stochastic term in \eqref{eq:DPD2} is $\Xthreeloc$ which in $d=4$ is of regularity $-2-$, this stops us from solving \eqref{eq:DPD2} in a space of positive regularity. 
There are two steps we take to extend our analysis to $\rmT^4_4$:
\begin{itemize}
\item We can remove the term $\Xthreeloc$ by doing a second shift of $\phi$ by $\underline{\cL}^{-1}\Xthreeloc$, this introduces new stochastic nonlocal products of the form $\cN(\bigcdot,\X,\underline{\cL}^{-1}\Xthreeloc$) and reveals to us the need to introduce a second renormalization constant (see Lemma~\ref{def:CT2}).  
   
\item Even after we introduce this additional renormalization to control these new products, we will see the renormalized products of the form $\X\times\X\times\underline{\mcL}^{-1}\Xthreeloc$ have regularity $-1-$.
This turns out to be a problem for our $L^2$ estimates used for tightness since the pairing of this object with $v$ requires more regularity from $v$ than is given by its $H^1$ norm \dash this good term only allows us to pair $v$ with something of regularity $-1+$ see the proof of \eqref{eq:Sv}.
To overcome this we introduce third shift (see \eqref{eq:third_tree}) after which the remaining purely stochastic terms are of regularity $-1/2-$.
\item We have one more challenge for local well-posedness, the estimate \eqref{eq:XdotX} in $d=4$ shows that if $v$ is of regularity $\alpha\in(\frac{1}{2},\frac{3}{2})$, then $\cL^{-1}\cN(\X,v,\X)$ is of regularity $\alpha-\epsilon$ for $\epsilon>0$, so that one can not close the fixed point. This is why we also reinject the remainder equation into itself to make it a well-posed remainder equation (see the discussion at the beginning of Section~\ref{subsec:4dfixedpoint}, and in Remark~\ref{rem:12}).
\end{itemize}

This second shift mentioned above is given by 
\begin{equs}
   \Pictwo_{N}\eqdef\underline{\cL}^{-1}\Xthreeloc_{N}\,.
\end{equs}
Note that $\Pictwo$ is stationary in space and time, and by using control over $\Xthreeloc$ and the Schauder estimates \eqref{eq:schauder2} it can be controlled, as $N\uparrow \infty$, as a space-time random field with regularity $-\frac{3d-12}{2}-\epsilon = -\epsilon$. We also sometimes use the notation $\Pictwo^c=\underline{\mcL}^{-1}\big(\cN^c(\X,\X,\X)-\mfC^{1,c}\X\big)$.

The main result of this section is then the following. 
\begin{theorem}\label{thm:4d} 
For $d = 4$, there exist constants $a_{N}$ such that one has uniform in $N$ control of the local in time solutions to \eqref{eq:regTensorLD} with initial condition of the form $\X_N(0) - \Pictwo_N(0) +  \Pi_{N}v(0)$ with $v(0) \in \cC^{\frac{3}{2}-\epsilon}$. 
In particular, using the ansatz $\phi=\X_N-\Pictwo_N+\Picthree_N+X+Y$ for the solution to equation with cut-off $N$, where $\Picthree_{N}$ is the explicit stochastic object \eqref{eq:third_tree} with regularity (uniform in $N$) of $1-$, one has that inexplicit remainders $(X,Y)$ converge, as $N\uparrow \infty$, in $C\big([0,\Bar{T}),\cC^{\frac{3}{2}-2\epsilon}(\T^4)\big)\times C\big([0,\Bar{T}),\cC^{\frac{3}{2}-\epsilon}(\T^4)\big)$ where $\Bar{T} \in (0,\infty]$ is a random existence time.  
\end{theorem}

\subsection{Second renormalization constant}\label{subsect:secondrenorm}

Since we aim to equally subtract $\Pictwo$ from the solution, as we did for $\X$, we have to take into account some potential divergences between $\X$ and $\Pictwo$. Indeed, it turns out that nonlinearities of the form $\cN(\bigcdot,\X,\Pictwo)$ or $\cN(\bigcdot,\Pictwo,\X)$ are divergent. To renormalize them, let us introduce a second renormalization constant to tame the divergence caused by the melonic pairings of $\X$ with $\Pictwo$. 

We now make precise what we mean by the contribution from the melonic pairing. 
We write
\begin{equs}
    \E[\cN^c(\phi,\X,\Pictwo^{c'})]&=    \E_{(1234)}[\cN^c(\phi,\X^1,\underline{\mcL}^{-1}\big(\cN^{c'}(\X^2,\X^3,\X^4)-\mfC^{1,c'}\X^2\big))]
    \\&= \E_{(12)(34)}[\cN^c(\phi,\X,\Pictwo^{c'})]+\E_{(13)(24)}[\cN^c(\phi,\X,\Pictwo^{c'})]+\E_{(14)(23)}[\cN^c(\phi,\X,\Pictwo^{c'})]\,,
\end{equs}
In the first equality above we view the four instances of the white noise as four different random variables (distinguished by their superscripts) that are equal almost surely. 
In the second equality, we work with two independent copies of the white noise \dash in $\E_{(12)(34)}$ we have $\xi^{i} = \xi^{j}$ almost surely for $(ij) \in \{ (12),(34)\}$ but $\xi^{1} \slash \xi^{2}$ is independent of $\xi^{3}\slash \xi^{4}$.
The notation $\E_{(13)(24)}$ is defined analogously \dash the second equality is then just a cumbersome way to write Wick's rule. 
The term with subscript $(12)(34)$ above is the \textit{melonic part} of $\E[\cN^c(\phi,\X,\Pictwo^{c'})]$. 

We then define, for $d\geqslant2$ and $c\in[d]$,
\begin{equs}    \mfC^{2,c}_{N}(d)&\eqdef 
\E_{(12)(34)}[\cN_N^c(1,\X_{N}(t),\Pictwo_{N}(t))(x)]+ 
 \E_{(12)(34)}[\cN_N^c(1,\Pictwo_{N}(t),\X_{N}(t))(x)]\,,
\end{equs}
where for the second term we are using the convention
\[
 \E[\cN_N^c(\phi,\Pictwo_{N}(t),\X_{N}(t))(x)]
=
 \E_{(1234)}[\cN^c(\phi,\underline{\mcL}^{-1}\big(\cN^{c'}(\X^2,\X^3,\X^4)-\mfC^{1,c'}\X^2\big),\X^1)]\;.
\]
By stationarity $\mfC^{2,c}_N (d)$ does not depend on $(t,x)$. 
The renormalized
products are given by setting for any $\phi$ (which one can view as deterministic)
\begin{equs}
    \cN^c_N(\phi,\X_N,\Pictwo_N)-\frac{1}{2}\mfC^{2,c}_N (d)\phi\, \text{  and  }\,\cN^c_N(\phi,\Pictwo_N,\X_N)-\frac{1}{2}\mfC^{2,c}_N (d)\phi\,.
\end{equs}

\begin{lemma}\label{def:CT2}
For $d\geqslant2$ and $c\in[d]$, 
\begin{equs}
 \mfC^{2,c}_{N}(d)= \sum_{c'\neq c}\sum_{m_{\hc}\in\Z_N^{d-1}}
 \sum_{n_{\hc}\in\Z^{d-1}}
\frac{1}{\langle m_{\hc}\rangle^4}\mfR^{1,c'}_{N}(\chi^{c'}(n,m))\,,
\end{equs}
where $\mfR^1_{N}$ is as defined in equation~\eqref{eq:R1}. Unlike $\mfC^{1,c}_{N}$,  $\mfC^{2,c}_{N}(d)$ does depend on $c$. Writing
\begin{equ}\label{def:CT2_formula2}
\mfC^2_{N}(d)\eqdef \sum_{c=1}^d \mfC^{2,c}_{N}(d)\,,
\end{equ}
we have, for $d\in\{2,3,4\}$ and any $\phi,f\in\cC^\infty(\T^d)$, 
\begin{equs}
    \sup_{N\in\N}\E[ \big(\cN_N(\phi,\X_{N},\Pictwo_{N})-\frac12\mfC^2_{N}\phi \big) (f)^2 ]&<\infty\,,\\
    \sup_{N\in\N}\E[\big(\cN_N(\phi,\Pictwo_{N},\X_{N})-\frac12\mfC^2_{N}\phi \big) (f)^2  ]&<\infty\,.
\end{equs}
Moreover, the same statement holds for the four other ways of substituting $\phi$, $\X_{N}$ and $\Pictwo$ into $\cN_{N}$ (with no renormalization).
We sometimes suppress $d$ from the notation, just writing $\mfC^2_N$.
\end{lemma}
\begin{proof}
Let us pause one moment to prove that $\mfC^{2}$ indeed renormalizes products of the form $\cN(\phi,\X,\Pictwo)$ and $\cN(\phi,\X,\Pictwo)$. This will be the occasion to verify that it is correctly defined, and that its second expression is correct. In the proof, to lighten the notations, we drop the dependence of the stochastic objects in $N$. Let us start from 
\begin{equs}
     \E&[\cN_N^c(\phi,\Pictwo(t),\X(t))(x)]\\
     &=
\sum_{m,n\in\Z_N^{d}}e^{-\imath
x\cdot m}\hat{\phi}_{-m_{\hc},-n_c}\E[\hat{\Pictwo}_{n}(t)\hat{\X}_{-n_{\hc},-m_c}(t)]  \\&=\sum_{m,n,p\in\Z_N^{d}}\sum_{c'=1}^de^{-\imath
x\cdot m}\hat{\phi}_{-m_{\hc},-n_c}\int_{-\infty}^t\rmd s\, e^{-(t-s)\angp{n}{2}}\\&\quad\E[\big(
\hat{\X}_{p_{\hat{c}'},n_{c'}}(s)\hat{\X}_{-p}(s)
\hat{\X}_{n_{\hat{c}'},p_{c'}}(s)-\mfC^{1,c'}_{N}\hat{\X}_{n}(s)
\big)\hat{\X}_{-n_{\hc},-m_c}(t)]
\\&=\sum_{m,n,p\in\Z_N^{d}}\sum_{c'=1}^de^{-\imath
x\cdot m}\hat{\phi}_{-m_{\hc},-n_c}\int_{-\infty}^t\rmd s\, e^{-(t-s)\angp{n}{2}}\big(\mathfrak{A}_{1}+\mathfrak{A}_{2}+\mathfrak{A}_{3}\big)\,,
\end{equs}
We expand the expectation as the sum of three terms $\mathfrak{A}_{1}$, $\mathfrak{A}_{2}$ and $\mathfrak{A}_{3}$ according to whether the fourth noise contracts with respectively third one, the second one, or the first one. $\mathfrak{A}_{1}$ corresponds to the melonic part of the expectation and includes the counterterm $\mfC^{1,c'}$. 
More explicitly, 
\begin{equs}    \mathfrak{A}_{1}=\delta_{n_{c},m_{c}}\delta_{p_{c'},n_{c'}}
    \frac{e^{-(t-s)\angp{n}{2}}}{\angp{n}{2}}\bigg(
    \frac{\Pi_N(p_{\hc'},n_{c'})}{\angp{(p_{\hc'},n_{c'})}{2}}  -
    \frac{ \Pi_N(p_{\hc'})}{\angp{p_{\hc'}}{2}}
    \bigg) 
    \,,
\end{equs}
while $\mathfrak{A}_{2}$ and $\mathfrak{A}_{3}$ are given by 
\begin{equs}
    \mathfrak{A}_{2}=\delta_{p,-n}\delta_{n_c,m_c}
    \frac{e^{-(t-s)\angp{n}{2}}}{\angp{n}{4}}\,,
   \quad   \mathfrak{A}_{3}&=\delta_{p_{\hc'},n_{\hc'}}\delta_{n_c,m_c}\frac{e^{-(t-s)\angp{n}{2}}}{\angp{n}{2}}\frac{\Pi_N(n_{\hc'},p_{c'})}{\angp{(n_{\hc'},p_{c'})}{2}}\,.
\end{equs}
The contributions from $\mathfrak{A}_{2}$ and $\mathfrak{A}_{3}$ are convergent while $\mathfrak{A}_{1}$ gives 
\begin{equs}
    &\sum_{m\in\Z^d;n_{\hc},p_{\hc'}\in\Z^{d-1}}\Pi_N(n_{\hc},m_c)
    \sum_{c'=1}^d
    e^{-\imath
x\cdot m}\hat{\phi}_{-m}
    \int^t_{-\infty}
    \frac{e^{-(t-s)\angp{(n_{\hc},m_c)}{2}}}{\angp{(n_{\hc},m_c)}{2}}
    \mfR^{1,c'}_{N}(\chi^{c'}(p,n))\\
    &\qquad=\frac{1}{2}\sum_{m\in\Z^d;n_{\hc},p_{\hc}\in\Z^{d-1}}\sum_{c'=1}^d e^{-\imath
x\cdot m}\hat{\phi}_{-m}
\frac{\Pi_N(n_{\hc},m_c)}{\angp{(n_{\hc},m_c)}{4}}
\mfR^{1,c'}_{N}(\chi^{c'}(p,n))\,.
\end{equs}
As we did we the first renormalization, we only subtract the value of the divergent expression at null external Fourier mode,  that is $m_{c} = 0$, since this is sufficient in order to obtain a convergent expression. Thus, the quantity we subtract from $ \E[\cN_N^c(\phi,\Pictwo_{N}(t),\X_{N}(t))(x)]$ is given by
\begin{equs}
\frac{1}{2}\phi(x)  
\sum_{n_{\hc}\in\Z_N^{d-1}}\sum_{p_{\hc}\in\Z^{d-1}}\sum_{c'=1}^d 
\frac{1}{\angp{n_{\hc}}{4}}
\mfR^{1,c'}_{N}(\chi^{c'}(p,n))\,.
\end{equs}
Here, we can observe that when $c'=c$, $\mfR_{N}^{1,c'}$ is evaluated at $(p_{\hc'},0)$ and we recall from \eqref{eq:R1} that $\mfR_{N}^{1,c'}(p_{\hc'},0)=0$. This means that the first renormalization totally renormalizes the melonic product of $\X$ times $\Pictwo^c$. Therefore, we have finally obtained that
\begin{equs}
 \E_{(12)(34)}[\cN_N^c(\phi,\Pictwo_{N}(t),\X_{N}(t))(x)]=\frac{1}{2}\mfC^{2,c}_{N}\phi(x)+\mathcal{O}(1)\,,
    \end{equs}
where the quantity $\mathcal{O}(1)$ is a convergent sum uniformly in $N$, while the non melonic part involving  $\E_{(13)(24)}$ and $\E_{(14)(23)}$ is finite. 
The computations for $\E[\cN_N^c(\phi,\X_{N}(t),\Pictwo_{N}(t))(x)]$ are similar.
Finally, setting $\phi=1$ proves \eqref{def:CT2_formula2}.
\end{proof}
Inspired by the previous discussion, we define for $d\geqslant4$ and $c,c'\in[d]$, $c\neq c'$ the second renormalized amplitude:
\begin{equs}
    \mfR^{2,c,c'}_{N}(m,n)\eqdef
    &\frac{1}{2}\bigg(
    \frac{\Pi_N(m)}{\angp{m}{4}}-
     \frac{\Pi_N(m_{\hc})}{\angp{m_{\hc}}{4}}
    \bigg)
    \mfR^{1,c'}(n) \,.
\end{equs}
The renormalized melonic pairing $\sum_{n\in\Z^d}\hat\phi_{-m_{\hc},-n_c}\E[\hat{\Pictwo}_{n}(t)\hat{\X}_{-n_{\hc},-m_c}(t)]-\frac12\mfC^{2}\hat\phi_m$ (as well as its counterpart with $\X$ and $\Pictwo$ switched) thus contains
\begin{equs}
\hat\phi_m \sum_{n_{\hc}\in\Z^{d-1}} \sum_{p_{\hc}\in\Z^{d-1}}\sum_{c'\neq c}\mfR_{N}^{2,c,c'}(\chi^c(n,m),\chi^{c'}(p,n))\,.
\end{equs}
We denote its value at $N=\infty$ by $ \mfR^{2,c,c'}(m,n)\eqdef\mfR^{2,c,c'}_\infty(m,n)$.
\begin{remark}
$\mfC^2_\infty$ is finite  when $d=3$ so we didn't encounter it in Section~\ref{sec:3d}.
$\mfC^2_{N}$ diverges like $\log N$ when $d=4$, while it diverges like $N^{2d-8}$ for $d \geqslant 5$. 
 \end{remark}
With the second renormalization constant $\mfC^2_{N}$ defined, our regularised and renormalized Langevin dynamic for $\mathrm{T}^4_4$ is given by
\begin{equs}
\label{eq:T^4_4}
    \cL\phi_{N}&=-\cN_N(\phi_{N},\phi_{N},\phi_{N})+(\mfC^1_{N}-\mfC^2_{N})\phi_{N}+\sqrt{2}\xi_{N}\,.
\end{equs}

We now define the third contribution to our shift, first stating its key stochastic estimate.
\begin{lemma}[Random fields 2]\label{lem:randfi2}
Let
 \begin{equs}
   \XtwoPictwo_{N}\eqdef &\big(\cN_N(\Pictwo_{N},\X_{N},\X_{N})-\mfC_{N}^1(d)\Pictwo_{N}\big)+\big(\cN_N(\X_{N},\Pictwo_{N},\X_{N})+\cN_N(\X_{N},\X_{N},\Pictwo_{N})-\mfC_{N}^2(d)\X_{N}\big)\,.
 \end{equs}
 Then, in $d=4$, 
 \begin{equs}
  \sup_{N\in\N} \E[\Vert \XtwoPictwo_{N}\Vert_{C_T\cC^{-1-\epsilon}(\T^4)}^p]<\infty\,.
 \end{equs}
\end{lemma}
We then set
\begin{equs}\label{eq:third_tree}
  \Picthree_{N}={\cL}^{-1}\XtwoPictwo_{N}\,.
\end{equs}
$\Picthree$ can be controlled by $\XtwoPictwo$ using the Schauder estimates \eqref{eq:schauder2}, and is of regularity $1-\epsilon$ in $d=4$.
We now introduce the total shift $\bX$ needed for us to solve the $\rmT^4_4$ dynamic, it is given by 
\begin{equs}
    \bX\eqdef \X-\Pictwo+\Picthree\,.
\end{equs}
The rough shift $\bX$ shares the regularity of the free field but for probabilistic estimates it is a more complicated non-Gaussian object. 
Moreover, $\bX$ is not stationary in time. 

As we did in Section~\ref{sec:3d} for the free field $\X$, we define the random operators $\bX^{(k)}$, $\pXtwom$, $\pXtwonm$ and $\pXdotX$ (see Lemma~\ref{lem:randop2}) \dash one just replaces the relevant occurrences of $\X$ with $\bX$.
We also define the random field
\begin{equs}
  \pXthreeloc_N\eqdef\cN_N(\bX_N,\bX_N,\bX_N)-(\mfC^1_N-\mfC^2_N)\bX_N\,.
\end{equs} 
They share the regularity properties of $\bX^{(k)}$, $\Xtwom$, $\Xtwonm$, $\XdotX$ and $\Xthreeloc$ stated in Lemmas~\ref{prop:stochastic_1} and \ref{lem:randop1}. In particular, we recall that in $d=4$, $\pXthreeloc$ is of regularity $-2-$.
\begin{notation}
In the end of the present section, we use again the pictorial notations introduced in the three dimensional case. In particular, recall that $\pvXX$ stands for $\cN(v,\bX,\bX)-(\mfC^1-\mfC^2)v$ and that we systematically drop the dependencies in $N$.
\end{notation}
Shifting the unknown by $\bX$ introducing $v = v_{N}\eqdef \phi_N-\bX_{N}$ in \eqref{eq:T^4_4} gives
\begin{equs}\label{eq:DPD3}
    \cL v=-\vvv-\pXvv-\pvXv-\pvvX-\pXXv-\pXvX-\pvXX-\cS\,,\,\,v(0)=\Pi_Nv_0\,,
\end{equs}
where the random field
\begin{equs}
    \cS_{N}\eqdef&\;\pXthreeloc_{N} -\Xthreeloc_{N}+\XtwoPictwo_{N}\\ =
    &\;\big(\cN_N(\Picthree_{N},\X_{N},\X_{N})-\mfC^1_N\Picthree_{N}\big)+  
    \big(\cN_N(\Pictwo_{N},\X_{N},\Pictwo_{N})-\frac12\mfC^2_N\Pictwo_{N}\big)+\cN_N(\X_{N},\X_{N},\Picthree_{N})+20 \text{ terms ,}
\end{equs}
collects the sum of all the at least $\cS$eptic remaining purely stochastic terms. Note that $\cS$ can be defined without any additional renormalization, and that in particular the paring of $\X$ with $\Picthree$ is not divergent. We have the following result on the regularity of $\cS$:
\begin{lemma}[Random fields 3]\label{lem:randfi4}
In $d=4$,
\begin{equs}
 \sup_{N\in\N} \E[\Vert \cS_N \Vert^p_{C_T\cC^{-\frac12-\epsilon}(\T^4)}]<\infty\,.
\end{equs}
\end{lemma}
\begin{remark}
The regularity of $\cS$ is not what we could naively expect from a power counting argument, see Notation~\ref{not:notation}. Indeed, once we try to insert an object of positive regularity (like $\Picthree$ in $d=4$) inside the nonlinearity, this objects rather acts like an object of null regularity, and we do not gain anything from its better behavior. This phenomenon is the same that makes positive renormalization necessary in the local case. In the non-local case, we deal with this difficulty differently, by mixing some stochastic and deterministic estimates, see the proof of Lemma~\ref{lem:randfi4} in Section~\ref{subsubsec:S}.
\end{remark}
\begin{remark}
We have not discussed yet our choice of initial condition. One way to show that the $\rmT^4_4$ measure is invariant under the $\rmT^4_4$ equation would be to show that the solution to the equation defines a Markov process on a state space on which the measure is supported -  this strategy was adopted in \cite{tsatsoulis2016spectral} in order to construct the $\Phi^4_2$ measure. 
However, for the deterministic equation $\cL\phi=-\cN(\phi,\phi,\phi)$, $s=-\frac{2}{3}$ is a critical value below which it is not possible to solve the equation with initial condition in $\cC^s$ and this makes it more difficult to adopt the same strategy, and so we instead simply consider an initial condition of the form dictated by our ansatz. 
Note that similar difficulties appear in the local theory \dash the threshold $-\frac{2}{3}$ is sufficient to deal with dynamical $\Phi^4_2$ and $\Phi^4_3$ but becomes a problem when trying to cover the full subcritical regime.
\end{remark}
At this stage, in contrast with the local case, \eqref{eq:DPD3} is fully renormalized in that it does not involve any explicit renormalization counterterm (they have all been consumed in the definition of stochastic objects). 
The pairings with $\bX$, $\pXtwom$, $\pXtwonm$ and $\pXdotX$ introduced below act on $v\in\cC^{1-}$, so if we could close \eqref{eq:DPD3}  in $v \in \cC^{1-}$ we would be done, 
However \eqref{eq:DPD3} cannot be solved in $\cC^{1-}$ but we will show we are able to overcome this by a re-injection of our equation into itself.

\subsection{Random operators for \TitleEquation{\rmT^4_4}{T44}}
As in three dimensions, we view the mixed terms in \eqref{eq:DPD3} as random operators acting on the solution. 
\begin{lemma}[Random operators 2]\label{lem:randop2}
Let $k\in[3]$ and $c\in[4]$. The random operators made with the rough shift $\bX$ are defined as
\begin{align*}
    \bX^{(k)}_{N}(f)&\eqdef \int_{\T^{4-k}} \bX_{N}(\cdot,y) f(y)\rmd y\,,&&\text{ for }f:\R_{\geqslant 0}\times\T^{4-k}\rightarrow\R\,,\\
     \pXtwom^c_{N}(f)&\eqdef \cN^c_N(f,\bX_N,\bX_N)-(\mfC^{1,c}_N-\mfC^{2,c}_N)f \,,&&\text{ for }f:\R_{\geqslant 0}\times\T_c\rightarrow\R\,,\\ \pXtwonm^c_{N}(f)&\eqdef \cN^c_N(f,\bX_N,\bX_N)\,,&&\text{ for }f:\R_{\geqslant 0}\times\T_{\hc}^3\rightarrow\R\,,\\ \pXdotX_{N}(f)&\eqdef \cN_N(\bX_{N},f,\bX_{N}) &&\text{ for }f:\R_{\geqslant 0}\times\T^4\rightarrow\R\,.
\end{align*}
Then, in $d=4$, for all $\alpha>\frac{2-k}{2}$, 
\begin{equs}\label{eq:boundpX}   
 \sup_{N\in\N} \E[\Vert\bX^{(k)}_{N}\Vert^p_{\cL(C_TH^{\alpha}(\T^{4-k}),C_T\cC^{\beta-\epsilon}(\T^k))}]<\infty \text{ with $\beta=\min(-\frac{k-2}{2},\alpha-1)$\;.}
\end{equs}
Moreover, in $d=4$, for all $c\in[4]$ and $\alpha>0$, 
\begin{equs}\label{eq:boundpXtwom}
    \sup_{N\in\N}\E[\Vert\pXtwom^c_{N}\Vert^p_{\cL(C_TH^{\alpha}(\T),C_T\cC^{\beta-\epsilon}(\T))}]<\infty \text{ with $\beta=\min(0,\alpha-\frac{3}{2})$ 
 .}
  \end{equs}  
Finally, in $d=4$, for all $c\in[4]$ and $\alpha>\frac{1}{2}$, 
\begin{equs}
{}& \sup_{N\in\N}   \E[\Vert\pXtwonm^c_{N}\Vert^p_{\cL (C_TH^{\alpha}(\T^3),C_T\cC^{\beta-\epsilon}(\T^3))}]<\infty \text{ with $\beta=\min(-\frac{1}{2},\alpha-\frac{3}{2})$ ,}  \label{eq:boundpXtwonm} \\
{}& \sup_{N\in\N} \E[\Vert\pXdotX_{N}\Vert^p_{\cL(C_TH^{\alpha}(\T^4),C_T\cC^{\beta-\epsilon}(\T^4))}]<\infty\, \text{ with $\beta=\min(-\frac{1}{2},\alpha-2)$ .} \label{eq:boundpXdotX}
\end{equs}
\end{lemma}
In the sequel, we also apply the notation introduced in Notation~\ref{not:cnotc} to the objects made with the rough ansatz $\pXtwom$ and $\pXtwonm$.

The previous estimate allows us to control the mixed terms $\pvvX$, $\pvXv$, $\pXvv$, $\pvXX$ and $\pXXv$ in $\cC^{-\frac{1}{2}-\epsilon}$ if $v$ is of regularity $1+$. 
However, the term $\pXvX$ remains problematic. 
Regardless of the our assumption on the regularity of $v$, the best regularity estimate we can hope for $\pXvX$ is $-\frac{1}{2}-$, which means we can't expect $v$ to better than  $-\frac{3}{2}-\epsilon$ in regularity. 
However, if $v\in\cC^{\frac{3}{2}-\epsilon}$, then $\pXvX\in\cC^{-\frac{1}{2}-2\epsilon}$, so that the argument never closes. 

We see that even with our random operator estimates for the \eqref{eq:T^4_4}, the term $\pXvX$ is too singular to be dealt with by a simple fixed point argument. 
This term in the $\rmT^4_4$ dynamic is the one that requires a bigger deparature from the approach we used for $\rmT^4_3$, in that sense it is analogous the product $\cherry v$ in dynamical $\Phi^4_3$. 
\subsection{Closing the fixed-point problem}\label{subsec:4dfixedpoint}
To further simplify the problem, let us perform the following change of variable, and rewrite \eqref{eq:DPD3} using the following ansatz: we split the solution $v$ as $v=X+Y$ where $X$ and $Y$ solve the system
\begin{subequations}
  \begin{empheq}[left=\empheqlbrace]{alignat=2}
   \label{eq:X}
    \cL X&=-\pXdotX(X)-\pXdotX(Y)\,,\\  \cL Y &=-\vvv-\pXvv-\pvXv-\pvvX-\pXXv-\pvXX-\cS\,,\label{eq:Y}
  \end{empheq}
\end{subequations}
with initial conditions $(X(0),Y(0))=(0,\Pi_Nv_0)$. 
Note that $v$ still stands for $X+Y$ in \eqref{eq:Y}. 
The RHS of \eqref{eq:Y} gathers all the terms that don't pose any problem for closing the fixed point problem, that is we fix some choice of $X \in C_T\cC^{1+}$ then we could close the equation \eqref{eq:Y} for  $Y \in C_T\cC^{\frac{3}{2}-\epsilon}$. 

Turning to \eqref{eq:X}, we see that if we fix $Y\in C_T\cC^{\frac{3}{2}-\epsilon}$ then we cannot close \eqref{eq:X} for $X$ \dash the map $X\mapsto -\cL^{-1}\pXdotX(X+Y)$ is well defined for $X\in C_T\cC^{\alpha}$ with $\alpha > \frac{1}{2}$ and that takes $C_T\cC^{\alpha}$ to $ C_T\cC^{\min(\frac{3}{2},\alpha)-\epsilon}$. 
We see that, whatever the value of $\alpha$, it is mapping $C_T\cC^{\alpha}$ into a larger space \dash while we can iterate this mapping and we can't use it to post a fixed point problem. 
While this slightly resembles how $\cherry v$ poses a problem for the $\Phi^4_3$ dynamic, the difficulty there is different in that it is impossible to compose $v \mapsto \cL^{-1} ( \cherry v)$ for generic $v$ of any fixed regularity. 

It turns out that we can overcome our difficulty with \eqref{eq:X} by injecting this problematic term into itself and then performing a stochastic estimate on this term, see the next lemma.  

\begin{lemma}[Random operators 3]\label{lem:randop_4}
Defining the operator $\pSym$ acting on $f:\R_{\geqslant 0}\times\T^4\rightarrow\R$ as
\begin{equs}
    \pSym_{N}(f)&\eqdef
    \pXdotX_{N}(\cL^{-1}\ 
    \pXdotX_{N}(f))=      
    \cN_N(\bX_{N},\cL^{-1}\cN_N(\bX_{N},f,\bX_{N}),
    \bX_{N})\,,
\end{equs}
then, in $d=4$ and for all $\alpha>\frac{1}{2}$,
\begin{equs}\label{eq:boundpSym}
 \sup_{N\in\N}     \E[\Vert\pSym_{N}\Vert^p_{\cL(C_TH^{\alpha}(\T^4),C_T\cC^{\beta-\epsilon}(\T^4))}]<\infty\text{ with $\beta=\min(-\frac{1}{2},\alpha-1)$}\,.
\end{equs}
\end{lemma}
\begin{notation}
Our use of symbolic notation in $\pSym$ is inconsistent with that of $\XtwoPictwo$ so we take a momeent to spell out the difference.  
Recalling that these trees are built with a non-commutative product, $\pSym$ contains only one of the three possible products - the one with the two $\bX$ terms in the first and third positions. 
This is different then what we refer to with the symbol $\XtwoPictwo$ which contains 3 terms for the 3 possibles products at the root.  
\end{notation}
Reinjecting $X=-\cL^{-1}\pXdotX(X+Y)$ in \eqref{eq:X} as described above gives us the system
\begin{subequations}    
\label{eq:system2}
\begin{empheq}[left=\empheqlbrace]{alignat=2}
\cL X&= \pSym(X) +  \pSym(Y)- \pXdotX(Y)\,,\label{eq:X2}\\
\cL Y&=-\vvv-\pXvv-\pvXv-\pvvX-\pXXv-\pvXX-\cS\,.\label{eq:Y2}\end{empheq}
\end{subequations}

\begin{remark}\label{rem:12}
This reinjection is a tool for proving the $N\uparrow \infty$ convergence of solutions to the original system with cut-off $N\uparrow \infty$. 
In particular, in the smooth setting (that is, $N$ finite) \textit{both} the original system \eqref{eq:X}+\eqref{eq:Y} and the modified system \eqref{eq:X2}+\eqref{eq:Y2} are well-posed and have same solutions. 
If we write the integral fixed point problem  for \eqref{eq:X}+\eqref{eq:Y} as
\begin{equ}\label{eq:problem1}
(x,y) \mapsto \big(\Theta_{X}(x,y),\Theta_{Y}(x,y) \big)
\end{equ} 
then the integral fixed point map for \eqref{eq:X2} + \eqref{eq:Y2} can be written as 
\begin{equ}\label{eq:problem2}
(x,y) \mapsto \Big(\Theta_{X} \big( \Theta_{X}(x,y) ,y\big),\Theta_{Y}(x,y) \Big)\;.
\end{equ}
Clearly if \eqref{eq:problem1} has a fixed point $(x_{\star},y_{\star})$ and we know \eqref{eq:problem2} has a unique fixed point, then the fixed point of  \eqref{eq:problem2}  must be given by $(x_{\star},y_{\star})$. 
\end{remark}

Define $\bbX^{\LD}_{f,4}$ the collection of all the random fields defined in Lemmas~\ref{lem:randfi2} and \ref{lem:randfi4}, and $\bbX^{\LD}_{o,4}$ the collection of all the random operators defined in Lemmas~\ref{lem:randop2} and \ref{lem:randop_4}. Then, as in three dimensions, the enhanced noise set $\bbX_4^{\LD}=\bbX^{\LD}_{f,4}\cup\bbX^{\LD}_{o,4}$ is an element of the product of the Banach spaces in which the random fields and random operators live. We endow $\bbX^{\LD}_4$ with the norm $\Vert\bbX_4^{\LD}\Vert_{T}$ defined as in \eqref{eq:normbbX} with $3$ replaced by $4$, and recall the notation $A\lesssim_{\bbX_4^{\LD}}B$ introduced in Section~\ref{sec:3d}, and the fact that at finite $T$, $\Vert\bbX^{\LD}_4\Vert_{T}\in L^p(\Omega)$ for every $1\leqslant p<\infty$.
We also define the Banach space $C_T\cC^{\frac{3}{2}-2\epsilon}\times C_T\cC^{\frac{3}{2}-\epsilon}$ with norm $\Vert X,Y\Vert_{sol}\eqdef \Vert X\Vert_{C_T\cC^{\frac{3}{2}-2 \epsilon}}\vee \Vert Y\Vert_{C_T\cC^{\frac{3}{2}-\epsilon}}$. 
 
\begin{proposition}\label{prop:4dfixedpoint} 
We have the following local well-posedness result for the system \eqref{eq:X2}+\eqref{eq:Y2}. 

For any $(X_0,Y_0) \in \cC^{\frac32 - 2 \epsilon}(\T^4) \times \cC^{\frac32-\epsilon}(\T^4)$, then there exists a random blow-up time $\Bar{T}$ such that the system \eqref{eq:system2} with initial condition $(X(0),Y(0))=( \Pi_{N} x_0,\Pi_N y_0)$ admits a unique solution $(X,Y)$ in $C([0,\Bar{T}),\cC^{\frac32-2\epsilon})\times C([0,\Bar{T}),\cC^{\frac32-\epsilon})$. 
This solution is maximal, and we have $\lim_{t\uparrow \Bar{T}} \Vert X(t)\Vert_{\cC^{\frac32-2\epsilon}}\vee \Vert Y(t)\Vert_{\cC^{\frac32-\epsilon}}=+\infty$. Moreover, if $T<\Bar{T}$, the solution on $[0,T]$ depends continuously on $(x_0,y_0)$ and on the enhanced noise set $\bbX^{\LD}_4$ (w.r.t the topology of $\Vert \bbX^{\LD}_4\Vert_{T}$).
\end{proposition}
\begin{proof}
    Let $\Phi(X,Y)=\big(\Phi_X(X,Y),\Phi_Y(X,Y)\big)$ be the fixed point map of \eqref{eq:system2}. We thus have
 \begin{subequations}    
\begin{empheq}[left=\empheqlbrace]{alignat=2}\nonumber
\Phi_X(X,Y)&= \int_0^tP_{t-s}\Big( \pSym(X) +  \pSym(Y)- \pXdotX(Y)\Big)(s)\rmd s\,,\\\nonumber
\Phi_Y(X,Y)&=P_tx-\int_0^tP_{t-s}\Big(\vvv+\pXvv+\pvXv+\pvvX+\pXXv+\pvXX+\cS\Big)(s)\rmd s\,,\label{eq:Y2}\end{empheq}
\end{subequations}
where $v=X+Y$. We show that $\Phi$ is indeed a contraction on $C_T\cC^{\frac{3}{2}-2\epsilon}\times C_T\cC^{\frac{3}{2}-\epsilon}$, while the proof that if maps a ball of $C_T\cC^{\frac{3}{2}-2\epsilon}\times C_T\cC^{\frac{3}{2}-\epsilon}$ into itself is left to the reader. We thus evaluate $\Phi$ at $(X_1-X_2,Y_1-Y_2)$. For the sake of notation, throughout this proof, we often write $X=X_1-X_2$, $Y=Y_1-Y_2$ and $v=X+Y$.

We first estimate $\Phi_X(X_1-X_2,Y_1-Y_2)$ in $C_T\cC^{\frac{3}{2}-2\epsilon}$.
The Schauder estimate \eqref{eq:schauder2} gives
\begin{equs}
    \Vert \Phi_X(X,Y)\Vert_{C_T\cC^{\frac{3}{2}-2\epsilon}}\lesssim T^{\frac{\epsilon}{4}}\Vert
    \pSym(X)+\pSym(Y)-\pXdotX(Y)\Vert_{C_T\cC^{-\frac{1+3\epsilon}{2}}}\,.
\end{equs}
Using \eqref{eq:boundpXdotX}, we obtain
\begin{equs}
   \Vert \pXdotX(Y)\Vert_{C_T\cC^{-\frac{1+3\epsilon}{2}}}   \lesssim \Vert \pXdotX\Vert_{\cL({C_T\cC^{\frac{3}{2}-\epsilon}},{C_T\cC^{-\frac{1+3\epsilon}{2}}})} \Vert Y\Vert_{C_T\cC^{\frac{3}{2}-\epsilon}}\,,
\end{equs}
and \eqref{eq:boundpSym} yields
\begin{equs}
    \Vert
    \pSym(X)\Vert_{C_T\cC^{-\frac{1+3\epsilon}{2}}}\lesssim
    \Vert \pSym\Vert_{\cL({C_T\cC^{\frac{3}{2}-2\epsilon}},{C_T\cC^{-\frac{1+3\epsilon}{2}}})} \Vert X\Vert_{C_T\cC^{\frac{3}{2}-2\epsilon}}\,,
\end{equs}
and 
\begin{equs}
    \Vert
    \pSym(Y)\Vert_{C_T\cC^{-\frac{1+3\epsilon}{2}}}\lesssim
    \Vert \pSym\Vert_{\cL({C_T\cC^{\frac{3}{2}-\epsilon}},{C_T\cC^{-\frac{1+3\epsilon}{2}}})} \Vert Y\Vert_{C_T\cC^{\frac{3}{2}-\epsilon}}\,.
\end{equs}
We thus have
\begin{equs}
     \Vert \Phi_X(X_1-X_2,Y_1-Y_2)\Vert_{C_T\cC^{\frac{3}{2}-2\epsilon}}\lesssim_{\bbX^{\LD}_4}T^{\frac{\epsilon}{4}}\Vert X_1-X_2,Y_1-Y_2\Vert_{sol}\,.
\end{equs}
Turning to the $C_T\cC^{\frac{3}{2}-\epsilon}$ norm of $\Phi_Y(X_1-X_2,Y_1-Y_2)$, by the Schauder estimate \eqref{eq:schauder2}, we have
\begin{equs}
    \Vert \Phi_Y(X,Y)\Vert_{C_T\cC^{\frac{3}{2}-\epsilon}}\lesssim T^{\frac{\epsilon}{4}}\Vert
    \vvv+\pXvv+\pvXv+\pvvX+\pXXv+\pvXX\Vert_{C_T\cC^{-\frac{1+2\epsilon}{2}}}\,.
\end{equs}
The estimates \eqref{eq:boundpX}, \eqref{eq:boundpXtwom} and \eqref{eq:boundpXtwonm} show that the stochastic objects made with the rough shift $\bX$ obey the same estimates than the objects made with $\X$. In particular, Lemma~\ref{lem:mixednonlin} stills holds in our context, and the terms $\pXvv$, $\pvXv$, $\pvvX$, $\pXXv$ and $\pvXX$ can be dealt with in the same way we did in three dimensions, in the proof of Proposition~\ref{prop:3dfixedpoint}.
Recalling \eqref{eq:boundvvv},
\begin{equs}
    \Vert \Phi_Y(X,Y)\Vert_{C_T\cC^{\frac{3}{2}-\epsilon}}&\lesssim T^{\frac{\epsilon}{4}}
     \big(1+\Vert v_1\Vert_{C_T\cC^{\frac32-2\epsilon}}^2+\Vert v_2\Vert_{C_T\cC^{\frac32-2\epsilon}}^2\big)\Vert v_1-v_2\Vert_{C_T\cC^{\frac32-2\epsilon}}\\&
     \lesssim T^{\frac{\epsilon}{4}}
     \big(1+\Vert X_1,Y_1\Vert_{sol}^2+\Vert X_2,Y_2\Vert_{sol}^2\big)\Vert X_1-X_2,Y_1-Y_2\Vert_{sol}\,,
\end{equs}
which gives
\begin{equs}
    &\Vert \Phi(X_1-X_2,Y_1-Y_2)\Vert_{sol}\lesssim T^{\frac{\epsilon}{4}}
     \big(1+\Vert X_1,Y_1\Vert_{sol}^2+\Vert X_2,Y_2\Vert_{sol}^2\big)\Vert X_1-X_2,Y_1-Y_2\Vert_{sol}\,.
\end{equs}
Consequently, by choosing $T$ small enough, depending only on $\Vert\bbX^{\LD}_4\Vert_T$ and $\Vert v_0\Vert_{\cC^{\frac32-\epsilon}}$, $\Phi$ is indeed a contraction. 
Iterating this argument yields the desired existence of maximal solutions and showing continuity of this solution in the stochastic data is a standard argument.
\end{proof}
In the next subsection, we aim to prove an a priori $L^2$ estimate on the solution $v=X+Y$ to \eqref{eq:DPD3} with initial condition $v_0\in\cC^{\frac32-\epsilon}$. To do so, we need the continuity in time of the solution to \eqref{eq:system2}:
\begin{lemma}\label{lem:continuity}
    Pick $\kappa<\frac{3}{4}$. Then for $T<\Bar{T}$, $t\mapsto v(t,x)$ belongs to $C_T^\kappa L^\infty(\T^4)$.
\end{lemma}
\begin{proof}
We prove the continuity in $t=0$, the result for $t>0$ can be proven similarly. 
Denote by $R(t,x)$ the RHS of \eqref{eq:DPD3}. 
By Proposition~\ref{prop:4dfixedpoint}, $v\in C_T\cC^{\frac{3}{2}-2\epsilon}$ so that $R\in C_T\cC^{-\frac12-3\epsilon}$. One has $v(t)-v(0)=P_t(x)-x+\int_0^tP_{t-s}(R(s))\rmd s$ so that 
   \begin{equs}
       \Vert v(t)-v(0)\Vert_{L^\infty}\lesssim \Vert (1-P_t)(x)\Vert_{L^\infty}+\Vert \int_0^tP_{t-s}(R(s))\rmd s\Vert_{L^\infty}\,.
   \end{equs} 
We deal with the first term using the heat kernel estimate $\Vert (1-P_t)(u)\Vert_{\cC^\beta}\lesssim t^{-\frac{\beta-\alpha}{2}} \Vert u\Vert_{\cC^\alpha}$ that holds for $\beta<\alpha$. Therefore, $\Vert (1-P_t)(x)\Vert_{L^\infty}\lesssim t^{\frac{3}{4}-\frac{\epsilon}{2}}\Vert x\Vert_{\cC^{\frac32-\epsilon}}$. On the other hand, using \eqref{eq:schauder1}, we have
\begin{equs}
    \Vert \int_0^tP_{t-s}(R(s))\rmd s\Vert_{L^\infty}&\lesssim
      \int_0^t \Vert P_{t-s}(R(s)) \Vert_{L^\infty} \rmd s \lesssim \int_0^t (t-s)^{-\frac14-\frac{3\epsilon}{2}} \rmd s \Vert R\Vert_{C_T\cC^{-\frac12-3\epsilon}}\,,
\end{equs}
and we can conclude that $\Vert v(t)-v(0)\Vert_{L^\infty}\lesssim t^{\frac34-\frac{3\epsilon}{2}}$.
\end{proof}

\subsection{$L^2$ a priori estimates and coming down from infinity}\label{subsec:L2apriori}
In this section, we show two different kinds of estimates on the $L^2$ norm of the solution to \eqref{eq:DPD3}. 
In Proposition~\ref{prop:comingdown} we use the damping effect of the non-linearity to establish a bound on the $L^{2}$ norm of $v$ that is uniform in the initial condition \dash this is sometimes called a ``coming down from infinity" estimate. 
In Proposition~\ref{prop:aprioriL2andH1} we use the nonlinearity to prove an energy estimate which is better suited for proving  tightness of the invariant measures. 

In order to use the nonlinearity as a damping effect, one has to show that all the mixed terms of the unknown with the stochastic objects are indeed bounded by the coercive terms. 

For the sake of clarity, the proofs of the bounds over the mixed terms are postponed to Appendix~\ref{sec:mixedterms}. 
\begin{proposition}[$L^2$ coming down from infinity]\label{prop:comingdown}
Recall from Section~\ref{subsec:4dfixedpoint} the definition of the enhanced noise $\bbX^{\LD}_4$ and its norm $\Vert \cdot\Vert_t$ (see equation~\eqref{eq:normbbX}). Pick $T<\Bar{T}$ and $v_0\in\cC^{\frac{3}{2}-2\epsilon}(\T^4)$, and let $v_x=X+Y$ with $(X,Y)$ the solution to \eqref{eq:system2} at fixed $N$ with initial condition $(0,\Pi_Nv_0)$. There exists a constant $\gamma>0$ such that for all $t\in[0,T]$, $v_x$ obeys the bound:
    \begin{equs}
    \Vert v(t,v_0)\Vert_{L^2(\T^4)}\lesssim \Vert \bbX^{\LD}_4\Vert^\gamma_t\max(1,t^{-\frac{3}{2}})\,.
    \end{equs}
Observe that the upper bound on $\Vert v(t,v_0)\Vert_{L^2(\T^4)}$ is a positive random variable, uniform in $N$, that only depends on the randomness on $[0,t]$ (not in the future of $t$), is in $L^p(\Omega)$ for all $1\leqslant p<\infty$, and that it is independent of the initial condition $v_0$.
\end{proposition}
\begin{remark}
    A generalization of the proof of previous estimate to some $L^p$ seems to be out of the reach. Indeed, when pairing the nonlinearity with some higher powers of $v$, we obtain expressions such as $(v^{p-1},\cN(v,v,v))$, which are neither norms, nor even easily provable to be positive. Moreover, establishing some $M^p$ estimates seems hopeless, since we would only obtain them for $p\geqslant2$ only, and they would therefore be weaker than the $L^2$ one. Finally, the non-locality of the interaction prevents us from easily using a maximum principle.
\end{remark}
\begin{remark}
We saw in the previous section that in order to prove local in time existence of \eqref{eq:DPD3}, we had to be careful dealing with the worst term $\pXvX$, leading us to the ansatz \eqref{eq:system2}. 
However, this ansatz is not necessary when proving global existence \dash when pairing equation paired with its solution to derive an $L^2$ energe estimate, $\pXvX$ becomes $\XvXv$, which, by Cauchy Schwarz inequality, is bounded by $\XXvv$. 
The crucial fact is that $\pXtwonm$ can be defined without any new renormalization \dash after the pairing the term $\pXvX$ becomes similar to the better-behaved term $\pXXv$ so the ansatz is not necessary.
This occurs in $d=4$ but isn't generic in the full subcritical regime. 
Closer to the critical dimension the terms $\pXXv$ and $\pXXv$ are also problematic and require a more complicated ansatz and re-injections, similar to what that we did for $\pXvX$. 
This observation is also used for treating $\rmT^4_4$ in Section~\ref{sec:BG}. 
\end{remark}
In order to prove Proposition~\ref{prop:comingdown}, we need this a priori estimate on the solution to \eqref{eq:DPD3}:
\begin{proposition}[A priori estimate]\label{prop:apriori}
Let $T<\Bar{T}\wedge1$ and $v=X+Y$ with $(X,Y)$ the solution to \eqref{eq:system2} on $[0,T]\times\T^4$ with initial condition $(0,\Pi_Nv_0)$ for $v_0\in \cC^{\frac{3}{2}-2\epsilon}$, and pick $s,t\in[0,T]$, $s<t$. There exists a constant $\kappa>0$ such that $v$ obeys the following a priori estimate:
\begin{equs}
     \Vert v(t)\Vert^2_{L^2}+\int^t_s\Big(\Vert v(r) \Vert_{H^1}^2+\Vert v(r)\Vert^4_{M^4}\Big)\rmd r\lesssim \Vert \bbX^{\LD}_4\Vert_t^\kappa+ \Vert v(s)\Vert_{H^1}^{\frac{3}{2}}+\Vert v(s)\Vert_{M^4}^{3}\,.
\end{equs}
\end{proposition} 
\begin{proof}
    Recall that by the solution theory established in Section~\ref{subsec:4dfixedpoint}, $v$ solves \eqref{eq:DPD3}. Therefore, let us pair \eqref{eq:DPD3} the solution $v$. For $r\in[s,t]$, we obtain
\begin{equs}\label{eq:paired}
    \frac{1}{2}\partial_r\Vert v(r)\Vert_{L^2}^2&+\Vert v(r) \Vert_{H^1}^2+\Vert \nonumber v(r)\Vert^4_{M^4}=-\Big(3\Xvvv+\XXvv+\XvXv+\vXXv+(\cS,v)\Big)(r)\,,
\end{equs}
where $\Xvvv\eqdef(\pvXv,v)$, $\XXvv\eqdef(\pXXv,v)$, $\XvXv\eqdef(\pXvX,v)$ and $\vXXv\eqdef(\pvXX,v)$. Note that by Lemma~\ref{lem:continuity}, the previous equation truly makes sens only integrated between two times, so that \eqref{eq:paired} should really be understood as $\int_s^t LHS = \int_s^t RHS$. However, for the sake of notation, time integration is postponed to the end of the proof. 

We now use the bounds on the mixed terms that are proven in Appendix~\ref{sec:mixedterms}. 
It is at this step that we take advantage of the fact that we have pushed the Da Prato - Debussche expansion far enough so that all the remaining purely stochastic terms on the RHS are of regularity $-1/2-$, this allows us to control the pairing $(\cS,v)$ using the $H^1$ norm of $v$. 
By \eqref{eq:Sv}, \eqref{eq:vXXv}, \eqref{eq:XXvv}, \eqref{eq:XvXv} and \eqref{eq:Xvvv}, we infer the existence of an exponent $\kappa>0$ such that
\begin{equs}
       \frac{1}{2}\partial_r&\Vert v(r)\Vert_{L^2}^2+\Vert v(r) \Vert_{H^1}^2+\Vert v(r)\Vert^4_{M^4}\leqslant 8(\delta^{-1}\Vert \bbX^{\LD}_4\Vert_r)^\kappa+8\delta\big(\Vert v(r) \Vert_{H^1}^2+\Vert v(r)\Vert^4_{M^4}\big)\,.
\end{equs}
Setting $\delta=\frac{1}{16}$ yields
\begin{equs}\label{eq:intermediatestep}
   \partial_r\Vert v(r)\Vert_{L^2}^2+\Vert v(r) \Vert_{H^1}^2+\Vert v(r)\Vert^4_{M^4}\lesssim \Vert \bbX^{\LD}_4\Vert_r^\kappa\,, 
\end{equs}
which we finally integrate between $s $ and $t$. We thus have
\begin{equs}
    \Vert v(t)\Vert^2_{L^2}+\int^t_s\Big(\Vert v(r) \Vert_{H^1}^2+\Vert v(r)\Vert^4_{M^4}\Big)\rmd r\lesssim \Vert \bbX^{\LD}_4\Vert_t^\kappa+ \Vert v(s)\Vert^2_{L^2}\,.
\end{equs}
We deduce from the interpolation inequality \eqref{eq:interpol2} that
\begin{equs}
    \Vert v(s)\Vert^2_{L^2}\lesssim \Vert v(s)\Vert_{M^4}\Vert v(s)\Vert_{H^1}\lesssim \Vert v(s)\Vert_{H^1}^{\frac{3}{2}}+\Vert v(s)\Vert_{M^4}^{3}\,,
\end{equs}  
\end{proof}
We need the following rewriting of the previous proposition:
\begin{corollary}
    Define $F(t)\eqdef  \Vert v(s)\Vert_{H^1}^{\frac{3}{2}}+\Vert v(s)\Vert_{M^4}^{3}$, uniform in  $T<\Bar{T}\wedge1$, $s,t\in [0,T]$, $s<t$, 
    \begin{equs}
{}& \Vert v(t)\Vert^2_{L^2}\lesssim \Vert \bbX^{\LD}_4\Vert_t^\kappa+F(s)\,, \label{eq:vL2&F} \\
{}& \int_s^t F(r)^{\frac{4}{3}}\rmd r\lesssim   \Vert \bbX^{\LD}_4\Vert_t^\kappa+F(s)\,. \label{eq:comparF}
    \end{equs}
\end{corollary}
Finally, we also need this comparison principle introduced by Mourrat \& Weber: 
\begin{lemma}[Comparison: \cite{MW17Phi43}, Lemma 7.3]\label{lem:compa}
 Let $G:[0,T)\rightarrow \mathbb{R}_+$ continuous and such that for all $s,t\in[0,T), s<t$, $ \int_s^t G(r)^{\frac{4}{3}}\rmd r\lesssim \mathfrak c G(s)$. Then there exists a sequence of times $0<t_1<\dots<t_N=T$ such that for all $n\in\{0,\dots,N-1\}$, we have
 \begin{equs}
     G(t_n)\lesssim \mathfrak{c}^3t_{n+1}^{-3}\,.
 \end{equs}
\end{lemma}
With these lemmas at hand, we can now prove Proposition~\ref{prop:comingdown}:
\begin{proof}[of Proposition~\ref{prop:comingdown}]
    Pick $t\in[0,T]$ for $T<\Bar{T}\wedge1$. If there exists $s<t$ such that $F(s)\leqslant \Vert \bbX^{\LD}_4\Vert^\kappa_t$, then using \eqref{eq:vL2&F}, we immediately conclude that $\Vert v(t)\Vert_{L^2}\lesssim\Vert \bbX^{\LD}_4\Vert_t^{\frac{\kappa}{2}}$. Otherwise, suppose that for all $s\in[0,t]$, we have $F(s)\geqslant  \Vert \bbX^{\LD}_4\Vert^\kappa_t$. Then, by Lemma~\ref{lem:compa}, we thus have a sequence $0<t_1<\dots<t_N=T$ such that $F(t_n)\lesssim t^{-3}_{n+1}$. On the other hand, there exists a $n$ such that $t\in[t_n,t_{n+1})$, so that taking $s=t_n$ in \eqref{eq:vL2&F} yields $\Vert v(t)\Vert_{L^2}\lesssim \max\big(1,t^{-\frac{3}{2}}_{n+1}\big)\lesssim \max\big(1,t^{-\frac{3}{2}}\big) $. Letting $T=1$, we then iterate the argument on $[1,2]$, starting the solution from $v(1)$, and then keep iterating in the same way. The proof then follows from the fact that the estimate we have obtained is uniform in the initial condition. 
\end{proof}

We also have the weaker bound on the $L^2$ norm of $v$, that depends on the initial condition:
\begin{proposition}\label{prop:aprioriL2andH1}
    Pick $p\geqslant2$ and $0\leqslant s<t<\Bar{T}$. There exists a constant $\gamma>0$ such that it holds
\begin{equs}\label{eq:aprioriL2}
    \int_s^t\Vert v(r)\Vert^p_{L^2(\T^d)}\rmd r&\lesssim(t-s)\Vert\bbX^{\LD}_4\Vert_t^\gamma+\Vert v(s)\Vert^p_{L^2(\T^d)}\,.
\end{equs}
\end{proposition}
\begin{proof}
We start from \eqref{eq:intermediatestep} and use the positivity of the $M^4$ norm to obtain
\begin{equs}
    \partial_r\Vert v(r)\Vert_{L^2}^2+\Vert v(r) \Vert_{H^1}^2\lesssim \Vert \bbX^{\LD}_4\Vert_r^\kappa\,.
\end{equs}    
Multiply this inequality by  $\Vert v(r)\Vert_{L^2}^{p-2}$ yields for all $\delta\in(0,1)$
\begin{equs}
    \partial_t\Vert v(r)\Vert_{L^2}^p+\Vert v(r) \Vert_{L^2}^p&\lesssim \Vert \bbX^{\LD}_4\Vert_r^\kappa \Vert v(r)\Vert_{L^2}^{p-2}\leqslant C_\delta \Vert \bbX^{\LD}_4\Vert_r^{\frac{p\kappa}{2}}+\delta \Vert v(r)\Vert_{L^2}^{p}\,,
\end{equs}  
where we used the embedding $H^1\hookrightarrow L^2$, and then Young's inequality. The $L^2$ norm can then be reabsorbed on the LHS, and integrating between $s$ and $t$ finally yields the desired result.
\end{proof}

\subsection{Tightness of the invariant measure}
We recall from the introduction the definition of $\nu_N$, \eqref{eq:local_phi_4_renorm} and the fact that \eqref{eq:T^4_4} is chosen in a way such that at fixed $N\in\N$, it leaves $\nu_N$ invariant. 
Combining this with our $L^2$ estimates for the Langevin dynamic allows us to show tightness of the regularized and renormalized $\rmT^4_4$ measures. 
\begin{proposition}\label{prop:tightness}
  The family $(\nu_N)_{N\in\N}$ is tight on $H^{-1-\epsilon}(\T^4)$. Moreover, for all $p\in[1,\infty)$, any subsequential limit $\nu$ satisfies the bounds: 
   \begin{equs}\label{eq:nu1}
       \E_\nu[\Vert\phi\Vert_{H^{-1-\epsilon}(\T^d)}^p]&<\infty\,,\\\label{eq:nu2}
       \E_\nu[\Vert \phi-\bX\Vert_{L^2(\T^d)}^p]&<\infty\,.
   \end{equs} 
\end{proposition}
In order to prove Proposition~\ref{prop:tightness}, we first rewrite \ref{eq:T^4_4} in Fourier space and for finite $N\in\N$ as a finite dimensional system of $(2N+1)^4$ stochastic ordinary differential equations:
\begin{equs}\label{eq:ODS}
    \rmd \hat{\phi}_m &=\Big( \mfC_N\hat{\phi}_m-\angp{m}{2}\hat{\phi}_m-\sum_{c=1}^4\sum_{n\in\Z^4_N} \hat{\phi}_{m_{\hc},n_c}\hat{\phi}_{-n}\hat{\phi}_{n_{\hc},m_c} \Big)\mrd t+\rmd B_t^m\\
    &=:b^m(\phi)\rmd t+\rmd B_t^m
    \,,\nonumber
\end{equs}
where $B^m_t\eqdef \int_{\T^4\times\R_{\geqslant 0}} \1_{[0,t]}(s) e^{\imath x\cdot m }\xi(\rmd x,\rmd s)$ is a complex Brownian motion and $\mfC_N$ is a short-hand notation for $\mfC^1_N-\mfC^2_N$, and we recall that we have $|\mfC_N|\lesssim N$. 
This  allows us to obtain global in time existence for the solution to \eqref{eq:T^4_4} at fixed $N$ using the following explosion criterion of Khasminskii.
\begin{lemma}[\cite{khasminskii}, Theorem 3.5]\label{thm:Khasminskii}
Consider a random vector $(X^i(t))_{i\in I}$ with $I$ a finite set solution to the system of stochastic differential equations
\begin{equs}
    \rmd X^i(t)=b^i(X(t))\rmd t+\sigma^i(X(t))\rmd B_t\,, \,\, i\in I\,.
\end{equs}
  Then if $b$ and $\sigma$ are locally Lipschitz, and if there exists a Lyapounov function $V(X)$ and a positive constant $C$ such that $\inf_{x:|x|>R}V(x)\rightarrow +\infty$ when $R\rightarrow +\infty$ and $LV(x)\leqslant CV(x)$ with $L\eqdef \sum_{i\in I}b^i(x)\partial_i+\frac{1}{2}\sum_{i,j\in I^2}\sigma^i(x)\sigma^j(x)\partial_i\partial_j$, the solution starting from $X_0$ at $t=0$ is a regular almost surely continuous Markov process. 
\end{lemma}
\begin{corollary}\label{coro:globaltime}
   For fixed $N\in\N$, \eqref{eq:ODS} admits global in time solutions for any initial data and the same holds for   \eqref{eq:T^4_4} for data of the form $\Pi_{N} v_{0}$.  
\end{corollary}
\begin{remark}
In contrast to Proposition~\ref{prop:4dfixedpoint}, we only prove Corollary~\ref{coro:globaltime}   for fixed cut-off $N$ make not  claim about the global in time existence of the limiting $N\uparrow \infty$ local in time solution. 
\end{remark}
\begin{proof}
We apply Lemma~\ref{thm:Khasminskii} to the system~\eqref{eq:ODS}. 
We use for a Lyapounov function the $L^2$ norm:
 \begin{equs}    V(\hat{\phi})&\eqdef1+\sum_{m\in\Z^4_N}|\hat{\phi}_m|^2=1+\Vert \Pi_N\phi\Vert^2_{L^2(\T^4)}\,.
\end{equs} 
Whenever $V(\hat{\phi})$ is a function of $(\hat{\phi}_m)_{m\in\Z^4_N}$ we write $\partial_m V\eqdef \partial_{\hat{\phi}_m}V$ and $\Bar{\partial}_m V\eqdef \partial_{\hat{\phi}_{-m}}V$. 
A small twist is given by the fact that in our case, the diffusion is complex, so that in the definition of $L$, $b^i(x)\partial_i$ is replaced with $\frac{1}{2}\big(b^i(x)\partial_i+\Bar{b}^i(x)\Bar{\partial}_i\big)$. With this notation, one has 
\begin{equs}    
LV(\hat{\phi})&=\frac12\sum_{m\in\Z_N^4}\Big(b^m(\hat{\phi})\partial_m V(\hat{\phi})+\Bar{b}^m(\hat{\phi})\Bar{\partial}_m V(\hat{\phi})+\Bar{\partial}_m\partial_m V(\hat{\phi})\Big)\\
    &=\sum_{m\in\Z_N^4}\Big(\big(\mfC_N-\angp{m}{2}\big)|\hat{\phi}_m|^2-\sum_{c=1}^4\sum_{n\in\Z^4_N}\hat{\phi}_{-m}\hat{\phi}_{m_{\hc},n_c}\hat{\phi}_{-n}\hat{\phi}_{n_{\hc},m_c}+\frac{1}{2}\Big)\\
&=\mfC_N\Vert \Pi_N\phi\Vert^2_{L^2(\T^4)}-\Vert \Pi_N\phi\Vert^2_{H^1(\T^4)}-\cI(\Pi_N\phi)+\frac{(2N+1)^4}{2}\\
&\leqslant |\mfC_N|\Vert \Pi_N\phi\Vert^2_{L^2(\T^4)}+\frac{(2N+1)^4}{2}\lesssim N^4V(\hat{\phi})\,,
\end{equs}
where we used the positivity of $\Vert \Pi_N\phi\Vert^2_{H^1(\T^4)}$ and $\cI(\Pi_N\phi)$. 
Lemma~\ref{thm:Khasminskii} shows the explosion time of $\hat{u}_m$ is almost surely infinite. Observe that the constant $C$ from Lemma~\ref{thm:Khasminskii} is of order $N^4$ and blows up with $N$, here we do not work uniformly in $N$.
\end{proof}
We turn to the stationary solution  $\phi_N$ to \eqref{eq:T^4_4}, that is to say with initial solution drawn according to $\nu_N$ as defined in \eqref{eq:tensor_phi_4_renorm} \dash in particular $(\hat{\phi}_{N,m})_{m\in\Z^4_N}$ solves \eqref{eq:ODS}. We also write $v_N=\phi_N-\bX_N$. 
While $\bX$ is not stationary in time, it note that $\bX-\Picthree=\X-\Pictwo$ and  $v_N+\Picthree_N$ are stationary in time.
\begin{lemma}\label{dynamic_theorem2}
For $p\in[1,\infty)$, it holds
\begin{equs}
    \sup_{N\in\N}\E[\Vert v_N(0)\Vert^{p}_{L^2(\T^4)}]&<\infty\,.\label{eq:tighTL^2}
\end{equs}
\end{lemma} 
\begin{proof}
Without loss of generality we  assume $p \geqslant 2$. 
By stationarity, 
\begin{equs}
   \E[\Vert v_N(0)\Vert^p_{L^2(\T^4)}]&=\E[\Vert v_N(0)+\Picthree_N(0)\Vert^p_{L^2(\T^4)}]=\frac{1}{T}\int_0^T
\E[\Vert v_N(r)+\Picthree_N(r)\Vert^p_{L^2(\T^4)}]\rmd r \\
&\lesssim\frac{1}{T}\int_0^T\E[\Vert v_N(r)\Vert^p_{L^2(\T^4)}]\rmd r+\frac{1}{T}\int_0^T\E[\Vert \Picthree_N(r)\Vert^p_{L^2(\T^4)}]\rmd r\\
&\lesssim\frac{1}{T}\int_0^T\E[\Vert v_N(r)\Vert^p_{L^2(\T^4)}]\rmd r+\E[\Vert \Picthree_N\Vert^p_{C_T\cC^{1-\epsilon}(\T^4)}]
\end{equs}
Combining this with \eqref{eq:aprioriL2} gives
\begin{equs}
    \E[\Vert v_N(0)\Vert^p_{L^2(\T^4)}] \lesssim \E[\Vert\bbX^{\LD}_4\Vert_T^\gamma]+\frac{1}{T}\E[\Vert v_N(0)\Vert^p_{L^2(\T^4)}]\,.
\end{equs}
The result follows from taking $T$ sufficiently large which is allowed due to Corollary~\ref{coro:globaltime}.
\end{proof}

\begin{proof}[of Proposition~\ref{prop:tightness}]
\eqref{eq:nu2} directly follows from \eqref{eq:tighTL^2}. Regarding \eqref{eq:nu1}, \eqref{eq:tighTL^2} implies
   \begin{equs}       \sup_{N\in\N}&\E[\Vert\phi_N(0)\Vert_{H^{-1-\epsilon}(\T^4)}^p]\\&\leqslant
\sup_{N\in\N}\E[\Vert\phi_N(0)-\bX_N(0)\Vert_{H^{-1-\epsilon}(\T^4)}^p]+\sup_{N\in\N}\E[\Vert\bX_N(0)\Vert_{H^{-1-\epsilon}(\T^4)}^p]       \\&\leqslant
\sup_{N\in\N}\E[\Vert\phi_N(0)-\bX_N(0)\Vert_{L^2(\T^4)}^p]+\sup_{N\in\N}\E[\Vert\bX_N(0)\Vert_{\cC^{-1-\frac\epsilon2}(\T^4)}^p]<\infty\,,    
   \end{equs} 
using the regularity estimate on $\bX$ in Lemma~\ref{lem:stoob}. The same estimate holds in $H^{-1-2\epsilon}(\T^4)$, and tightness stems for the fact that the embedding $H^{-1-\epsilon}\hookrightarrow H^{-1-2\epsilon}$ is compact. 
\end{proof}

\section{The variational approach for tensor field theories}\label{sec:BG}
In this section, we follow the approach introduced in \cite{BG20} to establish tightness of approximations of the $\rmT^4_4$ measure. 
We follow notations of \cite{BG20} which we briefly recall here. 
We consider $(\Omega, \cB,\P)$ a probability space endowed with a collection of complex Brownian motions $(B^m_t)_{m\in\Z^d}$, such that $B^m$ and $B^n$ are independent if $m\neq\pm n$ and $B^{-m}=\overline{B^m}$, and denote by $(\cF_t)_{t\geqslant0}$ the filtration of $\cB$ induced by these Brownian motions. 
$\E$ will denote the expectation with respect to $\P$.  

Let also $\varrho$ be a smooth decreasing function $\R_{\geqslant 0}\rightarrow\R_{\geqslant 0}$ such that $\varrho\vert_{[0,\frac12]}=1$, $\varrho\vert_{[1,+\infty)}=0$ and $|\varrho'|\leqslant1$ on $[\frac{1}{2},1]$. For $t > 0$ we write  $\varrho_t(\cdot)\eqdef\varrho(\ang{\cdot}/t)$ and also set for $x\in\R^4$ $\sigma_t(x)\eqdef \sqrt{\partial_t \varrho^2_t(x)}$. Finally, we define the Fourier multiplier $J_t\eqdef\sigma_t(\nabla)\angp{\nabla}{-1}$, which verifies the following estimate.
\begin{lemma}\label{lem:J}
 For $\alpha\in\R$, $J_t$ is continuous from $H^{\alpha}$ to $H^{\alpha+1}$ uniformly in $t\geqslant0$.
\end{lemma}
\begin{proof}
Using that $\ang{x}/t<1$ and $|\varrho'|(\ang{x}/t)\leqslant 1$ on supp$\varrho_t$, we have 
\begin{equs}
    \sigma_t(x)&= \sqrt{\partial_t \varrho^2_t(x)}=\frac{1}{\sqrt{t}}\sqrt{2(\varrho|\varrho'|\mathrm{Id})(\ang{x}/t)}\lesssim t^{-1/2}\1_{\ang{x}<t}\\
    &\lesssim   t^{-1/2}\1_{\ang{x}<t}\1_{t<1}+ t^{-1/2}\1_{\ang{x}<t}\1_{t\geqslant1}  \,.
\end{equs}
If $t\geqslant1$, then $t^{-1/2}\leqslant1$ while if $t<1$, $\1_{\ang{x}<t}\leqslant\1_{\ang{x}<1}=0$ (by definition, $\ang{x}\geqslant1$). This means $\sigma_t(x)\lesssim1$ uniformly in time so, for $\phi\in H^\alpha$, $J_t(m)\hat{\phi}_m\lesssim\angp{m}{-1}\hat{\phi}_m$ which gives $J_t\phi\in H^{\alpha+1}$.
\end{proof}
We also define, for $T \geqslant 0$ and $\phi: [0,T] \times\T^4\rightarrow\R$, $J(\phi):[0,T] \times\T^4\rightarrow\R,\,(t,x)\mapsto J_t\phi(t,x)$ and $I(\phi):[0,T] \times\T^4\rightarrow\R,\,(t,x)\mapsto \int_0^tJ_s\phi(s,x)\rmd s$, and for $t\in[0,T]$ we sometimes write $I_t(\phi)(x)$ instead of $I(\phi)(t,x)$. $I$ satisfies the following estimate.  
\begin{lemma}[\cite{BG20}, Lemma 2]Let $\alpha\in\R$. 
Uniformly in $T \geqslant 0$ we have
\begin{equs}\label{eq:J}
\Vert I(\phi) \Vert_{L^\infty_TH^{\alpha+1}(\T^4)} \lesssim \Vert \phi\Vert_{L^2_TH^\alpha(\T^4)}\,.
\end{equs}
In particular, Lemma~\ref{lem:J} implies $\Vert I(J(\phi)) \Vert_{L^\infty_TH^{\alpha+2}(\T^4)} \lesssim \Vert J(\phi) \Vert_{L^2_TH^{\alpha+1}(\T^4)} \lesssim  \Vert \phi \Vert_{L^2_TH^{\alpha}(\T^4)}$ so
\begin{equs}\label{eq:SchauderBG}
  \Vert I(J(\phi)) \Vert_{L^\infty_TH^{\alpha+2-\epsilon}(\T^4)} \lesssim \Vert \phi\Vert_{L^\infty_T\mcC^{\alpha}(\T^4)}\,.  
\end{equs}
\end{lemma}

As mentioned before, in the stochastic processes appearing in Barashkov-Gubinelli approach the time $t$ will play the role of a scale parameter. 
The most basic stochastic object is then given by $\X_t\eqdef \int_0^t J_s\rmd B_s$ where $ B_s(x)$ denotes $(s,x)\mapsto\sum_{m\in\Z^4}e^{-\imath x\cdot m }\rmd B^m_s$. $\X_t$ should be seen as a cut-off Gaussian free field for $t < \infty$. 
 Indeed, for fixed $t$, $\X_t$  is Gaussian with covariance $\varrho_t^2(\nabla)\na^{-2}$ in space so that $\mathrm{Law}_{\P}
\X_t=\varrho_t(\nabla)_{\#}\mcb{g}$ \dash in particular $\mathrm{Law}_{\P}\X_\infty=\mcb{g}$. 
We denote by $\bH_a$ the space of all progressively measurable processes with respect to $(\cF_t)_{t\geqslant0}$ and $\P$-almost surely in $L^2(\R_{\geqslant 0}\times\T^d)$. We can now state the Bou\'e-Dupuis formula.
 \begin{theorem}[Bou\'e-Dupuis formula] Let $F: C_t\cC^{-\frac{d-2}{2}-\epsilon}\rightarrow\R$ be Borel measurable and such that there exist $1<p,q<\infty$ with $1/p + 1/q = 1$ for which $\E[|F(\X)|^p]$ and $\E[e^{-qF(\X)}]$ are both finite. Then one has the identity 
 \begin{equs}\label{eq:Boue-Dupuis}
-\log \E[e^{-F(\X)}]= 
\inf_{u\in \mathbb H_a}\E[F(\X+I(u))+\frac12\Vert u\Vert^2_{L^2([0,t]\times\T^d)}]\,.
\end{equs}
\end{theorem}

The approach of Barashkov and Gubinelli in \cite{BG20} starts with writing the measure $\nu_t$ defined in \eqref{eq:tensor_phi_4_renorm} with $\varrho_t$ as above as
\begin{equs}
    \E_{\nu_t}[e^{-f(\phi)}]=\frac{1}{\cZ_t}\E_{\P}[e^{-f(\X_t)-\cI(\X_t)+a_t\Vert\X_t\Vert^2_{L^2}+b_t}]\,,
\end{equs} 
where $\cZ_t\eqdef\E_{\P}[e^{-\cI(\X_t)+a_t\Vert\X_t\Vert^2_{L^2}+b_t}]$, and to apply the Bou\'e-Dupuis formula to $F(\X)=V_t^f(\X_t)\eqdef f(\X_t)+\cI(\X_t)-a_t\Vert\X_t\Vert_{L^2(\T^4)}^2-b_t$. 
The constant $b_{t}$ is a ``vacuum renormalization'' which is explicit and guarantees $\cZ_t$ satisfies uniform bounds as $t \rightarrow \infty$. We thus have the following variational representation of $\nu_t$:
\begin{equs}\label{eq:varnut}
  -\log  \E_{\nu_t}[e^{-f(\phi)}]=\log\cZ_t+\inf_{u\in\bH_a}
    \E[V^f_t(\X_t+I_t(u))+\frac12\Vert u\Vert^2_{L^2([0,t]\times\T^4)}]\,.
\end{equs}

Renormalization beyond Wick renormalization (as is needed for $\Phi^4_3$) can be carried out by introducing a shift of the drift $u$ in \eqref{eq:Boue-Dupuis} \dash this shift closely resembles the ansatz used in the analysis of the Langevin dynamic.
For $\rmT^4_4$ this shift does not belong to the Cameron-Martin space of $\mcb{g}$ (the Sobolev space $H^1$) which suggests the $\rmT^4_4$ measure is singular with respect to $\mcb{g}$.

\subsection{Introducing the BG stochastic objects}\label{sec:BG-stochasticobjects}
To describe this shift, we introduce stochastic objects analogous to those introduced to study the Langevin dynamic \dash we distinguish these from the former by calling them BG stochastic objects. 
We first give the BG analogs of the renormalization constants of Lemmas~\ref{lem:CT1} and~\ref{def:CT2}. 
\begin{definition}[BG renormalization constants]
\begin{equs}  
\mfC^{1,c}_t&\eqdef \sum_{m_{\hc}\in\Z^{d-1}}\frac{\varrho_t^2( m_{\hc})}{\langle m_{\hc}\rangle^2}\,,\\
\mfC^{2,c}_t&\eqdef 2\sum_{c'\neq c}\sum_{m_{\hc}n_{\hat{c}'}\in\Z^{d-1}}\int_0^t
\frac{\varrho^2_{t\wedge s}( m_{\hc})\partial_s\varrho_s^2( m_{\hc})}{\langle m_{\hc}\rangle^4}\bigg(\frac{\varrho_s^2(n_{\hc'},m_{c'})}{\langle(n_{\hc'},m_{c'})\rangle^2}- \frac{\varrho_s^2( n_{\hc'})}{\langle n_{\hc'}\rangle^2}\bigg)\rmd s\,,
\end{equs}
and define $\mfC^1_t\eqdef \sum_{c=1}^4 \mfC^{1,c}_t$ as well as $\mfC^2_t\eqdef \sum_{c=1}^4 \mfC^{2,c}_t$.
\end{definition}
We now describe the BG enhancement $\bbX^{\BG}$ of $\X$ consisting of various random fields and operators. 
 \begin{notation}
In our pictorial representation of the objects in $\bbX^{\BG}$, it is our convention that for a trees $\tau$ different from the noise $\noise$ , we always have $\rXtau_t=J_t\tau_t$ and $\Xtau_t=\int_0^tJ_s^2\tau_s \rmd s$.
\end{notation}
\begin{definition}[BG random fields]\label{def:BGrandomfields}
The BG random fields are given by
\begin{equs}
\X_t&\eqdef\int_0^tJ_s\rmd B_s
\,,\quad  \Xthreeloc_t\eqdef 4\big(\cN(\X_t,\X_t,\X_t)-\mfC^1_t\X_t\big)\,,\quad
     \rPictwo_t\eqdef J_t\Xthreeloc_t\,,\quad
     \Pictwo_t\eqdef\int^t_0 J_s\rPictwo_sds\,,\\
       \XtwoPictwo_t&\eqdef 4\big(\cN(\Pictwo_t,\X_t,\X_t)-\mfC^1_t\Pictwo_t\big)
+4\big(\cN(\X_t,\Pictwo_t,\X_t)+\cN(\X_t,\X_t,\Pictwo_t)-\mfC^2_t\X_t\big)\,,  \quad \rPicthree_t\eqdef J_t\XtwoPictwo_t\,,\\
 \Picthree_t&\eqdef\int^t_0 J_s\rPicthree_s\rmd s\,,\quad\bX_t\eqdef \X_t-\Pictwo_t+\Picthree_t\,,\quad \pXthreeloc_t\eqdef4\big(\cN(\bX_t,\bX_t,\bX_t) -(\mfC^1_t-\mfC^2_t) \bX_t\big)\,,\\
   \cS_t&\eqdef \pXthreeloc_t-\Xthreeloc_t+\XtwoPictwo_t \,,\quad \XXXX_t\eqdef \cI(\bX_t)-2(\mfC^{1}_t-\mfC^{2}_t)\Vert\bX_t\Vert^2_{L^2(\T^4)}\,.
\end{equs}\label{def:stoobBG}
\end{definition}
\begin{definition}[BG random operators]\label{def:BGrandomop}
The BG random operators are defined as
\begin{align*}
   \X^{(k)}_t(f)&\eqdef4 \int_{\T^{4-k}} \X_t(\cdot,y) f(y)\rmd y\,,&&\text{ for }k\in[d-1]\;,\; f:\T^{4-k}\rightarrow\R\,,\\
    \Xtwom^c_t(f)&\eqdef 2\big(\cN^c(f,\X_t,\X_t)-\mfC^{1,c}_tf\big) \,,&&\text{ for }c\in[4]\,,\;f:\T_c\rightarrow\R\,,
       \\ \Xtwonm^c_t(f)&\eqdef 2\cN^c(\X_t,\X_t,f)\,,&&\text{ for }c\in[4]\,,\;f:\T_{\hc}^{3}\rightarrow\R\,,
       \\ \XdotX_t(f)&\eqdef 2\cN(\X_t,f,\X_t)\,,&&\text{ for }f:\T^{4}\rightarrow\R\,\;.
\end{align*}
We also define $\bX_t$, $\pXtwom^c_t$, $\pXtwonm^c_t$ and $\pXdotX_t$ analogously, substituting every occurrence of $\X_t$ with $\bX_t$, and replacing $\mfC^{1,c}_t$ with  $\mfC^{1,c}_t- \mfC^{2,c}_t$ in the definition of $\pXtwom^c$. Again, we drop the color subscript $c$ on the random operators when there is a sum over $c$.
\end{definition}

Recalling the notation $\bbX=\bbX_f\cup\bbX_o$ from Subsection~\ref{subsec:4dfixedpoint}, we now define the tuple  $\bbX^{\BG}_f$ of all the random fields defined in Definition~\ref{def:stoobBG} except $\rPictwo$, $\rPicthree$ and $\XXXX$, and the tuple $\bbX^{\BG}_o$ of all the random operators defined in Definition~\ref{def:BGrandomop}. 

Note that all the stochastic objects we have introduced in this section are processes in time taking values in distributions over space or operators on either space distributions or space-time distributions. 
For $t\in\R_{\geqslant 0}$, we endow $\bbX^{\BG}(t)\eqdef\{\tau(t)|\tau\in\bbX^{\BG}\}$ with the norm
\begin{align}\label{eq:bbXnormBG}
    \Vert \bbX^{\BG}(t)\Vert=\max\Big(  \max_{\tau\in\bbX^{\BG}_f}\Vert \tau(t)\Vert_{\cC^{\beta_\tau-\epsilon}},\max_{\tau\in\bbX^{\BG}_o}\Vert \tau(t)\Vert_{\cL(C_tH^{\alpha_\tau},\cC^{\beta_\tau-\epsilon})} \Big)\,,
\end{align}
where $\alpha_\tau$ and $\beta_\tau$ are the inner and outer regularities of the given stochastic objects which are chosen the same as in Lemmas~\ref{lem:stoob} and \ref{lem:ranop} for the corresponding LD objects.
The next lemma states that we can control the purely stochastic objects appearing in the variational formula. 
\begin{lemma}\label{lem:BGreg}
\begin{equs}
   \sup_{t\in\R_{\geqslant 0}} \E[\Vert \bbX^{\BG}(t)\Vert^p]<\infty\,.
\end{equs}
\end{lemma}
\begin{proof}
The statement is proved combining the fixed time Kolmogorov estimates for the random fields \eqref{eq:FixedTimeKolmoRF} and for the random operators \eqref{eq:FixedTimeKolmoRO} with our bounds on the second moments of the stochastic objects derived in Sections~\ref{sec:6_3} and \ref{sec:6_5}. In particular, the second moments are bounded uniformly in the scale parameter in the proof of Lemma~\ref{lem:amplitude_bound}. We conclude by studying individually the power countings of the amplitudes of each objects in Sections~\ref{sec:proofs} and \ref{sec:proofs2}.
\end{proof}
\subsection{Renormalizing the Bou\'e-Dupuis formula}
We are ready to introduce the shift of the drift that renormalizes the Bou\'e-Dupuis formula which we will be similar to the one that was introduced to deal with the $\rmT^4_4$ equation.
The main estimate for the variational approach to $\rmT^4_4$ is given in the following proposition. 
\begin{proposition}\label{prop:BG_energy}  
Let $f:\cC^{-1-\epsilon}\rightarrow\R$ of at most linear growth and $V_t^f(\phi)=f(\phi)+\cI(\phi)-a_t\Vert\phi\Vert_{L^2(\T^4)}^2-b_t$ where 
\begin{equs}
    a_t&\eqdef 2(\mfC^1_t-\mfC_t^2)\,,\;b_t\eqdef\E[\XXXX_t+\Vert\rPictwo-\rPicthree\Vert^2_{L^2([0,t]\times\T^4)}]\,.
\end{equs}
Then the change of variable
\begin{equs}
    l^t_s(u)&\eqdef u_s+\1_{s\leqslant t}\big(\rPictwo_s-\rPicthree_s \big)\,,  
\end{equs}
renormalizes the Bou\'e-Dupuis formula, and it holds
    \begin{equs}\label{eq:BGdecom}
       \nonumber &\E[V^f_t(\X_t+I_t(u))+\frac12\Vert u\Vert^2_{L^2([0,t]\times\T^4)}]\\&\qquad=\E[\mathfrak M^f_t(\bbX^{\BG},K_t(u))+\cI(K_t(u))
        +\frac12\Vert l^t(u)\Vert^2_{L^2([0,t]\times\T^4)}]\,,
    \end{equs}
where $    K_s(u)\eqdef I_s(l^t(u))$ and
\begin{equs}
  \fM^f_t(\bbX^{\BG},v) \eqdef f(\bX_t+v_t)+\Big((\cS,v)+\vXXv+\XvXv+\XXvv+\Xvvv\Big)(t)\,,
\end{equs}
is 
a term collecting all the mixed terms, and that obeys the bound
\begin{equs}\label{eq:boundmixed}
    |\fM^f_t(\bbX^{\BG},K_t(u))|\leqslant C\big(1+\Vert\bbX^{\BG}(t)\Vert^\gamma\big)+\frac12\Big(\cI(K_t(u))+\frac12\Vert l^t(u)\Vert^2_{L^2([0,t]\times\T^d)}\Big)\,,
\end{equs}
where $C,\gamma$ are two positive constants. 
\end{proposition}
\begin{remark}
Our choice of shift of the drift is made to ensure that, with the above notations, $\mfM^0_t(\bbX^{\BG},v)=(\mathfrak N_t(\bbX^{\BG},v),v)$ with $\mathfrak N_t(\bbX^{\BG},v)\in H^{-1}$ if $v\in H^1$.
\begin{itemize}
\item The first shift by $\1_{s\leqslant t}\rPictwo_s$ takes care of the term $\Xthreeloc_t$ in $\mathfrak N_t$ that would be of regularity $-2-$, while the second shift by $-\1_{s\leqslant t}\rPicthree_s$ takes care of all the three terms of the form $\Pictwo_t\times\X_t\times\X_t$ that would be of regularity $-1-$.
\item The analysis of the Langevin dynamic and the observation that $\cN(\X_t,v,\X_t)$ is of regularity $-1-$ if $v\in H^1$ strongly suggest that we should cancel it by introducing the non constant term $J_s\cN(\X_s,I_s(u),\X_s)$ into the shift of the drift. However, like for the a priori $L^2$ estimates in Section~\ref{subsec:L2apriori}, we will take advantage of the fact that the non-melonic products do not require any renormalization and we can bound the duality pairing $(v,\cN(\X,v,\X))$ by the better behaved $(v,\cN(\X,\X,v))$ thanks to the Cauchy-Schwarz inequality \eqref{eq:CS1} (see also the proof of Lemma~\ref{lem:CSinaction}), so that $\cN(\X_t,v,\X_t)$ is harmless even if it is not in $H^1$. 
Note that this simplification appears to only hold for $d \in \{3,4\}$, in particular, covering the full subcritical regime likely requires a shift that depends on the drift $u$.
\item Note that our final variational formula differs from that of \cite{BG20} in that the interaction term is evaluated at $K(u)$ instead of $I(u)$. 
This is due to the fact that our shift of $\X_{t}$ doesn't depend on $u$, which allows us to simplify our proofs compared to those of \cite{BG20}. 
\end{itemize} 
\end{remark}
\begin{proof}
We sketch the proof since it is very similar to the $L^2$ estimates of Section~\ref{subsec:L2apriori} \dash we use the same notation as well. 
We again use estimates on the mixed terms given in Appendix~\ref{sec:mixedterms}. 
We also write $v\eqdef K(u)$. 
We start by observing that we can carry out renormalization cancellations by expanding
\begin{align}
&\cI(\X_t+I_t(u))-a_t\Vert\X_t+I_t(u)\Vert_{L^2(\T^4)}^2=  \cI(\bX_t+K_t(u))-a_t\Vert\bX_t+K_t(u)\Vert_{L^2(\T^4)}^2\nonumber\\
&\quad= \cI(\bX_t+v_t)-a_t\Vert\bX_t+v_t\Vert_{L^2(\T^4)}^2
\nonumber\\
&\quad=\Big(\XXXX+\XXXv+\vXXv+\XvXv+\XXvv+\Xvvv\Big)(t)+\cI(v_t)\nonumber\\&\quad=\Big(\XXXX+(\cS,v)+(\Xthreeloc-\XtwoPictwo,v)+\vXXv+\XvXv+\XXvv+\Xvvv\Big)(t)+\cI(v_t)\,.\label{eq:expansionI}
\end{align}
Moreover, using the martingale properties of the stochastic objects, the third term in \eqref{eq:expansionI} rewrites as
\begin{equs}
    (\Xthreeloc_t-\XtwoPictwo_t,v_t)=
    \int_0^t(\Xthreeloc_s-\XtwoPictwo_s,\dot{v}_s)\rmd s+M_t=
    \int_0^t(\Xthreeloc_s-\XtwoPictwo_s,\dot{K}_s(u))\rmd s+M_t\,,
\end{equs}
where $M_t$ stands for a martingale of null expectation. On the other hand, we have 
\begin{equs}
    \frac12\Vert u\Vert^2_{L^2([0,t]\times\T^4)}&=\frac12\Vert l^t(u)-\1_{[0,t]}(\rPictwo-\rPicthree)\Vert^2_{L^2([0,t]\times\T^4)}\\
    &=\frac{1}{2}\Vert l^t(u)\Vert^2_{L^2([0,t]\times\T^4)}+\frac12\Vert \rPictwo-\rPicthree\Vert^2_{L^2([0,t]\times\T^4)}-\int_0^t(\rPictwo_s-\rPicthree_s,l^t_s(u))\rmd s\\
    &=\frac{1}{2}\Vert l^t(u)\Vert^2_{L^2([0,t]\times\T^4)}+\frac12\Vert \rPictwo-\rPicthree\Vert^2_{L^2([0,t]\times\T^4)}-\int_0^t(\Xthreeloc_s-\XtwoPictwo_s,\dot{K}_s(u))\rmd s\,, 
\end{equs}
which makes the cancellation with the third term in \eqref{eq:expansionI} manifest, so that \eqref{eq:BGdecom} follows. 

We now have to prove \eqref{eq:boundmixed}. In this setting, we still can use \eqref{eq:Sv}, \eqref{eq:vXXv}, \eqref{eq:XXvv}, \eqref{eq:XvXv} and \eqref{eq:Xvvv}, and the proof is therefore exactly similar to that of Proposition~\ref{prop:apriori}. Regarding, the term $f(\bX_t+v_t)$, using the fact that $f$ is at most of linear growth we have, by Sobolev embedding,
\begin{equs}
    |f(\bX_t+v_t)|&\lesssim \Vert \bX(t)+v(t)\Vert_{\cC^{-1-\epsilon}}\lesssim \Vert \bX(t)\Vert_{\cC^{-1-\epsilon}}+\Vert v\Vert_{C_t\cC^{-1-\epsilon}}
    \lesssim \Vert \bX(t)\Vert_{\cC^{-1-\epsilon}}+\Vert v\Vert_{C_t H^{1-\epsilon}}\\&\leqslant C_\delta\big(1+\Vert \bbX^{\BG}(t)\Vert\big)+\delta \Vert v\Vert_{C_tH^{1-\epsilon}}^2\,,
\end{equs}
where we recall that $\Vert v \Vert^2_{C_tH^{1-\epsilon}}=\Vert K(u)\Vert^2_{C_tH^{1-\epsilon}}\lesssim \Vert l^t(u)\Vert^2_{L^2([0,t]\times\T^4)}$ by \eqref{eq:J}. 
\end{proof}
We can now prove the main theorem of this section. 
\begin{proof}[Theorem~\ref{BG_theorem1}]
We first prove that $\log \cZ_t$ is bounded uniformly in $t$. By the Bou\'e-Dupuis formula~\eqref{eq:Boue-Dupuis}, and then \eqref{eq:BGdecom}, $\log \cZ_t$ verifies
\begin{equs}
    -\log  \cZ_t &=\inf_{u\in\bH_a}
    \E[V^0_t(\X_t+I_t(u)))+\frac12\Vert u\Vert^2_{L^2([0,t]\times\T^4)}]\\&=\inf_{u\in\bH_a}\E[\mathfrak M^0_t(\bbX^{\BG},K_t(u))+\cI(K_t(u))
        +\frac12\Vert l^t(u)\Vert^2_{L^2[0,t]\times\T^4)}]\,.
\end{equs}
For any adapted process $u\in\bH_a$, using \eqref{eq:boundmixed}, we thus have
\begin{equs}
&\E[\mathfrak M^0_t(\bbX^{\BG},K_t(u))+\cI(K_t(u))
        +\frac12\Vert l^t(u)\Vert^2_{L^2([0,t]\times\T^4)}]\\
&\quad\quad\quad\geqslant\E[-C\big(1+\Vert\bbX^{\BG}(t)\Vert^\gamma\big)+\frac{1}{2}\Big(\cI(K_t(u))
        +\frac12\Vert l^t(u)\Vert^2_{L^2([0,t]\times\T^4)}\Big)]
        \\
&\quad\quad\quad\geqslant-C\big(1+\E[\Vert\bbX^{\BG}(t)\Vert^\gamma]\big)\geqslant-C\big(1+\sup_{t\in\R_{\geqslant 0}}\E[\Vert\bbX^{\BG}(t)\Vert^\gamma]\big)\,, 
\end{equs}
which implies that $-\log\cZ_t\geqslant-C\big(1+\sup_{t\in\R_{\geqslant 0}}\E[\Vert\bbX^{\BG}(t)\Vert^\gamma]\big)$. 
On the other hand, for any $u\in\bH_a$
\begin{equs}
 -\log\cZ_t&\leqslant \E[V^0_t(\X_t+I_t(u)))+\frac12\Vert u\Vert^2_{L^2([0,t]\times\T^4)}]\\
 &\leqslant \E[C\big(1+\Vert\bbX^{\BG}(t)\Vert^\gamma\big)+\frac{1}{2}\Big(\cI(K_t(u))
        +\frac12\Vert l^t(u)\Vert^2_{L^2([0,t]\times\T^4)}\Big)]\,.
\end{equs}
Taking $u_s=-\1_{s\leqslant t}\big(\rPictwo_s-\rPicthree_s\big)$ (thus requiring that $l^t(u)=0$) yields an upper bound on $\log \cZ_t$ that only depends on the stochastic objects and is uniformly bounded in time. Finally, we have thus obtained that it holds uniformly in time 
\begin{equs}
    |\log \cZ_t|\lesssim 1+ \sup_{t\in\R_{\geqslant 0}}\E[\Vert \bbX^{\BG}(t)\Vert^\gamma]<\infty\,.
\end{equs}
The representation \eqref{eq:varnut} implies that repeating the previous argument with $V^f$ instead of $V^0$ immediately shows that $\nu_t$ verifies the same bound uniform in time, giving us the desired estimate \eqref{eq:laplace_bound}.  The fact that tightness of $(\nu_{t})_t$ follows from uniform bounds on the Laplace transforms of $(\nu_{t})_{t}$ is standard.
\end{proof}

\section{Kolmogorov arguments for random fields and operators}\label{sec:stochasticobjects}
In our analysis we will need to construct and obtain path-wise regularity estimates on various explicit stochastic objects. 

For most of these objects (which include both random fields and operators), we will proceed by formulating Kolmogorov criteria Lemmas~\ref{lem:stoobsobolev}, \ref{lem:randomopsob}, and~\ref{lem:randomopsob2}. 
When these criteria are combined with moment estimates, they will 
give the results stated in Lemmas~\ref{prop:stochastic_1}, \ref{lem:randop1}, \ref{lem:randfi2}, \ref{lem:randfi4}, \ref{lem:randop2}, \ref{lem:randop_4} and~\ref{lem:BGreg}. 

There are differences in both the definition and type of control we obtain on the objects needed for the Langevin dynamic (LD) versus those needed for the Barashkov \& Gubinelli approach (BG). 
On immediate difference is the role played by time, for the case of the free field recall that $\X^{\text{LD}}=\sqrt{2}\underline{\cL}^{-1}\xi$ while $\X^{\text{BG}}= \int_0^tJ_s\mrd B_s$.

The control of the LD objects is harder \dash we need to control space \textit{ and time} regularity of the LD object \dash for instance note that uniform space-time control is inside the expectation in the statements for the Lemmas~\ref{prop:stochastic_1}, \ref{lem:randop1}, \ref{lem:randfi2}, \ref{lem:randfi4}, \ref{lem:randop2}, and \ref{lem:randop_4} while for BG objects one needs space-regularity estimates along with control over expectations uniform in time - note that the supremum in time is outside of the expectation. in Lemma~\ref{lem:BGreg}. 

For this reason, our discussion will primarily focus on the LD objects, and we point where and how extra considerations should be kept in mind for the BG objects. 
 
\subsection{Random fields}
We first define the LD random fields and investigate their regularity for $d\in\{2,3,4\}$.
\begin{definition}[LD random fields]\label{def:randomfields}
Recall that the LD random fields are defined as
\begin{equs}
    \X_{N}&\eqdef \underline{\cL}^{-1}\xi_{N}\,,\quad
   \Xthreeloc_{N}\eqdef \cN_N(\X_{N},\X_{N},\X_{N})-\mfC^1_{N}(d)\X_{N}\,,\quad
     \Pictwo_{N}\eqdef \underline{\cL}^{-1}\Xthreeloc_{N}\,,\\
       \XtwoPictwo_{N}&\eqdef      \big(\cN_N(\Pictwo_{N},\X_{N},\X_{N})-\mfC_{N}^1(d)\Pictwo_{N}\big)       +         \big(\cN_N(\X_N,\Pictwo_{N},\X_{N})+\cN_N(\X_{N},\X_N,\Pictwo_{N})-\mfC^2_{N}(d)\X_{N}\big)
       \,,\\\Picthree_{N}&\eqdef {\cL}^{-1}\XtwoPictwo_{N}\,,\quad  \bX_{N}\eqdef \X_{N}-\Pictwo_{N}+\Picthree_{N}\,,
      \quad     \pXthreeloc_{N}\eqdef
    \cN_N(\bX_{N},\bX_{N},\bX_{N})-\big(\mfC^1_{N}(d)-\mfC^2_{N}(d)\big)\bX_{N}\,,\\
          \cS_{N}&\eqdef \pXthreeloc_{N}-\Xthreeloc_{N}+\XtwoPictwo_{N}\,.
\end{equs}
\end{definition}
\begin{notation}\label{not:notation}
In what follows we associate to each $\tau$ a power-counting $|\tau| \in \mathbb{R}$ as a useful notation. 
It is not always the case that $|\tau|$ is the regularity or order/homogenity of the random field associated to $\tau$ \dash this is because we do not positively renormalize as in regularity structures or use resonant products as in paracontrolled calculus.
\begin{figure}[H] 
\centering
\begin{tabular}{c|ccccccc}  $\tau$&$\X$&$\Xtwom/\dotXXloc$&$\Xtwonm/\XXdotloc$&$\Xthreeloc$&$\XdotX/\XdotXloc$&$\Sym$\\\hline $|\tau|$&$-\frac{d-2}{2}$&$-\frac{2d-5}{2}$&$-\frac{2d-5}{2}$&$-\frac{3d-8}{2}$&$-(d-2)$&$-(2d-7)$ 
\end{tabular}
\caption{Power counting of the main stochastic objects}
\end{figure}
If $\tau$ is a random operator of argument $f$, we partition its vertices in four categories: the vertices of the form $\cN(\tau_1,\tau_2,\tau_3)=\cXthreeloc$ for $\tau_1=\rnoise$, $\tau_2=\gnoise$ and $\tau_3=\bnoise $ some stochastic objects, of the form $\cN(\tau_1,\tau_2,f)=\cXXdotloc$, of the form $\cN(f,\tau_1,\tau_2)=\cdotXXloc$ and of the form $\cN(\tau_1,f,\tau_2)=\cXdotXloc$. All the vertices of a random field are of the form $\cXthreeloc$. With this convention, the general formula is given by 
\begin{equs}
|\tau|\eqdef &-\frac{d+2}{2}\#\text{noises}\;\noise\;(\tau)+\#\text{vertices}\;\cXthreeloc\;(\tau)+\frac{1}{2}\#\text{vertices}\cdotXXloc\;(\tau)\\&+\frac12\#\text{vertices}\;\cXXdotloc(\tau)+2\#\text{edges}\;|\;(\tau)\,.
\end{equs}
\end{notation}

All the stochastic objects we are considering are some variations of the seven main stochastic objects, either because we rather see them as random operators instead of random fields, or because we study them with the rough shift $\bX$ instead of $\X$.

Let us now recall the properties of the random fields:
\begin{lemma}\label{lem:stoob}In any dimension $d\geqslant2$, 
\begin{equs}\label{eq:regX}
\sup_{N\in\N}\E[\Vert \X_{N}\Vert_{C_T\cC^{|\X|-\epsilon}(\T^{d})}^p]&<\infty\,,
 \end{equs}
 Moreover, for $d\in\{2,3,4\}$, 
 \begin{equs}
\sup_{N\in\N} \E[\Vert \Xthreeloc_{N}\Vert_{C_T\cC^{|\Xthreeloc|-\epsilon}(\T^{d})}^p]&<\infty\,,\\ 
\label{eq:regXtwoPictwo}
\sup_{N\in\N}\E[\Vert \XtwoPictwo_{N}\Vert^p_{C_T\cC^{\min(-\frac{(d-3)}2,|\XtwoPictwo|)-\epsilon}(\T^{d})}]&<\infty\,,
   \\ 
   \sup_{N\in\N} \E[\Vert \cS_{N}\Vert_{C_T\cC^{-\frac{d-3}{2}-\epsilon}(\T^{d})}^p]&<\infty\,.
 \end{equs}
Furthermore, for every $\tau$ above of regularity $\beta_\tau$, the corresponding BG random field belongs to $\cC^{\beta_\tau-\epsilon}$ for all time $t\in\R_{\geqslant 0}$, and obeys a bound uniform in time:
\begin{equs}
    \sup_{t\in\R_{\geqslant 0}}\E[\Vert \tau(t)\Vert^p_{\cC^{\beta_\tau-\epsilon}}]<\infty\,.
\end{equs}
\end{lemma}
Note that we do not include $\Pictwo $, $\Picthree$ and $\bX$ in the previous list, since we can control them using the previous bounds, the Schauder estimate \eqref{eq:schauder2} and \eqref{eq:SchauderBG}.
\subsection{Random operators}
We state definitions and regularity estimates for the LD random operators.
\begin{definition}[LD random operators]\label{def:randomop}
Recall that the LD random operators are defined as
\begin{align*}  
  \X^{(k)}_{N}(f)&\eqdef \int_{\T^{d-k}} \X_{N}(\cdot,y) f(\cdot,y)\rmd y\,,  &&\text{ for }k\in[d-1]\;,\;f:\R_{\geqslant 0}\times\T^{d-k}\rightarrow\R\,,\\
    \Xtwom^c_{N}(f)&\eqdef \cN^c_N(f,\X_N,\X_N)-\mfC^{1,c}_N(d)f\,,   &&\text{ for }c\in[d]\,,\;f:\R_{\geqslant 0}\times\T_c\rightarrow\R 
    \,,\\ \Xtwonm^c_{N}(f)&\eqdef \cN^c_N(\X_N,\X_N,f)  \,, &&\text{ for }c\in[d]\,,\;f:\R_{\geqslant 0}\times\T_{\hc}^{d-1}\rightarrow\R 
  \,,\\ \XdotX_{N}(f)&\eqdef \cN_N(\X_{N},f,\X_{N}) \,,&&\text{ for }f:\R_{\geqslant 0}\times\T^d\rightarrow\R\,,\\
     \Sym_{N}(f)&\eqdef
    \XdotX_{N}(\cL^{-1}\ 
    \XdotX_{N}(f))\,, &&\text{ for }f:\R_{\geqslant 0}\times\T^d\rightarrow\R\,,
\end{align*}
and $\bX_{N}$, $\pXtwom^c_{N}$, $\pXtwonm^c_{N}$, $\pXdotX_{N}$ and $\pSym_{N}$ are defined accordingly, substituting every occurrence of $\X_{N}$ with $\bX_{N}$, and replacing $\mfC^{1,c}_N(d)$ with  $\mfC^{1,c}_N(d)- \mfC^{2,c}_N(d)$ in the definition of $\pXtwom^c$.
\end{definition}

\begin{lemma}\label{lem:ranop}
In any dimension $d\geqslant2$, for all $\alpha>\frac{d-k-2}{2}$, 
\begin{equ}\label{eq:X_rand_op}  
\sup_{N\in\N}\E[\Vert\X^{(k)}_{N}\Vert^p_{\cL(C_TH^{\alpha}(\T^{d-k}),C_T\cC^{\min(-\frac{k-2}{2},\alpha+|\X|)-\epsilon}(\T^k))}]<\infty\,.
\end{equ}
For $d\in\{2,3,4\}$, $c\in[d]$, and all $\alpha>-\frac{1}{2}$, 
\begin{equs}  \label{eq:regXtwom}  
 \sup_{N\in\N}  \E[\Vert\Xtwom^c_{N}\Vert^p_{\cL(C_TH^{\alpha}(\T),C_T\cC^{\min(\frac{1}{2},\alpha+|\Xtwom|)-\epsilon}(\T))}]<\infty\,.
\end{equs}
In any dimension $d\geqslant2$, for all $c\in[d]$ and $\alpha>\frac{d-3}{2}$, 
\begin{equs}\label{eq:regXtwonm}
   \sup_{N\in\N}  \E[\Vert\Xtwonm^c_{N}\Vert^p_{\cL(C_TH^{\alpha}(\T^{d-1}),C_T\cC^{\min(-\frac{d-3}{2},\alpha+|\Xtwonm|)-\epsilon}(\T^{d-1}))}]<\infty\,.
\end{equs}
In any dimension $d\geqslant2$, for all $\alpha>\frac{d-3}{2}$, 
\begin{equs}    \label{eq:regXdotX}
\sup_{N\in\N}\E[\Vert\XdotX_{N}\Vert^p_{\cL(C_TH^{\alpha}(\T^{d}),C_T\cC^{\min(-\frac{d-3}{2},\alpha+|\XdotX|)-\epsilon}(\T^{d}))}]<\infty\,.
\end{equs}
In any dimension $d\geqslant2$, for all $\alpha>\frac{d-3}{2}$, 
\begin{equs}    
    \sup_{N\in\N} \E[\Vert\Sym_{N}\Vert^p_{\cL(C_TH^{\alpha}(\T^{d}),C_T\cC^{\min(-\frac{d-3}{2},\alpha+|\Sym|)-\epsilon}(\T^{d}))}]<\infty\,.
\end{equs}
Moreover, the purple objects constructed with the rough shift $\bX$ obey the same bounds as their black counterparts constructed with $\X$, but with the restriction that the dimension must be $\leqslant4$. 

Furthermore, for $\tau$ as above of inner regularity $\alpha_\tau$ and outer regularity $\beta_\tau$, the corresponding BG random operator belongs to $\cL\big(C_tH^{\alpha_\tau},\cC^{\beta_\tau-\epsilon}\big)$ for any $t\in\R_{\geqslant 0}$, and obeys the bound 
\begin{equs}
    \sup_{t\in\R_{\geqslant 0}}\E[\Vert \tau(t)\Vert^p_{\cL(C_tH^{\alpha_\tau},\cC^{\beta_\tau-\epsilon})}]<\infty\,.
\end{equs}
\end{lemma}
\begin{remark}
    The proofs of the regularity properties of the random fields and the random operators are essentially similar. In particular repeating the argument used in order to establish \eqref{eq:regXtwom} and \eqref{eq:regXtwonm} yields the regularities of the random fields $\Xtwom$ and $\Xtwonm $ given in Lemma~\ref{prop:stochastic_1}.
\end{remark}

\subsection{Kolmogorov criteria}

In Lemmas~\ref{lem:stoobsobolev} and~\ref{lem:randomopsob} we give two estimates, the first being an estimate at a fixed time and the second an estimate on an homogenous H\"{o}lder norm in time, together these give control over the corresponding inhomogenous H\"{o}lder norm in time for the LD objects.
For the BG object, only the first estimate of  Lemmas~\ref{lem:stoobsobolev} and~\ref{lem:randomopsob} is relevant. 

Lemma~\ref{lem:randomopsob2} gives an estimate for an homogenous norm for particular class of random operators of depending on two times, and is solely used for the LD object $\Sym$. 

We also note, since in Section~\ref{sec:diagram} we obtain covariance estimates in terms of Fourier variables, it is convenient to do the same for Kolmogorov estimates and so we simplify our estimates using appropriate stationarity of the stochastic objects.

\subsubsection{Kolmogorov criteria for random fields}
We start with the Kolmogorov argument for the random fields. Below, $T$ is a fixed positive real number, and we write $\delta_{s,t}\tau\eqdef\tau(t)-\tau(s)$.

\begin{lemma}\label{lem:stoobsobolev}
For any fixed $k\geqslant1$ and $p \in 2\N$, uniform over smooth stationary in space random fields $\tau$ over $\R_{\geqslant0}\times\T^{d}$ belonging to the $k$-th inhomogeneous Gaussian chaos, 
one has 
\begin{equ}  \label{eq:FixedTimeKolmoRF} 
\E [\Vert\tau(t) \Vert_{\cC^{\beta-\epsilon}}^p ]\lesssim 
\bigg(
  \sum_{m\in\Z^d}\angp{m}{2\beta}\E [\hat{\tau}_m(t) \hat{\tau}_{-m}(t) 
  ]\bigg)^{\frac{p}{2}}
    \,.
\end{equ} 
In particular, for any fixed $\theta > \kappa > 0$ and uniform over smooth stationary in space random fields $\tau$ over $\R_{\geqslant 0} \times \T^{d}$ belonging to the $k$-th inhomogenous Gaussian chaos, one has 
\begin{equ}
\E [ 
\Vert \tau \Vert_{\dot{C}^{\theta - \kappa }_T\cC^{\beta-\epsilon}}^{p} ]
\lesssim \sup_{\substack{0\leqslant s,t\leqslant T\\|t-s|\leqslant 1}}{|t-s|^{-\frac{\theta p}{2}}}  
{\bigg(
    \sum_{m\in\Z^d}\angp{m}{2\beta}\E [\delta_{s,t}\hat{\tau}_m\delta_{s,t}\hat{\tau}_{-m}]\bigg)^{\frac{p}{2}}}
    \,.
\end{equ}
\end{lemma}
\begin{proof}
We focus on the second statement, the first one will be clear from our argument.  

Recall that for $k > \frac{d}{\epsilon}$ we have the compact embedding 
\begin{equ}\label{eq:embedding}
B^{\beta}_{ k, k}\hookrightarrow\cC^{\beta-\epsilon}\;.
\end{equ} 
We denote by $\Delta^h$ taking a Littlewood-Paley block in the $x$-variable (see Appendix~\ref{app:besov}). Therefore, we have
\begin{equs} 
\Vert \delta_{s,t}\tau \Vert_{\cC^{\beta-\epsilon}}^{p}
\lesssim 
\Vert \delta_{s,t}\tau \Vert_{B_{k,k}^{\beta}}^{p}    
   &=
  \bigg(\sum_{h\geqslant-1}2^{\beta hk}\int_{\T^d} |\Delta^h\delta_{s,t}\tau(x)|^k \rmd x\bigg)^{\frac{p}{k}} \\
  {}&= 
   \bigg(\sum_{h\geqslant-1}2^{\beta h k}\int_{\T^d} \Big|\sum_{m\sim 2^h}e^{-\imath x\cdot m}\delta_{s,t}\hat\tau_m\Big|^k \rmd x\bigg)^{\frac{p}{k}}
  \,.
\end{equs}
We now further enforce $k>p$ with $k$ even so that by Jensen's inequality we have
\begin{equs} 
\E [
\Vert \delta_{s,t}\tau\Vert_{\mcC^{\beta-\epsilon}}^p
]&
\lesssim\bigg(\sum_{h\geqslant-1}2^{\beta hk}\int_{\T^d}
\E \Big[\Big(\sum_{m\sim 2^h}e^{-\imath x\cdot m}\delta_{s,t}\hat\tau_m\Big)^k\Big]\rmd x\bigg)^{\frac{p}{k}}\,.
\end{equs}
Since $\tau$ sits in a finite inhomogenous Gaussian chaos, one has the hypercontractive moment bound
\begin{equs}
 {}&   \bigg(\sum_{h\geqslant-1}2^{\beta hk}\int_{\T^d}
\E \Big[\Big(\sum_{m\sim 2^h}e^{-\imath x\cdot m}\delta_{s,t}\hat\tau_m\Big)^k\Big]\rmd x\bigg)^{\frac{p}{k}}\\
{}&\lesssim
\bigg(\sum_{h\geqslant-1}2^{\beta hk}\int_{\T^d}
\E \Big[\Big(\sum_{m\sim 2^h}e^{-\imath x\cdot m}\delta_{s,t}\hat\tau_m\Big)^2\Big]^{\frac{k}{2}}\rmd x\bigg)^{\frac{p}{k}}\\
{}&=
\bigg(\sum_{h\geqslant-1}2^{\beta hk}\int_{\T^d}
 \Big(\sum_{m_1,m_2\sim 2^h}e^{-\imath x\cdot (m_1+m_2)}\E[\delta_{s,t}\hat\tau_{m_1}\delta_{s,t}\hat\tau_{m_2}]\Big)^{\frac{k}{2}}\rmd x\bigg)^{\frac{p}{k}}\;.
\end{equs}
Since the stationarity hypothesis implies that $\E[\delta_{s,t}\hat{\tau}_{m_1}\delta_{s,t}\hat{\tau}_{m_2} ]=0$ unless $m_1 +m_2=0$, we have
\begin{equs} 
\E [
\Vert \delta_{s,t}\tau\Vert_{\mcC^{\beta-\epsilon}}^p
]&
\lesssim
\bigg(\sum_{h\geqslant-1}2^{\beta hk}\int_{\T^d}
 \Big(\sum_{m\sim 2^h}\E[\delta_{s,t}\hat\tau_{m}\delta_{s,t}\hat\tau_{-m}]\Big)^{\frac{k}{2}}\rmd x\bigg)^{\frac{p}{k}}\\
{}&\lesssim
\bigg(
\sum_{h\geqslant-1}\sum_{m\sim 2^h}2^{2\beta h}
\E [
\delta_{s,t}\hat\tau_{m}\delta_{s,t}\hat\tau_{-m}
]\bigg)^{\frac{p}{2}}\\
&=\bigg(
\sum_{m\in\Z^d}\angp{m}{2\beta}
\E [
\delta_{s,t}\hat\tau_{m}\delta_{s,t}\hat\tau_{-m}
]\bigg)^{\frac{p}{2}}
\,.
\end{equs}
where we used $k\geqslant 1$. Finally, a more standard Kolmogorov argument in time along with hypercontractivity gives
\begin{equs}
  \E [ 
\Vert \tau \Vert_{\dot{C}^{\theta - \kappa }_T\cC^{\beta-\epsilon}}^{p} ]
\lesssim \sup_{\substack{0\leqslant s,t\leqslant T\\|t-s|\leqslant 1}}
\frac{
   \E [\|\delta_{s,t}{\tau}\|_{\cC^{\beta-\epsilon}}^2 ]^{p/2}}{|t-s|^{\frac{\theta p}{2}}}  
    \,.
\end{equs}
Combining the above with the first statement gives the second statement. 
\end{proof}

\subsubsection{Kolmogorov criteria for time-local random operators}

All the random operators in Definition~\ref{def:randomop} are presented as operators acting on a functions on space and time.
In this section we give Kolmogorov criteria for $\Xtwom$, $\Xtwonm$, $\pXtwom^{\sto}$ and $\XdotX$. 

We separate these cases since $\Sym$ is the only operator listed that does not act locally in time and therefore needs a slightly different argument -- see the next section. 

We note that, instead of being viewed as operators on space-time functions, the remaining operators $\tau$ in Definition~\ref{def:randomop} can be viewed as operators $\tilde{\tau}_{t}$ on space functions that themselves vary in time, by setting 
\begin{equs}\label{local-in-time}
    \tau(f)(t,x)=\tilde\tau_t\big(f(t,\cdot)\big)(x)
\end{equs}
for $f:\R_{\geqslant0} \times\T^{d_1}$ and a time-varying random operator $\tilde\tau$. Overloading the notation, we identify $\tilde\tau$ with $\tau$, also denoting it $\tau$ (or $\tau_{\bigcdot}$ when it is necessary to distinguish). Moreover, observe that we have
\begin{equs}   \|\tau(t)\|_{\mcL(C_TH^\alpha(\T^{d_1}),\mcC^{\beta-\epsilon}(\T^{d_2}))}\lesssim
\|\tau_t\|_{\mcL(H^\alpha(\T^{d_1}),\mcC^{\beta-\epsilon}(\T^{d_2}))}\,,
\end{equs} 
and, for $\theta > 0$, 
\begin{equs}\label{local-in-time-bound}    \|\tau\|_{\mcL(C_TH^\alpha(\T^{d_1}),C^\theta_T\mcC^{\beta-\epsilon}(\T^{d_2}))}\lesssim
\|\tau_{\bigcdot}\|_{C_T^\theta\mcL(H^\alpha(\T^{d_1}),\mcC^{\beta-\epsilon}(\T^{d_2}))}\,.
\end{equs} 
Our Kolmogorov criteria for local-in-time random operators will estimate the norm on the r.h.s. in both inequalities above. 

Our approach to estimating random operators in space is inspired by \cite[Sec.~10.2]{GP17}.

\begin{definition}\label{def:timevar}
For any $T>0$, a \textit{time-varying random smooth operator} over time $T$ will be an element $\tau_{\bigcdot} \in C_{T}\cL \big(H^\alpha(\T^{d_1}), \cC^{\beta-\epsilon} (\T^{d_2}) \big)$ which, for $x \in \T^{d_2}$ and $t\in[0,T]$, is of the form 
\begin{equs}
    \tau_t(f)(x)=\sum_{n\in\T^{d_1}}\mcF_i\tau_n(x,t)\hat f_{-n}
\end{equs}
where $\tau_t(f)$ denotes the action of $\tau$ evaluated a time $t$ on $f\in H^\alpha(\T^{d_1})$, and $\mcF_i\tau_n(x,t)$ stands for the Fourier transform of $\tau$ in the inner space variable evaluated at the mode $n$, that is to say the Fourier transform in the space variable acting on $f$. This implies that for $m\in\Z^{d_2}$,
\[
\cF_x\big(\tau_t(f)\big)(m) = 
\sum_{n \in \Z^{d_1}}
\hat{\tau}_{m,n}(t)
\hat{f}_{-n}
 \,,
\]
for $\hat{\tau}_{m,n}(t)$ rapidly decaying in $m$ and $n$. We say a time-varying random smooth operator $\tau$ is in the inhomogenous Gaussian chaos of order $k\in\N$ if all the $\hat{\tau}_{m,n}(t)$ are. 

For $h \in \T^{d_2}$, we write $T_{h}$ for the associated translated operator, that is $(T_{h}g)(\cdot) = g(\cdot - h)$ for functions $g$ on $\T^{d_2}$. We then say $\tau_{\bigcdot}$ as above is \textit{stationary in space}
if for every $h \in \T^{d_2}$,  $T_{h} \circ \tau_t(f)$ is equal in distribution to $\tau_t$ as a random element of $C_{T} \cL \big(H^\alpha(\T^{d_1}), \cC^{\beta-\epsilon} (\T^{d_2}) \big)$. Note that stationary in space implies that for $s,t \in [0,T]$ and $n\in\Z^{d_1}$, 
\begin{equ}\label{local-in-time-stationary}
\E[ \hat{\tau}_{m_1,n}(s) \hat{\tau}_{m_2,-n}(t) ] = 0 \text{ unless }
m_1 + m_2 = 0\,.
\end{equ}
\end{definition}

\begin{remark}
Note that it is easy to verify that with either the cut-off $N$ for the LD objects or a fixed time cut-off on the BG objects, the operators $\Xtwom$, $\Xtwonm$, $\pXtwom^{\sto}$ and $\XdotX$ satisfy the definition above -- in particular they have the stated stationarity property given above.
\end{remark}

We are now going to state some Kolmogorov estimates for $\tau_t$ which combined with \eqref{local-in-time-bound} will imply the desired bound for $\tau$.

Below, $T > 0$ is fixed, and we extend the notation $\delta_{s,t}\tau=\tau_t-\tau_s$ to a time-varying random operator $\tau_{\bullet}\;$.
\begin{lemma}\label{lem:randomopsob}
For any fixed $k\geqslant1$ and $p \in 2\N$, and uniform over smooth time-varying stationary in space random operators $\tau$ over time $T$ belonging to the $k$-th inhomogeneous Gaussian chaos, 
one has, for any $t \in [0,T]$, 
\begin{equs}\label{eq:FixedTimeKolmoRO}
\E [ \Vert \tau_t \Vert^p_{ \cL ( H^\alpha,\cC^{\beta-\epsilon} )} ]
\lesssim 
\bigg(
 \sum_{m \in\Z^{d_2},n\in\Z^{d_1}}\angp{m}{2\beta}\angp{n}{-2\alpha} \E[\hat\tau_{m,n}(t) \hat\tau_{-m,-n} (t)]\bigg)^{\frac{p}{2}}\;.
\end{equs}
In particular, for any fixed $\theta > \kappa > 0$ and uniform over smooth time-varying stationary in space random operators $\tau$ over time $T$ belonging to the $k$-th homogenous Gaussian chaos, one has 
\begin{equs}    
       \E&[\Vert\tau_{\bigcdot}\Vert_{\dot{C}_T^{\theta-\kappa}\cL(H^\alpha,\cC^{\beta-\epsilon})}^p]\lesssim \sup_{\substack{0\leqslant s,t\leqslant T\\|t-s|\leqslant 1}}{|t-s|^{-\frac{\theta p}{2}}}
       \bigg(    \sum_{m\in\Z^{d_2},n\in\Z^{d_1}}\angp{m}{2\beta}\angp{n}{-2\alpha}\E[\delta_{s,t}\hat\tau_{m,n}\delta_{s,t}\hat\tau_{-m,-n}]\bigg)^{\frac{p}{2}}\,.
 \end{equs}
\end{lemma}
\begin{proof}
We focus on the second statement, the first one will be clear from our argument. 
As before, we take $ k > d_2/\eps$ and use \eqref{eq:embedding}, giving 
\begin{equs}
   \Vert \delta_{s,t}\tau(f)\Vert^p_{\mcC^{\beta-\epsilon}}\lesssim
     \Vert \delta_{s,t}\tau(f)\Vert^p_{B_{k,k}^{\beta}}=  \bigg(\sum_{h\geqslant-1}2^{\beta hk}\int_{\T^d} \Big|\Delta^h\delta_{s,t}\tau(f)(x)\Big|^k \rmd x\bigg)^{\frac{p}{k}}
\end{equs}
Then, by Cauchy–Schwarz inequality,
\begin{equs}
    \bigg(\sum_{h\geqslant-1} & 2^{\beta hk}\int_{\T^{d_2}} \Big|\Delta^h\delta_{s,t}\tau(f)(x)\Big|^k \rmd x\Big)^{\frac{p}{k}}\\
    {}&\lesssim
      \bigg(\sum_{h\geqslant-1}2^{\beta hk}\int_{\T^{d_2}} \Big(\sum_{n\in\Z^{d_1}}|\Delta^h\delta_{s,t}\mcF_i\tau_{n}(x)\hat f_{-n}|\Big)^k \rmd x\bigg)^{\frac{p}{k}}\\&\lesssim
      \bigg(\sum_{h\geqslant-1}2^{\beta hk}\int_{\T^{d_2}} \Big(\sum_{n\in\Z^{d_1}}\angp{n}{-2\alpha}\big|\Delta^h\delta_{s,t}\mcF_i\tau_{n}(x)\big|^2\Big)^{\frac{k}{2}} \rmd x\bigg)^{\frac{p}{k}}      
      \Big(\sum_{n\in\Z^{d_1}}\angp{n}{2\alpha}|\hat{f}_n|^2\Big)^{\frac{p}{2}}\,.
      \end{equs}
We may now divide by $$  \Big(\sum_{n\in\Z^{d_1}}\angp{n}{2\alpha}|\hat{f}_n|^2\Big)^{\frac{p}{2}}=\|f\|_{H^\alpha}^p$$ and take the supremum over $f\in H^\alpha$ on the l.h.s. to obtain
\begin{equs}    \Vert\delta_{s,t}\tau\Vert_{\cL(H^\alpha,\cC^{\beta-\epsilon})}^p\lesssim
\bigg(\sum_{h\geqslant-1}2^{\beta hk}\int_{\T^{d_2}} \Big(\sum_{n\in\Z^{d_1}}\ang{n}^{-2\alpha}|\Delta^h\delta_{s,t}\mcF_i\tau_{n}(x)|^2\Big)^{\frac{k}{2}} \rmd x\bigg)^{\frac{p}{k}}\,.
\end{equs}
We are now ready to take the expectation. Again, taking $k > p$ and using Jensen’s inequality and
the hypercontractive moment bound, we have
\begin{equs}
  \E\bigg[  \bigg(\sum_{h\geqslant-1}2^{\beta hk}\int_{\T^{d_2}} \Big(&\sum_{n\in\Z^{d_1}}\angp{n}{-2\alpha}\big|\Delta^h\delta_{s,t}\mcF_i\tau_{n}(x)\big|^2\Big)^{\frac{k}{2}} \rmd x\bigg)^{\frac{p}{k}}\bigg] \\&\lesssim
  \bigg(\sum_{h\geqslant-1}2^{\beta hk}\int_{\T^{d_2}}  \E[  \Big(\sum_{n\in\Z^{d_1}}\angp{n}{-2\alpha}\big|\Delta^h\delta_{s,t}\mcF_i\tau_{n}(x)\big|^2\Big)^{\frac{k}{2}}] \rmd x\bigg)^{\frac{p}{k}}\\
  &\lesssim
   \bigg(\sum_{h\geqslant-1}2^{\beta hk}\int_{\T^{d_2}}  \Big(\sum_{n\in\Z^{d_1}}\angp{n}{-2\alpha}\E[  \Delta^h\delta_{s,t}\mcF_i\tau_{n}(x)
   \Delta^h\delta_{s,t}\mcF_i\tau_{-n}(x)]\Big)^{\frac{k}{2}} \rmd x\bigg)^{\frac{p}{k}}\\
   &\lesssim
     \bigg(\sum_{h\geqslant-1}2^{\beta hk}\int_{\T^{d_2}}  \Big(\sum_{n\in\Z^{d_1}}\angp{n}{-2\alpha}\sum_{m_1,m_2\sim 2^h}e^{-\imath x\cdot(m_1+m_2)}\E[  \delta_{s,t}\hat\tau_{m_1,n}\delta_{s,t}\hat\tau_{m_2,-n}]\Big)^{\frac{k}{2}} \rmd x\bigg)^{\frac{p}{k}}
     \\
   &\lesssim
     \bigg(\sum_{h\geqslant-1}2^{\beta hk}  \Big(\sum_{n\in\Z^{d_1}}\angp{n}{-2\alpha}\sum_{m\sim 2^h}\E[ \delta_{s,t}\hat\tau_{m,n}\delta_{s,t}\hat\tau_{-m,-n}]\Big)^{\frac{k}{2}} \bigg)^{\frac{p}{k}}\\
     &\lesssim
    \Big(\sum_{n\in\Z^{d_1}}\angp{n}{-2\alpha}\sum_{m\in\Z^{d_2}}\angp{m}{2\beta}\E[ \delta_{s,t}\hat\tau_{m,n}\delta_{s,t}\hat\tau_{-m,-n}]\Big)^{\frac{p}{2}} \,.
\end{equs}    
We can now conclude the proof by means of a standard Kolmogorov argument in time combined with hypercontractive moment estimates, which yields, for any $\kappa>0$, 
\begin{equs}    \E[\Vert\tau\Vert_{\cL(H^\alpha,\dot C^{\theta-\kappa}_T\cC^{\beta-\epsilon})}^p]&\lesssim\sup_{\substack{0\leqslant s,t\leqslant T\\|t-s|\leqslant 1}}{|t-s|^{-\frac{\theta p}{2}}}\E[\Vert\delta_{s,t}\tau\Vert_{\cL(H^\alpha,\cC^{\beta-\epsilon})}^2]^{\frac{p}{2}}\\
&\lesssim\sup_{\substack{0\leqslant s,t\leqslant T\\|t-s|\leqslant 1}}{|t-s|^{-\frac{\theta p}{2}}} \bigg( \sum_{m\in\Z^{d_2},n\in\Z^{d_1}}\angp{m}{2\beta }\ang{n}^{-2\alpha}\E[\delta_{s,t}\hat\tau_{m,n}\delta_{s,t}\hat\tau_{-m,-n}]
 \bigg)^{\frac{p}{2}}\,.
\end{equs}
\end{proof}

\subsubsection{Kolmogorov criteria for the time non-local random operator $\Sym$}\label{sec:non-local-in-time}
We now provide the Kolmogorov criterion we will use to control the LD random operator $\tau = \Sym$.
As described before, the feature of $\Sym$ that makes it challenging to estimate is that it is non-local in time -- this is because of the heat kernel connecting the root of $\Sym$ to the internal node pictured right above it. 

However, for this operator we can express this non-locality in time through an integral formula, exchanging \eqref{local-in-time} by writing, for $f:[0,T]\times\T^d\rightarrow\R$, 
\begin{equs}\label{nonlocal-in-time}
    \tau(f)(t,\cdot)=\int_{0}^t\tau_{u,t}\big(f(u,\cdot)\big)\rmd u\,.
\end{equs}
More explicitly, for $\tau = \Sym$, writing $\X_s(\cdot)=\X(s,\cdot)$ for $s\in\R$, we have 
\begin{equ}\label{non-local-def}
\tau_{u,t}\equiv\Sym_{u,t}(f)\eqdef\mcN\big(\X_t,P_{t-u}\mcN(\X_u,f,\X_u),\X_t\big) 
\text{ for }f:\T^d\rightarrow\R\;.
\end{equ}
The following lemma gives us an estimate analogous to \eqref{local-in-time-bound}. 

\begin{lemma}\label{lem:non-local-bound}
Let $\tau$ be as as in \eqref{nonlocal-in-time} and fix $\kappa>0$ and $\theta\in(\kappa,2+\kappa)$. 
Then, we have 
\begin{equs}\label{eq:FixedTimeBy2TimeObject}
        \|\tau(t)\|_{\mcL(C_TH^\alpha,\mcC^{\beta-\epsilon})}\lesssim \sup_{\substack{0\leqslant u\leqslant t}}  |t-u|^{1-\frac{\theta-\kappa}2}  \|\tau_{u,t}\|_{\mcL(H^\alpha,\mcC^{\beta-\epsilon})}\,,
\end{equs}
and
\begin{equs}\label{eq:non-local-bound}
      \|\tau\|_{\mcL(C_TH^\alpha,\dot C_T^{\theta-\kappa}\mcC^{\beta-\epsilon})} & \lesssim \sup_{\substack{0\leqslant u\leqslant t\leqslant T}}|t-u|^{1-\frac{\theta-\kappa}{2}}\|\tau_{u,t}\big(f(u,\cdot)\big)\|_{\mcL(H^\alpha,\mcC^{\beta-\epsilon})} \\
{}& \qquad +
    \sup_{\substack{0\leqslant u\leqslant s\leqslant t\leqslant T}}{|t-s|}^{-\frac{\theta-\kappa}2}|s-u|^{1-\frac{\theta-\kappa}{2}}\|\tau_{u,t}-\tau_{u,s}\|_{\mcL(H^\alpha,\mcC^{\beta-\epsilon})}\,.
\end{equs}
\end{lemma}
\begin{proof}
We start by performing the fixed-time estimate
\begin{equs}
      \|\tau(f)(t)\|_{\mcC^{\beta-\epsilon}}&\lesssim\int_0^t\|\tau_{u,t}\big(f(u,\cdot)\big)\|_{\mcC^{\beta-\epsilon}}\rmd u\lesssim      \int_0^t|t-u|^{-1+\frac{\theta-\kappa}2}
      \rmd u  \sup_{\substack{0\leqslant u\leqslant t\leqslant T}}  |t-u|^{1-\frac{\theta-\kappa}2}  \|\tau_{u,t}\big(f(u,\cdot)\big)\|_{\mcC^{\beta-\epsilon}}\\
      &\lesssim     t^{\frac{\theta-\kappa}2}
      \Big(  \sup_{\substack{0\leqslant u\leqslant t\leqslant T}}  |t-u|^{1-\frac{\theta-\kappa}2}  \|\tau_{u,t}\|_{\mcL(H^\alpha,\mcC^{\beta-\epsilon})}\Big)\|f\|_{C_TH^\alpha}\,.
\end{equs}
Dividing by $ \|f\|_{C_TH^\alpha}$ and taking the supremum over $f\in C_TH^\alpha$ in the l.h.s. gives us
\begin{equs}
    \|\tau(t)\|_{\mcL(C_TH^\alpha,\mcC^{\beta-\epsilon})}\lesssim \sup_{\substack{0\leqslant u\leqslant t\leqslant T}}  |t-u|^{1-\frac{\theta-\kappa}2}  \|\tau_{u,t}\|_{\mcL(H^\alpha,\mcC^{\beta-\epsilon})}\,.
\end{equs}

Similarly, using the fact that we have 
\begin{equs}
    \tau(f)(t,x)-\tau(f)(s,x)=\int_s^t\tau_{u,t}\big(f(u,\cdot)\big)\rmd u+\int_0^s (\tau_{u,t}-\tau_{u,s})\big(f(u,\cdot)\big)\rmd u\,,
\end{equs}
we obtain, denoting $\delta^o_{s,t}\tau_{u,\cdot}\eqdef\tau_{u,t}-\tau_{u,s}$
\begin{equs}
    \|&\delta_{s,t}\tau(f)\|_{\mcC^{\beta-\epsilon}}\\&\lesssim\int_s^t\|\tau_{u,t}\big(f(u,\cdot)\big)\|_{\mcC^{\beta-\epsilon}}\rmd u+\int_0^s\|\delta^o_{s,t}\tau_{u,\cdot}\big(f(u,\cdot)\big)\|_{\mcC^{\beta-\epsilon}}\rmd u\\
    &\lesssim|t-s|^{\frac{\theta-\kappa}{2}}\Big(\sup_{\substack{0\leqslant u\leqslant t\leqslant T}}|t-u|^{1-\frac{\theta-\kappa}{2}}\|\tau_{u,t}\big(f(u,\cdot)\big)\|_{\mcC^{\beta-\epsilon}}\\
    {}&  \qquad \qquad  \qquad +
    \sup_{\substack{0\leqslant u\leqslant s\leqslant t\leqslant T}}{|t-s|}^{-\frac{\theta-\kappa}2}|s-u|^{1-\frac{\theta-\kappa}{2}}\|\delta^o_{s,t}\tau_{u,\cdot}\big(f(u,\cdot)\big)\|_{\mcC^{\beta-\epsilon}}
    \Big)\\
    &\lesssim
    |t-s|^{\frac{\theta}{2}}\Big(\sup_{\substack{0\leqslant u\leqslant t\leqslant T}}|t-u|^{1-\frac{\theta-\kappa}{2}}\|\tau_{u,t}\big(f(u,\cdot)\big)\|_{\mcL(H^\alpha,\mcC^{\beta-\epsilon})}\\
    {}& \qquad \qquad  \qquad +
    \sup_{\substack{0\leqslant u\leqslant s\leqslant t\leqslant T}}{|t-s|}^{-\frac{\theta-\kappa}2}|s-u|^{1-\frac{\theta-\kappa}{2}}\|\delta^o_{s,t}\tau_{u,\cdot}\|_{\mcL(H^\alpha,\mcC^{\beta-\epsilon})}
    \Big)\|f\|_{C_TH^\alpha}\;,
\end{equs}
from which the desired estimate follows. 
\end{proof}

\begin{remark}
We allow ourselves the blow-up factor in $|t-u|$ in the estimates of Lemma~\ref{lem:non-local-bound} in anticipation of the estimates we expect for the LD object \eqref{lem:non-local-bound}, our estimates on $P_{t-u}$ blow-up as $(t-u) \downarrow 0$, see the estimate~\eqref{eq:assumption5}.  
\end{remark}

Our Kolmogorov criterion will be formulated for the quantities on the r.h.s. of \eqref{eq:non-local-bound}. 
Before stating this result we state the replacement for Definition~\ref{def:timevar}. 

\begin{definition}
For any $ T > 0$, we say  $\tau_{\bigcdot,\bigcdot}$ is a \textit{2-parameter time-varying random smooth operator} if $\tau_{\bigcdot,\bigcdot}$ is a continuous map
\[
\mathrm{Simp}_{T} = 
\{ (u,t) \in [0,T]^2\ni(u,t):\enskip 
u < t \}
\mapsto 
\tilde \tau_{u,t}\in\mcL(H^\alpha,\mcC^{\beta-\epsilon})
\]
which, for $x\in\T^d$, we often write in the form
\[
\tau_{u,t}(f)(x) = 
\sum_{n \in \Z^{d_1}}\cF_i
{\tau}_{n}(x,u,t)
\hat{f}_{-n}
 \,,
\]
and for $m \in \Z^{d}$
\[
\cF_x\big(\tau_{u,t}(f)\big)(m) = 
\sum_{n \in \Z^{d_1}}
\hat{\tau}_{m,n}(u,t)
\hat{f}_{-n}
 \,,
\]
and $\hat{\tau}_{m,n}(u,t)$ decaying rapidly in $m$ and $n$. 

As before, we say that $\tau_{\bigcdot,\bigcdot}$ belongs to the inhomogenous Gaussian chaos of order $k$ if $\hat{\tau}_{m,n}(u,t)$ always does.
We again say $\tau_{\bigcdot,\bigcdot}$ is stationary in space if $T_{h} \circ \tau_{u,t}(f)$ is equal in distribution to $\tau_{u,t}$ for any $h \in \T^d$. 
In analogy with \eqref{local-in-time-stationary}, this means we have, for any $(u,t),\; (u',t') \in \mathrm{Simp}_{T}$ and $n\in\Z^d$, 
\begin{equ}
\E[ \hat{\tau}_{m_1,n}(u,t) \hat{\tau}_{m_2,-n}(u',t') ] = 0 \text{ unless }
m_1 + m_2 = 0 \,.
\end{equ}

\begin{remark}
Note that $\Sym_{N}$ indeed satisfies the definition above, including the needed stationarity property. 
\end{remark}

Finally, to state a Kolmogorov estimate taking the blow up when $u\uparrow t$ into account, for $\theta,\eta\in(0,1)$, we endow the space of smooth functions $f: \mathrm{Simp}_{T}\rightarrow\R$ with the Hölder norms with blow up $C_{\mathrm{Simp}_{T}}^{\theta,\eta}$ defined as
\begin{equs}    \|f\|_{C_{\mathrm{Simp}_{T}}^{\theta,\eta}}\eqdef\sup_{(u,t)\in \mathrm{Simp}_{T}}|t-u|^{\eta}|f(u,t)|+\|f\|_{\dot C_{\mathrm{Simp}_{T}}^{\theta,\eta}}\,,
\end{equs}
where the homogeneous part of the norm is given by
\begin{equs}
    \|f\|_{\dot C_{\mathrm{Simp}_{T}}^{\theta,\eta}}\eqdef\sup_{((u,t),(r,s))\in \mathrm{Simp}^2_{T}}\big(|t-u|\wedge|s-r|\big)^{\eta}\big(|t-s|+|r-u|\big)^{-\frac{\theta}{2}}|f(u,t)-f(r,s)|\,.
\end{equs}
Note that the r.h.s. of \eqref{eq:FixedTimeBy2TimeObject} and \eqref{eq:non-local-bound} are bounded by 
\begin{equs}    \|\tau_{\bigcdot,\bigcdot}\|_{C^{\theta-\kappa,\eta}_{\mathrm{Simp}_{T}}\mcL(H^\alpha,\mcC^{\beta-\epsilon})}
\end{equs}
for $\eta=1-\frac{\theta-\kappa}{2}\in(0,1)$.

\end{definition}
The above discussion motivates the following lemma. Recall the notation $\delta^o_{s,t}\tau_{u,\cdot}=\tau_{u,t}-\tau_{u,s}$ the variation in the outer time. We also write $\delta^i_{u,r}\tau_{\cdot,t}\eqdef\tau_{r,t}-\tau_{u,t}$ the variation in the inner time.

\begin{lemma}\label{lem:randomopsob2}
    For any fixed $k\geqslant1$, $p\in2\N$, $\kappa>0$ and $\theta \in (\kappa,2+\kappa)$, writing $\eta=1-\frac{\theta-\kappa}{2}\in(0,1)$, and uniform over smooth 2-parameter time-varying stationary in space
random operators $\tau$ over time $T \in (0,1]$ belonging to the $k$-th inhomogenous Gaussian chaos, one has
\begin{equs}\label{eq:nonlocalhomogenous}
    \E&[\Big(  \|\tau_{\bigcdot,\bigcdot}\|_{ \dot{C}^{\theta-\kappa,\eta}_{\mathrm{Simp}_{T}}\mcL(H^\alpha,\mcC^{\beta-\epsilon})}\Big)^p]
    \\&\lesssim \sup_{\substack{0\leqslant u\leqslant r\leqslant t\leqslant T}}   |t-r|^{\eta p}|r-u|^{-\frac{\theta p}2}  \bigg(    \sum_{m\in\Z^{d},n\in\Z^{d}}\angp{m}{2\beta}\angp{n}{-2\alpha}\E[\delta^i_{u,r}\hat\tau_{m,n}(\cdot,t)\delta^i_{u,r}\hat\tau_{-m,-n}(\cdot,t)]\bigg)^{\frac{p}{2}}\\
    &\quad+\sup_{\substack{0\leqslant r\leqslant s\leqslant t\leqslant T}}|s-r|^{\eta p}{|t-s|}^{-\frac{\theta p}{2}} \bigg(    \sum_{m\in\Z^{d},n\in\Z^{d}}\angp{m}{2\beta}\angp{n}{-2\alpha}\E[\delta^o_{s,t}\hat\tau_{m,n}(r,\cdot)\delta^o_{s,t}\hat\tau_{-m,-n}(r,\cdot)]\bigg)^{\frac{p}{2}}\,.
\end{equs}
\end{lemma}
\begin{proof}
To lighten notation we just write $|\bigcdot|$ for $\| \bigcdot \|_{\mcL(H^\alpha,\mcC^{\beta-\epsilon})}$. 
Then for any random function $f:\mathrm{Simp}_T\rightarrow \mcL(H^\alpha,\mcC^{\beta-\epsilon})$ in the $k$-th inhomogenous Gaussian chaos, we again have by a classical Kolmogorov estimate and hypercontractive estimates,
\begin{equs}\label{eq:two-time-estimate}
    \E[\| f\|_{\dot{C}_{\mathrm{Simp}_{T}}^{\theta-\kappa,\eta}}^p]&\lesssim
\sup_{((u,t),(r,s))\in \mathrm{Simp}^2_{T}}\big(|t-u|\wedge|s-r|\big)^{\eta p}\big(|t-s|+|r-u|\big)^{-\frac{\theta p}{2}}\E[|f(u,t)-f(r,s)|^2]^{p/2}\,.\textcolor{white}{blablabalblabl}
    \end{equs}
By symmetry, we can take $t\geqslant s$ in the supremum. 
Using $f(u,t)-f(r,s)=\delta^i_{r,u}f(\cdot,t)+\delta^o_{s,t}f(r,\cdot)$ we can bound the r.h.s. of \eqref{eq:two-time-estimate} (up to an inessential constant) by
\begin{equs}
    \sup_{\substack{((u,t),(r,s))\in \mathrm{Simp}^2_{T}\\t\geqslant s}}
    & \big(|t-u|\wedge|s-r|\big)^{\eta p}|r-u|^{-\frac{\theta p}{2}}\E[|\delta^i_{r,u}f(\cdot,t)|^2]^{p/2} \\
    {}&+  \sup_{\substack{((u,t),(r,s))\in \mathrm{Simp}^2_{T}\\t\geqslant s}}\big(|t-u|\wedge|s-r|\big)^{\eta p}|t-s|^{-\frac{\theta p}{2}}\E[|\delta^o_{s,t}f(r,\cdot)|^2]^{p/2}\,.
\end{equs}
For the first term on the r.h.s. above, observe that $t\geqslant s$ with $s\geqslant r$ yields $|t-u|\wedge|s-r|\leqslant|t-u|\wedge|t-r|=|t-u\vee r|$.  
For the second term we observe that we have $|t-u|\wedge|s-r|\leqslant|s-r|$. 
Together, this allows to estimate the r.h.s of \eqref{eq:two-time-estimate} by 
\begin{equs}
  \lesssim  & \sup_{\substack{((u,t),(r,s))\in \mathrm{Simp}^2_{T}\\t\geqslant s}}|t-u\vee r|^{\eta p}|r-u|^{-\frac{\theta p}{2}}\E[|\delta^i_{r,u}f(\cdot,t)|^2]^{p/2} \\
 {}&\quad +  \sup_{\substack{((u,t),(r,s))\in \mathrm{Simp}^2_{T}\\t\geqslant s}}|s-r|^{\eta p}|t-s|^{-\frac{\theta p}{2}}\E[|\delta^o_{s,t}f(r,\cdot)|^2]^{p/2}\,.
\end{equs}
The desired estimate then follows by controlling expectations of $\| \bigcdot \|^2_{\mcL(H^\alpha,\mcC^{\beta-\epsilon})}$ by using \eqref{eq:FixedTimeKolmoRO} with $p=2$ to replace them by sums over Fourier modes. 
\end{proof}
\begin{remark}
While the above Lemma gives control on the homogenous bound, for any fixed $({u},{t}) \in \mathrm{Simp}_{T}$ we have
\[
\| f\|_{C_{\mathrm{Simp}_{T}}^{\theta-\kappa,\eta}} 
\lesssim  |t-u|^{\eta} |f({u},{t})| + 
\| f\|_{\dot{C}_{\mathrm{Simp}_{T}}^{\theta-\kappa,\eta}}\;,
\]
where putting in the factor $|t-u|^{\eta} |f({u},{t})|$  makes the implicit constant above independent of the choice of $({u},{t})$. Therefore, we have
\begin{equs}
      \E[\| \tau_{\bigcdot,\bigcdot}\|_{{C}_{\mathrm{Simp}_{T}}^{\theta-\kappa,\eta}}^p]
    & \lesssim |t-u|^{\eta p }  \bigg(    \sum_{m\in\Z^{d},n\in\Z^{d}}\angp{m}{2\beta}\angp{n}{-2\alpha}\E[\hat\tau_{m,n}(u,t)\hat\tau_{-m,-n}(u,t)]\bigg)^{\frac{p}{2}}+\text{r.h.s.}\;\text{of}\;\eqref{eq:nonlocalhomogenous}\,.
\end{equs}
Controlling the first term of the r.h.s. of the above equation uniformly in $(u,t)\in\mathrm{Simp}_T$ turns out to be similar to controlling the whole r.h.s. of \eqref{eq:nonlocalhomogenous}. We will focus on describing how to estimate this quantity, and point out how the second term above can be obtained by extracting small factors from increments of heat kernels -- see Remark~\ref{rem:continuityintime}.
\end{remark}

\begin{remark}
Those who are less familiar with path-wise approaches to stochastic analysis may be surprised we need to control H\"{o}lder regularity in $u$, as in the end we only need uniform estimates in $u$. 
However, the only way we have to control an expectation of a supremum over $u$ is by proving regularity in $u$. 
\end{remark}

\subsection{Regularity estimates for random field \TitleEquation{\X}{X} and operators \TitleEquation{\X^{(k)}}{X^{(k)}}}
As an instructive warm-up, we look at regularity estimates for $\X$ both as a random field and a random operator.
Using Lemma~\ref{lem:stoobsobolev}, we see that the key ingredient for \eqref{eq:regX} is estimating, for $\beta = |\X|$,
\begin{equ}\label{eq:Xrandomfieldest}
\sum_{m\in\Z^d}  \angp{m}{2\beta} \E[\delta_{s,t}\hat{\X}_m\delta_{s,t}\hat{\X}_{-m}] \lesssim\sum_{m\in\Z^d}\angp{m}{2\beta-2+2\theta}|t-s|^{\theta}\,,
\end{equ}
where we used \eqref{eq:covariance} to obtain
\begin{equ}\label{eq:covariance2}
\E[\delta_{s,t}\hat{\X}_m\delta_{s,t}\hat{\X}_{-m}]=\frac{2}{\angp{m}{2}}\big(1-e^{-|t-s|\angp{m}{2}}\big)\lesssim \angp{m}{-2+2\theta}|t-s|^\theta \,,
\end{equ}
for $\theta\in(0,1)$. 
The RHS of \eqref{eq:Xrandomfieldest} is convergent if and only if $\beta<-\frac{d-2}{2}-\theta$, and so we obtain the desired statement by taking $\theta $ small enough. 
 
For obtaining \eqref{eq:X_rand_op} for the random operator $\X^{(k)}$, Lemma~\ref{lem:randomopsob} tells us the key estimate is 
\begin{equs}\label{eq:lollipopproof}
    \sum_{n\in\Z^{d-k},m\in\Z^k}&\angp{n}{-2\alpha}\angp{m}{2\beta}\E[\delta_{s,t}\hat{\X}_{m,n}\delta_{s,t}\hat{\X}_{-m,-n}]\\&\lesssim\sum_{n\in\Z^{d-k},m\in\Z^k}\angp{n}{-2\alpha}\angp{m}{2\beta}\angp{(m,n)}{-2+2\theta}|t-s|^{\theta}\nonumber\;.   
\end{equs}
The situation here is more delicate, with the RHS of \eqref{eq:lollipopproof} being convergent if and only if the sum $\sum_{p\in\Z^d}\angp{p}{-2\alpha+2\beta-2+2\theta}$ is convergent along with both of the sums $\sum_{n\in\Z^{d-k}}\angp{n}{-2\alpha-2+2\theta}$ and $\sum_{m\in\Z^{k}}\angp{m}{2\beta-2+2\theta}$ are convergent. 
This gives us three conditions: 
\begin{equs}\label{eq:lollipopproofBis}
\beta<\alpha-\frac{d-2}{2}\,, \enskip \beta<-\frac{k-2}{2}\; \enskip \text{and} \enskip \alpha>\frac{d-k-2}{2}\;.    
\end{equs}
More generally, the regularity we can associate to a random field will be determined by how large we can take $\beta$ while keeping \eqref{eq:amplitude}  convergent. 
Similarly, for a general random operator $\tau$, we will aim to take $\alpha$ small and $\beta$ large while keeping \eqref{eq:amplitude} convergent. The exponents $\alpha$ and $\beta$ are, respectively, called the \textit{inner} and \textit{outer regularities} of $\tau$.
For a random field $\tau$, we also call the supremum over the values of $\beta$ that we can take the outer regularity of $\tau$.

Finally, since all the $\tau$ are polynomials in the Gaussian noise $\xi$, we can use Wick's rule to compute $C^{u,t}_{\tau}(m,n)$ as sum of various terms indexed by all the possible contractions of instances of $\xi$. 
To each of these contractions will correspond a (possibly renormalized) amplitude. 
In the next step, we show that these amplitudes can be indexed by a certain class of graphs, that we define hereafter. 

\begin{remark}\label{rem:continuityintime}    
Regarding H\"{o}lder estimates in time, recall that the required estimate for the covariance of $\delta_{s,t}\X$ is straightforward modification of the estimate of $\X(t)$ for fixed $t$ in \eqref{eq:covariance2}. 
For computing \eqref{eq:amplitude2} for more complicated stochastic estimates, we can always reorganize the terms by means of some telescoping sums so that we introduce a single time difference of free fields that gives us a good factor of $|t-s|$, see \cite{mourrat2017construction}. 

Moreover, the same technique can be applied to \eqref{eq:amplitude3} to prove Hölder-continuity in the inner time $u$ of the non-local in time random operator: each contribution can be rewritten as a sum over many contribution depending on $u$ and $r$ only through a single term $$\frac{1}{\angp{m}{2}}\big(1-e^{-|r-u|\angp{m}{2}}\big)\lesssim\angp{m}{-2+2\theta}|r-u|^\theta$$ giving us the needed small factor $|r-u|^\theta$.
Therefore, in the sequel, we will focus on bounding \eqref{eq:amplitude} with the understanding that the required estimate on \eqref{eq:amplitude2} and \eqref{eq:amplitude3} can be dealt with similarly. 
Note that this remark is only relevant for the LD objects. 
\end{remark}

\section{Diagrammatic methods and estimates on larger objects }\label{sec:diagram}
In Section~\ref{subsec:graphical_notation} we began to introduce some diagrammatic notation, we now introduce a more careful diagrammatic framework for obtaining the probabilistic estimates, that is covariance bounds, for the random fields and random operators that enter into our analysis.

\subsection{Covariances of the stochastic objects}\label{subsec:covar}
We introduce notation for covariances that will be used to estimate the r.h.s. of the Kolmogorov estimates, and which in this section we will estimate using sums over graphs.  

From now on, we fix $\theta>0$ sufficiently small and $\eta<1$ sufficiently close to $1$. 
\begin{definition}  
We introduce a short-hand notation for the covariances of the stochastic objects. For any random field $\tau$, we let
\begin{equs}
    C^{u,t}_{\tau}(m,n)&\eqdef\E[\hat{\tau}_{m}(t)\hat{\tau}_{-m}(t)]\delta_{n,0}\,,\\
    \tilde C^{r,s,t}_{\tau}(m,n)&\eqdef \E[\delta_{s,t}\hat{\tau}_{m}\delta_{s,t}\hat{\tau}_{-m}]\delta_{n,0}\,,
\end{equs}
for all random operators $\tau\neq\Sym$, we let
\begin{equs}
    C^{u,t}_{\tau}(m,n)&\eqdef \E[\hat{\tau}_{m,n}(t)\hat{\tau}_{-m,-n}(t)]\,,\\
   \tilde C^{r,s,t}_{\tau}(m,n)&\eqdef \E[\delta_{s,t}\hat{\tau}_{m,n}\delta_{s,t}\hat{\tau}_{-m,-n}]\,,
\end{equs}
and for $\tau=\Sym$ we let
\begin{equs}
  C^{u,t}_{\tau}(m,n)&\eqdef \E[\hat{\tau}_{m,n}(u,t)\hat{\tau}_{-m,-n}(u,t)]\,,\\
  \tilde C_\tau^{r,s,t}(m,n)&\eqdef\E[\delta^o_{s,t}\hat{\tau}_{m,n}(r,\cdot)\delta^o_{s,t}\hat{\tau}_{-m,-n}(r,\cdot)]\,,\\
 \hat C^{u,r,t}_{\tau}(m,n)&\eqdef \E[\delta^i_{u,r}\hat\tau_{m,n}(\cdot,t)\delta^i_{u,r}\hat\tau_{-m,-n}(\cdot,t)]\,.  
\end{equs}
\end{definition}
Using Lemmas~\ref{lem:stoobsobolev},\ref{lem:randomopsob} and \ref{lem:randomopsob2}, the needed space-time regularity estimates on LD stochastic objects reduces to estimating 
\begin{equs}\label{eq:amplitude2}
|s-r|^{\1\{\tau=\Sym\}2\eta}{|t-s|}^{-\theta}\sum_{m\in\Z^{d_2},n\in\Z^{d_1}}\angp{m}{2\beta}\angp{n}{-2\alpha}\tilde C^{r,s,t}_{\tau}(m,n)
\end{equs}
uniformly in the cut-off $N$ and in $(r,s,t)$, while uniform in time (for the LD objects) and in the cut-off (for both LD and BG objects) control of the space regularity reduces to estimating
\begin{equs}\label{eq:amplitude}
{|t-u|}^{\1\{\tau=\Sym\}2\eta}\sum_{m\in\Z^{d_2},n\in\Z^{d_1}}\angp{m}{2\beta}\angp{n}{-2\alpha} C^{u,t}_{\tau}(m,n)
\end{equs}
uniformly in the cut-off $N$ for the LD objects and in $(u,t)$. Finally, to handle $\tau=\Sym$, it is necessary to control the Hölder-continuity in the inner time $u$, which corresponds to estimating
\begin{equs}\label{eq:amplitude3}
   |t-r|^{2{\eta }}   |r-u|^{-\theta }   \sum_{m\in\Z^{d},n\in\Z^{d}}\angp{m}{2\beta}\angp{n}{-2\alpha} \hat C^{u,r,t}_{\tau}(m,n)
\end{equs}
uniformly in the cut-off $N$ and in $(u,r,t)$.

Note that with the above notation, when $\tau$ is a random field, the term $\delta_{n,0}$ makes that the sum over $n\in \Z^{d_1}$ vanishes, so that one can think of \eqref{eq:amplitude2} and \eqref{eq:amplitude} with $d_1=0$. 

\subsection{Stranded graphs}\label{subsubsec:strandedgraphs}
The first class of diagrams we will work with are called stranded graphs.  
\begin{definition}\label{def:strandedgraphs}
A \textit{stranded graph} $G$ is given by data $G=(V,L,\mfc,\mfp)$ of the following form. 

The data $(V,L)$ specifies a quartic graph, namely where each link $l \in L $ is incident to one or two vertices and each vertex $v \in N$ is incident to four links, in both cases counting multiplicity.  
We allow a link $l$ to be incident to a vertex $v$ with multiplicity two (which we still call being incident to two vertices), and we call such a link a tadpole. 
In particular, if we say $l$ is a tadpole at $v$, then $v$ is also incident to $l$ with multiplicity two. 
If any link $l$ is only incident to one vertex, then we call it an \textit{external link}.
Links incident to two vertices (including tadpoles) are called \textit{internal links}. 
We let $L^{int}(G)$ and $L^{ext}(G)$ be the internal link and external link set of $G$. 
A vertex which is only incident to internal links is called an \textit{internal vertex} and we denote by $V^{int}(G)$ the internal vertex set of $G$. 
If $|L^{ext}(G)|=0$ then $G$ is called a \textit{closed stranded graph} while if $|L^{ext}(G)| > 0$ we call $G$ an \textit{open stranded graph}.

The data $\mfc$ is a ``coloring map'' $\mfc: N \rightarrow [d]$ that indicates which term in $\sum_{c=1}^d\cI^c(\phi)$ is being associated to the vertex $v$. As an example, if we write $[d] = \{\textcolor{darkgreen}{1},\textcolor{brown}{2},\textcolor{red}{3}, \textcolor{blue}{4}\}$, then if $\mfc(v) = \textcolor{darkgreen}{1}$, $v$ would be associated to the following picture.   
\begin{figure}[H]
\centering
   \tikzsetnextfilename{fig6}
   \begin{tikzpicture}
\draw[darkgreen] (0,-1)--(0,0);
\draw[darkgreen] (1,-1)--(1,0);
\draw[red] (0,-1)--(1,-1);
\draw[red] (0,0)--(1,0);
\draw[brown] (0,-1)..controls(.3,-.8)and(.7,-.8)..(1,-1);
\draw[brown] (0,0)..controls(.3,-.2)and(.7,-.2)..(1,0);
\draw[blue] (0,0)..controls(.3,.2)and(.7,.2)..(1,0);
\draw[blue] (0,-1)..controls(.3,-1.2)and(.7,-1.2)..(1,-1);
\draw[black,thick,densely dotted] (1,0)--(1.3,0.3);
\draw[black,thick,densely dotted] (0,0)--(-0.3,0.3);
\draw[black,thick,densely dotted] (0,-1)--(-.3,-1.3);
\draw[black,thick,densely dotted] (1,-1)--(1.3,-1.3);
\node at(-.5,.5){$l_{1}$};
\node at(1.5,0.5){$l_{2}$};
\node at(1.5,-1.5){$l_{3}$};
\node at(-.5,-1.5){$l_{4}$};

\end{tikzpicture}
\end{figure}

The colored edges above are not links in $L$, only the dashed lines $l_{1},l_{2},l_{3},l_{4}$ are. 
Note that $l_{1},\dots,l_{4}$ need not all be distinct if we have a tadpole at $v$.  
Note also that the four nodes appearing above are not four elements of the set $V$, but the entire collection of them plus the colored edges correspond to one vertex in $V$. 

The data $\mfp$ consists of associating, to each $v \in V$, an ordered pair of partitions of the multi-set of links incident to $v$, which we write $\mfp(v) = (\hat{\mfp}(v), \bar{\mfp}(v))$. 
The data $\mfp(v)$ tells us how to glue the links incident to $v$ to the non-local vertex, $\hat{\mfp}(v)$ is the pairing induced by pairing links that are incident to nodes of the non-local vertex connected by the edges colored $\mfc(v)$, while  $\bar{\mfp}(v)$ is the pairing induced by pairing links that are connected by the other $d-1$ edges. The data $\mfp(v)$ furthermore comes with the following compatibility condition:
we enforce $\hat{\mfp}(v) \not = \bar{\mfp}(v)$, unless $\hat{\mfp}(v)  = \bar{\mfp}(v) = \big\{ \{l_t,l_1\} , \{ l_t,l_2\} \big\}$ with $l_t,l_1,l_2\in L$, and possibly $l_1=l_2$.

For example, in the picture above one has
\begin{equs}\label{eq:defl_i}
 \hat{\mfp}(v) = \Big\{ \{l_1,l_4\} , \{ l_2,l_3\} \Big\} \text{ and } \bar{\mfp}(v)= \Big\{ \{l_1,l_2\}, \{l_3,l_4\} \Big\}\;.
\end{equs} 
In what follows, we denote by $\{ l_1(v),\dots,l_4(v) \}$ the links incident to $v$ such that \eqref{eq:defl_i} holds.
 
Finally a \textit{rooted} stranded graph is a stranded graph with either exactly one marked vertex or exactly one marked link.
 
 \end{definition}

Note that stranded graphs for the local theory are obtained simply by ignoring the data $\mfc$ and $\mfp$. 
We close the subsection by describing the notion of stranded subgraph.

\begin{definition}
A \textit{stranded subgraph} $G'$ of a closed stranded graph $G$ is defined by a subset $L^{int}(G')\subset L(G)$. Its vertex set $V(G')$ is given by the set of all vertices incident to the links in $L^{int}(G')$ while its external link set $L^{ext}(G')$ is given by the set of all the links in $L^{ext}(G)$ that are attached to the vertices in $V(G')$ but are not in $L^{int}(G')$. 
$G'$ inherits the data $\mfc$ and $\mfp$ of $G$, so that it is indeed a stranded graph.
\end{definition}

\subsubsection{Drawing stranded graphs as ribbon graphs}
Why we call the objects ``stranded graphs'' becomes more clear once we draw them by expanding links as $d$ parallel strands. 

In Feynman diagrams each link $l$ carries a Fourier variable $m_{l}  = (m_{l,i})_{i=1}^{d} \in \N^{d}$ and contributes a factor $\angp{m_l}{-2}$, the total contribution of a given diagram comes from summing over the Fourier variable for every link subject to the constraints imposed by the vertices on Fourier modes of the links incident to them. 

At the level of links, drawing a a stranded graph using strands involves expanding each link into $d$-different edges corresponding to the $d$-different components of  it Fourier mode (which we view as carrying different colors $c \in [d]$). 
\begin{figure}[H]
    \centering
    \tikzsetnextfilename{fig7}
    \begin{tikzpicture}
\draw[black,thick,densely dotted] (-3,0)--(-2,0);
\node at(-2.5,.5){$l$};
\node at (-1,.15) {$\longleftrightarrow$};
\draw[blue] (0,0)--(1,0);
\draw[red] (0,.1)--(1,.1);
\draw[brown] (0,.2)--(1,.2);
\draw[darkgreen] (0,.3)--(1,.3);
\node at(.5,.5){$l$};
\node at (5,.15) {$\longleftrightarrow$\quad$\sum_{m_l=(\textcolor{darkgreen}{m_{l,1}},\textcolor{brown}{m_{l,2}},\textcolor{red}{m_{l,3}},\textcolor{blue}{m_{l,4}})\in\Z^4}\angp{m_l}{-2}$};
\end{tikzpicture}
    \caption{A link of a stranded graph representing a Fourier mode $m$}
\end{figure}

The Fourier relations \eqref{eq:fourier_local_vertex} and \eqref{eq:fourier_nonlocal_vertex} determine how links interact at vertices. 
We now describe how we write vertices using strands.
For the local theory, the relation  \eqref{eq:fourier_local_vertex} just produces a $d$-fold reproduction the original vertex.
\begin{figure}[H]
    \centering
    \tikzsetnextfilename{fig8}
    \begin{tikzpicture}
\draw[densely dotted, thick](0,0)--(1,0);
\draw[densely dotted, thick](0.5,-.5)--(.5,.5);
\node at(-0.2,0){$l_1$};
\node at(.50,-0.7){$l_2$};
\node at(1.2,0){$l_3$};
\node at(0.5,0.7){$l_4$};
\node at(3,0){$\longleftrightarrow \delta( \sum_{i=1}^4m_{e_i})$};
\end{tikzpicture}
\qquad 
\tikzsetnextfilename{fig9}
\begin{tikzpicture}
\draw[darkgreen](-0.05,-1.45)--(.95,-1.45);
\draw[darkgreen](0.45,-1.95)--(.45,-.95);
\draw[brown](0,-1.5)--(1,-1.5);
\draw[brown](0.5,-2)--(.5,-1);
\draw[red](0.05,-1.55)--(1.05,-1.55);
\draw[red](0.55,-2.05)--(.55,-1.05);
\draw[blue](0.1,-1.6)--(1.1,-1.6);
\draw[blue](0.6,-2.1)--(.6,-1.1);
\node at(4,-1.5){$\longleftrightarrow \prod_{c \in\{\textcolor{darkgreen}{1},\textcolor{brown}{2},\textcolor{red}{3}, \textcolor{blue}{4} \}} \delta( \sum_{i=1}^4m_{l_i,c} )$};
\node at(0,-2.5){};
\end{tikzpicture}
    \caption{The local vertex, and its trivial stranded representation}     \label{fig:localvertex}
\end{figure}
For larger Feynman diagrams in the local theory, expanding both links and vertices using strands in the local theory again just duplicates them, see the example below.
\[
\tikzsetnextfilename{fig10}
\begin{tikzpicture}
    \draw[black](-1,1)..controls(-.7,-.2)and(.7,-.2)..(1,1);
    \draw[black](-1,0)..controls(-.7,1.2)and(.7,1.2)..(1,0);
    \node at(2,.5){$\longrightarrow$};
    \draw[blue](3,1)..controls(3.3,-.2)and(4.7,-.2)..(5,1);
    \draw[red](3.05,1.05)..controls(3.3,-0.15)and(4.7,-0.15)..(4.95,1.05);
    \draw[brown](3.1,1.1)..controls(3.3,-0.1)and(4.7,-0.1)..(4.9,1.1);
    \draw[darkgreen](3.15,1.15)..controls(3.3,-0.05)and(4.7,-0.05)..(4.85,1.15);
    \draw[blue](3,0)..controls(3.3,1.2)and(4.7,1.2)..(5,0);
    \draw[red](3.05,-.05)..controls(3.3,1.15)and(4.7,1.15)..(4.95,-.05);
    \draw[brown](3.1,-.1)..controls(3.3,1.1)and(4.7,1.1)..(4.9,-.1);
    \draw[darkgreen](3.15,-.15)..controls(3.3,1.05)and(4.7,1.05)..(4.85,-.15);     
\end{tikzpicture}
\]
For the tensor field theory, we have a different vertex for each color $c$. 
The following picture shows how the relation \eqref{eq:fourier_nonlocal_vertex} can be pictured in terms of strands for $c = \textcolor{darkgreen}{1}$. 

\begin{figure}[H]\label{fig:nodestrandedgraph}
    \centering
\tikzsetnextfilename{fig11}
\begin{tikzpicture}
\draw[darkgreen] (0,1)..controls(0.1,0.8)and(0.9,.8)..(1,1);
\draw[brown] (-0.05,0.05)..controls(0.1,.1)and(0.1,.9)..(-.05,.95);
\draw[red] (-0.1,0.1)..controls(0.05,.1)and(0.05,.9)..(-.1,.9);
\draw[blue] (-0.15,0.15)..controls(0.0,.1)and(0.0,.9)..(-.15,.85);
\draw[darkgreen] (0,0)..controls(0.1,.2)and(0.9,.2)..(1,0);
\draw[brown] (1.05,0.05)..controls(0.9,.1)and(0.9,.9)..(1.05,.95);
\draw[red] (1.1,0.1)..controls(.95,.1)and(0.95,.9)..(1.1,.9);
\draw[blue] (1.15,0.15)..controls(1.0,.1)and(1.0,.9)..(1.15,.85);
\node at(-.3,1.2){$l_1$};
\node at(1.3,1.2){$l_2$};
\node at(-.3,-.2){$l_4$};
\node at(1.3,-.2){$l_3$};
\node  at(5.0,.5){$\longleftrightarrow$\quad$\delta^{\textcolor{darkgreen}{1}}_v(m_{l_1},m_{l_2},m_{l_3},m_{l_4})$};
\end{tikzpicture}
    \caption{The \textcolor{darkgreen}{green} vertex representing the \textcolor{darkgreen}{green} interaction $\cI^{\textcolor{darkgreen}{1}}$}
\end{figure}
where for $c \in [d]$ and Fourier modes $m,n,p,q \in \Z^{d}$, 
\begin{equs}    \delta^c_v(m,n,p,q)\eqdef\delta_{m_{c},-n_{c}}
\delta_{p_{c},-q_{c}}
\delta_{m_{\hc},-q_{\hc}}
\delta_{n_{\hc},-p_{\hc}}\,.
\end{equs}
An example of writing both links and vertices with strands in a larger graph is
\[
\tikzsetnextfilename{fig12}
\begin{tikzpicture}
\draw[black](-4,1)..controls(-3.7,-.2)and(-2.3,-.2)..(-2,1);
\draw[black](-4,0)..controls(-3.7,1.2)and(-2.3,1.2)..(-2,0);
\node at(-1,.5){$\longrightarrow$};
\draw[brown] (-0.05,0.05)..controls(0.1,.1)and(0.1,.9)..(-.05,.95);
\draw[red] (-0.1,0.1)..controls(0.05,.1)and(0.05,.9)..(-.1,.9);
\draw[blue] (-0.15,0.15)..controls(0.0,.1)and(0.0,.9)..(-.15,.85);
\draw[brown] (1.05,0.05)..controls(0.9,.1)and(0.9,.9)..(1.05,.95);
\draw[red] (1.1,0.1)..controls(.95,.1)and(0.95,.9)..(1.1,.9);
\draw[blue] (1.15,0.15)..controls(1.0,.1)and(1.0,.9)..(1.15,.85);
\draw[brown] (1.95,0.05)..controls(2.1,.1)and(2.1,.9)..(1.95,.95);
\draw[red] (1.9,0.1)..controls(2.05,.1)and(2.05,.9)..(1.9,.9);
\draw[blue] (1.85,0.15)..controls(2.0,.1)and(2.0,.9)..(1.85,.85);
\draw[brown] (3.05,0.05)..controls(2.9,.1)and(2.9,.9)..(3.05,.95);
\draw[red] (3.1,0.1)..controls(2.95,.1)and(2.95,.9)..(3.1,.9);
\draw[blue] (3.15,0.15)..controls(3.0,.1)and(3.0,.9)..(3.15,.85);
\draw[blue](1.14,0.855)..controls(1.4,0.9)and(1.6,0.9)..(1.86,0.855);
\draw[blue](1.14,0.145)..controls(1.4,0.1)and(1.6,0.1)..(1.86,0.145);
\draw[red](1.1,0.9)..controls(1.4,.95)and(1.6,.95)..(1.9,0.9);
\draw[red](1.1,0.1)..controls(1.4,.05)and(1.6,.05)..(1.9,0.1);
\draw[brown](1.05,0.95)..controls(1.4,1.)and(1.6,1.)..(1.95,0.95);
\draw[brown](1.05,0.05)..controls(1.4,0)and(1.6,0.)..(1.95,0.05);
\draw[darkgreen](1,.98)..controls(1.1,1.06)and(1.9,1.06)..(2,.98);
\draw[darkgreen](0,.98)..controls(.1,.9)and(.9,.9)..(1,.98);
\draw[darkgreen](2,.98)..controls(2.1,.9)and(2.9,.9)..(3,.98);
\draw[darkgreen](1,.02)..controls(1.1,-.06)and(1.9,-.06)..(2,.02);
\draw[darkgreen](0,.02)..controls(.1,.1)and(.9,.1)..(1,.02);
\draw[darkgreen](2,.02)..controls(2.1,.1)and(2.9,.1)..(3,.02);
\node at (-4.25,.75){$l_{e,1}$};
\node at (-1.75,.75){$l_{e,2}$};
\node at (-1.75,.25){$l_{e,4}$};
\node at (-4.25,.25){$l_{e,3}$};
\node at (-3,1.2){$l_{i,1}$};
\node at (-3,-.2){$l_{i,2}$};
\node at (-3.4,.5){$v_1$};
\node at (-2.6,.5){$v_2$};
\node at (.5,.5){$v_1$};
\node at (2.5,.5){$v_2$};
\node at (1.5,1.3){$l_{i,1}$};
\node at (1.5,-.3){$l_{i,2}$};
\node at (-.15,1.15){$l_{e,1}$};
\node at (3.15,1.15){$l_{e,2}$};
\node at (3.15,-.15){$l_{e,4}$};
\node at (-.15,-.15){$l_{e,3}$};
\end{tikzpicture}
\]
Going from the left to the right representation uses the data $(\mfc,\mfp)$ given by $\mfc(v_a)=\textcolor{darkgreen}{1}$, $\hat{\mfp}(v_a)=\{\{l_{i,1},l_{e,a}\},\{l_{i,2},l_{e,a+2}\}\}$ and $\bar{\mfp}(v_a)=\{\{l_{e,a},l_{e,a+2}\},\{l_{i,1},l_{i,2}\}\}$ for every $a\in\{1,2\}$. The graph above has three internal strands and eight external strands.

When we draw stranded graphs in their stranded representation, we see that a stranded graph can be associated to to a collection of graphs each of which is a single color. 
Individual colored strands pass through at through at least one link and one vertex, possibly more. 

If a strand runs through a least one external link, it is called an \textit{external strand}. Otherwise, a strand going through only internal links is called an \textit{internal strand}. We let $S(G)$ and $S^{int}(G)$ be the strand set and internal strand set of $G$.


 
By virtue of the delta functions imposed at vertices, each strand $s\in S(G)$, in addition to carrying a single color, is also associated to a one dimensional Fourier mode which we denote $m_s\in\Z$. 


\subsection{Bounding renormalized amplitudes with stranded graphs}\label{sec:6_3}
Let us go back to the evaluation of \eqref{eq:amplitude}. In the sequel, we have to partition the stochastic objects we are dealing with into two categories. We call stochastic objects of \textit{type I} all the random fields listed in Definition~\ref{def:randomfields} along with $\XdotX$ and $\Sym$, and stochastic objects of \textit{type II} $\Xtwom$, $\Xtwonm$ and $\pXtwom^{\sto}$ (which are defined in \eqref{eq:stoObUse}). 
Note that dealing with these objects is sufficient to prove Lemmas~\ref{lem:stoob} and \ref{lem:ranop}.

To a stochastic object $\tau$ of type I that contains $k$ instances of the noise we associate a collection of closed rooted stranded graphs $\G_\tau$ and also define a map $A^{u,t}_{T,N} : \G_{\tau} \rightarrow \R$ 
such that 
\begin{equs}\label{eq:covariance_representation}
{|t-u|}^{\1\{\tau=\Sym\}2\eta}\sum_{m\in\Z^{d_2},n\in\Z^{d_1}}\angp{m}{2\beta}\angp{n}{-2\alpha}C^{u,t}_{\tau}(m,n) = \sum_{G\in\G_\tau}A^{u,t}_{T,N}(G)\,.
\end{equs}
Each $G \in \G_{\tau}$ is determined by choosing one out of all the pairings of the $2k$ noises coming from Wick's rule and one of the colorings of all the vertices in the two copies of $\tau$. 
Graphs in $\G_\tau$ have links, vertices. Moreover, one or two of their links also carry a special labeling reflecting the terms $\angp{n}{-2\alpha}$ and $\angp{m}{2\beta}$ that probe the inner and outer regularities. We denote them as $l_\alpha$ and $l_\beta$, and $l_\beta$ is chosen to be the root of $G$. Note that $l_\alpha\in L(G)$ if and only if $\tau\in\{\XdotX,\Sym\}$.

If $\tau$ is of type II, then there is a slight subtlety related with the definition of $l_\alpha$ and $l_\beta$. 
Indeed, gluing two copies of $\tau$ and performing all the possible pairings of the $2k$ noises and choosing all the possible colorings of the vertices again produces terms that are indexed by a collection of closed rooted stranded graphs $\G_\tau$. 
However, for $G_s\in\G_\tau$, $l_\alpha$ and $l_\beta$ do not belong to $L(G_s)$ but are defined as follows. 
Since $G_s$ is representing the gluing of two copies of $\tau$, and $\tau$ contains one half-vertex, the gluing produces a new vertex $r$ of color $c\in[d]$, which is chosen to be the root of $G_s$. If $\tau\in\{\Xtwom,\pXtwom^{\sto}\}$, then $r$ is crossed by two strands $s_\alpha$ and $s_\beta$ of color $c$, and we have $\angp{m}{2\beta}=\angp{m_{s_\beta}}{2\beta}$ and $\angp{n}{-2\alpha}=\angp{m_{s_\alpha}}{-2\alpha}$, so that we define $l_\alpha\eqdef s_\alpha\cap r$, $l_\beta\eqdef s_\beta\cap r$, and we have
$m_{l_\alpha}=m_{s_\alpha}\in\Z$ and $m_{l_\beta}=m_{s_\beta}\in\Z$. The situation is similar if $\tau=\Xtwonm$, but now $r$ is crossed by two beams of $d-1$ strands of all the colors but $c$, that we denote $f_\alpha\eqdef(s_{\alpha,i})_{i\in[d]\setminus\{c\}}$ and $f_\beta\eqdef(s_{\beta,i})_{i\in[d]\setminus\{c\}}$. 
We also denote $m_{f_\alpha}\eqdef(m_{s_{\alpha,i}})_{i\in[d]\setminus\{c\}}$ and $m_{f_\beta}\eqdef(m_{s_{\beta,i}})_{i\in[d]\setminus\{c\}}$, that are such that $\angp{m}{2\beta}=\angp{m_{f_\beta}}{2\beta}$ and $\angp{n}{-2\alpha}=\angp{m_{f_\alpha}}{-2\alpha}$, and again we define $l_\alpha\eqdef f_\alpha\cap r$, $l_\beta\eqdef f_\beta\cap r$, so that $m_{l_\alpha}=m_{f_\alpha}\in\Z^{d-1}$ and $m_{l_\beta}=m_{f_\beta}\in\Z^{d-1}$. 
This definition implies that we do not really see the root $r$ as a vertex, but rather as two parallel links gluing the two copies of $\tau$. From now on, we thus identify $G_s$ with the graph $G$ which is obtained from $G_s$ removing the two times $1$ or $d-1$ edges that compose $l_\alpha$ and $l_\beta$ and replacing them with $l_\alpha$ and $l_\beta$. 
By definition, $L(G)\eqdef L(G_s)\sqcup\{l_\alpha\}\sqcup\{l_\beta\}$. See Figure~\ref{fig:gluing} for an illustration of the definitions of $l_\alpha$ and $l_\beta$ in the tensor graph representation. Going from $G_s$ to 
$G$ then simply involves seeing $l_\alpha$ and $l_\beta$ as two dotted lines instead of $1$ or $d-1$ parallel edges. 
\begin{figure}[H]
    \centering  
\tikzsetnextfilename{fig13} 
\begin{tikzpicture}
\draw[gray] (0,0)--(0,1);
\draw[gray] (1,0)--(1,1);
\draw[gray] (0,1)..controls(.3,1.2)and(.7,1.2)..(1,1);
\draw[gray] (0,1)..controls(.3,.8)and(.7,.8)..(1,1);
\draw[gray] (0,0)..controls(.3,.2)and(.7,.2)..(1,0);
\draw[gray] (0,0)..controls(.3,-.2)and(.7,-.2)..(1,0);
\draw[black,thick,densely dotted] (0,0)--(-.3,-.3) node[darkgreen,circle, fill=black, draw, solid, inner sep=0pt, minimum size=2pt] {};
\draw[black,thick,densely dotted] (0,1)--(-.3,1.3) node[darkgreen,circle, fill, draw, solid, inner sep=0pt, minimum size=2pt] {};
\draw[black,thick,densely dotted] (1,1)--(1.3,1.3) node[blue,circle, fill, draw, solid, inner sep=0pt, minimum size=2pt] {};
\draw[black,thick,densely dotted] (1,0)--(1.3,-.3) node[blue,circle, fill, draw, solid, inner sep=0pt, minimum size=2pt] {};
\node[black] at (-.2,.5){$l_\alpha$};
\node[black] at (1.2,.5){$l_\beta$};
\draw[gray] (3,0)--(3,1);
\draw[gray] (4,0)--(4,1);
\draw[gray] (3,1)..controls(3.3,1.2)and(3.7,1.2)..(4,1);
\draw[gray] (3,1)..controls(3.3,.8)and(3.7,.8)..(4,1);
\draw[gray] (3,0)..controls(3.3,.2)and(3.7,.2)..(4,0);
\draw[gray] (3,0)..controls(3.3,-.2)and(3.7,-.2)..(4,0);
\draw[black,thick,densely dotted] (3,0)--(2.7,-.3) node[red,circle, fill=black, draw, solid, inner sep=0pt, minimum size=2pt] {};
\draw[black,thick,densely dotted] (3,1)--(2.7,1.3) node[darkgreen,circle, fill, draw, solid, inner sep=0pt, minimum size=2pt] {};
\draw[black,thick,densely dotted] (4,1)--(4.3,1.3) node[darkgreen,circle, fill, draw, solid, inner sep=0pt, minimum size=2pt] {};
\draw[black,thick,densely dotted] (4,0)--(4.3,-.3) node[red,circle, fill, draw, solid, inner sep=0pt, minimum size=2pt] {};
\node[black] at (3.5,1.35){$l_\alpha$};
\node[black] at (3.5,-.35){$l_\beta$};
\end{tikzpicture}
\caption{The gluing of two copies of $\cdotXXloc=\Xtwom$ and of $\cXXdotloc=\Xtwonm$, with the positions of $l_\alpha$ and $l_\beta$}
\label{fig:gluing}
\end{figure}

\subsubsection{Renormalizing amplitudes}
Moreover, the formula for $A^{u,t}_{T,N}(G)$ must also take into consideration our renormalization coming from divergent subgraphs.
For every $G\in\G_\tau$, we thus define the (possibly empty) sets $\mfM^{1,c}(G)$ (respectively $\mfM^{2,c,c'}(G)$) that contain all the subgraphs of $G$ of the form $\mfM^{1,c}$ for $c\in[d]$ (respectively $\mfM^{2,c,c'}$ for $c\neq c'$). 
We write $\mfM^1(G)\eqdef\bigcup_{c=1}^d\mfM^{1,c}(G)$, $\mfM^2(G)\eqdef\bigcup_{c\neq c'}\mfM^{2,c,c'}(G)$ and $\mfM(G)\eqdef\mfM^1(G)\cup\mfM^2(G)$. 

We distinguish the links and vertices of $G$ that do not belong to some $M\in\mfM(G)$, defining $\tV(G)\eqdef V(G)\setminus\bigcup_{M\in\mfM^2}V^{int}(M)$ (this definition takes into account the fact that graphs in $\mfM^1$ do not have internal vertices) and $\tL(G)\eqdef L(G)\setminus\bigcup_{M\in\mfM(G)}L^{int}(M)$. 

Graphs $M\in\mfM^1$ only have one internal link $l_i=l_i(M)$, and we define $m_{M,i}\eqdef m_{l_i}$ the internal Fourier mode of $M$.
They have two external links $l_{e,1}=l_{e,1}(M)$ and $l_{e,2}=l_{e,2}(M)$ that carry the same Fourier mode, and we define $m_{M,e}\eqdef m_{l_{e,1}}=m_{l_{e,2}}$ the external Fourier mode of $M$. 
With this notation, if $M$ is color $c$ it carries the renormalized amplitude $\mfR^{1,c}(m_{M})$ defined in \eqref{eq:R1}. 

Graphs $M\in \mfM^{2}(G)$ have three internal links, among which two are attached to an external node of $M$ (they are thus themselves the two external links of some graph $M_1(M)\in\mfM^1(G)$). 
These two links are called \textit{superficial}, and we denote $l_{s,1}=l_{s,1}(M)$ and $l_{s,2}=l_{s,2}(M)$.
We call the remaining internal link of $M$ the \textit{deep link}, and it is denoted $l_d=l_d(M)$.

The structure of $M$ imposes that the two superficial links of $M$ carry the same Fourier mode which we denote $m_{M,s}\eqdef m_{l_{s,1}}=m_{l_{s,2}}$ while we denote $m_{M,d}\eqdef m_{l_d}$ the Fourier mode of the deep link of $M_2$ (which is also the internal link of $M_1(M_2)$). They also have two external links $l_{e,1}=l_{e,1}(M)$ and $l_{e,2}=l_{e,2}(M)$ that carry the same Fourier mode, and we define $m_{M,e}\eqdef m_{l_{e,1}}=m_{l_{e,2}}$ the external Fourier mode of $M$. With this notation, $M$ of colors $c,c'$ bears the amplitude $\mfR^{2,c,c'}(m_{M,s},m_{M,d})$. 

Finally, we let $\Tilde{\mfM}^{1,c}(G)$ be the set of all the graphs in $\mfM^{1,c}(G)$ that are not of the form $M_1(M_2)$ for some $M_2\in\mfM^{2,c,c'}(G)$ for $c'\neq c$. We also let $\Tilde{\mfM}^1(G)\eqdef\bigcup_{c=1}^d \Tilde{\mfM}^{1,c}(G)$ and $\Tilde{\mfM}(G)\eqdef\Tilde{\mfM}^{1}(G)\cup\mfM^2(G)$. A link $l$ can link two different graphs $M,M'\in\Tilde{\mfM}(G)$. If this is the case, we use the freedom we have in the labeling of $l_{e,1}$, $l_{e,2}$ to enforce the convention that if $l=l_{e,1}(M)$, then $l=l_{e,2}(M')$.
See Figure~\ref{fig:divergentgraphs} for an illustration of the notations introduced in this paragraph.
\begin{figure}[H]
    \centering
   \tikzsetnextfilename{fig14}
   \begin{tikzpicture}
\draw[darkgreen] (0,-1)--(0,0);
\draw[darkgreen] (1,-1)--(1,0);
\draw[red] (0,-1)--(1,-1);
\draw[red] (0,0)--(1,0);
\draw[brown] (0,-1)..controls(.3,-.8)and(.7,-.8)..(1,-1);
\draw[brown] (0,0)..controls(.3,-.2)and(.7,-.2)..(1,0);
\draw[blue] (0,0)..controls(.3,.2)and(.7,.2)..(1,0);
\draw[blue] (0,-1)..controls(.3,-1.2)and(.7,-1.2)..(1,-1);
\draw[black,thick,densely dotted] (1,0)arc(-45:225:.71);
\draw[black,thick,densely dotted] (0,-1)--(-.3,-1.3);
\draw[black,thick,densely dotted] (1,-1)--(1.3,-1.3);
\draw[blue] (3,0)--(3,1);
\draw[blue] (4,0)--(4,1);
\draw[brown] (3,0)--(4,0);
\draw[brown] (3,1)--(4,1);
\draw[darkgreen] (3,0)..controls(3.3,.2)and(3.7,.2)..(4,0);
\draw[darkgreen] (3,1)..controls(3.3,.8)and(3.7,.8)..(4,1);
\draw[red] (3,1)..controls(3.3,1.2)and(3.7,1.2)..(4,1);
\draw[red] (3,0)..controls(3.3,-.2)and(3.7,-.2)..(4,0);
\draw[red] (3,-2)--(3,-1);
\draw[red] (4,-2)--(4,-1);
\draw[darkgreen] (3,-2)--(4,-2);
\draw[darkgreen] (3,-1)--(4,-1);
\draw[blue] (3,-2)..controls(3.3,-1.8)and(3.7,-1.8)..(4,-2);
\draw[blue] (3,-1)..controls(3.3,-1.2)and(3.7,-1.2)..(4,-1);
\draw[brown] (3,-1)..controls(3.3,-.8)and(3.7,-.8)..(4,-1);
\draw[brown] (3,-2)..controls(3.3,-2.2)and(3.7,-2.2)..(4,-2);
\draw[black,thick,densely dotted] (4,1)arc(-45:225:.71);
\draw[black,thick,densely dotted] (3,-2)--(2.7,-2.3);
\draw[black,thick,densely dotted] (4,-2)--(4.3,-2.3);
\draw[black,thick,densely dotted] (3,-1)..controls(2.7,-.7)and(2.7,-.3)..(3,0) ;
\draw[black,thick,densely dotted] (4,-1)..controls(4.3,-.7)and(4.3,-.3)..(4,0) ;
\node[black] at(.5,-2){$M_1$};
\node[black] at(3.5,-3){$M_2$};
\node at(.5,1.4){$m_{M_1,i}$};
\node at(-.45,.65){$l_i$};
\node at(2.45,.-.5){$l_{s,1}$};
\node at(4.55,.-.5){$l_{s,2}$};
\node at(2.55,1.6){$l_d$};
\node at(3.5,2.4){$m_{M_2,d}$};
\node at(3.5,-.5){$m_{M_2,s}$};
\node at(5.1,1.2){$\Bigg\}M_1(M_2)$};
\node at(1.45,-.95){$l_{e,2}$};
\node at(-.5,-.95){$l_{e,1}$};
\node at(.5,-1.4){$m_{M_1,e}$};
\node at(4.45,-1.95){$l_{e,2}$};
\node at(2.5,-1.95){$l_{e,1}$};
\node at(3.5,-2.4){$m_{M_2,e}$};
\end{tikzpicture}
    \caption{A divergent graph $M_1$ of type $\mathfrak M^{1,\textcolor{darkgreen}{c}}$ and a divergent graph $M_2$ of type $\mathfrak M^{2,\textcolor{red}{c},\textcolor{blue}{c'}}$ pictured in the tensor graph representation introduced below}
    \label{fig:divergentgraphs}
\end{figure}

We also need the following partition of $\tL(G)$ which will enable us to reexpress \eqref{eq:amplitude}. 
First of all, note that $l_\alpha$ and $l_\beta$ belong to $\Tilde{L}(G)$ since they are attached to two different vertices in a non-melonic way. 
Links in $\Tilde{L}(G)$ are partitioned as $\tL(G)=\tL^i(G)\sqcup \Tilde{L}^c(G) \sqcup \{l_\alpha\}\sqcup \{l_\beta\}$ (if $\tau$ is a random fields, then $l_\alpha$ is absent in this decomposition) where links in $\tL^i(G)$ stand for \textit{time lines} while links in $\tL^c(G)$ represent the contraction of two noises. By time lines, we mean the links associated to the inverse heat operator in LD objects (i.e. the time integral of $P_{t-s}$), and to the time integral of $J^2$ in BG objects. Note that once all the integrations over the variables of the noises yielding some objects $\X_t$ are performed, then the times lines are the only links bearing a time integration. It turns out that $\tL^i(G)$ does not correspond to the set of all time lines, because each subgraph $M\in\mfM^{2}(G)$ only contains one time line (one of its two superficial links). 

We display below an example of graph $G\in\G_\tau$ for $\tau=\Sym$. Here, $l_\alpha$ and $l_\beta$ are indicated by their names, while $l_1$ and $l_2$ are the two elements of $\tL^i(G)$, and all the four remaining links belong to $\tL^c(G)$. Note that the same graph with $l_\alpha$ replaced by a link $l\in\tL^c(G)$ belongs to $\G_\tau$ for $\tau=\XtwoPictwo$ (more precisely it would be a contribution coming from the gluing of $\cN(\X,\Pictwo,\X)-\frac{1}{2}\mfC^2\X$ with itself).
\begin{figure}[H]
    \centering    
\tikzsetnextfilename{fig28}
\begin{tikzpicture}
\draw[gray] (0,0)--(0,1);
\draw[gray] (1,0)--(1,1);
\draw[gray] (0,1)..controls(.3,1.2)and(.7,1.2)..(1,1);
\draw[gray] (0,1)..controls(.3,.8)and(.7,.8)..(1,1);
\draw[gray] (0,0)..controls(.3,.2)and(.7,.2)..(1,0);
\draw[gray] (0,0)..controls(.3,-.2)and(.7,-.2)..(1,0);
\draw[gray] (2,0)--(2,1);
\draw[gray] (3,0)--(3,1);
\draw[gray] (2,1)..controls(2.3,1.2)and(2.7,1.2)..(3,1);
\draw[gray] (2,1)..controls(2.3,.8)and(2.7,.8)..(3,1);
\draw[gray] (2,0)..controls(2.3,.2)and(2.7,.2)..(3,0);
\draw[gray] (2,0)..controls(2.3,-.2)and(2.7,-.2)..(3,0);
\draw[black,thick,densely dotted] (1,0)..controls(1.3,-.3)and(1.7,-.3)..(2,0) ;
\node[black] at (1.5,.1){$l_\beta$};
\node[black] at (1.5,2.9){$l_\alpha$};
\node[black] at (3.5,1.5){$l_{2}$};
\node[black] at (-.5,1.5){$l_{1}$};
\draw[gray] (2,2)--(2,3);
\draw[gray] (3,2)--(3,3);
\draw[gray] (2,3)..controls(2.3,3.2)and(2.7,3.2)..(3,3);
\draw[gray] (2,3)..controls(2.3,2.8)and(2.7,2.8)..(3,3);
\draw[gray] (2,2)..controls(2.3,2.2)and(2.7,2.2)..(3,2);
\draw[gray] (2,2)..controls(2.3,1.8)and(2.7,1.8)..(3,2);
\draw[gray] (0,2)--(0,3);
\draw[gray] (1,2)--(1,3);
\draw[gray] (0,3)..controls(.3,3.2)and(.7,3.2)..(1,3);
\draw[gray] (0,3)..controls(.3,2.8)and(.7,2.8)..(1,3);
\draw[gray] (0,2)..controls(.3,2.2)and(.7,2.2)..(1,2);
\draw[gray] (0,2)..controls(.3,1.8)and(.7,1.8)..(1,2);
\draw[black,thick,densely dotted] (0,1)..controls(-.3,1.3)and(-.3,1.7)..(0,2) ;
\draw[black,thick,densely dotted] (3,1)..controls(3.3,1.3)and(3.3,1.7)..(3,2) ;
\draw[black,thick,densely dotted] (1,1)..controls(1.3,1.3)and(1.3,1.7)..(1,2) ;
\draw[black,thick,densely dotted] (2,1)..controls(1.7,1.3)and(1.7,1.7)..(2,2) ;
\draw[black,thick,densely dotted] (2,1)..controls(1.7,1.3)and(1.7,1.7)..(2,2) ;
\draw[black,thick,densely dotted] (1,3)..controls(1.3,3.3)and(1.7,3.3)..(2,3) ;
\draw[black,thick,densely dotted] (0,3)..controls(.5,4)and(2.5,4)..(3,3);
\draw[black,thick,densely dotted] (0,0)..controls(.5,-1)and(2.5,-1)..(3,0);
\end{tikzpicture}   
\caption{An example of graph $G\in\G_\tau$ for $\tau=\Sym$ with its links in $\{\tL^i(G)\}\sqcup\{l_\alpha\}\sqcup\{l_\beta\}$ indicated}
\end{figure}

With this notation in hand, we are ready to define, for any $\tau$ and $G\in\G_\tau$, a labelling of the links of $G$, $\ell:L(G)\rightarrow\R$:
\begin{equs}\label{eq:defell}
    \ell(l)\equiv\ell_\tau(l)\eqdef1-(1-\alpha)\1\{l=l_\alpha\}-(1+\beta)\1\{l=l_\beta\}-(1-\eta)\1\{\tau=\Sym\;\text{and}\;l\in\tL^i(G)\}\,.\textcolor{white}{bmabma}
\end{equs}

The tree structure of $\tau$ implies that for every $v\in \tV(G)$, there exists a unique path $\mfP_v$ made with only links in $\tL^i(G)$ between $v$ and the root (if $\tau$ is of type II) or one of the two vertices attached to the root (if $\tau$ is of type I). This induces a partial order on $\tV(G)$, and we say that $v_2\prec v_1$ if $v_1$ is on the path $\mfP_{v_1}$. 
For $l\in \tL^i(G)$, we denote by $v_1(l)$ and $v_2(l)$ the two extremities of $l$ such that $v_2\prec v_1$. 
For $l\in \tL^c(G)$, we make an arbitrary choice in labeling the two extremities of $l$ as $v_1(l)$ and $v_2(l)$.  
 
\subsubsection{Explicit formulae and bounds for renormalized amplitudes} 
We can now give an expression for $A^{u,t}_{T,N}(G)$. Recall that links of $G$ are partitioned as 
\begin{equs}
    L(G)=\bigsqcup_{M\in\Tilde{\mfM}(G)}L^{int}(M)\sqcup\tL^i(G)\sqcup \Tilde{L}^c(G) \sqcup \{l_\alpha\}\sqcup \{l_\beta\}\,.
\end{equs}
While $l_\alpha$ and $l_\beta$ bear the kernels $\angp{m_{l_\alpha}}{-2\alpha}$ and $\angp{m_{l_\beta}}{2\beta}$, links $l$ in $\tL^i(G)$ bear the integral kernel 
\[
\begin{cases}
\begin{array}{ll}
\int_{a_l}^{t_{v_1(l)}}{\tt I}(t_{v_1(l)},t_{v_2(l)},m_l) \rmd t_{v_2(l)}& \text{ if }\;\tau\neq\Sym\;, \\
{\tt I}(t_{v_1(l)},t_{v_2(l)},m_l)& \text{ if }\;\tau=\Sym\;.
\end{array}\end{cases}
\]
where $a_l\in\{-\infty,0\}$ and
\[
{\tt I}(t_1,t_2,m) =\begin{cases}
\begin{array}{ll}
\1\{|m|_\infty\leqslant N\}e^{-(t_1-t_2)\angp{m}{2}}& \text{ for a LD object }\;, \\
\frac{\partial_{t_2}\varrho_{t_2}^2(m)}{\angp{m}{2}}& \text{ for a BG object }\;.
\end{array}\end{cases}
\]
and links $l$ in $\tL^i(G)$ carry the contraction kernel ${\tt C}(t_{v_1(l)},t_{v_2(l)},m_l)$ given by
\begin{equs}
    {\tt C}(t_1,t_2,m)=\begin{cases}\begin{array}{ll}\1\{|m|_\infty\leqslant N\}\frac{e^{-|t_1-t_2|\angp{m}{2}}}{\angp{m}{2}}& \text{ for a LD object}\;,
    \\\frac{\varrho^2_{t_1\wedge t_2}(m)}{\angp{m}{2}}&\text{ for a BG object}\;.\end{array}\end{cases}.
\end{equs}
The remaining links belong to some $M\in\Tilde{\mfM}(G)$. The subamplitude $M(t_M,m_{M,e})$ of $M\in\Tilde{\mfM}^{1,c}(G)$ at $t_M=t_{v_2(l_{e,1})}$ is given by
\begin{equs}    M(t_M,m_{M,e})&=\sum_{m_{M,i}\in\Z^d}\delta_{m_{M,i,c},m_{M,e,c}}\big({\tt C}(t_M,t_M,m_{M,i})-{\tt C}(t_M,t_M,m_{M,i,\hc})\big)\,,
\end{equs}
and the subamplitude $M(t_M,m_{M,e})$ of $M\in{\mfM}^{2,c,c'}(G)$ at $t_M=t_{v_2(l_{e,1})}$ is given by
\begin{equs}  
M(t_M,&m_{M,e})=\sum_{m_{M,s},m_{M,d}\in\Z^d}
\delta_{m_{M,s,c},m_{M,e,c}}
\delta_{m_{M,d,c'},m_{M,s,c'}}\int_{a_{l_s}}^{t_M}\rmd s
\\&\big({\tt I}(t_M,s,m_{M,s}){\tt C}(t_M,s,m_{M,s})
-{\tt I}(t_M,s,m_{M,s,\hc}){\tt C}(t_M,s,m_{M,s,\hc})
\big)\\
&\big({\tt C}(s,s,m_{M,d})-{\tt C}(s,s,m_{M,d,\hc'})\big)\,.
\end{equs}

With this notation we have for $\tau\neq\Sym$ and $G\in\G_\tau$ 
\begin{equs}    A^{u,t}_{T,N}(G)=&\sum_{m_{l_1},\dots,m_{l_{|\tL(G)|}}\in\Z^{d(l)}}
   \angp{m_{l_\alpha}}{-2\alpha} \angp{m_{l_\beta}}{2\beta}
    \prod_{l\in \tL^i(G)}\int_{a_l}^{t_{n_1(l)}} \rmd t_{n_2(l)}{\tt I}(t_{n_1(l)},t_{n_2(l)},m_l) \\&\qquad    \prod_{M\in\Tilde{\mfM}(G)}M(t_M,m_{M,e})
     \prod_{v\in V(G)}\delta_v(m_{l_1(v)},\dots,m_{l_4(v)})
    \prod_{l\in \tL^c(G)}{\tt C}(t_{n_1(l)},t_{n_2(l)},m_l)
  \,,
\end{equs}
where $d(l)=d$ except, if $\tau$ is of type II, for $l_\alpha$ and $l_\beta$. Here, it is understood that the time of the root (if $\tau$ is of type II) and the times of the two vertices attached to the root (if $\tau$ is of type I) is $t$. 

Moreover, a similar formula holds for $\tau=\Sym$ and $G\in\G_\tau$, except that $n_1(l)=t$ and $n_2(l)=u$ for both $l\in \tL^i(G)$ (so that there is no integration over the variables $(t_{n_2(l)})_{l\in \tilde L^i(G)}$), and that the additional factor $|t-u|^{2\eta}$ is present. We state it for completeness:
\begin{equs}    A^{u,t}_{T,N}(G)=&\sum_{m_{l_1},\dots,m_{l_{|\tL(G)|}}\in\Z^{d(l)}}
\angp{m_{l_\beta}}{2\beta} \angp{m_{l_\alpha}}{-2\alpha} 
    \prod_{l\in \tL^i(G)}|t-u|^{\eta}{\tt I}(t,u,m_l)      
    \\&\qquad  \prod_{v\in V(G)}\delta_v(m_{l_1(v)},\dots,m_{l_4(v)})\prod_{l\in \tL^c(G)}{\tt C}(t_{n_1(l)},t_{n_2(l)},m_l)\,  
  \,.
\end{equs}

\begin{lemma}\label{lem:amplitude_bound}  
With all the previous notations introduced in Definition~\ref{def:strandedgraphs} and in the last paragraph at hand, for every stochastic object $\tau$ we have \eqref{eq:covariance_representation}.

Moreover, for every $G\in\G_\tau$, we have $A^{u,t}_{T,N}(G)\lesssim \tA(G)$ where $\tA(G)$ is independent of $u$, $t$, $T$ and $N$, and is given by 
 \begin{equs}  \label{eq:A(G)}   \tA(G)&\eqdef\sum_{m_{s_1},\dots,m_{s_{|S(G)|}}\in\Z}
     \prod_{l\in \tL(G)}\angp{m_l}{-2\ell(l)}\prod_{c\in[d]}\prod_{M\in\Tilde{\mfM}^{1,c}(G)}\frac{\angp{m_{M,e}}{2}\wedge\angp{m_{M,i}}{2}}{\angp{m_{M,i}}{2}\angp{m_{M,i,\hc}}{2}}\nonumber\\ 
       &\quad\prod_{c\neq c'}
      \prod_{M\in\mfM^{2,c,c'}(G)}
    \frac{\angp{m_{M,e}}{2}\wedge\angp{m_{M,s}}{2}}{\angp{m_{M,s}}{2}\angp{m_{M,s,\hc}}{4}}      \frac{\angp{m_{M,s}}{2}\wedge\angp{m_{M,d}}{2}}{\angp{m_{M,d}}{2}\angp{m_{M,d,\hc'}}{2}}
    \,.
 \end{equs}
 
 In particular, we have uniformly in $u,t\in\mathrm{Simp}_T$, $T\geqslant 0$ and $N\in\N$,
\begin{equs}\label{eq:BoundUnif1}
{|t-u|}^{\1\{\tau=\Sym\}2\eta}\sum_{m\in\Z^{d_2},n\in\Z^{d_1}}&\angp{m}{2\beta}\angp{n}{-2\alpha}C^{u,t}_{\tau}(m,n)\lesssim\sum_{G\in\G_\tau}\tA(G)\,.
\end{equs}

Moreover, a similar formula holds uniformly in $0\leqslant r\leqslant s\leqslant t\leqslant T<\infty$ and $N\in\N$ to control the time continuity of the LD objects: 
\begin{equs}\label{eq:BoundUnif2}
|s-r|^{\1\{\tau=\Sym\}2\eta}{|t-s|}^{-\theta}\sum_{m\in\Z^{d_2},n\in\Z^{d_1}}\angp{m}{2\beta}\angp{n}{-2\alpha}\tilde C^{r,s,t}_{\tau}(m,n)
\lesssim\sum_{G\in\G_\tau}\sum_{l_0\in \tL^r(G)}\mathring A(G,l_0)\,,\textcolor{white}{blablablabla}
\end{equs}
where $ \mathring A(G,l_0)$ is defined similarly to $\tA(G)$ but with $\ell(l)$ replaced by
\begin{equs}
    \tilde \ell(l,l_0)\equiv\tilde\ell_\tau(l,l_0)\eqdef\ell(l)-\theta\1\{l=l_0\}
\end{equs}
in \eqref{eq:A(G)}, and $\tL^r(G)$ is the set of all links of $G$ in $\tL^i(G)\sqcup\tL^c(G)$ that are attached to the root if $\tau$ is of type II or to the two vertices attached to the root if $\tau$ is of type I.

Finally, we also have a similar formula holding uniformly in $0\leqslant u\leqslant r\leqslant t\leqslant T<\infty$ and $N\in\N$ to control the inner time continuity of $\Sym$:
\begin{equs}\label{eq:boundUnif3}
   |t-r|^{2{\eta }}   |r-u|^{-\theta }   \sum_{m\in\Z^{d},n\in\Z^{d}}\angp{m}{2\beta}\angp{n}{-2\alpha} \hat C^{u,r,t}_{\tau}(m,n)\lesssim \sum_{G\in\G_\tau}\sum_{l_0\in \tL^i(G)}\mathring A(G,l_0)\,.
\end{equs}

\end{lemma} 
\begin{proof}
We have \eqref{eq:covariance_representation} by the definition of the quantities on the r.h.s of \eqref{eq:covariance_representation} and Wick's rule. 

We first show \eqref{eq:BoundUnif1}. In order to deal with both LD and BG objects in our estimates, we denote by $\tt{I}$ and $\tt{C}$ some general integration and contraction kernels verifying the following assumptions (for $a\in\{0,-\infty\}$):
\begin{subequations}
    \begin{equation}\label{eq:assumption1}
        \int_a^t{\tI}(t,s,m)\rmd s\leqslant\angp{m}{-2} \text{ for all }t\geqslant0\,,
    \end{equation}
    \begin{equation}\label{eq:assumption5}
    |t-u|^{\eta}   {\tI}(t,u,m)\lesssim\angp{m}{-2\eta} \text{ for all }t\geqslant u\geqslant0\,,
    \end{equation}
    \begin{equation}\label{eq:assumption2}
        {\tt C}(t_1,t_2,m)\leqslant\angp{m}{-2} \text{ for all }t_1,t_2\geqslant0\,,
    \end{equation}
    \begin{equs}\label{eq:assumption3} 
      &{\tt C}(t,t,m)=\Tilde{{\tt C}}(t,m)\angp{m}{-2} \text{ with }\Tilde{{\tt C}}(t)\leqslant1 \text{ for all }t\geqslant0\\&\text{and }-\partial_{m^2_i}\Tilde{{\tt C}}(t,m) =  \big| \partial_{m^2_i}\Tilde{{\tt C}}(t,m) \big| \lesssim\angp{m}{-2}\text{ for all }t\geqslant0\nonumber\,.
     \end{equs}
     \begin{equation}         
     \label{eq:assumption4}
       \int_a^t{\tI}(t,s,m){\tt C}(t,s,m)\rmd s\lesssim\Tilde{{\tt C}}(t,m)\angp{m}{-4} \text{ with }\Tilde{{\tt C}} \text{ as above}\,.
    \end{equation}
\end{subequations}
The integral and contraction kernels of both the LD and BG objects verify these assumptions, rather trivially for the LD objects for which $\Tilde{{\tt C}}(t,m)=1$. For the BG objects, we have $\Tilde{{\tt C}}(t,m)=\varrho^2_t(m)$. Note that we have
\[
-\partial_{m^2_i}\Tilde{{\tt C}}(t,m)
=
2\varrho_t(m)
\partial_{m^2_i}\frac{\ang{m}}{t}
\Big|\varrho'\Big(\frac{\ang{m}}{t}
\Big)
\Big|
=
\frac{\varrho_t(m)}{t\ang{m}}
\Big|\varrho'\Big(\frac{\ang{m}}{t} \Big)\Big|
\;.
\]
Recalling that $\supp(\varrho) \subset [\frac{1}{2},1]$ and $|\varrho'| \leqslant 1$ , we have $-\partial_{m^2_i}\Tilde{{\tt C}}(t,m)\lesssim\frac{1}{t\ang{m}}\lesssim\frac{1}{\angp{m}{2}}$. 
The estimate for $\int_a^t{\tI}(t,s,m){\tt C}(t,s,m)\rmd s$ follows similarly. Regarding \eqref{eq:assumption5}, it is only necessary in order to handle the LD object $\Sym$, and is easily verified, since
\begin{equs}
    |t-u|^{\eta}{\tt I}(t,u,m)=|t-u|^{\eta}e^{-|t-u|\ang{m}^2}=\angp{m}{-2\eta}\big(|t-u|\ang{m}^2\big)^{\eta}e^{-|t-u|\ang{m}^2}\,,
\end{equs}
and $r\mapsto r^{\eta}e^{-r}$ is bounded on $\R_{\geqslant0}$.

We first address the case where $G$ does not involve any renormalization. 
In this case, the proof directly follows from the assumptions \eqref{eq:assumption1}, \eqref{eq:assumption5} and \eqref{eq:assumption2} on ${\tt I}$ and ${\tt C}$. Indeed, once for all $l\in \tL^c(G)$, ${\tt C}(t_{v_1(l)},t_{v_2(l)},m_l)$ has been bounded by $\angp{m_l}{-2}$, there is no dependence left in $t_{v_2(l)}$ for $l\in \tL^i(G)$ apart in ${\tt I}(t_{v_1(l)},t_{v_2(l)},m_l)$. If $\tau\neq\Sym$, we can therefore perform all the integrations over the parameters $t_{v_2(l)}$, and they are in turn also bounded by $\angp{m_l}{-2}$, while if $\tau=\Sym$, we can make use of \eqref{eq:assumption5} to bound all the terms ${\tt I}(t_{v_1(l)},t_{v_2(l)},m_l)$ by $\angp{m_l}{-2\eta}$. Note that if $\tau=\Sym$, then $|\tL^i(G)|=2$ which matches the fact that we have precisely two factors $|t-u|^{\eta}$ at our disposal.

We now turn to $G$ that involve renormalization.
Below we sometimes abuse notation and write $\sqrt{u}\in\R$ for $\chi^c\big(0,\sqrt{u}(1,\dots,1) \big) \in\R^d$.  
For $M\in\Tilde{\mfM}^{1,c}(G)$, by assumption \eqref{eq:assumption3}, we have
\begin{equs}
   | {\tt C}(t_M&,t_M,m_{M,i})-{\tt C}(t_M,t_M,m_{M,i,\hc})|=\frac{
   \Tilde{{\tt C}}(t_M,m_{M,i,\hc})
   }{\angp{m_{M,i,\hc}}{2}}-\frac{\Tilde{{\tt C}}(t_M,m_{M,i})}{\angp{m_{M,i}}{2}}\\&=-\int_0^{m_c^2}\frac{\rmd}{\rmd u}\Big(
   \frac{
   \Tilde{{\tt C}}(t_M,m_{M,i,\hc}+\sqrt{u})
   }{(1+|m_{M,i,\hc}|^{2}+u)^{\frac{a}{2}}}\Big)\rmd u\; \text{ with $a=2$}\\
   &= 
   \int_0^{m_c^2}
   \frac{\frac{\rmd}{\rmd u}\big(-
   \Tilde{{\tt C}}(t_M,m_{M,i,\hc}+\sqrt{u})\big)
   }{(1+|m_{M,i,\hc}|^{2}+u)^{\frac{a}{2}}}\rmd u
   +\int_0^{m_c^2}
   \frac{
   \Tilde{{\tt C}}(t_M,m_{M,i,\hc}+\sqrt{u})
   }{(1+|m_{M,i,\hc}|^{2}+u)^{\frac{a+2}{2}}}\rmd u
   \\&\lesssim\int_0^{m_c^2}
   \frac{\rmd u }{(1+|m_{M,i,\hc}|^{2}+u)^{\frac{a+2}{2}}}=\frac{1}{\angp{m_{M,i,\hc}}{2}}-\frac{1}{\angp{m_{M,i}}{2}}
   \\&\lesssim\frac{m^2_{M,i,c}}{\angp{m_{M,i}}{2}\angp{m_{M,i,\hc}}{2}}\lesssim \frac{\angp{m_{M,e}}{2}\wedge\angp{m_{M,i}}{2}}{\angp{m_{M,i}}{2}\angp{m_{M,i,\hc}}{2}}\,.
\end{equs}

Regarding $M\in\mfM^{2,c,c'}(G)$, we first bound its subamplitude in $\mfM^{1,c'}(G)$ as
\begin{equs}
  |  {\tt C}(s,s,m_{M,d})-{\tt C}(s,s,m_{M,d,\hc'})|\lesssim \frac{\angp{m_{M,s}}{2}\wedge\angp{m_{M,d}}{2}}{\angp{m_{M,d}}{2}\angp{m_{M,d,\hc'}}{2}}\,.
\end{equs}
Then, we can integrate on $s$ which yields, using assumptions \eqref{eq:assumption3} and \eqref{eq:assumption4} and repeating the argument used to deal with $M\in\Tilde{\mfM}^{1,c}(G)$ with $a=4$,
\begin{equs}  
&\int_{a_{l_s}}^{t_M}\rmd s
|{\tt I}(t_M,s,m_{M,s}){\tt C}(t_M,s,m_{M,s})
-{\tt I}(t_M,s,m_{M,s,\hc}){\tt C}(t_M,s,m_{M,s,\hc})|\\
&=\frac{
   \Tilde{{\tt C}}(t_M,m_{M,s,\hc})
   }{\angp{m_{M,s,\hc}}{4}}-\frac{\Tilde{{\tt C}}(t_M,m_{M,s})}{\angp{m_{M,s}}{4}}\lesssim \frac{\angp{m_{M,e}}{2}\wedge\angp{m_{M,s}}{2}}{\angp{m_{M,s}}{2}\angp{m_{M,s,\hc}}{4}}\,.
\end{equs}

The proof of \eqref{eq:BoundUnif2} is similar, but we first have to pre-process a bit the expression of $\delta_{s,t}\tau$. First, observe that there are either two (for $\Xtwom,\Xtwonm$ and $\XdotX$) or three inverse heat operators attached to the root of $\tau$. Using 
\begin{equs}
    \int_a^tP_{t-u_{1}}\rmd u_1&\int_a^tP_{t-u_{2}}\rmd u_2-\int_a^sP_{s-u_{1}}\rmd u_1\int_a^sP_{s-u_{1}}\rmd u_2\\&=\Big(\int_a^tP_{t-u_{1}}\rmd u_1-\int_a^sP_{s-u_{1}}\rmd u_1\Big)\int_a^tP_{t-u_{2}}\rmd u_2+\Big(\int_a^tP_{t-u_{2}}\rmd u_2-\int_a^sP_{s-u_{2}}\rmd u_2\Big)\int_a^tP_{s-u_{1}}\rmd u_1
\end{equs}
in the first case and a similar telescoping sum in the second case, one can always create a difference of two inverse heat operators. We can now conclude, since proceeding as follows for $a\in\{-\infty,0\}$, we have that a difference of two inverse heat operators always yields a good factor $|t-s|^\theta$ (loosing a bit of the decay in $m$):
\begin{equs}
\Big|  \int_a^tP_{t-u_{1}}(m)\rmd u_1-\int_a^sP_{s-u_{1}}(m)\rmd u_1\Big| &=\Big| \int_s^tP_{t-u_1}(m)\rmd u_1+\int_a^s\big(P_{t-u_1}-P_{s-u_1}\big)(m)\rmd u_1\Big| \\
&\lesssim \angp{m}{-2}\big(1-e^{-|t-s|\angp{m}{2}}\big)\lesssim\angp{m}{-2+2\theta}|t-s|^\theta\,,
\end{equs}
where for $v\geqslant0$ we use the short-hand notation $P_v(m)=e^{-v\angp{m}{2}}$.

Finally, \eqref{eq:boundUnif3} is proved in the same way, pre-processing the expression of $\delta^i_{u,r}\tau$ using the fact that
\begin{equs}
    P_{t-r}(m)-P_{t-u}(m)=P_{t-r}(m)\big(1-P_{r-u}(m)\big)\lesssim P_{t-r}(m)\angp{m}{2\theta}|r-u|^\theta\,.
\end{equs}
\end{proof}





\subsection{Tensor graphs}
In Sobolev norms of the objects in terms of a sum over a class a graph, such that the amplitudes of the graphs are purely expressed in terms of Fourier mode. If this rewriting is useful, it is not yet sufficient to obtain the regularity of the objects. To obtain a better bound easier to handle, we need to introduce a new class of graphs.  
\begin{definition} 
A \textit{$(d+1)-$colored graph} $G=(N,E,\mfc)$ is a $(d+1)$-regular properly edge-colored graph. 
Properly edge-colored means the graph $G=(N,E)$ comes with a map $\mathfrak c $: $E\rightarrow\{0,\dots,d\}$ such that for every node $\mcb{n}\in N$ and any color $c\in\{0,\dots,d\}$, there exists a unique edge $e\in E(\mcb{n})$ such that $\mathfrak c(e)=c$ \dash here $E(\mcb{n})\subset E$ denotes the set of all edges incident to $\mcb{n}$.   
\end{definition}
\begin{definition} Given a $(d+1)-$colored graph $G$ and and two colors $0 \leqslant c_{1} < c_{2} \leqslant d$, a \textit{face} of $G$ of color $(c_1,c_2)$ is a (simple) cycle of edges all of which are colored either $c_1$ or $c_2$.
\end{definition}

\begin{definition} 
A \textit{tensor graph} $G=(N,E,\mathfrak c)$ is is $(d+1)-$colored graph where we also allow for a new kind of node, called ``external nodes'',  each of which are incident to precisely one edge of color $0$ (which we call an ``external edge'').
We also enforce the following constraints on tensor graphs:
\begin{itemize}
\item For $c_{1} > 0$, then any $(c_1,c_2)$-face has length $2$ or $4$. 
\item No two nodes can be connected by $d$ edges of colors $1,\dots,d$. 
\item Any face of colors $(c_1,c_2)$ with $c_1>0$ which is of length $4$ must, for $i \in \{1,2\}$, correspond to a ``vertex of color $c_{i}$'' , in particular one can label four nodes though which this face runs through as $\mcb{n}_1,\dots,\mcb{n}_4$ in a way such that 
\begin{itemize}
\item $\mcb{n}_1$ and $\mcb{n}_2$ are connected by a single edge of color $c_{i}$ and the same is true of  $\mcb{n}_3$ and $\mcb{n}_4$. 
\item $\mcb{n}_1$ and $\mcb{n}_4$ are connected by precisely $d-1$ edges of all the colors in $[d]\setminus\{c_i\}$,  and the same is true of $\mcb{n}_2$ and $\mcb{n}_3$ are connected by $d-1$ edges of all the colors in $[d]\setminus\{c_i\}$.
\end{itemize}
\end{itemize}

An edge of color $0$ is called a \textit{link}, we denote by $L(G)$, $L^{int}(G)$ and $L^{int}(G)$ the link set, internal link set and external link set of $G$. 
 We write $V(G)$ for the vertex set of $G$, a vertex is a set of four nodes $\{\mcb{n}_{1},\dots,\mcb{n}_{4}\}$ forming a vertex of color $c$ as described above.
For a vertex $v\in V(G)$ with nodes $\mcb{n}_1,\dots,\mcb{n}_4$, we denote by $l_i(v)$ the link attached to $\mcb{n}_i$. \\
Faces of color $(0,c)$ that run through at least one external link of $G$ are called ``\textit{external strands} of color $c$'' \dash by parity they must run through exactly two external links.  
Faces of color $(0,c)$ that are not \text{external strands} are called ``\textit{internal strands} of color $c$''.
We write $S$, $S^{int}(G)$, and $S^{ext}(G)$ the strand, internal strand, and external strand set of $G$.

A \textit{closed tensor graph} is an tensor graph with $L^{ext}(G)=0$, and a tensor graph that is not closed it called \textit{open}.
\end{definition}

 \begin{definition}
A \textit{tensor subgraph} of a closed tensor graph $G$ is an open tensor graph $G'$ determined by a subset $L^{int}(G')$ of $L(G)$. 
Its vertex set $V(G')$ is the set of all the extremities of all the links in $L^{int}(G')$, and its external link set $L^{ext}(G')$ are all the other links attached to the vertices in $V(G')$ that don't belong to $L^{int}(G')$. It inherits the data $\mfc$ of $G$, so that it is indeed a tensor graph. A subgraph $G'$ of a graph $G$ may not be connected, but its connected connected components are also tensor graphs themselves. 
\end{definition}
\begin{definition}
An open tensor graph $G=(N,E,\mathfrak c)$ has a \textit{boundary graph} $\partial G$ defined as follows: we remove all the internal links of $G$ and reconnect together the edges of colors $1,\dots,d$ that belong to the face going through one removed link. Note that the boundary graph is not necessarily connected. We denote by $C(\partial G)$ the number of connected components of $\partial G$. Finally, note that the boundary graph of a closed tensor graph is empty and has $C(\partial G)=0$.
\end{definition}
\begin{figure}[H]
    \centering
    \tikzsetnextfilename{fig15}
    \begin{tikzpicture}
\draw[gray] (0,0)--(1,0);
\draw[gray] (0,1)--(1,1);
\draw[gray] (1,0)..controls(.8,.3)and(.8,.7)..(1,1);
\draw[gray] (1,0)..controls(1.2,.3)and(1.2,.7)..(1,1);
\draw[gray] (0,0)..controls(-.2,.3)and(-.2,.7)..(0,1);
\draw[gray] (0,0)..controls(.2,.3)and(.2,.7)..(0,1);
\draw[black,thick,densely dotted] (0,0)--(-.3,-.3) ;
\draw[black,thick,densely dotted] (0,1)--(-.3,1.3) ;
\draw[gray] (2,0)--(3,0);
\draw[gray] (2,1)--(3,1);
\draw[gray] (2,0)..controls(1.8,.3)and(1.8,.7)..(2,1);
\draw[gray] (2,0)..controls(2.2,.3)and(2.2,.7)..(2,1);
\draw[gray] (3,0)..controls(2.8,.3)and(2.8,.7)..(3,1);
\draw[gray] (3,0)..controls(3.2,.3)and(3.2,.7)..(3,1);
\draw[black,thick,densely dotted] (1,0)..controls(1.3,-.3)and(1.7,-.3)..(2,0) ;
\draw[black,thick,densely dotted] (1,1)..controls(1.3,1.3)and(1.7,1.3)..(2,1) ;
\draw[black,thick,densely dotted] (3,1)--(3.3,1.3);
\draw[black,thick,densely dotted] (3,0)--(3.3,-.3);
\draw[gray] (5,0)..controls(4.8,.3)and(4.8,.7)..(5,1);
\draw[gray] (5,0)..controls(5.2,.3)and(5.2,.7)..(5,1);
\draw[gray] (6,0)..controls(5.8,.3)and(5.8,.7)..(6,1);
\draw[gray] (6,0)..controls(6.2,.3)and(6.2,.7)..(6,1);
\draw[black,thick,densely dotted] (5,0)--(4.7,-.3) ;
\draw[black,thick,densely dotted] (5,1)--(4.7,1.3) ;
\draw[black,thick,densely dotted] (6,0)--(6.3,-.3) ;
\draw[black,thick,densely dotted] (6,1)--(6.3,1.3) ;
\node[gray] at (.5,.15){$c$};
\node[gray] at (.5,1.15){$c$};
\node[gray] at (2.5,.2){$c'$};
\node[gray] at (2.5,1.2){$c'$};
\end{tikzpicture}
    \caption{An open tensor graph and its boundary graph for $c\neq c'$. Note that if $c=c'$, the boundary graph would be the node of color $c$.}
\end{figure}
\begin{definition}
We can associate to a $(d+1)-$colored graph $G$ a \textit{degree} $\delta(G)$ encoding some combinatorial information about the structure of $G$. We refer to \cite{randomtensors} for a gentle introduction to this notion of degree as well as its property. In the present work, we will simply use the following two facts: for an open tensor graph $G$, the degree is given by
\begin{equs}\label{eq:degree}
    \delta(G)=d-C(\partial G)+(d-1)|V(G)|-|S^{int}(G)|-\frac{d-1}{2}|L^{ext}(G)|\,,
\end{equs}
and is an integer, which if it is not 0, is at least $d-2$. Null degree graphs are called \textit{melonic}, and are easily identifiable by their tree like structure. Indeed, tensor graphs are in one to one correspondence with a class of maps, under the transformation that shrinks faces of size $d-1$ to vertices and stretches the vertices to edges. Melonic graphs correspond to the trees under this transformation.
\begin{figure}[H]
    \centering    \tikzsetnextfilename{fig16}
    \begin{tikzpicture}
\draw[gray] (0,0)--(1,0);
\draw[gray] (0,1)--(1,1);
\draw[gray] (1,0)..controls(.8,.3)and(.8,.7)..(1,1);
\draw[gray] (1,0)..controls(1.2,.3)and(1.2,.7)..(1,1);
\draw[gray] (0,0)..controls(-.2,.3)and(-.2,.7)..(0,1);
\draw[gray] (0,0)..controls(.2,.3)and(.2,.7)..(0,1);
\draw[gray] (2,0)--(3,0);
\draw[gray] (2,1)--(3,1);
\draw[gray] (2,0)..controls(1.8,.3)and(1.8,.7)..(2,1);
\draw[gray] (2,0)..controls(2.2,.3)and(2.2,.7)..(2,1);
\draw[gray] (3,0)..controls(2.8,.3)and(2.8,.7)..(3,1);
\draw[gray] (3,0)..controls(3.2,.3)and(3.2,.7)..(3,1);
\draw[black,thick,densely dotted] (1,0)..controls(1.3,-.3)and(1.7,-.3)..(2,0) ;
\draw[black,thick,densely dotted] (1,1)..controls(1.3,1.3)and(1.7,1.3)..(2,1) ;
\draw[black, ultra thick] (7,.5)--(6,.5);
\draw[black, ultra thick] (7,.5)--(8,.5);
\draw[black,thick,densely dotted] (0,1)arc(45:315:.71);
\draw[black,thick,densely dotted] (3,0)arc(-135:135:.71);
\node[white,circle, fill=black, draw, solid, inner sep=0pt, minimum size=8pt] at(6,.5) {};
\node[white,circle, fill=black, draw, solid, inner sep=0pt, minimum size=8pt] at(7,.5) {};
\node[white,circle, fill=black, draw, solid, inner sep=0pt, minimum size=8pt] at(8,.5) {};

\end{tikzpicture}
    \caption{A melonic graph and its tree representation}
\end{figure}
\end{definition}
\begin{remark}
Tensor graphs are in one-to-one correspondence with the stranded graphs introduced in Section~\ref{subsubsec:strandedgraphs} and we often identify them. 
In particular, the notions of link, vertex and strand match exactly when comparing a stranded graph to the tensor graph is corresponds to.
    
Writing $G$ for a stranded graph and $\tilde{G}$ for the corresponding tensor graph, note that the boundary graph $\partial \Tilde{G}$ of a tensor graph $\Tilde{G}$ is the tensor graph corresponding to the stranded graph $\partial G$ obtained by removing all internal strands from $G$. 
\begin{figure}[H]
    \centering
    \tikzsetnextfilename{fig17}
    \begin{tikzpicture}
\draw[darkgreen] (0,1)..controls(0.1,0.8)and(0.9,.8)..(1,1);
\draw[brown] (-0.05,0.05)..controls(0.1,.1)and(0.1,.9)..(-.05,.95);
\draw[red] (-0.1,0.1)..controls(0.05,.1)and(0.05,.9)..(-.1,.9);
\draw[blue] (-0.15,0.15)..controls(0.0,.1)and(0.0,.9)..(-.15,.85);
\draw[darkgreen] (0,0)..controls(0.1,.2)and(0.9,.2)..(1,0);
\draw[brown] (1.05,0.05)..controls(0.9,.1)and(0.9,.9)..(1.05,.95);
\draw[red] (1.1,0.1)..controls(.95,.1)and(0.95,.9)..(1.1,.9);
\draw[blue] (1.15,0.15)..controls(1.0,.1)and(1.0,.9)..(1.15,.85);
\draw[darkgreen] (2,0)--(3,0);
\draw[darkgreen] (2,1)--(3,1);
\draw[red] (2,0)--(2,1);
\draw[red] (3,0)--(3,1);
\draw[brown] (3,0)..controls(2.8,.3)and(2.8,.7)..(3,1);
\draw[blue] (3,0)..controls(3.2,.3)and(3.2,.7)..(3,1);
\draw[blue] (2,0)..controls(1.8,.3)and(1.8,.7)..(2,1);
\draw[brown] (2,0)..controls(2.2,.3)and(2.2,.7)..(2,1);
\draw[black,thick,densely dotted] (2,0)--(1.7,-.3) ;
\draw[black,thick,densely dotted] (3,1)--(3.3,1.3) ;
\draw[black,thick,densely dotted] (2,1)--(1.7,1.3) ;
\draw[black,thick,densely dotted] (3,0)--(3.3,-.3) ;
\end{tikzpicture}
    \caption{A vertex as a stranded graph vertex and also as a tensor graph}
\end{figure}
\end{remark}
Thinking in terms of tensor graphs and stranded graphs clarifies the definition of $l_\alpha$ and $l_\beta$ when $\tau$ is of type II. 
Indeed, these two links are actually one or $d-1$ of the edges constituting the rooted vertex of the tensor graph, as was shown on Figure~\ref{fig:gluing}.
\begin{remark}
    In two dimensions, stranded graphs are ribbon graphs. It turns out that the degree $\delta$ of a tensor graph $G$ is equal to the genus of the ribbon graph corresponding to $G$, and is thus a topological quantity. In higher dimension, however, the degree is not a topological invariant, even though melonic graphs are a particular class of tensor graphs that embed in the $d$-sphere.
\end{remark}
We have thus seen that for any object $\tau$, \eqref{eq:amplitude} is bounded by a sum indexed by $\G_\tau$, a certain class of closed tensor graphs. The number of vertices and some links of the elements of $\G_\tau$ are directly fixed by the structure of $\tau$ (the links $l_\alpha$, $l_\beta$, and the links corresponding to the time lines), while the remaining links correspond to the possible contractions of the noises. In the sequel, we call skeleton graph of $\tau$ the graph corresponding to $\sum_{n\in\Z^{d_1},m\in\Z^{d_2}}\angp{n}{-2\alpha}\angp{m}{2\beta}\hat{\tau}_{m,n}\hat{\tau}_{-m,-n}$ before the expectation is taken. Below we present the skeleton graph of $\Xtwom$ as a random operator on the left and the skeleton graph of $\Xthreeloc$ on the right: 
\begin{figure}[H]
    \centering 
\tikzsetnextfilename{fig18}
\begin{tikzpicture}
\draw[gray] (0,0)--(0,1);
\draw[gray] (1,0)--(1,1);
\draw[gray] (0,1)..controls(.3,1.2)and(.7,1.2)..(1,1);
\draw[gray] (0,1)..controls(.3,.8)and(.7,.8)..(1,1);
\draw[gray] (0,0)..controls(.3,.2)and(.7,.2)..(1,0);
\draw[gray] (0,0)..controls(.3,-.2)and(.7,-.2)..(1,0);
\draw[black,thick,densely dotted] (0,0)--(-.3,-.3) node[black,circle, fill=black, draw, solid, inner sep=0pt, minimum size=2pt] {};
\draw[black,thick,densely dotted] (0,1)--(-.3,1.3) node[black,circle, fill, draw, solid, inner sep=0pt, minimum size=2pt] {};
\draw[black,thick,densely dotted] (1,1)--(1.3,1.3) node[black,circle, fill, draw, solid, inner sep=0pt, minimum size=2pt] {};
\draw[black,thick,densely dotted] (1,0)--(1.3,-.3) node[black,circle, fill, draw, solid, inner sep=0pt, minimum size=2pt] {};
\node[black] at (-.2,.5){$l_\alpha$};
\node[black] at (1.2,.5){$l_\beta$};
\node[] at (0,-.7){$ $};
\end{tikzpicture}
\qquad 
\qquad 
\qquad 
\tikzsetnextfilename{fig19}
\begin{tikzpicture}
\draw[gray] (0,0)--(0,1);
\draw[gray] (1,0)--(1,1);
\draw[gray] (0,1)..controls(.3,1.2)and(.7,1.2)..(1,1);
\draw[gray] (0,1)..controls(.3,.8)and(.7,.8)..(1,1);
\draw[gray] (0,0)..controls(.3,.2)and(.7,.2)..(1,0);
\draw[gray] (0,0)..controls(.3,-.2)and(.7,-.2)..(1,0);
\draw[black,thick,densely dotted] (0,0)--(-.3,-.3) node[black,circle, fill=black, draw, solid, inner sep=0pt, minimum size=2pt] {};
\draw[black,thick,densely dotted] (0,1)--(-.3,1.3) node[black,circle, fill, draw, solid, inner sep=0pt, minimum size=2pt] {};
\draw[black,thick,densely dotted] (1,1)--(1.3,1.3) node[black,circle, fill, draw, solid, inner sep=0pt, minimum size=2pt] {};
\draw[gray] (2,0)--(2,1);
\draw[gray] (3,0)--(3,1);
\draw[gray] (2,1)..controls(2.3,1.2)and(2.7,1.2)..(3,1);
\draw[gray] (2,1)..controls(2.3,.8)and(2.7,.8)..(3,1);
\draw[gray] (2,0)..controls(2.3,.2)and(2.7,.2)..(3,0);
\draw[gray] (2,0)..controls(2.3,-.2)and(2.7,-.2)..(3,0);
\draw[black,thick,densely dotted] (1,0)..controls(1.3,-.3)and(1.7,-.3)..(2,0) ;
\draw[black,thick,densely dotted] (2,1)--(1.7,1.3) node[black,circle, fill, draw, solid, inner sep=0pt, minimum size=2pt] {};
\draw[black,thick,densely dotted] (3,1)--(3.3,1.3) node[black,circle, fill, draw, solid, inner sep=0pt, minimum size=2pt] {};
\draw[black,thick,densely dotted] (3,0)--(3.3,-.3) node[black,circle, fill, draw, solid, inner sep=0pt, minimum size=2pt] {};
\node[black] at (1.5,-.5){$l_\beta$};
\end{tikzpicture}   
\end{figure}
We have now gathered all the notations we needed in order to obtain a good upper bound on $\Tilde{A}(G)$. This is done through a multiscale analysis, which is introduced in the next section.
\subsection{Multiscale analysis}\label{sec:6_5}
\begin{definition}
An open tensor graph $G\in\G_\tau$ furthermore comes with a \textit{renormalized superficial degree of divergence} $\omega(G)$ defined as
\begin{equs}
        \omega(G)(\alpha,\beta)=&-2\sum_{l\in L^{int}(G)}\ell(l)+|S^{int}(G)|-2\1\{G=\mfM^1 \text{ }
        \mathrm{ or }\text{ }\mfM^2\}\,.
\end{equs}
Note that is the case where $\tau=\Sym$, $\omega(G)$ also depends on $\eta$, but we choose to drop this dependence since $\eta$ will be taken arbitrarily close to one. 
This definition reflects the fact that each internal strand comes with a sum over $\Z$, while each internal link brings a factor $\angp{m_l}{-2\ell(l)}$, and each renormalized subgraph has a power counting improved by two. This degree of divergence is superficial in the sens that its negativity is a necessary but not sufficient condition for $\tA(G)$ to be convergent. Moreover, the degree of divergence rewrites as
\begin{equs} \nonumber   \omega(G)(\alpha,\beta)&\eqdef d-(5-d)|V(G)|-\frac{d-3}{2}|L^{ext}(G)|-\delta(G)-C(\partial G)\\&\quad+2(1-\alpha)\1\{l_\alpha\in G\}+2(1+\beta)\1\{l_\beta\in G\}\label{eq:divergencedegree}+2(1-\eta)|\tL^i(G)|\1\{\tau=\Sym\}\textcolor{white}{blablab}\\&\quad-2\1\{G=\mfM^1\text{ or }\mfM^2\}\,,\nonumber
\end{equs}
and we say that $G$ is \textit{convergent} (resp. \textit{divergent}) if $\omega(G)$ is strictly negative (resp. positive or null). Note that the second expression is obtained by injecting the definition of the degree \eqref{eq:degree} in \eqref{eq:divergencedegree} and using the combinatorial relation $4|V(G)|=2|L^{int}(G)|+|L^{ext}(G)|$ as well as the definition of $\ell$, \eqref{eq:defell}. 
These computations also justify \eqref{eq:perturbative_degree_tensor} in Remark~\ref{remark:DPD}.
\end{definition}

To obtain good upper bounds on amplitudes we slice ``propagators" across scales. 
\begin{lemma}[Multiscale decomposition]   Pick $\ell\in(0,4]$. There exists a numerical constant $C>0$ such that for every $m\in\Z^d$, 
    \begin{equs}\label{eq:slacing1}
        \angp{m}{-2\ell}\lesssim \sum_{k\geqslant-1}2^{-2k\ell}e^{-C2^{-k}\sum_{c\in[d]} |m_c|}\,.
    \end{equs}
Moreover, for $\ell\in[0,4]$ and $k_0\geqslant0$, we have
\begin{equs}\label{eq:slacing2}
    e^{-2^{-2k_0}\angp{m}{2}} \angp{m}{-2\ell}\lesssim \sum_{k\leqslant k_0-1}2^{-2k\ell}e^{-C2^{-k}\sum_{c\in[d]} |m_c|}\,.
\end{equs}
\end{lemma}
\begin{proof}
We start by writing
\begin{equs}
    \angp{m}{-2\ell}&\lesssim\int_0^{+\infty}a^{\ell-1}e^{-a\angp{m}{2}}\rmd a\\&\lesssim\int_1^{+\infty}a^{\ell-1}e^{-a\angp{m}{2}}\rmd a+\sum_{k\geqslant0}\int_{2^{-2(k+1)}}^{2^{-2k}}a^{\ell-1}e^{-a\angp{m}{2}}\rmd a\,.
\end{equs}
In the second line, the first term corresponds to the slice $-1$ and is bounded by a constant while the $k$-th slice is bounded as
\begin{equs}    \int_{2^{-2(k+1)}}^{2^{-2k}}a^{\ell-1}e^{-a\angp{m}{2}}\rmd a&\lesssim 2^{-2k\ell}e^{-2^{-2(k+1)}\angp{m}{2}}\lesssim 2^{-2k\ell}e^{-C2^{-k}\sum_{c\in[d]}|m_c|}\,,
\end{equs}
where we used the equivalence of the $\ell^2$ and $\ell^1$ norms in finite dimension.

The proof of \eqref{eq:slacing2} goes the same way, starting from the observation that 
\begin{equs}
      e^{-2^{-2k_0}\angp{m}{2}} \angp{m}{-2\ell}\lesssim\int_{2^{-2k_0}}^{+\infty}a^{\ell-1}e^{-a\angp{m}{2}}\rmd a\,.
\end{equs}
Note that the integration doesn't include $a =0$ so it remains finite if $\ell=0$.
\end{proof}
\begin{lemma}[Multiscale analysis] Let $\tau$ a stochastic object and $G\in\G_\tau$. It holds
\begin{equs}
     \Tilde{A}(G)\lesssim \sum_{k_1,\dots,k_{|L(G)|}\geqslant-1}    \prod_{i\geqslant-1}\prod_{k=1}^{C(\cG^i)}2^{\omega(\cG^i_k)}\,,
    \end{equs}
    where $\cG^i$ is the subgraph of $G$ with internal link set $L^{int}(\cG^i)=\{l\in L(G):k_l\geqslant i\}$, and we denote by $(\cG^i_k)_k$ its $C(\cG^i)$ connected components. 
\end{lemma}
\begin{proof}
Recall the expression of $\tA(G)$ in \eqref{eq:A(G)}. We first deal with the contribution of the renormalized subgraphs, and begin with the study the possible subgraphs of the form $\Tilde{\mfM}^1$. Pick $M\in\Tilde{\mfM}^{1,c}(G)$. The contribution of $M\cup \{l_{e,1}(M)\}$ to $\tA(G)$ is given by
\begin{equs}
     &\frac{\angp{m_{M,e}}{2}\wedge\angp{m_{M,i}}{2}}{\angp{m_{M,e}}{2\ell(l_{e,1})    
     }\angp{m_{M,i}}{2}\angp{m_{M,i,\hc}}{2}}\\
      &\quad=\1\{\angp{m_{M,e}}{2}<\angp{m_{M,i}}{2}\}\angp{m_{M,e}}{2(1-\ell(l_{e,1}))}\angp{m_{M,i}}{-2}\angp{m_{M,i,\hc}}{-2}\\
      &\quad\textcolor{white}{=}+\1\{\angp{m_{M,i}}{2}\leqslant\angp{m_{M,e}}{2}\}\angp{m_{M,e}}{-2\ell(l_{e,1})}\angp{m_{M,i,\hc}}{-2}
      \\
      &\quad\lesssim\1\{\angp{m_{M,e}}{2}<\angp{m_{M,i}}{2}\}\angp{m_{M,e}}{2(1-\ell(l_{e,1}))}\angp{m_{M,i,\hc}}{-4}\\
      &\quad\textcolor{white}{=}+\1\{\angp{m_{M,i,\hc}}{2}\leqslant\angp{m_{M,e}}{2}\}\angp{m_{M,e}}{-2\ell(l_{e,1})}\angp{m_{M,i,\hc}}{-2}
      \\
      &\quad\lesssim\sum_{k_i\geqslant-1}2^{-4k_i}e^{-C2^{-k_i}\sum_{c'\neq c}|m_{M,i,c'}|}\1\{\angp{m_{M,e}}{2}<\angp{m_{M,i}}{2}\}\angp{m_{M,e}}{2(1-\ell(l_{e,1}))}\\
      &\quad\textcolor{white}{=}+\sum_{k_e\geqslant-1}2^{-2\ell(l_{e,1})k_e}e^{-C2^{-k_e}\sum_c |m_{M,e,c}|}\1\{\angp{m_{M,i,\hc}}{2}\leqslant\angp{m_{M,e}}{2}\}\angp{m_{M,i,\hc}}{-2}\,.
\end{equs}
In the first inequality, we use the fact that since $\angp{m_{M,i}}{2}\geqslant\angp{m_{M,i,\hc}}{2}$, we have 
\begin{equs}
    \1\{\angp{m_{M,i}}{2}\leqslant\angp{m_{M,e}}{2}\}\leqslant\1\{\angp{m_{M,i,\hc}}{2}\leqslant\angp{m_{M,e}}{2}\}\,,
\end{equs}
while in going from the first to the second inequality, we used \eqref{eq:slacing1}. To conclude, we bound the two indicator functions as $\1\{\angp{m_{M,e}}{2}$ $<\angp{m_{M,i}}{2}\}\lesssim
e^{-2^{-2k_i}\angp{m_{M,e}}{2}}
$ and $\1\{\angp{m_{M,i,\hc}}{2}\leqslant\angp{m_{M,e}}{2}\}\lesssim
e^{-2^{-2k_e}\angp{m_{M,i,\hc}}{2}}
$ and use \eqref{eq:slacing2}. This yields
\begin{equs}
     &\frac{\angp{m_{M,e}}{2}\wedge\angp{m_{M,i}}{2}}{\angp{m_{M,e}}{2\ell(l_{e,1})     
     }\angp{m_{M,i}}{2}\angp{m_{M,i,\hc}}{-2}}\\
      &\quad\lesssim\sum_{k_e<k_i}2^{-4k_i}e^{-C2^{-k_i}\sum_{c'\neq c}|m_{M,i,c'}|}
      2^{2(1-\ell(l_{e,1}))k_e}e^{-C2^{-k_e}\sum_{c\in[d]}|m_{M,e}|}
      \\
      &\quad\textcolor{white}{=}+\sum_{k_i\leqslant k_e}2^{-2\ell(l_{e,1})k_e}e^{-C2^{-k_e}\sum_c |m_{M,e,c}|}
      2^{-2k_i}e^{-C2^{-k_i}\sum_{c'\neq c} |m_{M,i,c'}|}\\
      &\quad\lesssim\sum_{k_i,k_e\geqslant-1}2^{-2k_i}
      2^{-2\ell(l_{e,1})
      k_e}
      2^{2(k_i\wedge k_e-k_i)}
      \\&\qquad\qquad
      e^{-C2^{-k_i}\sum_{c'\neq c}|m_{M,\hc,c'}|}
      e^{-C2^{-k_e}\sum_{c\in[d]}|m_{M,e}|}  \\
      &\quad\lesssim \sum_{k_i,k_e\geqslant-1}\prod_{l\in \{l_{i}\}\sqcup\{l_{e,1}\}}2^{-2\ell(l)k_l}\prod_{s\in S(G)}\prod_{l\in \{l_{i}\}\sqcup\{l_{e,1}\}|s\in l}e^{-C2^{-k_l}|m_s|}2^{-2((k_i-k_e)\vee 0)}\\
      &\quad\lesssim \sum_{k_i,k_e\geqslant-1}\prod_{l\in \{l_{i}\}\sqcup\{l_{e,1}\}}2^{-2\ell(l)k_l}\prod_{s\in S(G)}\prod_{l\in  \{l_{i}\}\sqcup\{l_{e,1}\}|s\in l}e^{-C2^{-k_l}|m_s|}\prod_{i=k_e+1}^{k_i}2^{-2}
      \,,
\end{equs}
where we introduced $k_{l_i(M)}=k_i$ for the internal link of $M$ and $k_{l_{e,1}(M)}=k_e$ for the external links of $M$ $l_{e,1}$. Finally, we note that $\prod_{i=k_e+1}^{k_i}2^{-2}=\prod_{i|\exists k, M=\cG^i_k}2^{-2}$. 

The proof for the graphs $M\in\mfM^{2}(G)$ is very similar so we only sketch it. 
Indeed, for $M$, one has to introduce three scales $k_e$, $k_s$ and $k_d$ that correspond to the slicing of the Fourier modes $m_{M_2,e}$, $m_{M_2,s}$ and $m_{M_2,d}$. One then has to consider separately all the six possible orderings of these scales, in the same way we did for $M\in\Tilde\mfM^1(G)$. This yields the following bound over the amplitude of $M$:
\begin{equs}
    \sum_{k_i,k_s,k_e\geqslant-1}
    \prod_{l\in L^{int}(M)\sqcup\{l_{e,1}\}}2^{-2\ell(l)k_l}
   \prod_{s\in S(G)}\prod_{l\in L^{int}(M)\sqcup\{l_{e,1}\}|s\in l}e^{-C2^{-k_l}|m_s|}\prod_{i=k_e+1}^{k_s\wedge k_d}2^{-2}\prod_{i=k_s+1}^{ k_d}2^{-2}
      \,,
\end{equs}
and we also note that $\prod_{i=k_e+1}^{k_s\wedge k_d}2^{-2}=\prod_{i|\exists k,M=\cG^i_k}2^{-2}$ and $\prod_{i=k_s+1}^{ k_d}2^{-2}=\prod_{i|\exists k,M_1(M)=\cG^i_k}2^{-2}$.

All the other links that are not in $\bigcup_{M\in\Tilde{\mfM}(G)}L^{int}(M)\cup\{l_{e,1}(M)\}$ are easily dealt with using \eqref{eq:slacing1}, and for $l$ such a link we have
\begin{equs}
\angp{m_l}{-2\ell(l)}&\lesssim\sum_{k_l\geqslant-1}2^{-2\ell(l)k_l}e^{-C2^{-k_l}\sum_{c\in[d]}|m_{l,c}|}\lesssim\sum_{k_l\geqslant-1}2^{-2\ell(l)k_l}\prod_{s\in S(G)|s\in l}e^{-C2^{-k_l}|m_s|}\,,
\end{equs}
where we say that $s\in l$ if $s$ runs through $l$. Therefore, collecting all the previous bounds together, we have obtained that 
\begin{equs}
    \tA(G)&\lesssim \sum_{m_{s_1},\dots,m_{s_{|S(G)|}}\in\Z}\sum_{k_1,\dots,k_{|L(G)|}\geqslant-1}\prod_{L\in L(G)}2^{-2\ell(l)k_l}\\    
    &\quad\prod_{s\in S(G)}\prod_{l\in L(G)|s\in l}e^{-C2^{-k_l}|m_s|}
    \prod_{M\in\mfM(G)}\prod_{i|\exists k,M=\cG^i_k}2^{-2}\,.
\end{equs}
We first rewrite 
\begin{equs}
    \prod_{M\in\mfM(G)}\prod_{i|\exists k,M=\cG^i_k}2^{-2}=\prod_{i\geqslant-1}\prod_{k=1}^{C(\cG^i)}2^{-2\1\{\cG^i_k=\mfM^1 \text{ or }\mfM^2\}}\,.
\end{equs}
Then, 
\begin{equs}
    \sum_{m_s\in\Z}\prod_{l\in L(G)|s\in l}&e^{-C2^{-k_l}|m_s|}\lesssim\sum_{p\in\N}e^{-C\sum_{l\in L(G)|s\in l}2^{-k_l}p}    \lesssim \frac{1}{\sum_{l\in L(G)|s\in l}2^{-k_l}}\lesssim 2^{k_s}\,,
\end{equs}
where $k_s\eqdef\min_{l\in L(G)|s\in l}k_l$. Moreover, 
\begin{equs}
    \prod_{s\in S(G)}2^{k_s}\lesssim\prod_{i\geqslant-1}\prod_{s\in S(G)|k_s\geqslant i}2^{1}\,,
\end{equs}
and since the condition $k_s\geqslant i$ implies that $s$ is an internal strand of some $\cG^i_k$, we finally have
\begin{equs}
    \prod_{s\in S(G)}2^{k_s}\lesssim\prod_{i\geqslant-1}\prod_{k=1}^{C(\cG^i)}2^{|S^{int}(\cG^i_k)|}\,.
\end{equs}
Let us finally rewrite, observing that the condition $k_l\geqslant i$ is equivalent to $l\in L^{int}(\cG^i_k)$ for some $k$,
\begin{equs}
    \prod_{l\in L(G)}2^{-2\ell(l)k_l}&=\prod_{l\in L(G)}\prod_{i=1}^{k_l}2^{-2\ell(l)}\lesssim\prod_{i\geqslant-1}\prod_{l\in L(G)|k_l\geqslant i}2^{-2\ell(l)}\lesssim\prod_{i\geqslant-1}\prod_{k=1}^{C(\cG^i)}2^{-2\sum_{l\in L^{int}(\cG^i_k)}\ell(l)}
\end{equs}
which proves the statement. 
\end{proof}

We immediately have the following corollary.
\begin{corollary}\label{coro:multiscale}
    If $G\in\G_\tau$, then $\tA(G)$ is finite if for all $i\geqslant-1$ and $k\in[C(\cG^i)]$, we have
    \begin{equs}   \max_{\cG^i_k|l_\alpha\in\cG^i_k\cap l_\beta\notin\cG^i_k}     \omega(\cG^i_k)(\alpha)<0\,,
   \\   \max_{\cG^i_k|l_\alpha\notin\cG^i_k\cap l_\beta\in\cG^i_k}     \omega(\cG^i_k)(\beta)<0\,,
   \\   \max_{\cG^i_k|l_\alpha\in\cG^i_k\cap l_\beta\in\cG^i_k}     \omega(\cG^i_k)(\alpha,\beta)<0\,.
    \end{equs} 
\end{corollary}
The proof of the regularities of the different stochastic objects reduces to identifying the worst of the possible subgraphs (including the full graphs) of all the graphs in $\G_\tau$ containing $l_\alpha$ and/or $l_\beta$. For the smallest objects, we exhibit all the possible contraction.
For bigger objects, we will describe the most divergent subgraph, and show that it is indeed maximal. All the worst contribution have a connected boundary graph if they are not a vacuum graph, so that the superficial degree of divergence of a subgraph $G$ is fully characterized by the triplet $\big(|V(G)|,|L^{ext}(G)|,\delta(G)\big)$. We therefore often use this quantity to refer to the graph, in order to lighten the notations. Moreover, we call $n$-point graph a graph $G$ such that $|L^{ext}(G)|=n$, and vacuum graph a zero-point graph. For example, we will call a melonic two-vertex four-point graph a graph $(2,4,0)$.

\subsection{Moment estimates for 
\TitleEquation{\Xtwom}{Xtwom}, 
\enskip \TitleEquation{\Xtwonm}{Xtwonm},
\enskip \TitleEquation{\Xthreeloc}{Xthreeloc},
\enskip \TitleEquation{\XdotX}{XdotX},
\enskip \TitleEquation{\XdotX}{XdotX},
\enskip \TitleEquation{\XtwoPictwo}{XtwoPictwo} }\label{sec:proofs}
In this section, we complete the diagrammatic estimates for the listed stochastic objects. 

As described in Lemma~\ref{lem:amplitude_bound}, BG objects obey the same power-counting, and using a Kolmogorov estimate for space regularity gives the necessary estimates on the symbols for Lemma~\ref{lem:BGreg}. 

The last estimate from Lemma~\ref{prop:stochastic_1} is proven in Section~\ref{subsubsec:regX3}, while the rest of the estimates follow from Lemma~\ref{lem:randop1} which is proven in Sections~\ref{subsubsec:regX2m}, \ref{subsubsec:regX2nm} and \ref{subsubsec:regQuadratic}. 
Lemma~\ref{lem:randfi2} is proven in Section~\ref{subsubsec:regQuintic}, Lemma~\ref{lem:randfi4} in Section~\ref{subsubsec:S} and Lemma~\ref{lem:randop2} in \ref{subsubsec:regQuartic}.

In Section~\ref{subsubsec:regRoughAnsatz} we argue that the relevant estimates in Lemmas~\ref{lem:randop2} and \ref{lem:randop_4} hold for the random operators defined by replacing instances of $\X$ with the random shift $\bX$. 
As described in Lemma~\ref{lem:amplitude_bound}, BG objects obey the same power-counting, and using a Kolmogorov estimate for space regularity gives Lemma~\ref{lem:BGreg}.

\subsubsection{The melonic second Wick power \TitleEquation{\Xtwom}{Xtwom}}\label{subsubsec:regX2m}
The renormalization of $\Xtwom$ has been described in Lemma~\ref{lem:CT1}, but now we turn to quantifying its regularity \dash to do this we study all the possible subgraphs of all the graphs that can be made of its skeleton graph.
\begin{figure}[H]
    \centering 
\tikzsetnextfilename{fig20}  
\begin{tikzpicture}
\draw[gray] (0,0)--(0,1);
\draw[gray] (1,0)--(1,1);
\draw[gray] (0,1)..controls(.3,1.2)and(.7,1.2)..(1,1);
\draw[gray] (0,1)..controls(.3,.8)and(.7,.8)..(1,1);
\draw[gray] (0,0)..controls(.3,.2)and(.7,.2)..(1,0);
\draw[gray] (0,0)..controls(.3,-.2)and(.7,-.2)..(1,0);
\draw[black,thick,densely dotted] (0,0)--(-.3,-.3) node[black,circle, fill=black, draw, solid, inner sep=0pt, minimum size=2pt] {};
\draw[black,thick,densely dotted] (0,1)--(-.3,1.3) node[black,circle, fill, draw, solid, inner sep=0pt, minimum size=2pt] {};
\draw[black,thick,densely dotted] (1,1)--(1.3,1.3) node[black,circle, fill, draw, solid, inner sep=0pt, minimum size=2pt] {};
\draw[black,thick,densely dotted] (1,0)--(1.3,-.3) node[black,circle, fill, draw, solid, inner sep=0pt, minimum size=2pt] {};
\node[black] at (-.2,.5){$l_\alpha$};
\node[black] at (1.2,.5){$l_\beta$};
\end{tikzpicture}
 \caption{Skeleton graph of the melonic second Wick power}
\end{figure}
This skeleton graph can give rise to the three following contractions:
\begin{figure}[H]
    \centering  
    \tikzsetnextfilename{fig21} 
   \begin{tikzpicture}
\draw[gray] (0,0)--(0,1);
\draw[gray] (1,0)--(1,1);
\draw[gray] (0,1)..controls(.3,1.2)and(.7,1.2)..(1,1);
\draw[gray] (0,1)..controls(.3,.8)and(.7,.8)..(1,1);
\draw[gray] (0,0)..controls(.3,.2)and(.7,.2)..(1,0);
\draw[gray] (0,0)..controls(.3,-.2)and(.7,-.2)..(1,0);
\draw[black,thick,densely dotted] (0,1)..controls(0.05,1.7)and(.95,1.7)..(1,1);
\draw[black,thick,densely dotted] (0,0)..controls(0.05,-.7)and(.95,-.7)..(1,0);
\node[black] at (-.2,.5){$l_\alpha$};
\node[black] at (1.2,.5){$l_\beta$};
\draw[gray] (2,0)--(2,1);
\draw[gray] (3,0)--(3,1);
\draw[gray] (2,1)..controls(2.3,1.2)and(2.7,1.2)..(3,1);
\draw[gray] (2,1)..controls(2.3,.8)and(2.7,.8)..(3,1);
\draw[gray] (2,0)..controls(2.3,.2)and(2.7,.2)..(3,0);
\draw[gray] (2,0)..controls(2.3,-.2)and(2.7,-.2)..(3,0);
\draw[black,thick,densely dotted] (2,0)..controls(1.3,0.05)and(1.3,.95)..(2,1);
\draw[black,thick,densely dotted] (3,0)..controls(3.7,0.05)and(3.7,.95)..(3,1);
\node[black] at (1.8,0.5){$l_\alpha$};
\node[black] at (3.2,0.5){$l_\beta$};
\draw[gray] (4,0)--(4,1);
\draw[gray] (5,0)--(5,1);
\draw[gray] (4,1)..controls(4.3,1.2)and(4.7,1.2)..(5,1);
\draw[gray] (4,1)..controls(4.3,.8)and(4.7,.8)..(5,1);
\draw[gray] (4,0)..controls(4.3,.2)and(4.7,.2)..(5,0);
\draw[gray] (4,0)..controls(4.3,-.2)and(4.7,-.2)..(5,0);
\draw[black,thick,densely dotted] (4,0)--(5,1);
\draw[black,thick,densely dotted] (5,0)--(4,1);
\node[black] at (3.8,.5){$l_\alpha$};
\node[black] at (5.2,.5){$l_\beta$};
\end{tikzpicture} 
 \caption{The three contractions: they have respectively degree 0, $d-2$ and $d-1$}\label{fig:Xtwom}
\end{figure}
This object is small enough for the bound \eqref{eq:A(G)} to be estimated without multiscale analysis, we write the amplitudes of the first two graphs in Figure~\ref{fig:Xtwom}.

The melonic contraction $(1,0,0)$ is thus the maximal subgraph that contains both $l_\alpha$ and $l_\beta$, and its amplitude \eqref{eq:A(G)} rewrites as
\begin{equs}
    \sum_{m\in\Z,a\in\Z^{d-1},b\in\Z^{d-1}}\angp{m}{-2\alpha+2\beta}\frac{\angp{m}{2}\wedge\angp{(a,m)}{2}}{\angp{(a,m)}{2}\angp{a}{2}}\frac{\angp{m}{2}\wedge\angp{(b,m)}{2}}{\angp{(b,m)}{2}\angp{b}{2}}\,.
\end{equs}
We see that the sum is convergent if and only if
\begin{equs}
    \beta<\alpha+|\Xtwom|\,.
\end{equs}
The only contraction with a subgraph containing only $l_\alpha/l_\beta$ is the graph  $(1,0,d-2)$, which has a subgraph $(1,2,d-2)$ that contains $l_\alpha/l_\beta$. Its amplitude is bounded as 
\begin{equs}
    \sum_{m\in\Z,n\in\Z,a\in\Z^{d-1}}\angp{(a,m)}{-2}\angp{(a,n)}{-2}\angp{m}{-2\alpha}\angp{n}{2\beta}\,.
\end{equs}
For this sum to be convergent, the conditions $2\beta-2\alpha+d-3<0$,
\begin{equs}
    \beta<1/2,\text{ and }\alpha>-1/2
\end{equs}
have to be fulfilled. Indeed, it may be that the sums over $n$ and $a$ are convergent, even without the factor $\angp{(a,m)}{-2}$, so that we would be left with $\sum_{m\in\Z}\angp{m}{-2-2\alpha}$, which indeed requires $\alpha<-\frac{1}{2}$.
\subsubsection{The non-melonic second Wick power \TitleEquation{\Xtwonm}{Xtwonm}}\label{subsubsec:regX2nm}
For the rest of our estimates we leverage the multiscale bound of Corollary~\ref{coro:multiscale}. 
We therefore identify all the possible graphs made with the skeleton graph, along with all their possible subgraphs, and compute their divergence degree using \eqref{eq:divergencedegree}.

This object doesn't involve any renormalization and we turn to its regularity properties.
\begin{figure}[H]
    \centering  
\tikzsetnextfilename{fig22}     
\begin{tikzpicture}
\draw[gray] (0,0)--(0,1);
\draw[gray] (1,0)--(1,1);
\draw[gray] (0,1)..controls(.3,1.2)and(.7,1.2)..(1,1);
\draw[gray] (0,1)..controls(.3,.8)and(.7,.8)..(1,1);
\draw[gray] (0,0)..controls(.3,.2)and(.7,.2)..(1,0);
\draw[gray] (0,0)..controls(.3,-.2)and(.7,-.2)..(1,0);
\draw[black,thick,densely dotted] (0,0)--(-.3,-.3) node[black,circle, fill=black, draw, solid, inner sep=0pt, minimum size=2pt] {};
\draw[black,thick,densely dotted] (0,1)--(-.3,1.3) node[black,circle, fill, draw, solid, inner sep=0pt, minimum size=2pt] {};
\draw[black,thick,densely dotted] (1,1)--(1.3,1.3) node[black,circle, fill, draw, solid, inner sep=0pt, minimum size=2pt] {};
\draw[black,thick,densely dotted] (1,0)--(1.3,-.3) node[black,circle, fill, draw, solid, inner sep=0pt, minimum size=2pt] {};
\node[black] at (.5,1.35){$l_\alpha$};
\node[black] at (.5,-.35){$l_\beta$};
\end{tikzpicture}
 \caption{Skeleton graph of the non-melonic second Wick power}
\end{figure}
As its melonic counterpart, this symbol gives rise to the three one-vertex vacuum graphs.
\begin{figure}[H]
    \centering
    \tikzsetnextfilename{fig23}
    \begin{tikzpicture}
\draw[gray] (0,0)--(0,1);
\draw[gray] (1,0)--(1,1);
\draw[gray] (0,1)..controls(.3,1.2)and(.7,1.2)..(1,1);
\draw[gray] (0,1)..controls(.3,.8)and(.7,.8)..(1,1);
\draw[gray] (0,0)..controls(.3,.2)and(.7,.2)..(1,0);
\draw[gray] (0,0)..controls(.3,-.2)and(.7,-.2)..(1,0);
\draw[black,thick,densely dotted] (0,1)..controls(0.05,1.7)and(.95,1.7)..(1,1);
\draw[black,thick,densely dotted] (0,0)..controls(0.05,-.7)and(.95,-.7)..(1,0);
\node[black] at (2.5,1.35){$l_\alpha$};
\node[black] at (2.5,-.35){$l_\beta$};
\draw[gray] (2,0)--(2,1);
\draw[gray] (3,0)--(3,1);
\draw[gray] (2,1)..controls(2.3,1.2)and(2.7,1.2)..(3,1);
\draw[gray] (2,1)..controls(2.3,.8)and(2.7,.8)..(3,1);
\draw[gray] (2,0)..controls(2.3,.2)and(2.7,.2)..(3,0);
\draw[gray] (2,0)..controls(2.3,-.2)and(2.7,-.2)..(3,0);
\draw[black,thick,densely dotted] (2,0)..controls(1.3,0.05)and(1.3,.95)..(2,1);
\draw[black,thick,densely dotted] (3,0)..controls(3.7,0.05)and(3.7,.95)..(3,1);
\node[black] at (.5,1.35){$l_\alpha$};
\node[black] at (.5,-.35){$l_\beta$};
\draw[gray] (4,0)--(4,1);
\draw[gray] (5,0)--(5,1);
\draw[gray] (4,1)..controls(4.3,1.2)and(4.7,1.2)..(5,1);
\draw[gray] (4,1)..controls(4.3,.8)and(4.7,.8)..(5,1);
\draw[gray] (4,0)..controls(4.3,.2)and(4.7,.2)..(5,0);
\draw[gray] (4,0)..controls(4.3,-.2)and(4.7,-.2)..(5,0);
\draw[black,thick,densely dotted] (4,0)--(5,1);
\draw[black,thick,densely dotted] (5,0)--(4,1);
\node[black] at (4.5,1.35){$l_\alpha$};
\node[black] at (4.5,-.35){$l_\beta$};
\end{tikzpicture} 
 \caption{The three contractions}
\end{figure}
Once again, the melonic contraction $(1,0,0)$ is the maximal subgraph that contains both $l_\beta$ and $l_\alpha$ \dash this yields the constraint
 \begin{equs}
    \beta<\alpha+|\Xtwonm|\,.
\end{equs}
It turns out that it also has subgraph $(1,2,0)$ that only contains $l_\alpha/l_\beta$ (observe that here, if there is a subgraph that seems to be a melonic tadpole, it is actually not the case, because it containes $l_\alpha/l_\beta$, so that it does not require renormalization, and does not benefit from the +2 effect in the power counting coming from renormalization) and yields the constraints
\begin{equs}
    \beta<-\frac{d-3}{2} \text{ and }\alpha>\frac{d-3}{2}\,.
\end{equs}
Note that this is the only subgraph of the three graphs that contains $l_\alpha/l_\beta$.
\subsubsection{The cubic random field  \TitleEquation{\Xthreeloc}{Xthreeloc}}\label{subsubsec:regX3}
The renormalization of $\Xthreeloc$ is given by a``Wick'' renormalization coming from the rightmost two factors of $\X$ in this product, corresponding to the graph in $\mfM^{1}$ and the counterterm $\mfC^1$.
Below we draw the corresponding skeleton graph. 
\begin{figure}[H]
    \centering 
\tikzsetnextfilename{fig24}    
\begin{tikzpicture}
\draw[gray] (0,0)--(0,1);
\draw[gray] (1,0)--(1,1);
\draw[gray] (0,1)..controls(.3,1.2)and(.7,1.2)..(1,1);
\draw[gray] (0,1)..controls(.3,.8)and(.7,.8)..(1,1);
\draw[gray] (0,0)..controls(.3,.2)and(.7,.2)..(1,0);
\draw[gray] (0,0)..controls(.3,-.2)and(.7,-.2)..(1,0);
\draw[black,thick,densely dotted] (0,0)--(-.3,-.3) node[black,circle, fill=black, draw, solid, inner sep=0pt, minimum size=2pt] {};
\draw[black,thick,densely dotted] (0,1)--(-.3,1.3) node[black,circle, fill, draw, solid, inner sep=0pt, minimum size=2pt] {};
\draw[black,thick,densely dotted] (1,1)--(1.3,1.3) node[black,circle, fill, draw, solid, inner sep=0pt, minimum size=2pt] {};
\draw[gray] (2,0)--(2,1);
\draw[gray] (3,0)--(3,1);
\draw[gray] (2,1)..controls(2.3,1.2)and(2.7,1.2)..(3,1);
\draw[gray] (2,1)..controls(2.3,.8)and(2.7,.8)..(3,1);
\draw[gray] (2,0)..controls(2.3,.2)and(2.7,.2)..(3,0);
\draw[gray] (2,0)..controls(2.3,-.2)and(2.7,-.2)..(3,0);
\draw[black,thick,densely dotted] (1,0)..controls(1.3,-.3)and(1.7,-.3)..(2,0) ;
\draw[black,thick,densely dotted] (2,1)--(1.7,1.3) node[black,circle, fill, draw, solid, inner sep=0pt, minimum size=2pt] {};
\draw[black,thick,densely dotted] (3,1)--(3.3,1.3) node[black,circle, fill, draw, solid, inner sep=0pt, minimum size=2pt] {};
\draw[black,thick,densely dotted] (3,0)--(3.3,-.3) node[black,circle, fill, draw, solid, inner sep=0pt, minimum size=2pt] {};
\node[black] at (1.5,-.5){$l_\beta$};
\end{tikzpicture}
\caption{Skeleton graph of the third Wick power}
\end{figure}
The contractions are all the connected two-vertex vacuum graphs, and the melonic one $(2,0,0)$ is therefore maximal \dash this yields the condition \begin{equs}
     \beta<|\Xthreeloc|\,.
\end{equs} 
Since $l_\beta$ connects two different vertices, it cannot belong to a $(1,2,0)$ subgraph, and $(2,2,0)$ graphs have a better power counting than $(2,0,0)$ graphs, which is why we don't have to show all the possible contractions. 
\subsubsection{The quadratic random operator  \TitleEquation{\XdotX}{XdotX}}\label{subsubsec:regQuadratic}
This stochastic estimate does not involve any renormalization.
\begin{figure}[H]
    \centering  
\tikzsetnextfilename{fig25} 
\begin{tikzpicture}
\draw[gray] (0,0)--(0,1);
\draw[gray] (1,0)--(1,1);
\draw[gray] (0,1)..controls(.3,1.2)and(.7,1.2)..(1,1);
\draw[gray] (0,1)..controls(.3,.8)and(.7,.8)..(1,1);
\draw[gray] (0,0)..controls(.3,.2)and(.7,.2)..(1,0);
\draw[gray] (0,0)..controls(.3,-.2)and(.7,-.2)..(1,0);
\draw[black,thick,densely dotted] (0,0)--(-.3,-.3) node[black,circle, fill=black, draw, solid, inner sep=0pt, minimum size=2pt] {};
\draw[black,thick,densely dotted] (0,1)..controls(.5,2)and(2.5,2)..(3,1);
\draw[black,thick,densely dotted] (1,1)--(1.3,1.3) node[black,circle, fill, draw, solid, inner sep=0pt, minimum size=2pt] {};
\draw[gray] (2,0)--(2,1);
\draw[gray] (3,0)--(3,1);
\draw[gray] (2,1)..controls(2.3,1.2)and(2.7,1.2)..(3,1);
\draw[gray] (2,1)..controls(2.3,.8)and(2.7,.8)..(3,1);
\draw[gray] (2,0)..controls(2.3,.2)and(2.7,.2)..(3,0);
\draw[gray] (2,0)..controls(2.3,-.2)and(2.7,-.2)..(3,0);
\draw[black,thick,densely dotted] (1,0)..controls(1.3,-.3)and(1.7,-.3)..(2,0) ;
\draw[black,thick,densely dotted] (2,1)--(1.7,1.3) node[black,circle, fill, draw, solid, inner sep=0pt, minimum size=2pt] {};
\draw[black,thick,densely dotted] (3,0)--(3.3,-.3) node[black,circle, fill, draw, solid, inner sep=0pt, minimum size=2pt] {};
\node[black] at (1.5,-.5){$l_\beta$};
\node[black] at (1.5,2){$l_\alpha$};
\end{tikzpicture}   
 \caption{Skeleton graph of the quadratic random operator}
\end{figure}
We deduce that there are three different contractions with this random operator, but two of them are trivial and have only $d$ faces, and thus no non-trivial subgraphs. The third one is $(2,0,d-2)$ which is of maximal degree, and is the maximal graph containing both $l_\alpha$ and $l_\beta$ since no melonic subgraph can contain both $l_\alpha$ and $l_\beta$ \dash this gives the condition 
\begin{equs}
    \beta<\alpha+|\XdotX|\,.
\end{equs}
Moreover, this graph has a $(2,4,0)$ subgraph that contains only $l_\alpha/l_\beta$ and is therefore maximal (since the skeleton graph implies that $l_\alpha/l_\beta$ cannot belong to any $(2,2,0)$ subgraph) \dash this gives constraints
\begin{equs}
     \beta<-\frac{d-3}{2} \text{ and }\alpha>\frac{d-3}{2}\,.
\end{equs}
\subsubsection{The quartic random operator \TitleEquation{\Sym}{Sym}}\label{subsubsec:regQuartic}
This random operator does not involve any renormalization, below we draw the corresponding skeleton graph. 
\begin{figure}[H]
    \centering    
\tikzsetnextfilename{fig26}
\begin{tikzpicture}
\draw[gray] (0,0)--(0,1);
\draw[gray] (1,0)--(1,1);
\draw[gray] (0,1)..controls(.3,1.2)and(.7,1.2)..(1,1);
\draw[gray] (0,1)..controls(.3,.8)and(.7,.8)..(1,1);
\draw[gray] (0,0)..controls(.3,.2)and(.7,.2)..(1,0);
\draw[gray] (0,0)..controls(.3,-.2)and(.7,-.2)..(1,0);
\draw[black,thick,densely dotted] (0,0)--(-.3,-.3) node[black,circle, fill=black, draw, solid, inner sep=0pt, minimum size=2pt] {};
\draw[gray] (2,0)--(2,1);
\draw[gray] (3,0)--(3,1);
\draw[gray] (2,1)..controls(2.3,1.2)and(2.7,1.2)..(3,1);
\draw[gray] (2,1)..controls(2.3,.8)and(2.7,.8)..(3,1);
\draw[gray] (2,0)..controls(2.3,.2)and(2.7,.2)..(3,0);
\draw[gray] (2,0)..controls(2.3,-.2)and(2.7,-.2)..(3,0);
\draw[black,thick,densely dotted] (1,0)..controls(1.3,-.3)and(1.7,-.3)..(2,0) ;
\draw[black,thick,densely dotted] (2,1)--(1.7,1.3) node[black,circle, fill, draw, solid, inner sep=0pt, minimum size=2pt] {};
\draw[black,thick,densely dotted] (2,2)--(1.7,1.7) node[black,circle, fill, draw, solid, inner sep=0pt, minimum size=2pt] {};
\draw[black,thick,densely dotted] (3,0)--(3.3,-.3) node[black,circle, fill, draw, solid, inner sep=0pt, minimum size=2pt] {};
\node[black] at (1.5,-.5){$l_\beta$};
\node[black] at (1.5,3.5){$l_\alpha$};
\draw[gray] (2,2)--(2,3);
\draw[gray] (3,2)--(3,3);
\draw[gray] (2,3)..controls(2.3,3.2)and(2.7,3.2)..(3,3);
\draw[gray] (2,3)..controls(2.3,2.8)and(2.7,2.8)..(3,3);
\draw[gray] (2,2)..controls(2.3,2.2)and(2.7,2.2)..(3,2);
\draw[gray] (2,2)..controls(2.3,1.8)and(2.7,1.8)..(3,2);
\draw[gray] (0,2)--(0,3);
\draw[gray] (1,2)--(1,3);
\draw[gray] (0,3)..controls(.3,3.2)and(.7,3.2)..(1,3);
\draw[gray] (0,3)..controls(.3,2.8)and(.7,2.8)..(1,3);
\draw[gray] (0,2)..controls(.3,2.2)and(.7,2.2)..(1,2);
\draw[gray] (0,2)..controls(.3,1.8)and(.7,1.8)..(1,2);
\draw[black,thick,densely dotted] (0,1)..controls(-.3,1.3)and(-.3,1.7)..(0,2) ;
\draw[black,thick,densely dotted] (3,1)..controls(3.3,1.3)and(3.3,1.7)..(3,2) ;
\draw[black,thick,densely dotted] (1,3)..controls(1.3,3.3)and(1.7,3.3)..(2,3) ;
\draw[black,thick,densely dotted] (1,1)--(1.3,1.3) node[black,circle, fill, draw, solid, inner sep=0pt, minimum size=2pt] {};
\draw[black,thick,densely dotted] (1,2)--(1.3,1.7) node[black,circle, fill, draw, solid, inner sep=0pt, minimum size=2pt] {};
\draw[black,thick,densely dotted] (0,3)--(-.3,3.3) node[black,circle, fill=black, draw, solid, inner sep=0pt, minimum size=2pt] {};
\draw[black,thick,densely dotted] (3,3)--(3.3,3.3) node[black,circle, fill=black, draw, solid, inner sep=0pt, minimum size=2pt] {};
\end{tikzpicture}   
\caption{Skeleton graph of the quartic random operator}
\end{figure}
The vacuum contraction of highest degree ($d-2$) itself yields the following constraint on $\alpha$ and $\beta$:
\begin{equs}
    \beta<\alpha+|\Sym|-2(1-\eta)\,.
\end{equs}
It is maximal since inspecting the skeleton graph shows that any melonic subgraph $G$ containing both $l_\alpha$ and $l_\beta$ would a least be $(2,4,0)$, so that their power counting would be better, since improving $|L^{ext}(G)|=$ by 2 yields a factor $d-3$ and improving $C(\partial G)$ by one yields a factor $1$, which at least compensates the gain of $d-2$ due to the improvement of the degree.

This contraction has the same $(2,4,0)$ subgraph as $\XdotX$ containing either $l_\alpha$ or $l_\beta$, and this subgraph is still maximal since $l_\beta$ cannot belong to any $(2,2,0)$ subgraph. Again, this gives the conditions
\begin{equs}
     \beta<-\frac{d-3}{2} \text{ and }\alpha>\frac{d-3}{2}\,.
\end{equs}

\subsection{Estimates on \TitleEquation{\XtwoPictwo}{XtwoPictwo}, \enskip \TitleEquation{\cS}{cS}, \enskip and objects with the rough shift $\bX$}\label{sec:proofs2}
We now discuss obtaining estimates on $\XtwoPictwo$, $\cS$ and also estimates on random operators built using the rough shift.
For these objects we use a combination of stochastic and deterministic estimates. 
In particular, in Section~\ref{subsubsec:regRoughAnsatz} we argue that the relevant estimates in Lemmas~\ref{lem:randop2} and \ref{lem:randop_4} hold for the random operators defined by replacing instances of $\X$ with the random shift $\bX$.

\subsubsection{The quintic random field \TitleEquation{\XtwoPictwo}{XtwoPictwo}}\label{subsubsec:regQuintic} 
Recall the product denoted by the root in symbol $\XtwoPictwo$ actually corresponds to three terms due to the three ways to order this product -- see Definition~\ref{def:randomfields}.

The product $\cN(\Pictwo, \X,\X)$ requires a ``Wick'' renormalization coming from the two factors of $\X$ in this product to the graph in $\mfM^{1}$ and counter-term $\mfC^{1}$. 

The other two products $\cN(\X,\Pictwo,\X)$ and $\cN(\X,\X,\Pictwo,\X)$ contain a melonic pairing of $\X$ and $\Pictwo$ that is divergent in $d=4$ and corresponds to graphs in $\mfM^2$ and the counterterm $\frac{1}{2}\mfC^2$ (recall Lemma~\ref{def:CT2}). 

Some care has to be taken here since the regularity of $\XtwoPictwo$ is not determined by its power counting $|\XtwoPictwo|=-\frac{5d-18}{2}$. The skeleton graphs coming from all the three products in $\XtwoPictwo$ have ten leads, and thus gives rise to 945 graphs. However, we do not need to draw them because some of them are melonic \dash $(4,0,0)$ \dash this gives the condition
\begin{equs}
    \beta<|\XtwoPictwo|\,.
\end{equs}
It turns out that these melonic vacuum graphs are maximal only in $d=4$. Indeed, $(2,2,0)$ and $(2,4,0)$ subgraphs have a worst power counting, and some of these melonic vacuum graphs have $(2,2,0)$ and $(2,4,0)$ subgraphs containing $l_\beta$. This yields the conditions $\beta<-\frac{(d-3)}{2}$ and $\beta<-(d-3)$, which becomes better behaved only when $d \geqslant 4$ but dominates when $d \in \{2,3\}$ (however in the latter case we don't need the estimate $\XtwoPictwo$).

\subsubsection{The septic random field \TitleEquation{\cS}{cS}}\label{subsubsec:S}
This symbol requires care, one reason being that its regularity is not given by our usual power counting $|\cS|=-\frac{7d-28}{2}$ just as was the case for $\XtwoPictwo$.

However, a second issue here is that we won't obtain the needed regularity estimate on $\CS$ solely through a stochastic estimate. In particular, we split $\cS$ as $\cS=\cS^{\det}+\cS^{\sto}$ with 
\begin{equs}
    \cS^{\det}&\eqdef \cN(\Picthree,\X,\X)-\mfC^1\Picthree\,,\\\cS^{\sto}&\eqdef \big(\cN(\Pictwo,\X,\Pictwo)-\frac12\mfC^2\Pictwo\big)+\cN(\X,\X,\Picthree)+20 \text{ terms}\,.
\end{equs}
Here $\cS^{\sto}$ will be treated as a whole using a stochastic estimate (similarly to the previous stochastic objects).
However, trying to use a stochastic estimate to control $\cS^{\det}$ would yield a regularity estimate with exponent $-(d-3)$ coming from a subgraph $(2,2,0)$, which is worst than expected and not enough to close our argument when $d=4$. 
We instead write $\cS^{\det}$ in terms of smaller random operators and combine their estimates, writing 
\[
\cS^{\det}(x)=\sum_{c=1}^d\Xtwom^c \big(\Picthree(\cdot,x_{\hc}) \big)(x_c)\;.
\]
By \eqref{eq:bilocbesov-} and \eqref{eq:bilocbesov+}, we have
\begin{equs}
    \Vert \Xtwom^c(\Picthree)\Vert_{C_T\cC^{-\frac{d-3}{2}-3\epsilon}}&\lesssim
    \Vert \Xtwom^c(\Picthree)\Vert_{C_TL^\infty_{x_{\hc}}\cC_{x_c}^{-\frac{d-3}{2}-2\epsilon}}\\&\lesssim\Vert\Xtwom^c\Vert_{\cL(C_T\cC^{-(d-5)-2\epsilon},C_T\cC^{-\frac{d-3}{2}-2\epsilon})}\Vert\Picthree\Vert_{C_TL^\infty_{x_{\hc}}\cC_{x_c}^{-(d-5)-2\epsilon}}\\   &\lesssim\Vert\Xtwom^c\Vert_{\cL(C_T\cC^{-(d-5)-2\epsilon},C_T\cC^{-\frac{d-3}{2}-2\epsilon})}\Vert\Picthree\Vert_{C_T\cC^{-(d-5)-\epsilon}}\,.
\end{equs}
The estimate \eqref{eq:regXtwoPictwo} then implies that $\Picthree$ is of regularity $-(d-5)>-\frac{1}{2}$. 
We combine this with the random operator estimate \ref{eq:regXtwom} that states that for $\alpha>-\frac{1}{2}$, $\Xtwom$ goes from $C_TH^\alpha$ to $C_T\cC^{\beta-\epsilon}$ with $\beta=\min\big(\frac{1}{2},\alpha-\frac{2d-5}{2}\big)\geqslant-\frac{d-3}{2}$ for $\alpha=-(d-5)$. Thus, we can conclude that the right hand side is finite, so that $\cS^{\det}$ is indeed of regularity $-\frac{d-3}{2}-$.

We now turn to showing that $\cS^{\sto}$ is of regularity $-\frac{d-3}{2}-$.
$\cS^{\sto}$ is the sum of 22 terms, but all of them have an important difference with $\cN(\Picthree,\X,\X)-\mfC^1\Picthree$ that we describe here. The skeleton graphs of all these 22 terms give rise to some maximal $(6,0,0)$ graph which would yield $\beta<|\cS|$. However, one has to take care the contractions do not contain some $(2,2,0)$ or $(2,4,0)$ subgraphs, that would yields a worse power counting. It turns out that some of these maximal $(6,0,0)$ graph do have a $(2,4,0)$ subgraph containing $l_\beta$, which thus yields the sharper condition\begin{equs}
    \beta<-\frac{d-3}{2}\,.
\end{equs}
However, this constraint is maximal since no contraction can produce $(2,2,0)$ subgraph with $l_\beta$ inside it, in view of the position of $l_\beta$. This is the key difference betwenn the 22 terms in $\mcS^\sto$ (like for instance $\cN(\X,\X,\Picthree)$) and $\cN(\Picthree,\X,\X)-\mfC^1\Picthree$, because the latter can give rise to a $(2,2,0)$ subgraph, and would thus be of regularity $-(d-3)$, which is why we had to deal with it deterministically.

\subsubsection{The random operators made of the rough shift $\bX$}\label{subsubsec:regRoughAnsatz}
In $d=4$, we are interested in the previous random operators defined with $\X$ substituted with $\bX$. 
Our regularity estimates for the more basic random fields implies that $\bX$ shares the regularity of $\X$. 
Moreover, going though the proofs of the regularities of the random operators with one or several occurrence of $\X$ substituted with $\Pictwo$ or $\Picthree$ always yields better results, because this will give rise to bigger graphs that have a better power counting, which is the key feature of superrenormalizability/subcriticality.

For the sake of completeness, we provide here a short argument that confirms that these objects share the regularity of their counterparts built on the free field. 
Indeed, like we did above for $\cS$, we can decompose a random operator $\tau\in\{\pXtwom,\pXtwonm,\pXdotX,\pSym\}$ made with $\bX$ as $\tau^{\sto}+\tau^{\det}$ where $\tau^{\sto}$ will be controlled with a stochastic estimate on the whole object, while $\tau^{\det}$ will be treated using deterministic estimates.

One key input for this deterministic estimate will be stochastic estimates on the smaller random operators 
$(\bX-\X)^{(k)}$ defined by setting 
\[
(\bX-\X)^{(k)}_{N}(f)\eqdef \int_{\T^{d-k}} 
(-\Pictwo_N+\Picthree_N)(\cdot,y) f(y)\rmd y
\quad
\text{ for }
\quad
f:\R_{\geqslant 0}\times\T^{d-k}\rightarrow\R\;. 
\]
Denoting $|\bX-\X|=-\frac{3d-12}{2}$, our key estimate is then
\begin{lemma}
For $d\in\{2,3,4\}$ and all $\alpha>\frac{3d-12-k}{2}$, 
\begin{equs}
\label{eq:bX-X_rand_op}  
\sup_{N\in\N}\E[\Vert(\bX-\X)^{(k)}_{N}&\Vert^p_{\cL(C_TH^{\alpha}(\T^{d-k}),C_T\cC^{\beta}(\T^k)}<\infty\,\\
\text{ where }&
\beta = \min\Big(-\frac{2d+k-12}{2},\alpha+|\bX-\X| \Big)-\epsilon\;.
\end{equs}
\end{lemma} 
\begin{proof}
The proof generalizes that for $\X^{(k)}$ -- see \eqref{eq:lollipopproof}. When performing the covariance computations for $\X^{(k)}$ and $(\bX-\X)^{(k)}$, one encounters a class of graphs similar to the graphs in $G_\tau$ for $\tau$ an object of type II, but with a root which is given by a link instead of a vertex. This link is attached to two other links at its two extremities, and its strands divide into $k$ strands corresponding to $l_\beta$ and $d-k$ strands corresponding to $l_\alpha$. In the case of $\tau=\X^{(k)}$, there is a unique $G\in\G_\tau$ which is the graph for which the root is attached to a single link connected to both its extremities. This graph has three subgraphs (including itself), hence the estimates \eqref{eq:lollipopproofBis}. 

In the case of $\tau=(\bX-\X)^{(k)}$ , the skeleton graphs obtained when gluing together two copies of $\Pictwo$, one of $\Pictwo$ with one of $\Picthree$, and two copies of $\Picthree$, can give rise to several graphs, all with several subgraphs. Again, we can read from the skeleton graphs the numbers of external links, the numbers of vertices, and the degrees of all the possible graphs in $\G_\tau$ and their subgraphs, and therefore estimate their contributions. The most divergent graph is the melonic two-vertex graph arising when gluing two copies of $\Pictwo$. It also contains the subgraphs dictating the additional conditions on $\alpha$ and $\beta$.

\end{proof}
For every $\tau\in\{ \pXtwonm,\pXdotX,\pSym \}$,
we take $\tau^{\sto}$ to be the term containing only $\X$'s (for instance $\pSym^{\sto}=\Sym$) that has been constructed in the previous sections, and $\tau^{\det}$ the sum of all the remaining cross terms that contain at least one $\bX-\X$. Regarding $\pXtwom$, we chose 
\begin{equs}\label{eq:stoObUse}    \pXtwom^{\sto}:f\mapsto\Xtwom(f)+\cN(f,\Pictwo,\X)+\cN(f,\X,\Pictwo)-{\mfC^2}f\;.
\end{equs}
We have the following lemma about the regularity of $\pXtwom^{\sto}$.
\begin{lemma}
    In $d=4$, for all $\alpha>-\frac{1}{2}$, it holds
\begin{equ}
\sup_{N\in\N}\E[\Vert  \pXtwom^{\sto}_N-\Xtwom_N\Vert^p_{\cL(C_TH^{\alpha}(\T),C_T\cC^{\min(\frac{1}{2},\alpha-\frac12)-\epsilon}(\T))}]<\infty\,.
\end{equ}
\end{lemma}
\begin{proof}
    We deal separately with $f\mapsto \cN(f,\Pictwo,\X)-\frac{\mfC^2}{2}f $ and $f\mapsto \cN(f,\X,\Pictwo) -\frac{\mfC^2}{2}f $.
   \begin{figure}[H]
    \centering   
\tikzsetnextfilename{fig27}
\begin{tikzpicture}
\draw[gray] (0,0)--(0,1);
\draw[gray] (1,0)--(1,1);
\draw[gray] (0,1)..controls(.3,1.2)and(.7,1.2)..(1,1);
\draw[gray] (0,1)..controls(.3,.8)and(.7,.8)..(1,1);
\draw[gray] (0,0)..controls(.3,.2)and(.7,.2)..(1,0);
\draw[gray] (0,0)..controls(.3,-.2)and(.7,-.2)..(1,0);
\draw[black,thick,densely dotted] (1,1)--(1.3,1.3) node[black,circle, fill, draw, solid, inner sep=0pt, minimum size=2pt] {};
\draw[black,thick,densely dotted] (1,0)--(1.3,-.3) node[black,circle, fill, draw, solid, inner sep=0pt, minimum size=2pt] {};
\draw[black](-.3,1.3)--(-.5,1.5)node[black,circle, fill, draw, solid, inner sep=0pt, minimum size=2pt] {};
\draw[black](-.3,1.3)--(-.35,1.53)node[black,circle, fill, draw, solid, inner sep=0pt, minimum size=2pt] {};
\draw[black](-.3,1.3)--(-.53,1.35)node[black,circle, fill, draw, solid, inner sep=0pt, minimum size=2pt] {};
\draw[black] (0,1)--(-.3,1.3) node[draw,black,diamond, fill=blue, solid, inner sep=0pt, minimum size=3pt] {};
\node[black] at (-.2,.5){$l_\alpha$};
\node[black] at (1.2,.5){$l_\beta$};
\draw[black](-.3,-.3)--(-.5,-.5)node[black,circle, fill, draw, solid, inner sep=0pt, minimum size=2pt] {};
\draw[black](-.3,-.3)--(-.35,-.53)node[black,circle, fill, draw, solid, inner sep=0pt, minimum size=2pt] {};
\draw[black](-.3,-.3)--(-.53,-.35)node[black,circle, fill, draw, solid, inner sep=0pt, minimum size=2pt] {};
\draw[black] (0,0)--(-.3,-.3) node[draw,black,diamond, fill=blue, solid, inner sep=0pt, minimum size=3pt] {};
\end{tikzpicture}
 \caption{Schematic representation of the skeleton graph of $f\mapsto \cN(f,\Pictwo,\X)-\frac{\mfC^2}{2}f$. The one of $f\mapsto \cN(f,\X,\Pictwo)-\frac{\mfC^2}{2}f$ is obtained by exchanging $l_\alpha$ and $l_\beta$. }
\end{figure} 
    
Both contributions generate all the vacuum three-vertex graphs, so that the most divergent possible graph or subgraph that contains both $l_\alpha$ and $l_\beta$ is the melonic graph $(3,0,0)$ that brings the constrain $\beta<\alpha-\frac{1}{2}$. 
The first contribution has a subgraph $(1,2,d-2)$ (obtained by gluing the two noises coming from the two terms $\X$) containing only $l_\beta$, and the second one the same subgraph containing only $l_\alpha$, which yields the additional constraints. 
    These subgraphs are maximal because in view of their positions, $l_\alpha$ can not belong to a melonic subgraph that would not also contain $l_\beta$. 
\end{proof}
We can now deal with $\tau^{\det}$, the terms constructed thanks to a deterministic analysis. The idea is to express all the contributions in $\tau^{\det}$ as some compositions of random operators, and to show that this composition exists. 
For example, the worst term in $\pXtwonm^{\det}$ is given by $f\mapsto\cN(\X,\bX-\X,f)$ which we thus reexpress as 
\begin{equs}
\cN(\X,\bX-\X,f)(x)=\sum_{c\in[4]}\X^{(3)}\big((\bX-\X)^{(1)}(f)(\cdot,x_c)\big)(x_{\hc})\,.
    \end{equs}
 If $f\in H^\alpha$ with $\alpha>\frac{1}{2}$, \eqref{eq:bX-X_rand_op} implies that $(\bX-\X)^{(1)}(f)(y_c,x_c)$ will be of regularity $\min(\frac{1}{2},\alpha-\frac12)$ in the direction of $y_c$, which is bigger than $-\frac12$, so that one can use \eqref{eq:X_rand_op} to obtain than $\cN(\X,\bX-\X,f)$ is of regularity $\min(-\frac{1}{2},\alpha-\frac32)$. Therefore, $f\mapsto\cN(\X,\bX-\X,f)$ is indeed better behaved than $\Xtwonm(f)$. All the remaining contributions in $\pXtwonm^{\det}$ are even better behaved, and can be treated in the same way. The proof is also similar for all the terms in $\pXtwom^{\det}$ and $\pXdotX^{\det}$, that again we write as some compositions of operators.

To help makes things clear, we detail how this argument works in the case of the biggest operator $\pSym$. 
The worst terms in $\pSym^{\det}$ are 
\[
f\mapsto\cN(\bX-\X,\cL^{-1}\XdotX(f),\X) \quad
\text{and}
\quad
f\mapsto\XdotX\big(\cL^{-1}\cN(\bX-\X,f,\X)\big)\;.
\]
If $f\in H^\alpha$ with $\alpha>\frac{1}{2}$, by the estimate \eqref{eq:regXdotX}, we have that $\XdotX(f)$ is of regularity $\min(-\frac12,\alpha-2)$, so that $\cL^{-1}\XdotX(f)$ is of regularity $\min(\frac32,\alpha)>\frac12$. 
It follows by the deterministic estimate stated in Lemma~\ref{lem:A2} that one can take the non-melonic pairing of $\cL^{-1}\XdotX(f)$ with $\bX-\X$ which is of regularity $-\epsilon$ as soon as $\epsilon < 1/2$.

On the other hand, one can rewrite 
$$\cN(\bX-\X,\cL^{-1}\XdotX(f),\X)(x)=\sum_{c\in[4]}\X^{(1)}\big((\bX-\X,\cL^{-1}\XdotX(f))_{L^2(\T)}(x_{\hc},\cdot)\big)(x_c)\,.$$ Now, since $\cL^{-1}\XdotX(f)$ is of regularity $>\frac{1}{2}$ (and hence so is $(\bX-\X,\cL^{-1}\XdotX(f))_{L^2(\T)}(x_{\hc},y_{\hc})$ in the direction $y_{\hc}$), one can use the estimate \eqref{eq:X_rand_op} to obtain that $\cN(\bX-\X,\cL^{-1}\XdotX(f),\X)$ is as expected of regularity $\min(0,\alpha-1)$.

Regarding $\XdotX\big(\cL^{-1}\cN(\bX-\X,f,\X)\big)$ for $f\in H^\alpha$ with $\alpha>\frac12$, we still can perform the non-melonic pairing of $f$ with $\bX-\X$ using Lemma~\ref{lem:A2}, and $\X^{(1)}$ can still act on the non-melonic pairing of $f$ with $\bX-\X$ and we can conclude since $\cL^{-1}$ increases the regularity by 2 while $\XdotX$ decreases it by two. All the other terms are better behaved, and follow in the same way, which confirms the fact that the objects made with the rough shift share the analytic properties of the objects made with $\X$.


\begin{appendix}
\section{Besov spaces}\label{app:besov}
We first give a precise definition of the Besov spaces we use in our analysis. 
We define Littlewood-Paley blocks, which are Fourier multipliers $(\Delta^{j})_{j=-1}^{\infty}$, by setting $\Delta^{-1} = \bone_{[0,2^{-1})}(|\nabla|_\infty)$ and, for $j \in \N$, $\Delta^j = \bone_{[2^{-1},1)}(2^{-j}|\nabla|_\infty)$ (we have $\Pi_N=\sum_{i\leqslant \log_2N}\Delta^i$ for $N$ dyadic). 

We then define two norms on $C^{\infty}(\T^d)$ by setting
\[
  \Vert\bigcdot\Vert_{H^\alpha(\T^d)}\eqdef \Vert\na^\alpha\bigcdot\Vert_{L^2(\T^d)} \text{ and } \Vert \bigcdot \Vert_{B^\alpha_{p,q}(\T^d)}\eqdef \Big\Vert 2^{i\alpha}\Vert \Delta^i\bigcdot\Vert_{L^p(\T^d)}\Big\Vert_{\ell^q_{i\geqslant-1}}\,,
\]
where for a countable set $A$ we write $\ell^q_{A} = \ell^{q}(\R^A)$. 
We define $B^{\alpha}_{p,q}$ to be the completion of $C^{\infty}(\T^d)$  under the norm $ \Vert \bigcdot \Vert_{B^\alpha_{p,q}(\T^d)}$. 
We make a frequent use of the fact that $B^\alpha_{2,2}=H^\alpha$, and of the following facts about Besov spaces.
\begin{lemma}
    Pick $\alpha\in\R$, $\delta>0$ and $p,q,p_+,p_-,q_+,q_-\in[1,\infty]$, $p_-<p_+$, $q_-<q_+$. Then, the following embeddings are compact:
    \begin{equs}
        &B^{\alpha}_{p,q}\hookrightarrow B^{\alpha-\delta}_{p,q}\,, \,\, B^{\alpha}_{p_+,q}\hookrightarrow B^{\alpha}_{p_-,q}\,,\,\, B^{\alpha}_{p,q_-}\hookrightarrow B^{\alpha}_{p,q_+}\,,\\&B^{\alpha}_{p_-,q}\hookrightarrow B^{\alpha-d(\frac{1}{p_-}-\frac{1}{p_+})}_{p_+,q}\,,\,\,B^{\alpha}_{p,q_+}\hookrightarrow B^{\alpha-\delta}_{p,q_-}\,.
    \end{equs}
\end{lemma}
\begin{lemma}\label{lem:A2}
Pick $\alpha\in\R$ and $p,p',q,q'\in[1,\infty]$, such that $\frac{1}{p} + \frac{1}{p'} = \frac{1}{q} + \frac{1}{q'} = 1$. Then, the $L^2$ pairing can be extended to a bilinear map $B^\alpha_{p,q}\times B^{-\alpha}_{p',q'}\rightarrow\R$.
\end{lemma}

Finally, we have the following standard heat kernel and Schauder estimates. 
\begin{lemma}
For $\alpha, \beta \in\R$ and $t > 0$ one has that $P_t:\cC^\alpha\rightarrow\cC^\beta$ is bounded with 
\begin{equs}\label{eq:schauder1}
    \Vert P_t( u )\Vert_{\cC^\beta}\lesssim t^{-\frac{\beta-\alpha}{2}} \Vert  u \Vert_{\cC^\alpha}\,.
\end{equs}
\end{lemma}

\begin{lemma}
For any $\alpha \in (-2,0) \setminus \N$ and $\epsilon>0$, 
\begin{equs}\label{eq:schauder2}
    \Vert \cL^{-1}( u) \Vert_{C_T\cC^{\alpha+2-\epsilon}}\lesssim T^{\frac{\epsilon}{2}}\Vert  u \Vert_{C_T\cC^\alpha}\,.
\end{equs}
\end{lemma}

\subsection{Bilocal Besov regularity}\label{app:bilocbesov}
We studying mixed terms in the equation such as $\XXv$ or $\Xvv$ by treating them as distributions over $\T^n_x\times\T^m_y$. 
In our analysis it is key to keep track of their anisotropic behavior, for instance being bounded in the the direction $y$ but only being $\cC^{\alpha}$ for $\alpha < 0$ in $x$. 
While such distributions certainly belong to $\cC^\alpha_{x,y} \eqdef \cC^\alpha  (\T^n_x\times\T^m_y)$, it will be better to work in a space like $L_y^\infty\cC_x^\alpha$. The following lemma gives straightforward but useful estimates for these spaces.  

\begin{lemma}\label{lem:bilocreg+}
Let $\alpha\geqslant0$, and $\epsilon>0$. Then, for $u\in \cC^{\alpha+\epsilon}(\T^n_x\times\T^m_y)$,
\begin{equs}\label{eq:bilocbesov+}
\Vert u\Vert_{L^\infty_y\cC^\alpha_x}= \bV \Vert u \Vert_{\cC^\alpha_x}\bV_{L^\infty_y}\lesssim \Vert u\Vert_{\cC^{\alpha+\epsilon}_{x,y}}\,.
\end{equs}
\end{lemma}
\begin{proof}
For a function of two variables $(x,y)$, we let $\Delta^i_x u(x,y)$ stand for the Littlewood-
Paley block $\Delta^i$ applied only in the variable $x$. When no subscript is present, it is understood that $\Delta^j\equiv\Delta^j_{x,y}$ acts on both variables $x$ and $y$.
\begin{equs}
    \bV \Vert u \Vert_{\cC^\alpha_x}\bV_{L^\infty_y} &=\sup_{y\in\T^m} \Vert u(y)\Vert_{\cC^{\alpha}_{x}}=\sup_{x,y\in\T^n\times\T^m}\sup_{i\geqslant -1} 2^{\alpha i} |  \Delta^i_x u (x,y)|\\&=\sup_{x,y\in\T^n\times\T^m}\sup_{i\geqslant -1} 2^{\alpha i} |  \Delta^i_x \sum_{j\geqslant i}\Delta^j u (x,y)|\leqslant \sup_{i\geqslant -1}2^{\alpha i}\sum_{j\geqslant i}\Vert \Delta^i_x\Delta^j u\Vert_{L^\infty_{x,y}}\\&\leqslant \sup_{i\geqslant -1}\sum_{j\geqslant i}2^{\alpha j}\Vert \Delta^j u\Vert_{L^\infty_{x,y}}\,,
\end{equs}
where we used the fact that since $\Delta^j u $ is bounded, $\Vert \Delta^i_x(\Delta^j u)\Vert_{L^\infty}\lesssim  \Vert \Delta^j u\Vert_{L^\infty}$, as well as the positivity of $\alpha$. Now, since all the terms in the sum over $j$ are positive, the supremum over $i$ is reached when $i=-1$, which implies that 
\begin{equs}
    \bV \Vert u \Vert_{\cC^\alpha_x}\bV_{L^\infty_y}\lesssim\sum_{j\geqslant -1}2^{\alpha j}\Vert \Delta^j u\Vert_{L^\infty_{x,y}}=\Vert u  \Vert_{B^{\alpha}_{\infty,1}(\T^n_x\times\T^m_y)}\,,
\end{equs}
and the claim follows using the embedding $\cC^{\alpha+\epsilon}\hookrightarrow B^{\alpha}_{\infty,1}$.
\end{proof}
We also have a Sobolev space version of the previous lemma:
\begin{lemma}
  Let $\alpha\geqslant0$, and $\epsilon>0$. Then the following inequality holds for $u\in H^{\alpha+\epsilon}(\T^n_x\times\T^m_y)$:
  \begin{equs}\label{eq:sobobilocal}
\Vert u\Vert_{L^2_y H_x^{\alpha}  } =   \bV\Vert u \Vert_{H^\alpha_x}\bV_{L^2_y}\lesssim\Vert u\Vert_{H_{x,y}^{\alpha+\epsilon}}\,.
  \end{equs}
\end{lemma}
\begin{proof}
Recall the notation $\Delta^i_x u(x,y)$ for the Littlewood-
Paley block acting the on the variable $x$. We have
\begin{equs}
  \Big\Vert \Vert u \Vert_{H^\alpha_x}\Big\Vert_{L^2_y}&=\Big(\int\sum_{i\geqslant-1}2^{2\alpha i}\Vert \Delta^i_xu\Vert_{L^2_x}^2(y)\rmd y \Big)^\frac{1}{2}\\ &=\Big(\sum_{i\geqslant-1}2^{2\alpha i}\int\Vert \Delta^i_xu\Vert_{L^2_x}^2(y)\rmd y \Big)^\frac{1}{2} \\&=\Big(\sum_{i\geqslant-1}2^{2\alpha i}\Vert \Delta^i_xu\Vert_{L^2_{x,y}}^2  \Big)^\frac{1}{2}=\bV 2^{\alpha i}\Vert \Delta^i_xu\Vert_{L^2_{x,y}}\bV_{\ell^2_i}\\&=\bV2^{\alpha i}\Vert \Delta^i_x\sum_{j\geqslant i}\Delta^ju\Vert_{L^2_{x,y}}\bV_{\ell^2_i}\leqslant \bV 2^{\alpha i}\sum_{j\geqslant i}\Vert \Delta^i_x\Delta^ju\Vert_{L^2_{x,y}}\bV_{\ell^2_i}\\&= \bV 2^{-\frac{\epsilon}{2} i}\sum_{j\geqslant i}2^{(\alpha+\frac{\epsilon}{2})j}\Vert \Delta^i_x\Delta^ju\Vert_{L^2_{x,y}}\bV_{\ell^2_i}\,.
\end{equs}
where we used the fact that since $\Delta^j u $ is bounded, $\Vert \Delta^i_x(\Delta^j u)\Vert_{L^2_{x,y}}\lesssim  \Vert \Delta^j u\Vert_{L^2_{x,y}}$, as well as the positivity of $\alpha+\frac{\epsilon}{2}$. Now, since all the terms of the sum over $j$ are positive, the sum over $j\geqslant i$ is bounded by the sum over $j\geqslant -1$, which implies that 
\begin{equs}
    \bV \Vert u \Vert_{H^\alpha_x}\bV_{L^2_y}\leqslant \bV 2^{-\frac{\epsilon}{2} i}\sum_{j\geqslant -1}2^{(\alpha+\frac{\epsilon}{2})j}\Vert \Delta^ju\Vert_{L^2_{x,y}}\bV_{\ell^2_i}=\Vert2^{-\frac{\epsilon}{2} i}\Vert_{\ell^2_i}\Vert u\Vert_{B_{2,1}^{\alpha+\frac{\epsilon}{2}}(\T^n_x\times\T^m_y)}\,,
\end{equs}
and the claim follows using the embedding $H^{\alpha+\epsilon}\hookrightarrow B^{\alpha+\frac{\epsilon}{2}}_{2,1}$.
\end{proof}
Finally, the following lemma deals with distributions of negative bilocal H\"older-Besov regularity: 
\begin{lemma}\label{lem:bilocreg-}
Let $\alpha\leqslant0$, and $\epsilon>0$. Then the following inequality holds for $u\in L^\infty_y\cC_x^{\alpha+\epsilon}(\T^n_x\times\T^m_y)$:
\begin{equs}\label{eq:bilocbesov-}
 \Vert u \Vert_{\cC^\alpha_{x,y}}\lesssim\Big\Vert \Vert u\Vert_{\cC^{\alpha+\epsilon}_{x}}\bV_{L^\infty_y}=\Vert u\Vert_{L^\infty_y\cC_x^{\alpha+\epsilon}}\,.
\end{equs}
\end{lemma}
\begin{proof}
Again, $\Delta^j_x$ is a Littlewood-
Paley block in $x$ only, as opposed to $\Delta^i$. We start from
\begin{equs}
     \Vert u \Vert_{\cC^\alpha_{x,y}}&=
     \sup_{i\geqslant -1}2^{\alpha i} \Vert \Delta^i u\Vert_{L^\infty_{x,y}}     =\sup_{y\in\T^m}\sup_{i\geqslant -1}2^{\alpha i} \Vert \Delta^i\sum_{j\leqslant i}\Delta^j_x u(y)\Vert_{L^\infty_{x}}  \\
     &\lesssim \sup_{y\in\T^m}\sup_{i\geqslant -1}2^{\alpha i} \sum_{j\leqslant i}\Vert \Delta^i\Delta^j_x u(y)\Vert_{L^\infty_{x}} 
     \lesssim  \sup_{y\in\T^m} \sup_{i\geqslant -1} \sum_{j\leqslant i}2^{\alpha j}\Vert \Delta^j_x u(y)\Vert_{L^\infty_{x}}\,,
\end{equs}
where we used the fact that $\alpha$ is negative and that $\Delta^j_x u $ being bounded in $x$ implies $\Vert \Delta^i(\Delta^j_x u)(y)\Vert_{L^\infty_x}$ $\lesssim$ $ \Vert \Delta^j_x u(y)\Vert_{L^\infty_x}$.
Now, since all the terms in the sum over $j$ are positive, the supremum over $i$ is reached when $i=\infty$, which implies that
\begin{equs}
    \Vert u \Vert_{\cC^\alpha_{x,y}}\lesssim\sum_{j\geqslant -1}2^{\alpha j}\Vert \Delta^j u\Vert_{L^\infty_{x,y}}=\Vert u  \Vert_{B^{\alpha}_{\infty,1}(\T^n_x\times\T^m_y)}\,,
\end{equs}
and the claim follows using the embedding $\cC^{\alpha+\epsilon}\hookrightarrow B^{\alpha}_{\infty,1}$.
\end{proof}

\section{Facts about the nonlinearity}\label{sec:nonlin_facts}
In this section, we finally prove some fact about the non-local interaction appearing in the $\mathrm{T}^4_{d}$ field theory. 
We first show that it is built from a norm in a wider family of norms over $\cC^\infty(\T^d)$ which we denote as the $M^p_c$ norms. 
We then consider the analog of the Sobolev norms constructed with the $M^p_c$ norms, and establish some interpolation inequalities and embeddings for these norms. 
Finally, we prove that it verifies two important Cauchy-Schwarz inequalities. 
 
Throughout this section, we fix $d \in \N$ with $d \geqslant 2$ and always take $c\in[d]$.

\begin{definition}[$M^p$ spaces]
Let $ u \in \cC^{\infty}(\T^d)$. 
We define a bounded and positive operator $M_c( u )\in \cL(\ell^{2}(\Z))$ on the space of square integrable complex-valued sequences indexed by $\Z$ by setting, for $a_{c},b_{c} \in \Z$, the corresponding matrix entry of $M_{c}( u )$ to be given by 
\begin{equs}    M_c( u )_{a_c,b_c}\eqdef\sum_{m_1,...,m_d\in\Z}\hat{ u }_{m_1,...,a_c,...,m_d}\hat{ u }_{m_1,...,b_c,...,m_d}\,.
\end{equs}
For $p\in(1,+\infty)$ we define the Banach space $M^p_c(\T^d)$ as the completion of $\cC^{\infty}(\T^d)$ under the norm:
\begin{equs}
    \Vert  u  \Vert_{M_c^p}\eqdef \mathrm{Tr} \Big( M_c( u )^{\frac{p}{2}} \Big)^{\frac{1}{p}}\,.
    \end{equs}
\end{definition}
Positivity of the above is straightforward to check. 
We also define what we call the $M^p$ norm as 
\begin{equ}
\Vert  u \Vert^p_{M^p}\eqdef \displaystyle\sum_{c=1}^d \Vert  u \Vert^p_{M^p_c}\;.
\end{equ}

\begin{remark}
    The name tensor field theory stems from the fact that a Fourier cut-off distribution $u$ over $\T^d$ can be seen as a tensor $\widehat{\Pi_N u}=(\hat{u}_m)_{m\in\Z_N^d}$. 
    Then, for $p$ an even integer, $\Vert u\Vert_{M^p_c}^p$ is a trace invariant of $\hat{u}$, that is a function of $\hat{u}$ invariant under the action of $\mathrm{O}(2N+1)^{\otimes d}$ on the indices of $\hat{u}$.
\end{remark}
\begin{lemma}
Let $1\leqslant p<q<\infty$. Then we have the compact embedding 
\begin{equs}
    \Vert  u  \Vert_{M_c^q}\lesssim \Vert  u  \Vert_{M_c^p}\,,
\end{equs}
which directly follows from the embedding $\ell^p(\Z)\hookrightarrow\ell^q(\Z)$. \end{lemma}
\begin{remark}
For any $c\in[d]$,  $\mathrm{Tr}M_c( u )=\Vert u \Vert_{L^2(\T^d)}^2<\infty$ by Parseval, 
so the space $M_c^2 (\T^d)$ is just $L^2(\T^d)$ \dash in particular,  $   \cI( u )^{\frac14}\lesssim\Vert u \Vert_{L^2}$. 
\end{remark}
Next we state some interpolation and embedding inequalities. 
\begin{lemma}[$M^4_c-$Sobolev interpolation 1]
For any $\beta\geqslant \alpha \geqslant 0$,
    \begin{equs}\label{eq:interpol}
        \Vert \langle \nabla_c\rangle^\alpha u\Vert_{M^4_c}\lesssim \Vert u\Vert_{M^4_c}^{1-\frac{\alpha}{\beta}}\Vert \langle\nabla_c\rangle^{\beta}u\Vert^{\frac{\alpha}{\beta}}_{M^4_c}\,.
    \end{equs}
\end{lemma}
\begin{proof}
In the proof, we denote $M_c=M_c(u)$. Using H\"older's inequality, we have
\begin{equs}
    \Vert \langle \nabla_c\rangle^\alpha u\Vert_{M^4_c}^4&=\sum_{p_c,q_c} \langle p_c\rangle^{2\alpha}\langle q_c\rangle^{2\alpha}M_c(p_c,q_c)^2\\&=\sum_{p_c,q_c} M_c(p_c,q_c)^{2(1-\frac\alpha\beta)}\Big( \langle p_c\rangle^{2}\langle q_c\rangle^{2}M_c(p_c,q_c)^{\frac{2}{\beta}}\Big)^{\alpha}\\&\leqslant\Vert M_c(p_c,q_c)^{2(1-\frac\alpha\beta)}\Vert_{\ell^{\frac{1}{1-\frac{\alpha}{\beta}}}_{p_c,q_c}}\Vert \Big( \langle p_c\rangle^{2}\langle q_c\rangle^{2}M_c(p_c,q_c)^{\frac{2}{\beta}}\Big)^{\alpha}\Vert_{\ell^{\frac{\beta}{\alpha}}_{p_c,q_c}}\\&=\Vert M_c(p_c,q_c)\Vert_{\ell_{p_c,q_c}^{2}}^{2(1-\frac{\alpha}{\beta})}\Vert\langle p_c\rangle^{\beta}\langle q_c\rangle^{\beta}M_c(p_c,q_c)\Vert_{\ell^2_{p_c,q_c}}^{2\frac{\alpha}{\beta}}\,.
\end{equs}
\end{proof}
\begin{lemma}[$M^4_c-$Sobolev embedding]
    Let $\alpha>\frac{1}{4}$. It holds
    \begin{equs}\label{eq:M4sobemb}
        \Vert u\Vert_{L^2}\lesssim\Vert \angp{\nabla_c}{\alpha}u\Vert_{M^4_c} \,.
        \end{equs}
\end{lemma}
\begin{proof}
Again denoting $M_c=M_c(u)$, by Cauchy-Schwarz we have
\begin{equs}
    \Vert u \Vert_{L^2}^2&=\mathrm{Tr}M_c=\sum_{p_c}M_c(p_c,p_c)=\sum_{p_c,q_c} M_c(p_c,q_c)\delta_{p_c,q_c}\\
    &=\sum_{p_c,q_c} M_c(p_c,q_c)\angp{p_c}{\alpha}\angp{q_c}{\alpha}\delta_{p_c,q_c}\angp{p_c}{-\alpha}\angp{q_c}{-\alpha}\\
    &\leqslant \Big(\sum_{p_c,q_c} M_c(p_c,q_c)^2\angp{p_c}{2\alpha}\angp{q_c}{2\alpha}\Big)^{\frac{1}{2}}
     \Big(\sum_{p_c,q_c}\delta_{p_c,q_c}^2 \angp{p_c}{-2\alpha}\angp{q_c}{-2\alpha}\Big)^{\frac{1}{2}}\\
     &\leqslant \Vert \angp{\nabla_c}{\alpha}{u} \Vert^2_{M^4_c} \Big(\sum_{p_c} \angp{p_c}{-4\alpha}\Big)^{\frac{1}{2}}\lesssim\Vert \angp{\nabla_c}{\alpha}u\Vert^2_{M^4_c}\,.
\end{equs}
\end{proof}
\begin{corollary}[$M^p_c-$Sobolev interpolation 2]
For all $\theta\in(\frac{1}{4},1]$, we have
 \begin{equs}\label{eq:interpol2}
     \Vert u\Vert_{L^2}\lesssim \Vert u\Vert_{M^4_c}^{1-\theta}\Vert u\Vert_{H^{1}}^{\theta}\,.
 \end{equs}
\end{corollary}        
\begin{proof}
   Let $\theta_\epsilon=\frac{\frac{1}{4}+\epsilon}{1-\epsilon}$. 
Using \eqref{eq:M4sobemb}, \eqref{eq:interpol}, and \eqref{eq:sobobilocal}, uniform in $\epsilon > 0$, we have
    \begin{equs}
        \Vert u\Vert_{L^2}&\lesssim \Vert \angp{\nabla_c}{\frac{1}{4}+\epsilon}u\Vert_{M^4_c}\lesssim \Vert u\Vert^{1-\theta_{\epsilon}}_{M^4_c}\Vert \angp{\nabla_c}{1-\epsilon}u\Vert^{\theta_{\epsilon}}_{M^4_c}\\
        &\lesssim \Vert u\Vert^{1-\theta_\epsilon}_{M^4_c}\Vert \angp{\nabla_c}{1-\epsilon}u\Vert^{\theta_\epsilon}_{L^2}\lesssim
        \Vert u\Vert^{1-\theta_\epsilon}_{M^4_c}\Vert u\Vert_{H^1}^{\theta_\epsilon}\,.
    \end{equs}
We finish by taking $\epsilon \downarrow 0$.
\end{proof} 

\begin{lemma}[Cauchy-Schwarz inequalities]
    Fix two smooth functions $\phi,\psi$ on $\T^d$. It holds
    \begin{equs}        \label{eq:CS1}|\big(\cN^c(\phi,\psi,\phi),\psi\big)_{L^2(\T^d)}|&\leqslant    \big(\cN^c(\psi,\psi,\phi),\phi\big)_{L^2(\T^d)}     \,,\\    \label{eq:CS2}|\big(\cN^c(\phi,\psi,\phi),\phi\big)_{L^2(\T^d)}|&\leqslant    \big(\cN^c(\psi,\psi,\phi),\phi\big)^{\frac12}_{L^2(\T^d)}  \|\phi\|_{M_c^4(\T^d)}   \,.
    \end{equs}
\end{lemma}
\begin{proof}
    The proof immediately follow from an application of the Cauchy-Schwarz inequality, after rewriting a bit the pairings. For the first one, denoting by $\theta(x_c,y_c)=\big(\phi(x_c,\cdot),\psi(y_c,\cdot)\big)_{L^2(\T_{\hc}^{d-1})}$ the melonic pairing of $\phi$ with $\psi$, we have by Cauchy-Schwarz inequality
\begin{equs}    |\big(\cN^c(\phi,\psi,\phi),\psi\big)_{L^2(\T^d)}|=\Big|\int_{\T^2}\theta(x_c,y_c)\theta(y_c,x_c)\rmd x_c\rmd y_c\Big|\lesssim \Big|\int_{\T^2}\theta ^2(x_c,y_c)\rmd x_c\rmd y_c\Big|=\big(\cN^c(\psi,\psi,\phi),\phi\big)_{L^2(\T^d)}\,.
\end{equs}
Similarly, writing $\varphi=(x_c,y_c)=\big(\phi(x_c,\cdot),\phi(y_c,\cdot)\big)_{L^2(\T_{\hc}^{d-1})}$, we have
\begin{equs}    |\big(\cN^c(\phi,\psi,\phi),\psi\big)_{L^2(\T^d)}|=\Big|\int_{\T^2}\theta(x_c,y_c)\varphi(y_c,x_c)\rmd x_c\rmd y_c\Big|\lesssim \Big|\int_{\T^2}\theta ^2(x_c,y_c)\rmd x_c\rmd y_c\Big|^{\frac12}\Big|\int_{\T^2}\varphi ^2(x_c,y_c)\rmd x_c\rmd y_c\Big|^{\frac12}\,,
\end{equs}
which yields the desired result.
\end{proof}

\section{Controlling mixed terms in $d=4$}\label{sec:mixedterms}
In this section, we gather the bounds over the mixed term of the solution to the Langevin dynamic/the Barashkov \& Gubinelli drift and the stochastic objects, that we need in order to study the global in time existence of the Langevin dynamic as well as the renormalization of the Bou\'e-Dupuis formula, in Propositions~\ref{prop:apriori} and~\ref{prop:BG_energy}. 
For the sake of concreteness, we state these results in $d=4$ but they also carry over as (much less optimal) estimates for $d< 4$.

The regularity exponents for both of the above settings are the same and the arguments are essentially the same as well, therefore we  adopt a notation that allows us to state the needed mixed term estimates for both settings. 

Throughout this section, $v$ will be a smooth function $[0,T]\times\T^d\rightarrow\R$, and we will write $v_t=v(t,\cdot)$. For such a function $v$, we use the notation $\cK(v_t)=\Vert v_t\Vert^2_{H^{1-\epsilon}(\T^4)}$ and $\cI(v_t)=\Vert v_t\Vert^4_{M^{4}(\T^4)}$. 
We also recall that for $d\in\{3,4\}$, $\bbX_d^{\LD}=\bbX^{\LD}_{f,d}\cup\bbX^{\LD}_{o,d}$ is the collection of all stochastic objects necessary to handle the fixed point problem in dimension $d$, and that we see it as an element of the product of the Banach spaces in which the objects live. We denote $\bbX_d^{\LD}(t)\eqdef\{\tau(t)|\tau\in\bbX_d^{\LD}\}$ which we endow with the norm:
\begin{equs}
    \Vert \bbX_d^{\LD}(t)\Vert=\max\Big(  \max_{\tau\in\bbX^{\LD}_{f,d}}\Vert \tau(t)\Vert_{\cC^{\beta_\tau-\epsilon}},\max_{\tau\in\bbX^{\LD}_{o,d}}\Vert \tau(t)\Vert_{\cL(C_tH^{\alpha_\tau},\cC^{\beta_\tau-\epsilon})} \Big)\,,
\end{equs}
where $\alpha_\tau$ and $\beta_\tau$ are the inner and outer regularities of $\tau$ as stated in Lemmas~\ref{lem:stoob} and \ref{lem:ranop}. We also recall that $\bbX^{\LD}_d$ is endowed with the norm $\Vert\cdot\Vert_T$ introduced in $d=3$ in \eqref{eq:normbbX} and the we extended to the case $d=4$ in Section~\ref{sec:T44}. Note that for every $t\in[0, T]$, it holds $\Vert\bbX_d^{\LD}(t)\Vert\lesssim\Vert\bbX_d^{\LD}\Vert_T$.

In the sequel, we often drop the subscript and superscript notation, writing $\|\bbX(t)\|$ to refer to both $\Vert \bbX_d^{\LD}(t)\Vert$ and $\Vert \bbX^{\BG}(t)\Vert$ as given in \eqref{eq:bbXnormBG}. 

Throughout this section, we write  $A(t)\lesssim_{\bbX} B(t)$ whenever there exists two positive constants $C$ and $c$ such that $A(t) \leqslant C \Vert \bbX(t)\Vert^c B(t)$.

\begin{lemma} 
Recall the random field $\cS_{t}$ defined in Definition~\ref{def:randomfields} as an LD object and Definition~\ref{def:BGrandomfields} as a BG object. There exists two positive constants $C,\kappa$ such that it holds
   \begin{equs}\label{eq:Sv}
       |(\cS_t,v_t)|\leqslant C(\delta^{-1}\Vert\bbX(t)\Vert)^\kappa+\delta \cK(v_t)
   \end{equs} 
uniformly in $t\in[0,T]$ and $\delta\in(0,1]$.
\end{lemma}
\begin{proof}
    \begin{equs}
         |(\cS_t,v_t)|&\leqslant \Vert\cS(t)\Vert_{\cC^{-\frac12-\epsilon}}\Vert v_t\Vert_{B^{\frac12+\epsilon}_{1,1}}\lesssim_{\bbX} \Vert v_t\Vert_{H^{\frac{1}{2}+2\epsilon}}\lesssim_{\bbX}\cK( v_t)^\frac{1}{2}\,,
    \end{equs}
    and we conclude using Young's inequality.
\end{proof}
\begin{lemma}\label{lem:CSinaction}
 There exists two positive constants $C,\kappa$ such that it holds
 \begin{equs}
     \label{eq:vXXv}\bv\vXXv(t)\bv&\leqslant C(\delta^{-1}\Vert\bbX(t)\Vert)^\kappa+\delta\big(\cI(v_t)+\cK(v_t)\big)\,,\\\label{eq:XXvv}\bv\XXvv(t)\bv&\leqslant C(\delta^{-1}\Vert\bbX(t)\Vert)^\kappa+\delta\big(\cI(v_t)+\cK(v_t)\big)\,,\\\label{eq:XvXv}\bv\XvXv(t)\bv&\leqslant C(\delta^{-1}\Vert\bbX(t)\Vert)^\kappa+\delta\big(\cI(v_t)+\cK(v_t)\big)\,,
 \end{equs}
uniformly in $t\in[0,T]$ and $\delta\in(0,1]$.
\end{lemma}
\begin{proof}
We first prove \eqref{eq:vXXv}. We view $\pXtwom^c$ as a time indexed space operator, and at fixed $t\geqslant0$ we denote its kernel by $\pXtwom^c_t(x_c,y_c)$. As a first step, observe that
\begin{equs}
\vXXv^c(t)&=\int  \pXtwom^c_t(x_c,y_c) v_t(x_{\hc},x_c)v_t(x_{\hc},y_c)\rmd x_{\hc}\rmd x_c\rmd y_c\\&=\int  \langle\nabla_{x_c}\rangle^{-\frac{3}{4}-\frac\epsilon2}\langle\nabla_{y_c}\rangle^{-\frac{3}{4}-\frac\epsilon2}\pXtwom^c_t(x_c,y_c)
\langle\nabla_c\rangle^{\frac{3}{4}+\frac\epsilon2}v_t(x_{\hc},x_c)\langle\nabla_c\rangle^{\frac{3}{4}+\frac\epsilon2}v_t(x_{\hc},y_c)\rmd x_{\hc} \rmd x_c\rmd y_c\\&=\int \langle\nabla\rangle^{-\frac{3}{4}-\frac\epsilon2}\circ\pXtwom^c_t\circ\langle\nabla\rangle^{-\frac{3}{4}-\frac\epsilon2}\Big(\int \langle\nabla_c\rangle^{\frac{3}{4}+\frac\epsilon2}  v_t(x_{\hc},x_c)\langle\nabla_c\rangle^{\frac{3}{4}+\frac{\epsilon}{2}} v_t(x_{\hc},\cdot)\rmd x_{\hc}\Big)(x_c)\rmd x_c\,,
\end{equs}
where the notation $\circ$ denotes the composition of operators. Therefore, 
\begin{equs}
\bv\vXXv^c(t)\bv&\leqslant \Big\Vert \langle\nabla\rangle^{-\frac{3}{4}-\frac\epsilon2}\circ\pXtwom^c_t\circ\langle\nabla\rangle^{-\frac{3}{4}-\frac\epsilon2}\Big(\Vert \int\langle\nabla_c\rangle^{\frac{3}{4}+\frac\epsilon2}  v_t(x_{\hc},x_c)\langle\nabla_c\rangle^{\frac{3}{4}+\frac\epsilon2} v_t(x_{\hc},\cdot)\rmd x_{\hc}\Vert_{L^2_{x_c}}\Big)\Big\Vert_{L^2}\\&\leqslant\Vert  \langle\nabla\rangle^{-\frac{3}{4}-\frac\epsilon2}\circ\pXtwom^c_t\circ\langle\nabla\rangle^{-\frac{3}{4}-\frac\epsilon2} \Vert_{\cL(C_TL^2,C_TL^2)} \bV  \Vert \int\langle\nabla_c\rangle^{\frac{3}{4}+\frac\epsilon2}  v_t(x_{\hc},x_c)\langle\nabla_c\rangle^{\frac{3}{4}+\frac\epsilon2} v_t(x_{\hc},y_c)\rmd x_{\hc}\Vert_{L^2_{x_c}} \bV_{L^2_{y_c}}\\&\leqslant \Vert  \langle\nabla\rangle^{-\frac{3}{4}-\frac\epsilon2}\circ\pXtwom^c_t\circ\langle\nabla\rangle^{-\frac{3}{4}-\frac\epsilon2} \Vert_{\cL(C_TL^2,C_TL^2)} \bV \int\langle\nabla_c\rangle^{\frac{3}{4}+\frac\epsilon2}  v_t(x_{\hc},x_c)\langle\nabla_c\rangle^{\frac{3}{4}+\frac\epsilon2} v_t(x_{\hc},y_c)\rmd x_{\hc} \bV_{L^2_{x_c,y_c}}\,.
\end{equs}
Here, we use the fact that for every $m\in\R$, $\langle \nabla\rangle^m$ is a continuous operator $H^{s}\rightarrow H^{s-m}$ for all $s\in \R$. Hence,
\begin{equs}
    \Vert  &\langle\nabla\rangle^{-\frac{3}{4}-\frac\epsilon2}\circ\pXtwom^c_t\circ\langle\nabla\rangle^{-\frac{3}{4}-\frac\epsilon2} \Vert_{\cL(C_TL^2,C_TL^2)}\\&\leqslant \Vert\langle\nabla\rangle^{-\frac{3}{4}-\frac\epsilon2}\Vert_{\cL(L^2,H^{\frac{3}{4}+\frac\epsilon2})}\Vert\pXtwom(t)\Vert_{\cL(C_tH^{\frac34+\frac\epsilon2},H^{-\frac34-\frac\epsilon2})}\Vert\langle\nabla\rangle^{-\frac{3}{4}-\frac\epsilon2}\Vert_{\cL(H^{-\frac{3}{4}-\frac\epsilon2},L^2)}\\&\lesssim\Vert\pXtwom(t)\Vert_{\cL(C_tH^{\frac34+\frac\epsilon2},H^{-\frac34-\frac\epsilon2})}\lesssim_{\bbX}1\,.
\end{equs}
Moreover, using the interpolation formula \eqref{eq:interpol}, we have that
\begin{equs}
     &\bV \int\langle\nabla_c\rangle^{\frac{3}{4}+\frac\epsilon2}  v_t(x_{\hc},x_c)\langle\nabla_c\rangle^{\frac{3}{4}+\frac\epsilon2} v_t(x_{\hc},y_c)\rmd x_{\hc} \bV_{L^2_{x_c,y_c}}^{\frac{1}{2}}=\Vert \langle\nabla_c\rangle^{\frac{3}{4}+\frac\epsilon2} v_t\Vert_{M^4_c}\lesssim \Vert  v_t\Vert_{M^4_c}^{\frac{1-10\epsilon}{4-8\epsilon}}\Vert \langle\nabla_c\rangle^{1-2\epsilon}  v_t\Vert_{M^4_c}^{\frac{3+2\epsilon}{4-8\epsilon}}\,,
\end{equs}
and $\Vert \langle\nabla_c\rangle^{1-2\epsilon}  v_t\Vert_{M^4_c}\lesssim \Vert \langle\nabla_c\rangle^{1-2\epsilon}  v_t\Vert_{L^2}=\Vert  v_t \Vert_{L^2_{x_{\hc}}H^{1-2\epsilon}_{x_c}}\lesssim\Vert  v_t\Vert_{H^{1-\epsilon} }  $ by \eqref{eq:sobobilocal}, which yields 
\begin{equs}
   \bv\vXXv (t) \bv\lesssim_{\bbX}   \cI(v_t)^{\frac{1-10\epsilon}{16-32\epsilon}}\cK(v_t)^{\frac{3+2\epsilon}{8-16\epsilon}}\,.
\end{equs}
Since the sum of the exponents of $\cI(v_t)$ and $\cK(v_t)$ is smaller than $1$, the claim follows by Young's inequality.

The proof of \eqref{eq:XXvv} is very similar and we obtain
\begin{equs}\label{eq:interm}
   \bv\XXvv (t)\bv\lesssim_{\bbX} \cI(v_t)^{\frac{1-10\epsilon}{16-32\epsilon}}\cK(v_t)^{\frac{3+2\epsilon}{8-16\epsilon}}\,,
\end{equs}
while \eqref{eq:XvXv} follows from \eqref{eq:XXvv} by the Cauchy-Schwarz inequality \eqref{eq:CS1} we implies that
\begin{equs}
    \bv\XvXv(t)\bv\leqslant\bv\XXvv(t)\bv\,.
\end{equs}
At this stage, we can conclude using the fact that $\pXtwonm$ is a stochastic object that can be defined without any renormalization (in the sense that $\pXtwonm(f)=\cN(\bX,\bX,f)$). 
\end{proof}
\begin{remark}
    Note that in the above proof, when applying the Cauchy-Schwarz inequality, it is of utmost importance that one can leverage the fact that $\pXtwonm(f)=\cN(\bX,\bX,f)$. Indeed, it is of course also true that 
    \begin{equs}           \bv\XvXv(t)\bv\leqslant\big(\cN(v_t,\bX_t,\bX_t),v_t\big)\,,
    \end{equs}
    but $\big(\cN(v_t,\bX_t,\bX_t),v_t\big)\neq \vXXv$, and it is therefore not controlled uniformly in the cut-off.
\end{remark}
\begin{lemma} 
There exists two positive constants $C,\kappa$ such that it holds
 \begin{equs}
     \label{eq:Xvvv}\bv\Xvvv(t)\bv&\leqslant C(\delta^{-1}\Vert\bbX(t)\Vert)^\kappa+\delta\big(\cI(v_t)+\cK(v_t)\big)
\end{equs}
uniformly in $t\in[0,T]$ and $\delta\in(0,1]$.
\end{lemma}
\begin{proof}
Using the Cauchy-Schwarz inequality \eqref{eq:CS2}, and again leveraging the fact that $\pXtwonm$ can be defined without any new renormalization, we have
\begin{equs}
    \bv\Xvvv(t)\bv\leqslant \bv\vvvv(t)\bv^{\frac{1}{2}} \bv\XXvv(t)\bv^{\frac{1}{2}}=\cI(v_t)^{\frac{1}{2}}\bv\XXvv(t)\bv^{\frac{1}{2}}\,.
\end{equs} 
Now, using \eqref{eq:interm} yields
\begin{equs}
    \bv\Xvvv(t)\bv\lesssim_{\bbX}  \cI(v_t)^{\frac{1}{2}+\frac{1-10\epsilon}{32-64\epsilon}}\cK(v_t)^{\frac{3+2\epsilon}{16-32\epsilon}}\,,
\end{equs}
and since the sum of the exponents of $\cI(v)$ and $\cK(v)$ is smaller than one, the claim follows by Young's inequality.
\end{proof}

\end{appendix}

\bibliographystyle{Martin}
\bibliography{refs.bib} 

\end{document}